\documentclass{amsart}

\usepackage{graphicx,caption,subcaption}

\usepackage{amssymb,amsmath,amsfonts,amsthm,amscd,amstext,mathtools}
\usepackage{epstopdf}
\usepackage{microtype}
\usepackage{mathrsfs}
\usepackage{stmaryrd}	
\usepackage{enumitem}	
\usepackage{verbatim} 
\usepackage{subcaption}	
\usepackage{url} 
\usepackage{hyperref}
\usepackage{xcolor}
\usepackage{tikz}

\numberwithin{equation}{section}
\theoremstyle{plain}

\newtheorem{theorem}{Theorem}[section]
\newtheorem{claim}[theorem]{Claim}
\newtheorem{lemma}[theorem]{Lemma}
\newtheorem{conj}[theorem]{Conjecture}

\newtheorem{proposition}[theorem]{Proposition}
\newtheorem{corollary}[theorem]{Corollary}
\theoremstyle{definition}
\newtheorem{definition}[theorem]{Definition}
\newtheorem*{definition*}{Definition}

\newtheorem{remark}{Remark}[section]

\newcommand{\ind}{{\sf 1}}

\newcommand{\bP}{\mathbf{P}}

\newcommand{\bE}{\mathbf{E}}
\newcommand{\bbP}{\mathbb{P}}
\newcommand{\bbE}{\mathbb{E}}

\newcommand{\R}{\mathbb{R}}
\newcommand{\bbN}{\mathbb{N}}
\newcommand{\N}{\mathbb{N}}

\newcommand{\Var}{\mathbb{V}\mathrm{ar}}

\newcommand{\cA}{{\ensuremath{\mathcal A}} }
\newcommand{\cF}{{\ensuremath{\mathcal F}} }
\newcommand{\cC}{{\ensuremath{\mathcal C}} }
\newcommand{\cN}{{\ensuremath{\mathcal N}} }
\newcommand{\cL}{{\ensuremath{\mathcal L}} }
\newcommand{\cZ}{{\ensuremath{\mathcal Z}} }
\newcommand{\cM}{{\ensuremath{\mathcal M}} }
\newcommand{\cX}{\mathcal{X}}
\newcommand{\cY}{\mathcal{Y}}
\renewcommand{\epsilon}{\varepsilon}
\renewcommand{\phi}{\varphi}

\newcommand{\ga}{\alpha}
\newcommand{\gb}{\beta}
\newcommand{\gga}{\gamma}            
\newcommand{\gd}{\delta}
  
\newcommand{\eps}{\varepsilon}      

\newcommand{\gk}{\kappa}

\newcommand{\gl}{\lambda}

\newcommand{\gs}{\sigma}

\newcommand{\gth}{\theta}
\newcommand{\gU}{\Upsilon}

\newcommand{\gh}{\eta}

\newcommand{\dd}{\mathrm{d}}
\renewcommand{\preceq}{\preccurlyeq}		
\renewcommand{\succeq}{\succcurlyeq}		
\renewcommand{\hat}{\widehat}
\renewcommand{\tilde}{\widetilde}
\newcommand{\ol}{\overline}
\newcommand{\ul}{\underline}

\newcommand{\mueps}{\mu_{\eps}}
\newcommand{\llbis}{\leq_{\gth}}
\newcommand{\ggbis}{\geq_{\gth}}
\newcommand{\Ai}{\ensuremath{\mathrm{Ai}}}
\newcommand{\Bi}{\ensuremath{\mathrm{Bi}}}
\newcommand{\rma}{\mathrm{a}}
\newcommand{\bbT}{\mathbb{T}}
\newcommand{\supercrit}{{\mathrm{sup}}}
\newcommand{\crit}{{\mathrm{crit}}}
\newcommand{\subcrit}{{\mathrm{sub}}}
\newcommand{\supdecr}{ {\mathrm{sup}\text{-}\mathrm{d}} }
\newcommand{\shifted}{{\mathrm{red}}}
\newcommand{\twhite}{{\mathrm{white}}}
\newcommand{\tred}{{\mathrm{red}}}
\newcommand{\anyrg}{{\ensuremath{\ast}}}

\newcommand{\shft}{{\ensuremath{Q_T}}}
\newcommand{\shftend}[1]{{\ensuremath{Q_T^{#1}}}}
\newcommand{\gsens}{\ensuremath{\mathcal{S}_{\gh}}}
\newcommand{\gsensrg}{\ensuremath{{\gsens^{\anyrg}}}}
\newcommand{\rgens}{\{\supercrit,\crit,\subcrit\}}
\newcommand{\meas}{\mathtt{C}}
\newcommand{\measN}{\mathtt{C}_N}
\newcommand{\crem}{{\mathtt{CREM}}}
\newcommand{\ncrem}{{N\text{-}\mathtt{CREM}}}

\newcommand{\Ber}{\mathrm{Ber}}
\newcommand{\nat}{\mathrm{nat}}

\title[Time-inhomogeneous N-BBM]{Time-inhomogeneous \texorpdfstring{$N$}{N}-particle Branching Brownian Motion and the continuous random energy model}

\author[A. Legrand]{Alexandre Legrand}
\address{Department of Mathematics, Università degli Studi di Padova and Istituto Nazionale di Alta Matematica.}
\author[P. Maillard]{Pascal Maillard}
\address{Institut de Mathématiques de Toulouse, Université de Toulouse and Institut Universitaire de France.}
\email{legrand@math.unipd.it; pascal.maillard@math.univ-toulouse.fr}


\makeatletter
\@namedef{subjclassname@2020}{%
	2020 Mathematics Subject Classification}
\makeatother

\begin{document}
\begin{abstract}
The $N$-particle branching Brownian motion ($N$-BBM) is a branching Markov process which describes the evolution of a population of particles undergoing reproduction and selection. It has attracted a lot of interest due to its relations to the study of front propagation phenomena on the one hand, and to (hierarchical) physical $p$-spin models on the other hand, among which the continuous random energy model (CREM). This paper investigates the asymptotic displacement of the $N$-BBM in a time-inhomogeneous setting, and when the time horizon $T$ and the number of particles $N$ jointly tend to infinity. We estimate the maximal displacement of the process up to the second order, and show that the latter undergoes a \emph{transition} at the scale $\log N\approx T^{1/3}$. In particular when $\log N\ll T^{1/3}$ we recover the \emph{Brunet-Derrida behavior} which was proven in a time-homogeneous setting and for $T\to+\infty$ then $N\to+\infty$. Furthermore, our results can also be interpreted from the perspective of \emph{algorithmic optimisation} on some spin glass models, since the time-inhomogeneous $N$-BBM can be seen as the realization of an optimization procedure called \emph{beam search} on the aforementioned CREM. The CREM has been proven by L.~Addario-Berry and the second author to undergo an \emph{algorithm hardness threshold} phenomenon, and the results of the present paper describe precisely the efficiency of the beam search algorithm around that threshold.
\end{abstract}
\maketitle
\tableofcontents
	
\section{Introduction and main results}\label{sec:intro}
\subsection{Branching Brownian motion and $N$-BBM}
	
	The \emph{branching Brownian motion} (BBM) can be described as follows. At time $t=0$ we consider a (non-empty) initial configuration of particles on the real line, which all start moving as standard Brownian motions until some exponentially distributed random times with parameter $\gb_0$. All those movements and exponential ``clocks'' are taken independently from one another. When one of the exponential clocks rings, the corresponding particle splits into a random number $\xi\geq2$ of new ones at its location. Then, those particles start evolving the same way, independently, with their own exponential clocks. Following a generalization first introduced in~\cite{DS88}, in this paper we will be interested in BBM's with \emph{time-inhomogeneous} motion. More precisely, let $\gs:[0,1]\to (0,\infty)$ be a smooth function (twice continuously differentiable is enough): then for some fixed final time $T>0$, we assume that, at time $t\in[0,T]$, the infinitesimal variance of all the Brownian motions involved in the construction is given by $\gs^2(t/T)$. 
	
The BBM has generated a lot of interest in the last decades, notably due to its relation to the study of \emph{front propagation phenomena}, see Section~\ref{sec:motivation:NBBM} below and~\cite{BD97, BDMM06} in the time-homogeneous setting. Following the ideas therein, we define the time-inhomogeneous $N$-\emph{particle branching Brownian motion} ($N$-BBM) by adding the following \emph{selection} mechanism to the time-inhomogeneous BBM: we start from an initial configuration containing at most $N$ particles, $N\in\N$. At any time of a splitting event, we only keep the $N$ particles at the highest positions. We denote by $\cX^{N}_T$ the particle configuration of the $N$-BBM at time $T$, seen as a (finite) counting measure on $\R$ (full formal notation and construction of the $N$-BBM are presented more extensively below). We write $\max(\cX^{N}_T)$ for the maximal displacement of the process at time $T$, i.e. the position of the highest living particle from the $N$-BBM. Similarly, $\cX_T$ denotes the particle configuration of the BBM (without selection) at time $T$, and $\max(\cX_T)$ the maximal displacement of the BBM.
	
	The maximum displacement of the BBM has been extensively investigated in the literature, both for the time-homogeneous and inhomogeneous cases, see~\cite{Bra83, Bra78, DS88,McK75} or more recently~\cite{Aid13, MZ16, Mal15} among other works. On the other hand, studying the maximum of the homogeneous $N$-BBM is a more difficult matter: it was first done in~\cite{BD97} with heuristic methods, and later in \cite{Mai16} (see also \cite{BG10}). 
	This paper, to our knowledge, is the first to investigate the time-inhomogeneous $N$-BBM. More precisely, we study its asymptotic displacement, including the maximum position at time $T$.  See Figure~\ref{fig:NBBM:simu1} for a simulation.
	
	\begin{figure}[ht]
		\centering
		\includegraphics[width=0.75 \textwidth]{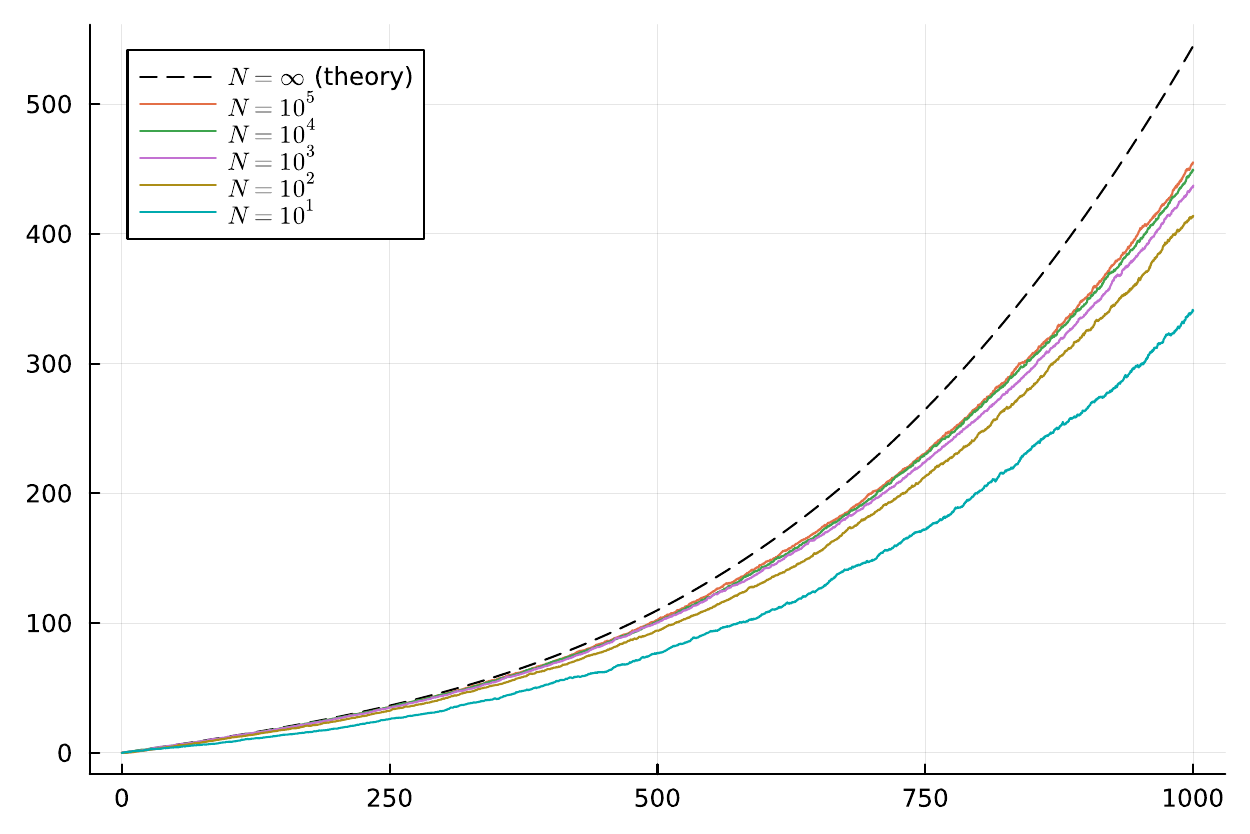}
		\caption{Running maximum of simulations of time-inhomogeneous $N$-BBM with varying $N$. Parameters: $T = 1000$, $\sigma(t) = 0.125+t^2$. The dashed curve corresponds to the theoretical position of the running maximum of the time-inhomogeneous BBM without selection ($N=\infty$), which is attained up to random $O(1)$ fluctuations, see \cite{BH15}.}
		\label{fig:NBBM:simu1}
	\end{figure}
	
	Moreover, in this paper we will be interested in the case where $N$ depends on $T$, with $N = N(T)\to\infty$ as $T\to\infty$. More precisely, we define
	\begin{equation}\label{eq:defL}
		L(T)\;:=\;\log N(T)\,,
	\end{equation}
	and we shall consider the regime where $1 \ll \log N(T) \ll T$. Note that the original BBM formally corresponds to $N=+\infty$. We will discuss in Sections~\ref{sec:crem} and~\ref{sec:beam-search} below why this setting has natural motivations and applications in the field of \emph{algorithmic optimisation} on some \emph{spin glass} models, beyond being interesting for the $N$-BBM in its own right. 
	
	\begin{remark}
		We stick to the convention that $N$ is chosen as a function of $T$, but we could also take $T$ as a function of $N$ or $N$ and $T$ as a function of a third parameter such that $N$ and $T$ both go to infinity. The results can be adapted straightforwardly.
	\end{remark}

	\subsection{Statement of results on the $N$-BBM}
	\subsubsection*{Notations}
	Throughout this paper, $\cC^2([0,1])$ denotes the set of $\cC^2$ functions from $[0,1]$ to $(0,+\infty)$ (in particular they are positive). Let the infinitesimal variance of the $N(T)$-BBM be given by $\gs^2(t/T)$, $t\in[0,T]$, for some $\gs\in\cC^2([0,1])$. Define
	\begin{equation}\label{eq:defv}
		v(s) \;=\; \int_0^s \sigma(s)\,\dd s\;,\qquad s\in[0,1],
	\end{equation}
	which is called the \emph{natural speed} of the BBM. Recall that $\xi$ denotes the offspring distribution of particles, and $\gb_0$ denotes the branching rate: in this paper we assume $\bbE[\xi^2]<+\infty$, and, using the Brownian scaling property, one can assume without loss of generality that $\gb_0=(2(\bbE[\xi]-1))^{-1}$.
	
	Let $\meas$ denote the set of all finite particle configurations, i.e. all finite counting measures on $\R$; and for $N\geq1$, let $\measN:=\{\mu\in\meas,\,\mu(\R)\leq N\}$. For $\mu\in\meas$ (resp. $\mu\in\measN$), the law of the BBM (resp. $N$-BBM) starting from the initial configuration $\mu$ will be denoted by $\bbP_{\mu}$ (we use the same notation for both, and what branching process is considered at any given time will always be clear from context). 
	
	For any $\mu\in\meas_{N(T)}$, define
	\begin{align}\label{eq:defshift}
		\shft(\mu)\;:=\; &\sup\big\{q\in\R\,\big|\, \exists \gk\in[0,1]: \mu\big([q-\gk\gs(0)L(T),+\infty)\big)\,\geq\, N(T)^{\gk}\big\}\\
		\notag =\;&\inf\big\{q\in\R\,\big|\, \forall \gk\in[0,1],\, \mu\big([q-\gk\gs(0)L(T),+\infty)\big)\,<\, N(T)^{\gk}\big\}\,.
	\end{align}
	We will see below that this term determines how the choice of the initial configuration $\mu\in\meas_{N(T)}$ reverberates on the displacement of the $N(T)$-BBM after a long time (see Section~\ref{sec:coupling:mainprops} for more details). It can be interpreted as some kind of ``entropy-position trade-off'' for the initial configuration: for instance, if the process starts from $M\in\{1,\ldots,N(T)\}$ particles all located at $y\in\R$, then the maximum displacement of the $N(T)$-BBM after a long time will be shifted approximately by $\shft(M \gd_{y})=y+\gs(0)\log M$.
	
	Let $\Ai$ and $\Bi$ denote respectively the Airy functions of first and second kind, and define
	\begin{equation}\label{def:Psi}
		\Psi(q) \;:=\; \frac{q^{2/3}}{2^{1/3}} \sup\big\{\gl\leq0\,\big|\, \Ai(\gl)\Bi(\gl+(2q)^{1/3})=\Ai(\gl+(2q)^{1/3})\Bi(\gl)\big\}\,<\,0\;,\quad q>0\;,
	\end{equation}
	and $\Psi(-q):=q+\Psi(q)$ for $-q<0$; and finally $\Psi(0):=-\frac{\pi^2}{2}$. It is proven in~\cite[Lemmata 1.7, A.4]{Mal15} that $q\mapsto \Psi(q)$ is well-posed and convex (hence continuous) on $\R$. Moreover, let  $\rma_1$ denotes the absolute value of the largest root of $\Ai$ (i.e. $\rma_1= 2.33811\ldots$); then one has $\Psi(q)\sim-\frac{\rma_1 q^{2/3}}{2^{1/3}}$ as $q\to+\infty$.
	
	Let us denote the positive and negative parts of a real number with
	\begin{equation}
		(x)^+\,:=\, x\ind_{\{x\geq0\}}\;,\qquad   (x)^-\,:=\, -x\ind_{\{x\leq0\}}\;,\qquad \forall\,x\in\R\,,
	\end{equation}
	and, for a function $f:\R\to\R$, let us write similarly $(f)^+(x):=(f(x))^+$ and $(f)^-(x):=(f(x))^-$, $x\in\R$. Furthermore, $o_{\bbP}(f(T))$ denotes a random quantity which, when divided by $f(T)$, converges to 0 in $\bbP_{\mu_T}$-probability as $T\to+\infty$. Finally, we say that a function $\gs\in\cC^2([0,1])$ \emph{changes its monotonicity finitely many times} if there exists $u_0=0<u_1<\ldots<u_p=1$ such that $\gs$ is monotonic (i.e. non-increasing or non-decreasing) on each $[u_{i-1},u_i]$, $1\leq i\leq p$. 
	
	\subsubsection*{Main result: asymptotic of the maximum}
	We now state the main result of this paper. 
	\begin{theorem}\label{thm:main} Let $\gs\in\cC^2([0,1])$, let $L(T)=\log N(T)\to+\infty$ as $T\to+\infty$, and denote by $\cX_T^{N(T)}$ the empirical measure on $\R$ at time $T>0$ of a $N(T)$-BBM with infinitesimal variance $\gs^2(\cdot/T)$, started from some initial configuration $\mu_T\in\meas_{N(T)}$, $T\geq0$. Let $\max(\cX_T^{N(T)})$ denote the maximal displacement of the process at time $T$. We have the following.
		
		$(i)$ \emph{(Sub-critical regime)} Assume $1\ll L(T)\ll T^{1/3}$. Then,
		\begin{equation}\label{eq:thm:subcrit}
			\max(\cX_T^{N(T)}) = \shft(\mu_T) + v(1)T\left(1 - \frac{\pi^2}{2L(T)^2}\right) + o_{\bbP}\left(\frac{T}{L(T)^2}\right)\,.
		\end{equation}
		
		$(ii)$ \emph{(Critical regime)} Assume $L(T)\sim \ga T^{1/3}$ for some $\ga\in(0,+\infty)$. Then,
		\begin{equation}\label{eq:thm:crit}
			\max(\cX_T^{N(T)}) = \shft(\mu_T) + v(1)T + \left[\int_0^1\frac{\gs(u)}{\ga^2}\Psi\Big(-\ga^3\frac{\gs'(u)}{\gs(u)}\Big)\dd u \right]T^{1/3} + o_{\bbP}\big(T^{1/3}\big)\,.
		\end{equation}
		
		$(iii)$ \emph{(Super-critical regime)} Assume $T^{1/3}\ll L(T)\ll T$, and that $\gs$ changes its monotonicity finitely many times. Then,
		\begin{equation}\label{eq:thm:supercrit}
			\max(\cX_T^{N(T)}) = \shft(\mu_T) + v(1)T + \left[\int_0^1(\gs')^+(u)\,\dd u\right]L(T) + o_{\bbP}\big(L(T)\big)\,.
		\end{equation}
	\end{theorem}
	In the super-critical regime~\eqref{eq:thm:supercrit}, if $\gs$ is non-increasing, Theorem~\ref{thm:main} gives little information. However, we complete it with the following result.
	\begin{proposition}\label{prop:main:gsdecreasing}
		Let $\gs\in\cC^2([0,1])$ be strictly decreasing, let $T^{1/3}\ll L(T)\leq +\infty$, and consider the initial particle configuration $\gd_0$ (i.e. one particle at the origin). Then as $T\to+\infty$, one has,
		\begin{equation}
			\max(\cX^{N(T)}_T) = v(1)T - \frac{\mathrm{a}_1}{2^{1/3}}\left[\int_0^1\gs(u)^{1/3}|\gs'(u)|^{2/3}\,\dd u\right]T^{1/3} + o_{\bbP}\big(T^{1/3}\big)\,.
		\end{equation}
	\end{proposition}
	Some assumptions in this proposition (namely, constraining the initial configuration to be $\gd_0$, and assuming $\gs$ is \emph{strictly} decreasing) are not as general as one would expect when compared to Theorem~\ref{thm:main}: we further discuss them in the proof, see Section~\ref{sec:compl:gsdecreasing} below.
	
	\begin{remark}\label{rem:BBMspeed}
		Let us point out that Proposition~\ref{prop:main:gsdecreasing} also holds for $L(T)=+\infty$, that is for the maximum of the BBM without selection when $\gs$ is decreasing: this has already been proven in~\cite{MZ16}. However, when $\gs$ is not decreasing, the maximum of the BBM without selection is $v_{\max}T+o(T)$ for some $v_{\max}>v(1)$ (this is a well-known result, which follows e.g. from direct adaptations of~\cite{BK04-2} or~\cite{Mal15}). This is in line with~\eqref{eq:thm:supercrit}, since one can let $L(T)$ be arbitrarily close to $T$, hence it implies that the maximum of the BBM without selection is larger than any $v(1)T+o(T)$ for $\gs$ not decreasing.
	\end{remark}

	\subsubsection*{Notational convention and rephrasing of the main result.}
	In order to write Theorem~\ref{thm:main} and upcoming statements in a more condensed form, let us denote the three regimes (i.e. $T^{1/3}\ll L(T)\ll T$, $L(T)\sim \ga T^{1/3}$ and $1\ll L(T)\ll T^{1/3}$) respectively with the superscripts ``sup'', ``crit'' and ``sub''. We will also occasionally consider the regime $L(T) \gg T^{1/3}$ in the case of non-increasing $\sigma$, which we denote with the superscript ``$\supdecr$''. In what follows, \emph{we will always assume that $N$ depends on $T$ according to one of the four regimes.} Furthermore, in the ``sup'' regime we always implicitly assume that $\gs$ changes its monotonicity finitely many times (this technical restriction is further discussed below) and in the ``$\supdecr$'' regime we assume that $\sigma$ is decreasing. Then, we define the \emph{error scaling terms} for each regime with
	\begin{equation}\label{eq:defbanyrg}
		b^\supercrit_T\,:=\, L(T)\,,\qquad
		b^\crit_T\,=\, b^\supdecr_T
		\,:=\, T^{1/3}\,,\qquad\text{and}\qquad
		b^\subcrit_T\,:=\, \frac{T}{L(T)^2}\,,
	\end{equation}
	for $T\geq0$; and the \emph{limiting terms} with,
	\begin{equation}\label{eq:defmanyrg}\begin{aligned}
			m^\supercrit_T\,&:=\, v(1)T + \left[\int_0^1(\gs')^+(u)\,\dd u\right]L(T)\,,\\
			m^\crit_T\,&:=\,v(1)T + \left[\int_0^1\frac{\gs(u)}{\ga^2}\Psi\Big(-\ga^3\frac{\gs'(u)}{\gs(u)}\Big)\dd u\right] T^{1/3}  \,,\\
			m^\subcrit_T\,&:=\, v(1)T\left(1 - \frac{\pi^2}{2L(T)^2}\right) \,,\\
			\text{and}\qquad m^\supdecr_T\,&:=\, v(1)T - \frac{\mathrm{a}_1}{2^{1/3}}\left[\int_0^1\gs(u)^{1/3}|\gs'(u)|^{2/3}\,\dd u\right]T^{1/3}.
		\end{aligned}
	\end{equation}
	Theorem~\ref{thm:main} and Proposition~\ref{prop:main:gsdecreasing} can then be summarized as follows:
	\begin{theorem}[Rephrasing of Theorem~\ref{thm:main} and Proposition~\ref{prop:main:gsdecreasing}]
		\label{thm:main_rephrased}
		Let the assumptions of Theorem~\ref{thm:main} (in the regimes ``$\mathrm{sup}$'', ``$\mathrm{crit}$'' or ``$\mathrm{sub}$'') or of Proposition~\ref{prop:main:gsdecreasing} (in the regime ``$\mathrm{sup}$-$\mathrm{d}$'') hold. Then, for every $\anyrg\in\{\supercrit,\crit,\subcrit,\supdecr\}$, we have as $T\to+\infty$,
		\begin{equation}\label{eq:mainresult:manyrg}
			\max(\cX_T^{N(T)}) = \shft(\mu_T)+ m^\anyrg_T + o_{\bbP}\big(b^\anyrg_T\big).
		\end{equation}
	\end{theorem}
	
	\subsubsection*{Complementary result: empirical measure and diameter.}
	We complement the main result with a statement, in the critical and super-critical regimes, about the empirical measure of the particles below the asymptotic maximum and the diameter of the configuration at the final time. We do not expect the very same claim to hold in general in the subcritical regime: especially for~\eqref{eq:prop:endtimedist:1}, a random centering would be required, see Remark~\ref{rem:enddist:subcrit}. We write $\log_+(x) = \log(x\vee 1)$.
	\begin{proposition}\label{prop:endtimedistribution}
		Suppose the assumptions of Theorem~\ref{thm:main} hold and that we are in the critical or super-critical regime, i.e.~$\anyrg\in\{\crit,\supercrit\}$. Then, as $T\to\infty$,
		\begin{equation}\label{eq:prop:endtimedist:1}
			\sup_{y\in[0,1]}\left|\,\frac{\log_+\cX_T^{N(T)}\big([\shft(\mu_T) + m_T^\anyrg - y\gs(1)L(T),+\infty)\big) }{L(T)}\,-\,y\,\right|\;\longrightarrow\;0\,,
		\end{equation}
		in $\bbP_{\mu_T}$-probability. 
		Additionally, if $\min(\cX_T^{N(T)})$ denotes the position of the minimum of the $N(T)$-BBM at time $T$, then
		\begin{equation}\label{eq:prop:endtimedist:2}
			\max(\cX_T^{N(T)}) - \min(\cX_T^{N(T)}) = \sigma(1)L(T) + o_{\bbP}(L(T)).
		\end{equation}
	\end{proposition}
	\begin{remark}
		We can rephrase the first part of Proposition~\ref{prop:endtimedistribution} as follows: in the critical and super-critical regimes, one has for $T$ large, \[\cX_T^{N(T)}([\shft(\mu_T) + m_T^\anyrg - y\gs(1)L(T),+\infty)) \,=\, N(T)^{y+o_{\bbP}(1)},\] uniformly in $y\leq1$ not too close to zero.
	\end{remark}
	
	
	\subsubsection*{$N$-BBM with deterministic branching times}
	Our results also apply to a variant of the $N$-BBM in which particles branch simultaneously at deterministic times on a time grid $a\N$, for some $a>0$. Let $\xi$ and $\sigma$ be as above. The process is then defined as follows: Given $T>0$, particles diffuse independently according to time inhomogeneous branching Brownian motions with infinitesimal variance $\gs^2(t/T)$ at time $t$. Furthermore, at each time which is a multiple of $a = 2\log\bbE[\xi]$, each individual is replaced independently from the others by a random number of particles with the same distribution as $\xi$. In the following, the BBM with deterministic branching times will be called ``BBMdb'', and its counterpart with selection will be called ``$N$-BBMdb''.
	
	\begin{proposition}\label{prop:main:detbranching}
		Let $N(T)\to+\infty$ as $T\to+\infty$, then the results of Theorem~\ref{thm:main}, Propositions~\ref{prop:main:gsdecreasing} and~\ref{prop:endtimedistribution} also apply to the $N(T)$-BBMdb.
	\end{proposition}
	\begin{remark}
		Note that we have chosen the value of $a$ above for convenience, but we can handle any value of $a$ by a time-change and using a different function $\gs$.
	\end{remark}

	\subsection{The continuous random energy model, and algorithmic hardness}
	\label{sec:crem}
	We now present one of the main motivations and applications of our results. A large body of literature is concerned with algorithmic hardness thresholds for combinatorial optimization problems on random instances. This has been a very active research area in the last two decades, drawing extensively on results and methods from the theory of \emph{spin glasses} in statistical mechanics. See e.g.~\cite{ElAlaoui2021,Gamarnik2021} and the references therein. A stylized model of a spin glass is the continuous random energy model (CREM), initially introduced by Derrida and Spohn~\cite{DS88} and Bovier and Kurkova~\cite{BK04-2}. The CREM is a certain Gaussian process on the rooted binary tree of depth $T$, which we now introduce.
	
	
	Consider a function $A$ from $[0,1]$ to $[0,1]$ which is non-decreasing and satisfies $A(0)=0$, $A(1)=1$. Let $\bbT_T:=\{\emptyset\}\cup\bigcup_{i=1}^T\{0,1\}^i$ denote the rooted binary tree of depth $T\in\N$, and for $0\leq i\leq T$, write $V_i$ for the set of vertices $u\in \bbT_T$ with depth $|u|=i$. Let $X^\crem_\emptyset:=0$. For $u\in \bbT_{T-1}$, and $ui$ one of its two offspring, $i\in\{0,1\}$, let $X^\crem_{ui}-X_u$ be a centered Gaussian random variable with variance
	\[T\,\big(A\big(\tfrac{|u|+1}T\big)-A\big(\tfrac {|u|}T\big)\big)\,,\]
	and assume the $X^\crem_{ui}-X^\crem_u$, $u\in \bbT_{T-1}$, $i\in\{0,1\}$ are independent. Then the (Hamiltonian of the) CREM with parameters $A(\cdot)$ and $T$ is given by the values of the process $X$ on the leaves of $\bbT_T$, that is $(X^\crem_u)_{u\in V_T}$. One may consider the CREM as an isotropic centered Gaussian process on the rooted binary tree, since the covariance function only depends on the distance between the vertices and their distance to the root, analogously to isotropic Gaussian processes on Euclidean space, see e.g.~Berman~\cite{Ber80} and the references therein.
	
	If $A$ is smooth, then the CREM can also be constructed from a variant of the time-inhomogeneous BBMdb presented above. Indeed, consider $(\tilde\cX_t)_{t\in[0,T]}$ a BBMdb with infinitesimal variance $\gs^2(s):=A'(s)$, $s\in[0,1]$, and where all particles branch simultaneously into $\xi\equiv2$ offspring each at each time of the grid $a\N$, where we take $a:=2\log 2$. We assume $T\in a\N$, but the process ends at time $T$ \emph{before} branching. Assuming $\tilde\cX$ starts from two particles at the origin (or that it branches instantly at time 0), one obtains the following equality in law:
	\begin{equation}\label{eq:BBM-CREMidentity}
		\tilde\cX_T \;\overset{(d)}{=}\; \big(\sqrt{2\log 2}\,X^\crem_u\big)_{u\in V_{T/(2\log2)}}\;.
	\end{equation}

	\subsubsection*{Algorithmic hardness of the CREM}
	The algorithmic hardness of optimizing the Hamiltonian of the CREM has been studied by Addario-Berry and Maillard~\cite{ABM19}. We recall their main result, which states informally as follows (we denote by $V_T$ the leaves of the tree):
	
	\begin{theorem}[from \cite{ABM19}]\label{thm:ABM19}
		Consider the CREM with parameters $A(\cdot)$ and $T\in\N$, the former being a continuously differentiable  function on $[0,1]$. Let $\gs^2(\cdot):=A'(\cdot)$ and $v_c \coloneqq \sqrt{2\log2} \int_0^1\gs(s)\dd s$. Let $\eps>0$, then the following holds:
		
		$(i)$ There exists an algorithm with run-time linear in $T$ that finds a vertex $u\in V_T$ such that $X^\crem_u\geq (v_c-\eps)T$ with high probability.
		
		$(ii)$ There exists $\gga=\gga(A,\eps)>0$ such that for $T$ sufficiently large, for any algorithm, the number of queries performed before finding a vertex $u\in V_T$ such that $X^\crem_u\geq (v_c+\eps)T$ is stochastically bounded from below by a geometrically distributed random variable with parameter $\exp(-\gga T)$.
	\end{theorem}
	
	In other words, Theorem~\ref{thm:ABM19} proves the existence of an \emph{algorithmic hardness threshold} for the CREM: finding a vertex with a value greater than $(v_c+\varepsilon)T$ for a given $\varepsilon > 0$ typically requires a number of queries exponential in $T$, whereas values smaller than $(v_c-\varepsilon)T$ can be obtained in linear time. An \emph{algorithm} in this context is, roughly speaking, any random sequence of vertices, such that the choice of the next vertex only depends on the values of the previous vertices. Note that this result is not interesting from a purely theoretical computer science perspective, since the input size of the problem is exponential in $T$. However, as argued in Addario-Berry and Maillard~\cite{ABM19} and in Section~\ref{sec:spin_glasses} below, this result sheds light on the efficiency of algorithms on other spin glass models, for which the input complexity is indeed polynomial in $T$.
	
	In light of Theorem~\ref{thm:ABM19}, it is natural to ask about the complexity of finding vertices in the CREM with value \emph{near the threshold} $v_c T$. The present paper provides a partial answer to this question. Indeed, our results on the BBM allow us to study in detail the efficiency of a particular algorithm---the $N$-CREM---which has complexity (i.e. number of queries) $O(TN)$, and which we now introduce.
	
	Let $N\in\N$, and let $(X^\crem_u)_{u\in \bbT_T}$ be the construction of the CREM on the whole tree $\bbT_T$ as presented above. We may construct the $N$-CREM with the following procedure: perform a breadth-first exploration of the tree $\bbT_T$, noting encountered values of $X^\crem$ at depth $k$ with $X_{k,1}, X_{k,2},\ldots$ Then, remove from the tree all vertices (as well as the sub-tree they support) from $V_k$ which are not associated with one of the $N$ highest values from the sequence $X_{k,i}$, $i\geq1$. Repeat that procedure at depth $k+1$, considering only the offspring of vertices which were not removed. When that procedure ends, it yields a family which we denote $(X^\ncrem_{T,i})_{i\leq N(T)}\subset (X^\crem_u)_{u\in V_T}$: this can be seen as an optimization algorithm on the CREM, with complexity ---that is, the number of queries throughout the procedure--- of order $O(TN)$. In particular, choosing a specific sequence $N=N(T)$ allows for any complexity in $T$ for that algorithm. One can also interpret $(X^\ncrem_{T,i})_{i\leq N}$ as the final values of a (discrete time) branching random walk with selection. We refer to Section~\ref{sec:beam-search} below for further discussion on this algorithm and its interpretation as a \emph{beam search} procedure.
	
	Recall that we defined the $N$-BBMdb above, i.e. the BBM with selection and deterministic branching times. Then the BBM--CREM correspondence~\eqref{eq:BBM-CREMidentity} also applies to the $N$-particles variants: more precisely, one has
	\begin{equation}\label{eq:NBBM-NCREMidentity}
		\tilde\cX^{2N}_T \;\overset{(d)}{=}\; \big(\sqrt{2\log 2}\;X^\ncrem_{T/(2\log2),i}\big)_{1\leq i\leq N}\;.
	\end{equation}
	\begin{remark}
		The presence of a $2N$ in the l.h.s. of~\eqref{eq:NBBM-NCREMidentity} comes from the fact that, in the $N$-CREM, one has to consider $2N$ Gaussian increments, then select the $N$ highest before having the particles reproduce; whereas in the $N$-BBMdb, the selection happens just after the reproduction event. Notice that, if $N=N(T)\to+\infty$ as $T\to+\infty$, one has $\log (2N)\sim\log N=:L(T)$: so the maximal displacement of the $N$-BBM and $(2N)$-BBM have the same asymptotics in Theorem~\ref{thm:main}.
	\end{remark}
	
	We now present our main results on the $N$-CREM. 
	Let us write $\max(X^\ncrem_T):=\max\{X^\ncrem_{T,i},\, 1\leq i\leq N(T)\}$, and recall~\eqref{eq:defbanyrg} and \eqref{eq:defmanyrg}. We have the following, which is an immediate corollary of Proposition~\ref{prop:main:detbranching} and~\eqref{eq:NBBM-NCREMidentity}.
	\begin{theorem}\label{thm:crem}
		Consider the CREM with parameters $A(\cdot)$ and $T\in\N$. Let $N=N(T)\to+\infty$ as $T\to+\infty$, and consider the associated $N$-CREM. Let $\anyrg\in\{\supercrit,\crit,\subcrit,\supdecr\}$ denote the regime satisfied by $L(T):=\log N(T)$, and $\gs^2(\cdot):=A'(\cdot)\in\cC^2([0,1])$. Then as $T\to+\infty$, one has,
		\begin{equation}\label{eq:crem:main}
			\max(X^\ncrem_T)\,=\, (2\log2)^{-1/2}\,m^\anyrg_{(2\log 2)T} + o_\bbP\big(b^\anyrg_{(2\log 2)T}\big)\,.
		\end{equation}
		Moreover, assume $\anyrg\in\{\crit,\supercrit\}$; then one has as $T\to+\infty$,
		\begin{equation}
			\sup_{y\in[0,1]}\left|\,\frac{\log_+\#\big\{i\leq N,\, X^\ncrem_{T,i} \geq (2\log2)^{-1/2}[m^\anyrg_{(2\log2)T} - y\gs(1)L(T)]\big\} }{L(T)}\,-\,y\,\right|\,\longrightarrow\,0\,,
		\end{equation}
		in $\bbP_{\mu_T}$-probability, and 
		\begin{equation}
			\max(X^\ncrem_T) - \min(X^\ncrem_T) = (2\log2)^{-1/2} \sigma(1)L(T) + o_{\bbP}(L(T)).
		\end{equation}
	\end{theorem}
	
	\begin{remark}
		Equation~\eqref{eq:crem:main} from Theorem~\ref{thm:crem} can be stated in words as follows. Let $N = \exp(T^\kappa)$. When $\kappa < 1/3$, then with high probability, the value found by the $N$-CREM algorithm (i.e. its \emph{output}) is far below the threshold, more precisely, at $v_c T - O(T^{1-2\kappa})$.
		When transitioning into the critical regime ($\kappa=1/3$), the second order term of the output escapes from a singularity at the scale $T^{1/3}$ (see also Figure~\ref{fig:NBRW:simu} below). 
		On the other hand, if $\kappa > 1/3$, then with high probability the output of the algorithm is above the threshold and of order $v_cT + O(T^{\kappa})$---unless $\gs$ is non-increasing, in which case the output of the algorithm is $o(T^{1/3})$-close to the maximum value of the CREM, which is of order $v_cT - O(T^{1/3})$ with high probability. 
	\end{remark}
	
	\subsection*{Organization of the paper}
	In Section~\ref{sec:motivations_and_comparison} we compare our results with the literature on spin glasses and branching Brownian motion; we also present an overview of the proof in Section~\ref{sec:overview_proof}, as well as extensive numerical simulations on a related discrete-space model in Section~\ref{sec:motivation:BRWconj}. 
	
	In Section~\ref{sec:coupling} we detail the construction of the time-inhomogeneous BBM and $N(T)$-BBM, and we establish some coupling results between these processes that are used throughout this paper. Moreover, in Section~\ref{sec:coupling:mainprops} we state two key results, Propositions~\ref{prop:main:LB} and~\ref{prop:main:UB}, which provide respectively a lower bound and an upper bound on the maximal displacement of the $N(T)$-BBM for some specific initial configurations. Using these and a coupling argument, we deduce Theorem~\ref{thm:main} for any initial configuration. 
	
	Sections~\ref{sec:prelim} through~\ref{sec:UB} are dedicated to the proofs of Propositions~\ref{prop:main:LB} and~\ref{prop:main:UB}. The core idea of the proof is to study the BBM with certain well-chosen \emph{killing barriers}, which approximates the $N(T)$-BBM. In Section~\ref{sec:prelim} we introduce the relevant barriers depending on the regime: super-critical, sub-critical and critical; as well as some preliminary results. Then in Section~\ref{sec:moments} we compute moment estimates on some functions of the BBM between barriers in each of those regimes. With these moment estimates, the proofs of Propositions~\ref{prop:main:LB} and~\ref{prop:main:UB} are performed in Sections~\ref{sec:LB} and~\ref{sec:UB} respectively. 
	This completes the proof of Theorem~\ref{thm:main}.
	
	Finally, Section~\ref{sec:compl} contains the proofs of all remaining statements from Section~\ref{sec:intro}, that is Propositions~\ref{prop:main:gsdecreasing},~\ref{prop:endtimedistribution} and~\ref{prop:main:detbranching} (from which one deduces Theorem~\ref{thm:crem}), as well as some complementary results on the $N$-BBM with time-inhomogeneous selection. All of these rely on Theorem~\ref{thm:main}, or arguments from its proof presented in previous sections.

	\section{Motivations and comments}
	\label{sec:motivations_and_comparison}

	\subsection{Algorithms on spin glasses}
	\label{sec:spin_glasses}
	In this section, we discuss algorithmic optimization in general spin glasses. In the recent years, several authors have proposed optimization algorithms for mixed $p$-spin models inspired by Parisi ultrametricity and the TAP approach for spin glasses, see Subag~\cite{subag2021following} for spherical spin glasses, as well as Montanari~\cite{montanari2021optimization} and El Alaoui, Montanari, Sellke~\cite{ElAlaoui2021} for Ising spin glasses. See also Huang and Sellke~\cite{huang2022tight,huang2023algorithmic} for hardness results in this setting. These algorithms can be regarded as analogs to the algorithms considered here and in Addario-Berry and Maillard \cite{ABM19} for the CREM. Note that in all of these models, it is by now understood or generally believed that a necessary and sufficient obstruction to efficient approximation of the ground state is provoked by the so-called ``overlap gap property'' ---this is also called the ``overlap gap conjecture''. The overlap gap property, roughly speaking, states that the support of the Parisi measure has a ``gap'', i.e. is not an interval, at sufficiently small temperature. See e.g. Gamarnik and Jagannath~\cite{gamarnik_overlap_2021}, as well as the survey by Gamarnik~\cite{Gamarnik2021}. The overlap gap conjecture indeed holds for the CREM under quite general assumptions. To wit, it is known for the CREM that the algorithmic hardness threshold is equal to the ground state exactly if the covariance function $A(\cdot)$ is concave \cite{ABM19}. Moreover, it has long been known that the Parisi measure is supported on the extreme points of the concave hull of $A(\cdot)$ \cite{BK07}. This proves the overlap gap conjecture, apart from boundary cases, where the function $A(\cdot)$ is concave and has affine parts. However, a more precise and more restrictive statement of the overlap gap property requires that the probability of two replicas to have an overlap in the gap is exponentially small~\cite{gamarnik_overlap_2021} which, we believe, rules out such cases. For example, in the simplest case, where $A(x) \equiv x$, i.e.~for branching random walks, it is known that this probability decays only polynomially fast \cite{derrida_genealogy_2016}. 
	
	In the same spirit, we hope that the current article may serve as a starting point to the study of efficiency of optimization algorithms \emph{close to the algorithmic hardness threshold in general spin glass models}. In fact, the results from the current article shed a light on the transition from polynomial to exponential complexity of optimization algorithms near the algorithmic hardness threshold, particularly when this threshold is strictly below the ground state, i.e.~when the overlap gap property holds. For the CREM, we establish the appearance of three different asymptotic regimes and a deep relation with the Brunet--Derrida correction for the noisy FKPP equation (see Section~\ref{sec:motivation:NBBM} below). It would be very interesting to study this transition in other spin glass models, such as (mixed) $p$-spin models.

	Finally, let us mention that $p$-spin models admit a natural dynamics which is reversible with respect to the Gibbs measure: Langevin dynamics (spherical models) and Glauber dynamics (Ising models). These dynamics have also been proven to be asymptotically optimal in some of these models, see Sellke~\cite{Sel24}. 
	Spin-flip dynamics have also been considered for the random energy model in relation to aging and metastability, see for example \cite{benarous_glauber_2003,baity-jesi_activated_2018}. However, this dynamics does not seem to be efficient for algorithmic purposes in the case of the CREM, in fact, it is not hard to show that the natural spin-flip dynamic has exponentially large mixing time for every positive inverse temperature $\beta$.

	The efficiency of more general optimization algorithms for the CREM is an interesting open problem. We believe that the $N$-CREM considered here is close to optimal within the class of algorithms of a given complexity. Indeed, due to the particular structure of the CREM (in particular, the branching property), it is always favorable (on average) to explore subtrees of vertices of large values as opposed to subtrees of vertices of smaller values. The $N$-CREM is therefore a very natural candidate for an asymptotically optimal algorithm for this model.

	\subsection{Beam-search algorithms}
	\label{sec:beam-search}
	The $N$-CREM considered in this article can be viewed as a particular greedy-type algorithm, in a similar spirit as the algorithm studied in \cite{ABM19}. More precisely, it may be regarded as a \emph{beam search} algorithm \cite{Sha87AI}, with the parameter $N$ being the \emph{width of the beam}. Our main result (Theorem~\ref{thm:crem}) then precisely describes how the output of the algorithm depends on the width of the beam $N$, with a phase transition happening at $\log N \approx T^{1/3}$.

	A \emph{practical takeaway} might be the following. For the beam search algorithm, on the one hand, increasing the width $N$ of the beam \emph{substantially improves the output in the subcritical regime} $\log N\ll T^{1/3}$, due to the singular second order term $-T/(\log N)^2$ in the value of the output. 
	For example, passing from $N=T$ to $N = T^2$ makes this term four times smaller, increasing the value of the output by a term of order $T/(\log T)^2$, asymptotically when $T$ is large. 
	On the other hand, in the super-critical regime $\log N \gg T^{1/3}$, increasing the width of the beam comes with a more tenuous improvement of the output; more precisely, it only grows logarithmically in $N$.

	The efficiency of beam search algorithms is still an active research area, see e.g.~\cite{Lemons2022} and the references therein. The beam search algorithm considered here is quite special, due to the nature of the CREM. For example, the output of the algorithm is a non-decreasing function of the width of the beam, which is in general not the case \cite{Lemons2022}. Nevertheless, we hope that our results shed light on the behavior of general beam search algorithms for hard optimization problems on random instances, as the width of the beam grows to infinity.

	\subsection{$N$-BBM: Comparison with previous results}\label{sec:motivation:NBBM}
	\subsubsection*{The Brunet-Derrida behavior in the sub-critical regime.} 
	Results similar to~\eqref{eq:thm:subcrit} have already been obtained for some (time-homogeneous) branching processes with selection, see e.g.~\cite{BG10, Mai16} respectively for the $N$-particles branching random walk ($N$-BRW) and the $N$-BBM. In those papers, the authors prove for fixed $N\in\N$ the existence of an \emph{asymptotic speed} $v_N:=\lim_{T\to+\infty} \max(\cX^N_T)/T$, where $(\cX^N_T)$ denotes either the $N$-BRW or $N$-BBM. Then, when $N\to+\infty$, this asymptotic speed converges very slowly ---like $(\log N)^{-2}$--- to $v_{\infty}:=\lim_{T\to+\infty} \max(\cX_T)/T$, the asymptotic speed of the corresponding (time-homogeneous) branching process without selection. This slow convergence has been called \emph{Brunet-Derrida behavior}: it was first observed in~\cite{BD97} with heuristic methods and numerical simulations, and it is expected to hold for many models that fall under the universality class of the FKPP equation (see~\cite{BDMM06}). 
	
	Our Theorem~\ref{thm:main} generalizes the Brunet-Derrida behavior to time-inhomogeneous BBM with selection, for parameters $(T,N)$ in the range 
	$1\ll \log N \ll T^{1/3}$. We can recover from this an asymptotic valid as $T\to\infty$, for large but fixed $N$:
	\begin{corollary}\label{corol:finiteN}
		Let $(\cX_T^{N})_{T\geq0}$ an $N$-BBM (or $N$-BBMdb) with infinitesimal variance $\gs^2(\cdot/T)$. For any $\eps>0$ and sequence $(\mu_T)_{T\geq0}$ in $\meas$, one has,
		\begin{align}\label{eq:corol:finiteN}
			\limsup_{N\to+\infty} \, \limsup_{T\to+\infty} \,\bbP_{\mu_T}\left(\frac{(\log N)^2}{T}\left| \max(\cX_T^{N}) -\shft(\mu_T)- v(1)T\left(1-\frac{\pi^2}{2(\log N)^2}\right)\right|>\eps\right)&\\
			=\,0\,.&\notag
		\end{align}
	\end{corollary}
	This corollary follows naturally from~\eqref{eq:thm:subcrit} and a diagonal argument: if~\eqref{eq:corol:finiteN} does not hold, then one can construct a sequence $(N_k,T_k)_{k\geq1}$ for which the probability above remains large. However, one can freely choose $(T_k)_{k\geq1}$ such that $T_k \gg (\log N_k)^3$ for large $k$, and this directly contradicts the sub-critical result from Theorem~\ref{thm:main} (we leave the details of the proof to the reader). Furthermore, let us point out that in the time homogeneous case, the Brunet-Derrida behavior is expected to hold up to a number of particles $N$ satisfying $\log N\ll T^{1/2}$, see e.g. the discussion after Theorem 1.1 of Mallein~\cite{Mal17}. In contrast, in the time-inhomogeneous setting, this is true only in the range $\log N\ll T^{1/3}$, as Theorem~\ref{thm:main} shows.
	
	We remark that the asymptotics of the $N$-BBM when $N\to\infty$ and $T$ is fixed (the so-called \emph{hydrodynamic limit}) have also generated a lot of interest \cite{DeMasiFerrariPresuttiSoprano-LotoHydrodynamicsNBBM2019,Berestycki2018,AtarWeakFormulation2025,BerardFrenaisHydrodynamicLimit2023}. 
	
	\subsubsection*{1:3 space-time scaling in branching Brownian motion and branching random walks}
	The 1:3 space-time scaling has appeared many times in the study of branching Brownian motion and branching random walks. For the time-homogeneous versions of these processes, it appears in the $N$-particle process mentioned above as well as in the process with absorption at a linear space-time barrier, with the earliest appearance being, to our knowledge, in Kesten \cite{Kesten1978} and later developments by many authors \cite{Berard2011,Berestycki2010a,BBS13,Berestycki2014a,Berestycki2015,Derrida2007,GHS11,YaglomBBM,BBMsmalltime,Pemantle2009}. Pemantle \cite{Pemantle2009} is motivated by algorithmic aspects, inspired by Aldous~\cite{AldousGreedy,Aldous1998}. The 1:3 scaling also appears in the study of the particles in BBM or BRW without selection remaining close to the running maximum throughout their trajectory, which is usually called a ``consistent(ly) maximal displacement'' \cite{Fang2010,Faraud2012b,Jaffuel2012,Roberts2012}. 
	In these different references, it is observed that, generally speaking, a branching process constrained to a region of space of size $L$ for a time $T$ undergoes some critical behavior at scale $L\asymp T^{1/3}$. Since a branching process limited to $N$ particles occupies a region of size of order $\log N$ (recall~\eqref{eq:prop:endtimedist:2}), this matches the critical regime observed in Theorem~\ref{thm:main}.
	
	In the context of time-inhomogeneous BBM or BRW, the 1:3 scaling has been considered to our knowledge only in the study of extremal particles, in the regimes where their trajectories stay close to the running maximum during a macroscopic time \cite{Mal15,MZ16}. Relatedly, it appears in the time-inhomogeneous Fisher-KPP equation \cite{Nolen2014}, due to a duality relation with the time-inhomogeneous BBM.

	\subsection{Overview of the proof}
	\label{sec:overview_proof}
	As mentioned above, the long-time behavior of homogeneous $N$-BBM and its variants have been quite intensively studied in the mathematical literature since the seminal article by Bérard and Gouéré~\cite{BG10}. However, it appears that these studies have been focused on two specific time-scales: 
	\begin{enumerate}
	    \item $T\to\infty$, then $N\to\infty$ (i.e.~$N$ large but constant), \cite{BG10,BerestyckiZhaoShapeMultidimensional2018,Mallein2015a,Pain2015,BerestyckiToughSelectionPrinciple2024}
	    \item $\log N\asymp T^{1/3}$ \cite{Mai16,Mal17}
	\end{enumerate}
	All of these papers rely on more or less precise comparisons with BBM or BRW killed at certain space-time curves, which are amenable to explicit first and second moment calculations. Additionally, in the setting where $T\to\infty$, then $N\to\infty$, arguments based on regeneration times and Kingman's ergodic theorem are generally applied, sometimes also Birkhoff's ergodic theorem \cite{BerestyckiToughSelectionPrinciple2024}. These arguments do not extend to a time-inhomogeneous setting. Even in the time-homogeneous setting, there is little hope that they could be used to deal with all regimes of time $T$ depending on $N$. This leads to important technical challenges which we overcome in this article, depending on the regime.
	
	Before detailing the approaches for each regime, we wish to highlight that we have taken great care in avoiding repetitions: many lemmas are applicable in two or more regimes. The general idea is to bound the $N$-BBM from above and from below by auxiliary models, called $N^+$-BBM and $N^-$-BBM, similar to Maillard~\cite{Mai16}. These models are defined using barriers at which particles are killed, and which ensure that, with high probability, the number of surviving particles is larger or smaller than $N$, respectively. Contrary to previous works on \emph{extremal} particles in time-inhomogeneous BBM such as Mallein~\cite{Mal15} or Maillard and Zeitouni~\cite{MZ16}, it turns out that it is not enough here to consider only one upper or lower barrier but both at the same time, which complicates their definition, leading in particular to the appearance of the function $\Psi$ defined in \eqref{def:Psi} in the bounds for the maximum, in the \emph{critical} regime.
	
	The number of particles surviving between the barriers is controlled using first and second moment estimates. The first and second moment estimates rely on the study of the heat kernel of a time-inhomogeneous Brownian motion killed upon exiting certain space-time tubes. Some of these estimates are similar to those from Mallein~\cite{Mal15}, who considers the regime $T\asymp (\log N)^3$. In order to deal with all regimes, we derive them from scratch, benefiting from the fact that the Brownian motion simplifies the calculations compared to the branching random walk. 	It is these estimates which are particularly delicate in the \emph{super-critical} regime. In this regime, the time-inhomogeneous $N$-BBM behaves radically different to the time-homogeneous $N$-BBM. In particular, we observe that the main contribution to the first moment comes from trajectories that localize near one of the boundaries of the tube, see Figure~\ref{fig:supercrit_strategy}, leading to additional constraints that have to be introduced. We refer to Section~\ref{sec:moments:supercrit} for details.
	
	The comparison argument is particularly challenging in the proof of the lower bound in the \emph{subcritical} regime $\log N \ll T^{1/3}$. Here, comparison with a single BBM with barriers does not work as errors blow up. Since the process behaves like a concatenation of homogeneous $N$-BBM, it is natural to apply results or methods for this model. However, the arguments mentioned above based on regeneration times are not strong enough to cover the whole subcritical regime. We therefore develop a new proof method, which does not make use of regeneration times. This proof method can potentially be applied to the study of other variants of the (homogeneous) $N$-BBM. Roughly speaking, we slice the time interval into pieces of length slightly larger than $(\log N)^3$, and compare with a BBM with barriers in each piece. We then distinguish among two cases. If the number of pieces is not too large ($\log N \gg \log T$ works), then it is enough to restrict to a certain event ensuring that each piece behaves typically. On the other hand if the size of the pieces is small enough ($\log N \ll T^{1/8}$ works), we can bound expectation and variance of the displacement of the maximum in between each piece to obtain the result. This argument appears in Section~\ref{sec:LB}, making use of estimates from Section~\ref{sec:moments:subcrit}.
	
	Another difficulty arises in the proof of the upper bound, when bounding the $N$-BBM from above by an $N^+$-BBM. In order to control the particles in the $N^+$-BBM, it is necessary to add an additional upper barrier. However, this breaks the comparison to $N$-BBM. To circumvent this, we treat the particles hitting the upper barrier separately, in the spirit of Berestycki-Berestycki-Schweinsberg \cite{BBS13} and Maillard~\cite{Mai16}. We refer to the beginning of Section~\ref{sec:UB} for details.
	
	\subsection{Generic branching random walk and numerical simulations}\label{sec:motivation:BRWconj}
	In Theorem~\ref{thm:crem} we provided an asymptotic of $\max(\cX^{\ncrem}_T)$ as $T\to+\infty$ for the $N$-CREM, $N=N(T)$, and we commented above 
	that it may be seen as an optimization algorithm on realizations of the CREM. We conjecture that such results adapt to more general (i.e. non-Gaussian) branching random walks (BRW). In this section we present the conjectured formulae that would extend Theorem~\ref{thm:main} to a generic BRW, and then we present numerical simulations in the case of a Bernoulli BRW. 
	
	\subsubsection*{Conjecture for the general BRW}
	We introduce some notation in the vein of~\cite{Mal15}, which we only use in this section. Let $(\cL_s)_{s\in[0,1]}$ be a family of laws of point processes. Then, the BRW with offspring distributions $(\cL_s)_{s\in[0,1]}$ is constructed until time $T\in\N$, starting from some initial configuration $\mu_T\in\meas$, by induction: at generation $t<T$, an individual $u\in\cN_t$ located in $x\in\R$ generates $|L^u_{t/T}|$ children located respectively in $x+Y$ for $Y\in L^u_{t/T}$, where the point processes $L^u_{t/T}\sim \cL_{t/T}$ are independent in $u$, $t$. 
	We write $\xi_{t/T}\sim |L^u_{t/T}|$ for the law of the number of children at generation $t$, and we assume $\bbE[\xi_s^2]<+\infty$ and $m_s:=\bbE[\xi_s]\geq 1$ for all $s\in[0,1]$. We write,
	\begin{equation}\label{eq:app:laplace}
		\gk_{s}(\gth)\,:=\, \log \bbE\left[\sum_{Y\in L_s} e^{\gth Y}\right]\,,\qquad \gth\geq0\,,\,s\in[0,1]\,,
	\end{equation}
	for the log-Laplace transform of the offspring point processes. Consider its Fenchel-Legendre transform, that is
	\begin{equation}\label{eq:app:legendre}
		\gk_s^*(v)\,=\,\sup_{\gth>0} \,[v\gth-\gk_s(\gth)]\,,\qquad v\in\R\,,\,s\in[0,1]\,.
	\end{equation}
	Morally, if $a=(a_s)_{s\in[0,1]}$ denotes a bounded, measurable function (called the ``speed profile''), the number of particles from the BRW that remain close to $\sum_{s=1}^t a_{s/T}$ at all times $t\in\{0,\ldots,T\}$, is roughly $\exp(- \sum_{s=1}^T \gk^*_{s/T}(a_{s/T}))$, see e.g.~\cite{Big95, Aldous1998} or more recently~\cite{Mal15}. In particular, letting $v_{\max}$ be the first order of the speed of the maximum of the BRW (without selection), it satisfies,
	\[
	v_{\max}=\sup\left\{\int_0^1 a_s\dd s\,\middle|\, (a_s)_{s\in[0,1]} \text{ such that }  \forall s\leq 1,\, \int_0^s \gk^*_{u}(a_{u}) \dd u \leq 0\right\}.
	\]
	Assume that for all $s\in[0,1]$, there exists a greatest root $v_s$ of $\kappa_s^*(v)=0$, and that $\kappa_s^*$ is finite in a neighborhood of $v_s$; then the ``natural speed'' of the process is defined by
	\begin{equation}\label{eq:app:natspeed}
		v_{\nat}=\sup\left\{\int_0^1 a_s\dd s\,\middle|\, (a_s)_{s\in[0,1]} \text{ such that }  \forall s\leq 1,\, \gk^*_{s}(a_{s}) \leq 0\right\}\,=\,\int_0^1v_s\dd s\,\leq\,v_{\max}\,.
	\end{equation}
	Finally, one defines $(\theta_s)_{s\in[0,1]}$, $(\gs_s)_{s\in[0,1]}$ with,
	\begin{equation}\label{eq:app:def:phiandgs}
		\forall\, s\leq 1\,,\quad \theta_s:=\partial_v\gk^*_s(v_s) =(\partial_\gth \gk_s)^{-1}(v_s)\,,\quad\text{and}\quad \gs_s^2=\partial^2_\gth \gk_s(\theta_s) = 1/\partial_v^2 \gk^*_s(v_s)\,.
	\end{equation}
	\begin{remark}\label{rem:conjBRW:gaussian}
		In the case of a centered Gaussian BRW (i.e. the variables $Y\in L_s$ are independent with law $\cN(0, \gs^2(s) )$), then one has $\gk_s(\gth)=\log m_s + \gth^2\gs^2(s)/2$. In particular one has $v_s:=\gs(s)\sqrt{2\log m_s}$, $s\in[0,1]$; and $v_s$, $\gs(s)$ do satisfy~\eqref{eq:app:def:phiandgs} with $\theta_s:=\sqrt{2\log m_s}/\gs(s)$. In particular for $m_s=m$ constant for all $s\in[0,1]$, $v_{\nat}=v(1)\sqrt{2\log m}$ matches the definition of $v(1)$ from Section~\ref{sec:intro} (up to a scaling factor coming from our initial choice of branching rate $\gb_0$).
	\end{remark}
	
	\begin{conj}\label{conj:BRW}
		With the notation above, let $N(T)=e^{L(T)}$ and consider $\cX^{N(T)}_T$ the configuration at generation $T$ of an $N(T)$-BRW started from a single particle at the origin (i.e. $\mu_T=\gd_0$) and with offspring distributions $(\cL_{t/T})$, $t\leq T$. Assume that $(v_s)_{s\in[0,1]}$ is well-defined. Then, if $L(T)\sim \ga T^{1/3}$ for some $\ga\in\R$, one has as $T\to+\infty$,
		\begin{equation}\label{eq:conj:crit}
			\max\big(\cX^{N(T)}_T\big) = v_{\nat} T +  \Bigg[\int_0^1 \frac{\theta_s\gs^2_s}{\ga^2}\Psi\Bigg(\frac{\ga^3\dot\theta_s}{\theta_s^3 \gs_s^2}\Bigg) \dd s\Bigg]T^{1/3} + o_\bbP\big(T^{1/3}\big)\,.
		\end{equation}
		If $1\ll L(T)\ll T^{1/3}$, then
		\begin{equation}\label{eq:conj:subcrit}
			\max\big(\cX^{N(T)}_T\big) = v_{\nat} T - \left(\frac{\pi^2}2 \int_0^1 \theta_s\gs^2_s \,ds \right)\frac{T} {L(T)^2} + o_\bbP\left(\frac{T}{L(T)^2}\right)\,.
		\end{equation}
		If $T^{1/3}\ll L(T)\ll T$, then
		\begin{equation}\label{eq:conj:supercrit}
			\max\big(\cX^{N(T)}_T\big) = v_{\nat} T + \Bigg[\int_0^1 \frac{(\dot\theta_s)^-}{\theta_s^2} \dd s\Bigg]L(T) + o_\bbP\big(L(T)\big)\,,
		\end{equation}
		where $(\cdot)^-$ denotes the negative part; and if additionally $\dot\theta_s\geq0$ for all $s\in[0,1]$, then
		\begin{equation}\label{eq:conj:supercrit:gsdecreasing} 
			\max\big(\cX^{N(T)}_T\big) = v_{\nat} T - \frac{\mathrm{a}_1}{2^{1/3}}\Bigg[\int_0^1 \frac{(\dot\theta_s\gs_s)^{2/3}}{\theta_s} \dd s\Bigg]T^{1/3} + o_\bbP\big(T^{1/3}\big)\,.
		\end{equation}
	\end{conj}
	
	On different matter, recall that we left the case $L(T)\asymp T$ (that is $N(T)=e^{\gamma T}$ for some $\gamma>0$) completely open. Considering the definitions above, we may write the following conjecture, which matches~\cite[(5.1)]{ABM19} in particular.
	\begin{conj}\label{conj:NBRW:Nlarge}
		For the $N$-BRW with $N(T)=e^{\gamma T}$, $\gamma>0$, one has $\max(\cX^{N(T)}_T) = v_{\max}^\gamma T + o_{\bbP}(T)$ as $T\to+\infty$, where
		\[
		v_{\max}^{\gamma}:=\sup\left\{\int_0^1 a_s\dd s\,\middle|\, (a_s)_{s\in[0,1]} \text{ such that }  \forall s\leq 1,\, -\gamma \leq \int_0^s \gk^*_{u}(a_{u}) \dd u \leq 0\right\}.
		\]
	\end{conj}
	
	\subsubsection*{Example: Bernoulli BRW}
	We now turn to the case of Bernoulli increments: for $p\in[0,1]$, we write $Y\sim \Ber(p)$ if $\bP(Y=+1)=1-\bP(Y=0)=p$. In particular $\bbE[Y]=p$ and $\Var(Y)=p(1-p)$. Let $p:[0,1]\to(0,1)$ a $\cC^2$ function (we write $p_s:=p(s)$), and assume that $\cL_s$, $s\in[0,1]$ is such that, for $L_s\sim\cL_s$, then the variables $Y\in L_s$ are independent with law $\Ber(p_s)$. Then the log-Laplace and Fenchel-Legendre transforms from~(\ref{eq:app:laplace}--\ref{eq:app:legendre}) can be written,
	\begin{equation*}
		\gk_s(\gth)\,=\, \log m_s + \log(1+ p_s(e^\gth -1))\,,
	\end{equation*}
	and
	\[
	\notag   \gk_s^*(v)\,=\, -\log m_s + D_{KL}(\Ber(a)||\Ber(p_s))\,=\, -\log m_s + (1-v)\log\frac{1-v}{1-p_s} + v\log\frac{v}{p_s}\,,
	\]
	where $D_{KL}(\cdot||\cdot)$ denotes the Kullback–Leibler divergence. Moreover, the speed profile $(v_s)_{s\in[0,1]}$ and the natural speed $v_{\nat}$ in~\eqref{eq:app:natspeed} are well defined if $m_sp_s < 1$ for all $s\in[0,1]$. 
	
	In the following, we take $\xi_s\equiv 2 = m_s$ for all $s\in[0,1]$ (i.e. branching is binary), and $p_s<1/2$. Then, the functions $v_s$, $\theta_s$ and $\gs_s$, $s\in[0,1]$ are well defined and can be explicitly expressed in terms of each other. One can numerically calculate $v_s$ as the greatest root of $\kappa_s^*(v) = 0$, and express $\theta_s$ and $\sigma_s$ in terms of $v_s$ as follows:
	\begin{align}
		\theta_s &= \partial_v \kappa_s^*(v_s) = \log\left(\frac{(1-p_s)v_s}{(1-v_s)p_s}\right),\\
		\sigma_s^2 &= 1/\partial_v^2 \gk^*_s(v_s)= v_s(1-v_s)\,.
	\end{align}
	
	\begin{figure}[ht]\begin{center}
			\begin{subfigure}[t]{0.65\textwidth}
				\includegraphics[width=\textwidth]{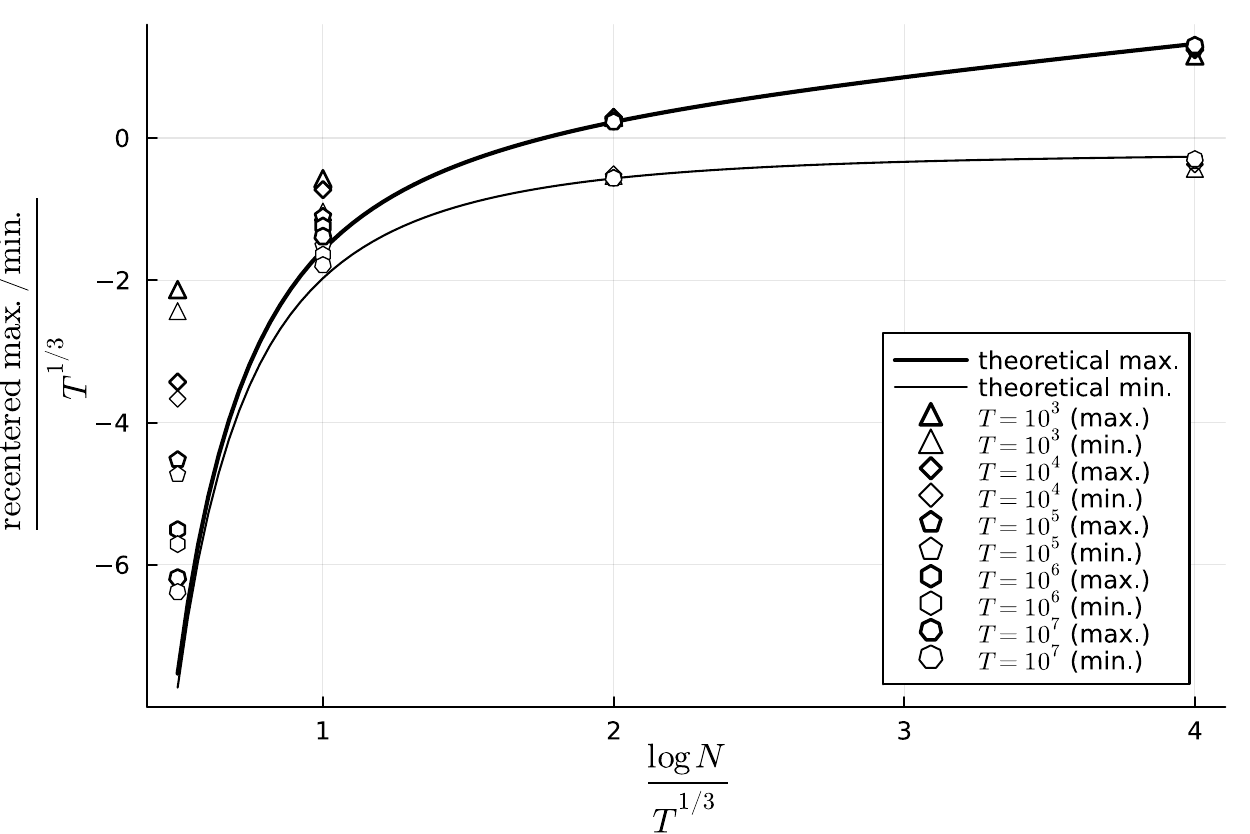}
				\caption{$p_s=0.4-0.3s$ ($\gth_s$ decreasing in $s$)}
			\end{subfigure}
			\begin{subfigure}[t]{0.65\textwidth} 
				\includegraphics[width=\textwidth]{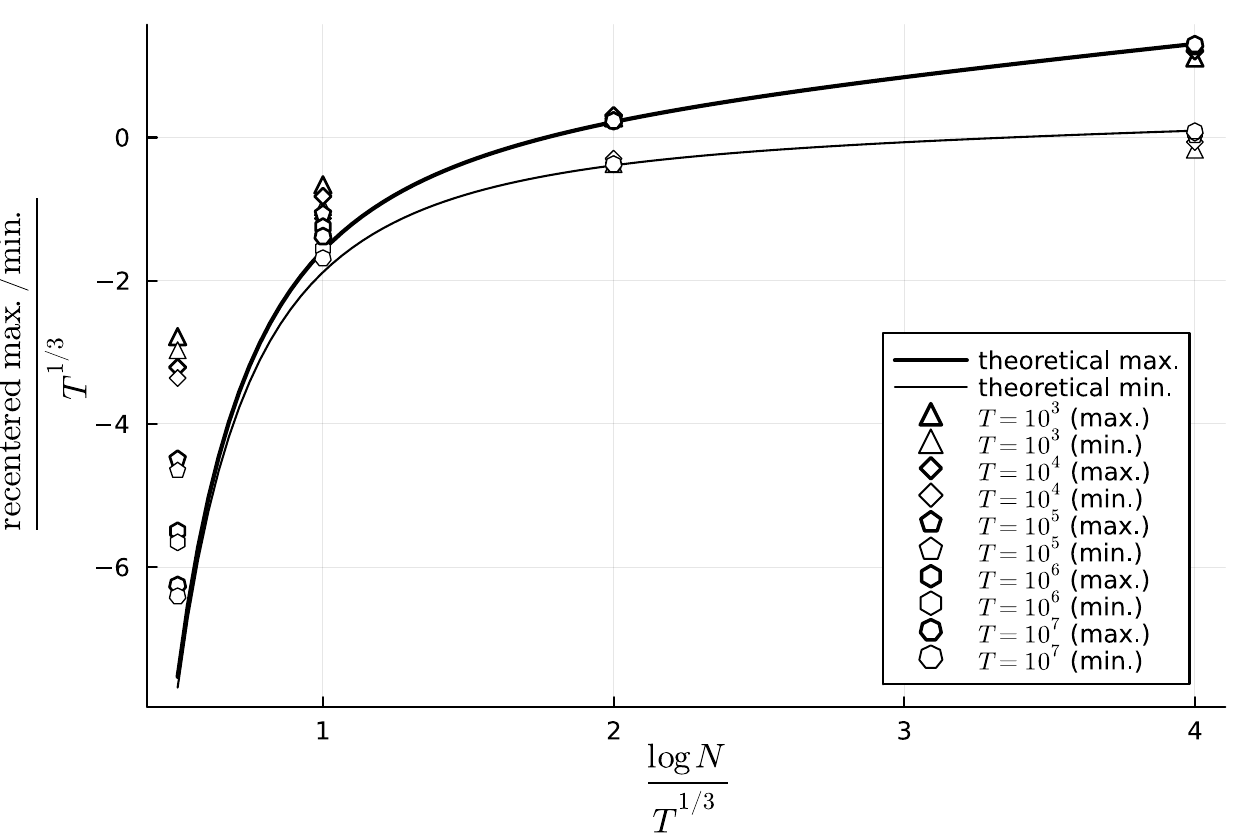}
				\caption{$p_s=0.1+0.3s$ ($\gth_s$ increasing in $s$)}
			\end{subfigure}
			\caption{Numerical simulations of the recentered, rescaled maximum and minimum positions $\overline{\mathrm{max}}^{N(T)}_T$ and $\overline{\mathrm{min}}^{N(T)}_T$ (see \eqref{eq:max_min_recentered_rescaled}) of the binary, Bernoulli $N$-BRW and comparison with the theoretical values. See text for details.}
			\label{fig:NBRW:simu}
	\end{center}\end{figure}
	
	We conducted extensive numerical simulations of maximal (and minimal) displacements of this particular Bernoulli $N$-BRW. These simulations were made for two different choices of $p_s$, $s\in[0,1]$ and various values of $T$ and $N=N(T,\ga)$, the latter being chosen such that $(\log N)/T^{1/3}$ takes a predetermined value $\ga>0$, with $\ga \in\{0.5,1,2,4\}$. For every simulation, we start with $N$ particles in 0, and we plot 
	\begin{equation}
		\label{eq:max_min_recentered_rescaled}
		\overline{\mathrm{max}}^{N(T)}_T \coloneqq \frac{\max\big(\cX^{N(T)}_T\big) - v_{\nat}T}{T^{1/3}}, \quad \overline{\mathrm{min}}^{N(T)}_T \coloneqq \frac{\min\big(\cX^{N(T)}_T\big) - v_{\nat}T}{T^{1/3}}
	\end{equation}
	the position of the maximum and minimum, recentered by the first-order term provided by Conjecture~\ref{conj:BRW} and rescaled by $T^{1/3}$, as a function of $\ga$.
	We further compare the output with the theoretical result as $T\to\infty$, with fixed $\alpha$. The results of the simulations are presented in Figure~\ref{fig:NBRW:simu}.
	
	The simulations use a trick from Brunet and Derrida~\cite{Brunet1999}, which consists of storing the number of particles at each site, instead of the position of every particle individually. This allows for an algorithm with a complexity of $O(\alpha L(T) T)$ arithmetical operations, since only $O(\alpha L(T))$ sites are occupied at every time, with very high probability. The code, written in Julia, took several hours to run on a 2020 MacBook Pro with M1 chip.

	\subsection{Some perspectives}
	Let us discuss several ways in which our work could be expanded.\smallskip
	
	\noindent \emph{Technical restrictions.} There are a few technical assumptions in Theorem~\ref{thm:main} which we do not expect to be optimal. For instance, we believe that our results hold for all $\gs(\cdot)\in\cC^1([0,1])$, however proving this does not seem to be straightforward. 
%
	\smallskip
	
	\noindent \emph{Genealogy of the $N$-BBM.} In~\cite{BBS13}, the authors study the genealogy of a sample of particles in a (time homogeneous) BBM with drift and absorption, and prove that it converges to the genealogy of the Bolthausen-Sznitman coalescent. More specifically, they choose a near critical drift depending on some constant $L>0$, such that the process contains roughly $e^L$ particles throughout a time interval of length $L^3$, and they remain in a space interval of length $L$; moreover the absorption only kills the bottom-most particles of the process. Comparing these properties with those of the $N(T)$-BBM, we therefore expect the same convergence to hold for the genealogy of the $N(T)$-BBM in the critical regime $\log N\approx T^{1/3}$, up to a time-change of the coalescent due to the inhomogeneity in time. 
	\smallskip
	%
	
	\noindent \emph{General time-inhomogeneous BRW.} 
	Our current results only apply to the $N$-BBM and Gaussian $N$-BRW, but we conjecture that they can be extended to more general BRW laws, as presented in Conjecture~\ref{conj:BRW}. Moreover, let us stress that the largest part of Sections~\ref{sec:LB} and~\ref{sec:UB} does \emph{not} rely on the Gaussian distribution (nor the random branching times, see Section~\ref{sec:compl:detbranching}). Therefore, most of the required work should come from obtaining moment estimates as in Section~\ref{sec:moments}, which we expect to be technically involved ---even more than in the Gaussian case, which is the object of the present paper and already involves significant bookkeeping.\smallskip

	\section{Construction and couplings of the \texorpdfstring{$N$}{N}-BBM}\label{sec:coupling}
	\subsection{Definition of the time-inhomogeneous BBM}
	\label{sec:coupling:construct}
	Let us start this section by recalling elementary facts on time-inhomogeneous Brownian motions, and introducing some notation. 
	Throughout this paper, the standard, time-homogeneous Brownian motion on $\R$ will be denoted by $(W_t)_{t\geq0}$. Let $T>0$ and $\gs\in\cC^2([0,T])$: then the \emph{time-inhomogeneous Brownian motion} on $[0,T]$, with infinitesimal variance $\gs^2(\cdot/T)$ and started from $x\in\R$, is the Gaussian process  $(B_t)_{t\in[0,T]}$ with continuous sample paths such that,
	\[\bE_x[B_t] = x\quad \text{and} \quad \bE_x[(B_s-x)(B_t-x)]\,=\,\int_0^{s\wedge t}\gs^2(u/T)\,\dd u\;,\qquad\forall\, s,t\in[0,T]\,.\]
	It satisfies
	\begin{equation}\label{eq:BM:timechange}
		(B_t)_{t\in[0,T]}\,\overset{\text{(d)}}{=}\, (x+W_{J(t)})_{t\in[0,T]}\;,\qquad\text{where}\quad J(t):=\int_0^t\gs^2(s/T)\,\dd s\,.
	\end{equation}
	In this paper, the law and expectation of a (non-branching) Brownian motion started from $x\in\R$ will always be denoted by $\bP_x$ and $\bE_x$ respectively. Similarly, the law of the Brownian motion started from time-space location $(s,x)\in[0,T]\times \R$ (i.e. shifted in time by $s$) will be denoted by $\bP_{(s,x)}$. It will always be clear from context whether the process considered is the time-homogeneous ($W$) or inhomogeneous ($B$) variant.\medskip
	
	\subsubsection*{The branching Brownian motion (BBM)}
	We now turn to the branching Brownian motion. We do not expand too much on precise definitions of branching Markov processes, but the reader can refer to~\cite{INW69Part1, INW69Part3} for a very complete and general construction, or e.g.~\cite{AN72} for a more accessible presentation.
	
	The \emph{(time-inhomogeneous) branching Brownian motion} (BBM) on $[0,T]$ can be described with some random families, $(\cN_t)_{t\in[0,T]}$ and $(X_u(t))_{u\in\cN_t,t\in[0,T]}$, where the (finite) set $\cN_t$ denotes the labels of \emph{particles} alive at time $t\in[0,T]$, and $X_u(t)$ denotes the \emph{position} at time $t\in[0,T]$ of a particle $u\in\cN_t$; which satisfies the following properties:
	
	--- each individual $u\in\cN_t$, $t\in[0,T]$ dies at rate $\gb_0\geq0$, and is immediately replaced at the same position by a random number of descendants following the law of a given random variable $\xi\geq2$,
	
	--- for $u\in\cN_t$, $t\in[0,T]$, the function $(X_u(s))_{s\in[0,t]}$ denotes the positions of $u$ and its ancestors throughout $[0,t]$: it has the same law as a time-inhomogeneous Brownian motion $(B_s)_{s\in[0,T]}$ started from $X_u(0)$,
	
	--- the evolution of any particle (lifespan, number of descendants and displacement), once born, is independent of the other living particles.
	
	\begin{remark} Throughout the remainder of this paper, unless stated otherwise, the branching processes we consider have offspring distribution $\xi\geq2$ with $\bbE[\xi^2]<+\infty$, and branching rate $\gb_0=(2(\bbE[\xi]-1))^{-1}$; in particular, if the process is started from a single particle at $x\in\R$, a standard computation yields $\bbE_{\gd_x}[|\cN_t|]=e^{t/2}$ for all $t\in[0,T]$ (see e.g.~\cite{AN72}). We do not write those assumptions again.
	\end{remark}
	
	Recall that $\meas$ denotes the set of all finite counting measures on $\R$. Then, the family $(\cX_t)_{t\in[0,T]}$ defined by $\cX_t:=\sum_{u\in\cN_t} \gd_{X_u(t)}$  defines a Markov process on $\meas$, which completely describes the particle configurations of the BBM ---in the following, we only write the sets of labels $(\cN_t)_{t\in[0,T]}$ explicitly if they are needed. Since we assumed $\bbE[\xi^2]<+\infty$, the total population of the process does not blow up on $[0,T]$ with probability 1 (see e.g.~\cite{Sav69} for a proof). For $\mu\in\meas$, the law and expectation of the (time-inhomogeneous) BBM started from the initial configuration $\cX_0=\mu$ will be denoted by $\bbP_\mu$ and $\bbE_\mu$ respectively throughout this paper. When it is started from a single particle at the origin (i.e. $\mu=\gd_0$), we shall sometimes omit the subscript and write $\bbP$, $\bbE$.
	
	With a slight abuse of notation, any finite counting measure $\mu\in\meas$ can be written as a finite subset of $\R$, with possible repetition of its elements. In particular, for $\mu,\nu\in\meas$, one may write $\mu\subset\nu$ if all atoms in the counting measure $\mu$ are also present in $\nu$. Regarding the BBM, one has $\max(\cX_t)=\max_{u\in\cN_t} X_u(t)$ with that notation. Finally, let us mention that one can consider a time-homogeneous BBM very similarly by replacing $(B_s)_{s\in[0,T]}$ with $(W_s)_{s\in[0,T]}$ in the definition above; but unless specified otherwise, we shall only consider time-inhomogeneous BBM's throughout this paper.
	

	\subsection{Selection mechanisms and \texorpdfstring{$N$}{N}-BBM}\label{sec:coupling:NBBM}
		Recall that $\measN$ denotes the set of counting measures on $\R$ with total mass at most $N$. The \emph{$N$-particles branching Brownian motion} ($N$-BBM) started from $\mu\in\measN$ can be defined from the original BBM $(\cX_t)_{t\in[0,T]}$ by only keeping its $N$ highest particles at all time, \emph{killing} (i.e. removing from the process) the others as well as their offspring. If several particles are located at the same height, we break ties arbitrarily (e.g.~by using the lexicographic order on the set of labels $U$, see below). Note that we allow the $N$-BBM to start with fewer than $N$ particles. Its particle configuration and set of (living) particles at time $t\in[0,T]$ are respectively denoted by $\cX^N_t$ and $\cN^N_t$ (the positions of particles are still denoted by $X_u(t)$, $u\in\cN^N_t$, $t\in[0,T]$).
		
		A formal construction of the $N$-BBM can be achieved through the use of \emph{stopping lines}, which is the analogue of a stopping time for a branching Markov process. We briefly recall the basic definitions, referring the reader to~\cite{Big04, Chau91, Jag89} or~\cite[Appendix~1]{Mai16} for more details.
Consider $U=\{\emptyset\}\cup\bigcup_{n\ge 1}\bbN^{n}$ the set of finite words over the alphabet $\bbN$, and for $u,v\in U$ let us write $u\preceq v$ if $u$ is a prefix of $v$. Following Chauvin~\cite{Chau91}, we label\footnote{In the literature it is standard to consider the BBM started from a single particle labelled with the empty word $\emptyset$; however one can also start with finitely many particles labelled by integers, i.e.\ words of length 1.} the particles of the BBM in the set $U$ in such a way that the genealogical order matches the prefix order on $U$: this means $\bigcup_{t\in[0,T]}\cN_t\subset U$ and for $s\leq t$, $v\in\cN_s$, $u\in\cN_t$, one has $v\preceq u$ if and only if $v$ is an ancestor of $u$. For $(v,s),(u,t)\in U\times[0,T]$, we write $(v,s)\preceq(u,t)$ if $v\preceq u$ and $s\leq t$; and $(v,s)\prec(u,t)$ if additionally $(v,s)\neq(u,t)$. A subset $\ell\subset U\times[0,T]$ is a \emph{line} if for all $(u,t)\in\ell$, one has $(v,s)\notin \ell$ for all $(v,s)\prec (u,t)$. One can extend that order relation in the following way: for $(u,t)\in U\times[0,T]$ and a line $\ell$, we write $\ell\preceq(u,t)$ if there exists $(v,s)\in\ell$, $(v,s)\preceq(u,t)$; and for $A\subset U\times[0,T]$, we write $\ell \preceq A$ if $\ell\preceq (u,t)$ for all $(u,t)\in A$. In the case of the BBM and for any line $\ell$, we can define the $\gs$-algebra
\[
\mathcal{F}_\ell \coloneqq \gs\left(\{u\in\cN_t\},X_u(t)\,;\, (u,t)\in U\times[0,T], \ell \nprec (u,t)\right)\,,
\] which, informally, contains all information from the BBM except for the descendants of particles in $\ell$. Then, an \emph{(optional) stopping line} for the BBM is a random line $\cL$ such that for all $(u,t)\in\cL$, $u\in\cN_t$, and such that for every line $\ell$, $\{\cL\preceq\ell\}\in\cF_\ell$: in other words, this means that determining if an individual $(u,t)$ is in $\cL$ does not depend on the descendants of $(u,t)$.
    
Therefore, for any stopping line $\cL$, one can define the BBM with selection mechanism $\cL$ by removing particles as soon as they ``hit'' $\cL$: i.e. it is the particle process which contains all particles $(u,t)$ from the BBM (that is $u\in\cN_t$, $t\in[0,T]$) such that $\cL\not\preceq(u,t)$. In particular, the $N$-BBM is an example of a BBM with selection mechanism: by induction  over the sequence $(t_n)_{t\ge1}$ of (random) branching epochs of the BBM, one can easily construct a stopping line $\cL$, such that the BBM with selection mechanism $\cL$ is an $N$-BBM.

We further introduce the following two classes of selection mechanisms:
    \begin{definition}\label{def:N-N+BBM} Let $N\in\N$ and consider a BBM with some selection mechanism $\cL$.
    
    $(i)$ We say that it is an $N^-$-BBM if, whenever at least $N$ particles are above another particle, the latter gets killed (but possibly more particles get killed).
    
    $(ii)$ We say that it is an $N^+$-BBM if, whenever a particle gets killed, there are at least $N$ particles above it (but the process may contain more than $N$ particles).
    \end{definition}
Notice that the $N$-BBM is both an $N^-$-BBM and an $N^+$-BBM.
    

	\subsection{Monotone couplings}\label{sec:coupling:NBBMcoupl}
	For $\mu,\nu\in\meas$, we write $\mu\prec\nu$ if $\mu([x,+\infty))\leq\nu([x,+\infty))$ for all $x\in\R$; in particular this implies $\mu(\R)\leq \nu(\R)$ and $\max(\mu)\leq\max(\nu)$. Moreover, for two random counting measures $\cX$, $\cY$ on $\R$, we say that ``$\cX$ is stochastically dominated by $\cY$'' if there exists a coupling between $\cX$ and $\cY$ such that $\bbP(\cX\prec\cY)=1$. In this section, we are interested in couplings between BBM's and/or $N$-BBM's which preserve the comparison $\prec$ through time.

We first present a monotone coupling result between the $N^-$-BBM, $N^+$-BBM and $N$-BBM from~\cite{Mai16}, as well as an immediate corollary which will be used many times in the present paper. 
	\begin{lemma}[Lemma~2.9 in \cite{Mai16}]\label{lem:N-N+couping} Let $\mu_0^-$, $\mu_0^N$ and $\mu_0^+$ denote (possibly random) finite counting measures on $\R$ with $\mu_0^N(\R)\leq N$ and $\mu_0^-\prec\mu_0^N\prec\mu_0^+$. Let $\cX^N$ (resp. $\cX^{N^-}$ $\cX^{N^+}$) be a time-inhomogeneous $N$-BBM (resp. $N^-$-BBM, $N^+$-BBM) with diffusion $\gs:[0,1]\to(0,+\infty)$ and started from $\mu_0^N$ (resp. $\mu_0^-$, $\mu_0^+$). Then there exists a coupling between the three processes such that, with probability 1, $\cX_t^{N^-}\prec \cX_t^{N}\prec \cX_t^{N^+}$  for all $t\in[0,T]$.
	\end{lemma}
    \begin{remark}
    There are two minor differences between this statement and~\cite[Lemma~2.9]{Mai16}: first, the latter only considers the time-homogeneous branching Brownian motion, and second, in the Definition~\ref{def:N-N+BBM} we removed the assumption that ``Only left-most particles are killed''. With a thorough reading of~\cite[Section~2.3]{Mai16}, one can check that these assumptions actually play no role in the proof of the coupling, provided that the same diffusion function $\gs:[0,1]\to(0,+\infty)$ is used in the three processes. For the sake of conciseness, we do not reproduce the proof in this paper.
    \end{remark}    
    \begin{corollary}\label{cor:coupling:N}
    Let $N_1, N_2\in\N$, $N_1\leq N_2$. Let $\mu_1\in\meas_{N_1}$ and $\mu_2\in\meas_{N_2}$ which satisfy $\mu_1\prec\mu_2$: then there exists $(\cX^{N_1}_{t})_{t\in[0,T]}$ and $(\cX^{N_2}_{t})_{t\in[0,T]}$ respectively an $N_1$- and an $N_2$-BBM, such that $\cX^{N_1}_{0}=\mu_1$, $\cX^{N_2}_{0}=\mu_2$ and $\cX^{N_1}_{t}\prec\cX^{N_2}_{t}$ for all $t\in[0,T]$ with probability 1.
    \end{corollary}
\begin{proof}
This follows e.g.\ from the fact that, for $N_1\le N_2$, an $N_2$-BBM is also an $(N_1)^+$-BBM.
\end{proof}    
    
Moreover, we provide some additional coupling statements that will be used in this paper. Recall that, with an abuse of notation, any counting measure $\mu\in\meas$ can be seen as a finite subset of $\R$ (with possible repetition of its elements); and notice that, for $\mu,\nu\in\meas$, the statement ``$\mu\subset\nu$'' is strictly stronger than ``$\mu\prec\nu$''.
\begin{proposition}\label{prop:coupling:BBM}
$(i)$ For $N\in\N$, $\mu\in\measN$, there exists a coupling between a BBM (without selection) and an $N$-BBM both started from $\mu$, such that, with probability 1, one has $\cX^N_t\subset\cX_t$ for all $t\in[0,T]$.

$(ii)$ Let $\mu_1,\mu_2\in\meas$ such that $\mu_1\prec\mu_2$. There exists a coupling between two BBM's $(\cX_{1,t})_{t\in[0,T]}$, $(\cX_{2,t})_{t\in[0,T]}$ (without selection) such that $\cX_{1,0}=\mu_1$, $\cX_{2,0}=\mu_2$ and $\cX_{1,t}\prec\cX_{2,t}$ for all $t\in[0,T]$ with probability 1. 
\end{proposition}
\begin{proof}
The first statement $(i)$ is a direct consequence of the definition of the $N$-BBM as a BBM with selection mechanism, so we only need to prove $(ii)$.
Let $k_1=\mu_1(\R)$, $k_2=\mu_2(\R)$ (so $k_1\leq k_2$). There exists $(x_i)_{1\leq i\leq k_1}$ and $(y_i)_{1\leq i\leq k_2}$ such that 
\[
\mu_1=\sum_{i=1}^{k_1}\gd_{x_i}\,,\quad x_1\geq\ldots\geq x_{k_1}\,,\qquad
\text{and}\qquad \mu_2=\sum_{i=1}^{k_2}\gd_{y_i}\,,\quad y_1\geq\ldots\geq y_{k_2}\,.
\]
Moreover, the assumption $\mu_1\prec\mu_2$ implies $x_i\leq y_i$ for all $1\leq i\leq k_1$.

We let $(\cM^1_t,\cZ^1_t)_{t\in[0,T]},\ldots,(\cM^{k_2}_t\cZ^{k_2}_t)_{t\in[0,T]}$ be $k_2$ i.i.d. copies of a time-inhomogeneous BBM, all starting from the initial configuration $\gd_0\in\cC$, and we write for any fixed $x\in\R$, $i\leq k_2$ and $t\in[0,T]$,
\[x+\cZ^i_t:=\sum_{u\in\cM^i_t}\gd_{x+X_u(t)}.\]
		Therefore, letting for $t\in[0,T]$,
\[
\cX_{1,t} := \sum_{i=1}^{k_1} (x_i+\cZ^i_t) \,,\qquad\text{and}\qquad \cX_{2,t} := \sum_{i=1}^{k_2} (y_i+\cZ^i_t)\,,
\]
one notices that $(\cX_{1,t})_{t\in[0,T]}$ and $(\cX_{2,t})_{t\in[0,T]}$ are two BBM's respectively started from $\mu_1$ and $\mu_2$, and they satisfy $\cX_{1,t}\prec\cX_{2,t}$ for all $t\in[0,T]$, finishing the proof. 
\end{proof}

We conclude this section with a direct consequence of Corollary~\ref{cor:coupling:N} regarding the \emph{quantiles} of the $N$-BBM configurations. 
For $M\in\N$, and $\mu\in\meas_N$, let us define
\begin{equation}\label{def:quantile:M}
q_M(\mu)\,\coloneqq\, \inf\{x\in\R; \mu([x,+\infty))< M \}\,=\, \sup\{x\in\R; \mu([x,+\infty))\ge M \}\,.
\end{equation}
In other words, $q_M(\mu)$ is the position of the $M$-th highest particle in the configuration $\mu$, with $q_M(\mu)=-\infty$ if $\mu(\R)<M$. By definition, for any $\mu,\nu\in\meas$ with $\mu\prec\nu$, one has $q_M(\mu)\leq q_M(\nu)$ for all $M\in\N$. In particular, we have the following direct consequence of Corollary~\ref{cor:coupling:N}.
\begin{corollary}\label{corol:quantile:NBBM}
Let $N\in\N$. Let $\mu_1,\mu_2\in\meas_{N}$ which satisfy $\mu_1\prec\mu_2$, and $(\cX^{N}_{1,t})_{t\in[0,T]}$, resp. $(\cX^{N}_{2,t})_{t\in[0,T]}$, an $N$-BBM started from $\mu_1$, resp. $\mu_2$. Then there exists a coupling between the two processes such that with probability one, one has $q_M(\cX^N_{1,t})\leq q_M(\cX^N_{2,t})$ for all $t\in[0,T]$, $M\in\N$.
\end{corollary}

	\subsection{Main propositions and proof of Theorem~\ref{thm:main}}\label{sec:coupling:mainprops}
	Let $N=N(T)\to+\infty$ as $T\to+\infty$, and define $L(T)=\log N(T)$. Using the coupling propositions presented above, notably Corollary~\ref{cor:coupling:N}, we claim that it is sufficient to prove Theorem~\ref{thm:main} for some specific initial configurations, and the main result follows. In order to condense all upcoming statements, let us recall the following notation: the three regimes ($L(T)\ll T^{1/3}$, $L(T)\gg T^{1/3}$ and $L(T)\sim\ga T^{1/3}$) are respectively denoted by the abbreviations $\subcrit, \supercrit$ and $\crit$. Recall the definitions of the scaling and limiting terms in all regimes from~\eqref{eq:defbanyrg} and~\eqref{eq:defmanyrg}. In the remainder of this paper, we shall write ``let $\anyrg\in\{\supercrit,\subcrit,\crit\}$'' instead of ``let $1\ll L(T)\ll T$ which satisfies either $L(T)\ll T^{1/3}$, $L(T)\gg T^{1/3}$ or $L(T)\sim \ga T^{1/3}$ for some $\ga>0$ as $T\to+\infty$''; and the symbol $\anyrg$ shall denote the regime corresponding to the choice of $L(T)$. In particular, many upcoming statements are formulated in terms of $b^\anyrg_T$, $m^\anyrg_T$ instead of $b^\supercrit$, $b^\subcrit\ldots m^\crit_T$, (and similarly for any upcoming notation). 
	
	Let us introduce two specific families of initial configurations. 
	On the one hand, for $\gk\in[0,1]$ we shall consider the measure $\lfloor N^\gk\rfloor \gd_{-\gk\gs(0)L(T)}\in\measN$. On the other hand we define for $\eps\in(0,1)$,
	\begin{equation}\label{eq:defmuK:sec2}
		\mu_{\eps}\;:=\; \sum_{k=0}^{\lceil\eps^{-1}\rceil} \Big\lceil N^{k\eps+\frac\eps2} \Big\rceil \gd_{-k\eps\gs(0)L(T)}\;\in\; \meas\;.
	\end{equation}
	\begin{remark}
		Notice that $\mu_\eps$ contains more than $N$ particles: when starting an $N$-BBM from $\mu_\eps$, we instantaneously kill all particles which are not in the $N$ highest.
	\end{remark} 
	The measure $\mu_\eps$ is a discrete approximation of an exponential distribution of (roughly) $N$ particles over the interval $[-\gs(0)L(T),0]$. More precisely, $\mu_\eps$ is composed of finitely many atoms that contribute similarly to the maximal displacement of the $N$-BBM that spawns from it. Indeed, recall the definition of the entropy-position trade-off function $Q_T$ in \eqref{eq:defshift}: then one observes that, for $\gk\in[0,1]$ and $\eps\in(0,1)$,
	\[Q_T(\lfloor N^{\gk}\rfloor \gd_{-\gk\gs(0)L(T)})=\gs(0)\log\big(\lfloor N^\gk\rfloor N^{-\gk}\big)= o(1)\,.\]
	
	It will be convenient to formulate statements which hold uniformly over some class of variance functions $\gs^2(\cdot/T). $ 
	Therefore, we define for $\gh>0$ small,
	\begin{equation}\label{eq:defgsens}
		\gsens\;:=\;\big\{\gs\in\cC^2([0,1])\,\big|\, \forall u\in[0,1],\, |\gs'(u)|\leq \gh^{-1},\, |\gs''(u)|\leq \gh^{-1}\,\text{ and }\,\gh\leq \gs(u)\leq \gh^{-1} \big\}\,,
	\end{equation}
	in particular $\cC^2([0,1])=\bigcup_{\gh>0}\,\gsens$, and $\gs\in\gsens$ implies
	\begin{equation}\label{eq:gsapprox}
		\forall\; 0\leq u, v\leq 1\;,\qquad \left|\frac{\gs(u)}{\gs(v)}-1\right|\,\leq\, \gh^{-2}|u-v|\,.
	\end{equation}
	For the convenience of the notation, we also define $\gsens^{\crit}= \gsens^{\subcrit}:= \gsens$, and
	\begin{align}\label{eq:defgsens:anyrg}
		&\gsens^{\supercrit}\;\\
\notag		&:=\;\left\{\gs\in\gsens\,\middle|\, 
		\exists n\ge 1,\ 0=u_0<\ldots<u_n=1:\,\begin{aligned}&\forall 1\leq i\leq n,\,\gs \text{ is monotonic on } [u_{i-1},u_i];\\ &\forall 1\leq i\leq n,\, u_i-u_{i-1}\ge\gh \end{aligned}
		\right\}.
	\end{align}
	With those definitions, we have the following results. 
	\begin{proposition}\label{prop:main:LB}
		Let $\anyrg\in\rgens$, and $\gl, \gh>0$. Then,
		\begin{equation}\label{eq:prop:main:LB}
			\lim_{T\to+\infty}\,\sup_{\gk\in[0,1]} \,\sup_{\gs\in\gsensrg} \,\bbP_{N^{\gk}\gd_{-\gk\gs(0)L(T)}}\left(\frac{1}{b^\anyrg_T}\left(\max(\cX_T^{N(T)}) - m^\anyrg_T\right)\leq -\gl \right) \;=\; 0\,.
		\end{equation}
	\end{proposition}
	
	\begin{proposition}\label{prop:main:UB} 
		Let $\anyrg\in\rgens$, and $\gl,\gh>0$. Then,
		\begin{equation}\label{eq:prop:main:UB}
			\lim_{\eps\to0}\,\limsup_{T\to+\infty}\,\sup_{\gs\in\gsensrg}\,\bbP_{\mu_\eps}\left(\frac{1}{b^\anyrg_T}\left(\max(\cX_T^{N(T)}) - m^\anyrg_T\right)\geq \gl \right) \;=\; 0\,.
		\end{equation}
	\end{proposition}
	
	Propositions~\ref{prop:main:LB} and~\ref{prop:main:UB} may be seen as particular cases of Theorem~\ref{thm:main} (with some added uniformity in $\gs(\cdot)$). Their proofs are contained in Sections~\ref{sec:LB} and~\ref{sec:UB} respectively, and rely on moment estimates from Section~\ref{sec:moments}. In the remainder of this section, we deduce Theorem~\ref{thm:main} from these propositions.
	
	\begin{proof}[Proof of Theorem~\ref{thm:main} subject to Propositions~\ref{prop:main:LB} and \ref{prop:main:UB}] 
		Let $\mu_T\in\meas_{N(T)}$. Recall the definition of $\shft(\cdot)$ from~\eqref{eq:defshift} and notice that, for any $x\in\R$,
		\begin{equation}\label{eq:shft:invariance}
			\shft(\mu_T(\cdot-x))\,=\, x+\shft(\mu_T))\,.
		\end{equation}
		Hence, by shifting the process and initial configuration by $-\shft(\mu_T)$, $T\geq0$, we may assume without loss of generality that $\shft(\mu_T)=0$. Let $\gl>0$, and let us write with a union bound,
		\begin{align}\label{eq:prf:thm:main:1}
			&\bbP_{\mu_T}\left(\frac{1}{b^\anyrg_T}\left|\max(\cX_T^{N(T)}) - m^\anyrg_T\right|\geq \gl \right)\\
			\notag &\qquad\leq\, \bbP_{\mu_T}\left(\frac{1}{b^\anyrg_T}\left(\max(\cX_T^{N(T)}) - m^\anyrg_T\right)\leq -\gl \right) + \bbP_{\mu_T}\left(\frac{1}{b^\anyrg_T}\left(\max(\cX_T^{N(T)}) - m^\anyrg_T\right)\geq \gl \right).
		\end{align}
		Then we treat both terms separately.
		
		Let $\eps>0$. Since $\shft(\mu_T)=0$, there exists $\gk=\gk(\eps,T)\in[0,1]$ such that 
		\[\mu_T\big([-\eps-\gk\gs(0)L(T),+\infty)\big)\,\geq\, N(T)^{\gk}\,.\]
		In particular, this implies $\mu_T \succ N^{\gk} \gd_{-\eps-\gk\gs(0)L(T)}$. Therefore, Corollary~\ref{cor:coupling:N} and a shift by $\eps$ yield,
		\begin{align*}
		&\bbP_{\mu_T}\left(\frac{1}{b^\anyrg_T}\left(\max(\cX_T^{N(T)}) - m^\anyrg_T\right)\leq -\gl \right)\\
		&\qquad\leq\,\bbP_{N^{\gk} \gd_{(-\gk\gs(0)L(T))}}\left(\frac{1}{b^\anyrg_T}\left(\max(\cX_T^{N(T)}) -\eps - m^\anyrg_T\right)\leq -\gl \right),\end{align*}
		and since $\eps/b^\anyrg_T<\gl/2$ for $T$ sufficiently large, we deduce from Proposition~\ref{prop:main:LB} that the first term in~\eqref{eq:prf:thm:main:1} vanishes as $T\to+\infty$.
		
		On the other hand, for $\eps>0$ the definition of $\shft(\mu_T)$ implies
		\[\forall\,\gk\in[0,1]\,,\qquad \mu_T\big([-(\gk-\eps)\gs(0)L(T),+\infty)\big)\,<\, N^{\gk}.\]
		Recall the definition of $\mu_\eps$ from~\eqref{eq:defmuK:sec2}. In particular, one notices for all $\gk\in[0,1]$,
		\[
		\mu_\eps\big([-(\gk+\eps)\gs(0)L(T),+\infty)\big)\,\geq N^\gk\,.
		\]
		Recalling that $\mu_T(\R)\leq N(T)$ by assumption, one obtains that $\mu_T\prec \mu_\eps(\cdot-2\eps\gs(0)L(T))$. Therefore, applying Corollary~\ref{cor:coupling:N} to an appropriately shifted process yields,
		\begin{align*}
		&\bbP_{\mu_T}\left(\frac{1}{b^\anyrg_T}\left(\max(\cX_T^{N(T)}) - m^\anyrg_T\right)\geq \gl \right)\\
		&\qquad\leq\,\bbP_{\mu_\eps}\left(\frac{1}{b^\anyrg_T}\left(\max(\cX_T^{N(T)}) +2\eps\gs(0)L(T) - m^\anyrg_T\right)\geq \gl \right).
		\end{align*}
		Recall that $L(T) \le b^\anyrg_T$. Taking $\eps$ sufficiently small and letting $T$ be large, we deduce from Proposition~\ref{prop:main:UB} that the second term in~\eqref{eq:prf:thm:main:1} can be arbitrarily small, which concludes the proof of the theorem.
	\end{proof}
	
	\section{Preliminaries on the BBM with barriers}\label{sec:prelim}
	Let us put aside the $N$-BBM for now, and consider the branching Brownian motion \emph{between barriers}, a variant of the BBM which is the cornerstone of the proof of Theorem~\ref{thm:main}. 
	This section assembles all our notation on the BBM killed at certain barriers, as well as preliminary results. We first introduce some notation which is used throughout the remainder of this paper, 
	then we present the main ideas and tools for the proofs of Propositions~\ref{prop:main:LB} and~\ref{prop:main:UB}.
	
	\subsection{Preliminaries and notation}\label{sec:prelim:notation}
	Recall~(\ref{eq:defgsens}--\ref{eq:gsapprox}), where we fix $\gh>0$ sufficiently small so that $\gs\in\gsens$. In the following, for any function $\gth:\R_+\to\R^*_+$ such that $\gth(T)\to0$ as $T\to+\infty$, we write for $f,g:\R_+\to\R_+$, and $T\geq0$,
	\begin{equation}\label{eq:defllbis}
		f(T) \llbis g(T)\qquad\text{if}\qquad f(T)\leq \gth(T) g(T)\,,
	\end{equation}  
	and, symmetrically, $g(T)\ggbis f(T)$ if $g(T)\geq \gth(T)^{-1} f(T)$. In particular, having~\eqref{eq:defllbis} for some $\gth(\cdot)$ and all $T$ large implies $f(T)\ll g(T)$; and, conversely, having $f(T)\ll g(T)$ implies that there exists some $\gth(\cdot)$ such that $f(T)\llbis g(T)$ for $T$ sufficiently large.
	
	In the following we fix $1\ll L(T)\ll T$ such that $L(T)\ll T^{1/3}$, $L(T)\gg T^{1/3}$ or $L(T)\sim \ga T^{1/3}$, $\ga>0$ as $T\to+\infty$ ; and let $\anyrg\in\{\supercrit,\subcrit,\crit\}$ denote the matching regime. Then, let $\tilde\gth(\cdot)$ be an (arbitrary) function taking values in $(0,1]$ and vanishing at infinity, which may depend on $L(T)$, such that, for $T$ sufficiently large, one has
	\begin{equation*}\left\{\begin{aligned}
			1\le_{\tilde\gth} L(T)\le_{\tilde\gth} T^{1/3}\,,&\qquad\text{if}\quad\anyrg=\subcrit\,,\\
			1\le_{\tilde\gth} T^{1/3}\le_{\tilde\gth} T\,,&\qquad\text{if}\quad\anyrg=\crit\,,\\
			T^{1/3}\le_{\tilde\gth} L(T)\le_{\tilde\gth} T\,,&\qquad\text{if}\quad\anyrg=\supercrit\,.
		\end{aligned}\right.
	\end{equation*}
	We then set $\theta \coloneqq \sqrt{\tilde \gth}$. We have
		\begin{equation}\label{eq:gthasymp}\left\{\begin{aligned}
			\theta^{-1}(T)\llbis L(T)\llbis \theta(T)T^{1/3}\,,&\qquad\text{if}\quad\anyrg=\subcrit\,,\\
			\theta^{-1}(T)\llbis T^{1/3}\llbis \theta(T)T\,,&\qquad\text{if}\quad\anyrg=\crit\,,\\
			\theta^{-1}(T)T^{1/3}\le_{\tilde\gth} L(T)\llbis \theta(T)T\,,&\qquad\text{if}\quad\anyrg=\supercrit\,.
		\end{aligned}\right.
	\end{equation}
	A pair of \emph{barriers}, which we usually write $(\gga^\anyrg_T(\cdot), \ol\gga^\anyrg_T(\cdot))$ in the remainder of this paper, is a pair of (smooth) functions from $[0,T]$ to $\R$, depending on $L(T)$, which satisfy the following for some $h>x>0$:
	\begin{equation}\label{eq:defxh}
		\gga_T^\anyrg(0)\,:=\,-x\gs(0)L(T)\,<\,0\,,\quad
		\text{and}\quad \ol\gga_T^\anyrg(r) - \gga_T^\anyrg(r)\,:=\,h\gs(r/T)L(T)\,,\quad \forall\,r\in[0,T]\,.
	\end{equation}
	We refer to $\gga^\anyrg_T(\cdot)$ (resp. $\ol\gga^\anyrg_T(\cdot)$) as the \emph{lower} (resp. \emph{upper}) barrier. Therefore, throughout the article and all regimes, the parameter $h$ denotes (up to a scaling term) the gap in-between the two barriers, and $x$ denotes (up to a scaling term) the distance from the origin to the lower barrier at time $t=0$. When we want to explicit the parameters $h>x>0$ for which the barriers satisfy~\eqref{eq:defxh}, we shall add them as superscripts by writing $\gga^{\anyrg,h,x}_T$, $\ol\gga^{\anyrg,h,x}_T$ (when they are clear from context we shall not write them, to lighten formulae). 
	
	For $t\in[0,T]$ and an interval $I\subset[0,h]$, we denote the set of particles which remained between the barriers throughout $[0,t]$ and ended in $\gga_{T}^\anyrg(t)+\gs(t/T)L(T)\cdot I$ at time $t$ with
	\begin{equation}\label{eq:defA:allregimes}
		A_{T,I}^\anyrg(t)\,:=\,\Big\{u\in\cN_t\;\Big|\; \forall s\in[0,t],\,X_u(s)\in \big[\gga_{T}^\anyrg(s), \ol\gga_{T}^\anyrg(s)\big]\,;\, \tfrac{X_u(t)-\gga_{T}^\anyrg(t)}{\gs(t/T)L(T)}\in I \Big\}\,.
	\end{equation}
	To lighten notation, we shall also write $A_{T}^\anyrg(t):=A_{T,[0,h]}^\anyrg(t)$ and $A_{T,z}^\anyrg(t):=A_{T,[z,h]}^\anyrg(t)$ respectively for the specific cases $I=[0,h]$ (no constraint on the final height within the barriers) and $I=[z,h]$, for some $z\in[0,h)$ (lower constraint only). Furthermore, for $0\leq s\leq t\leq T$, we denote with $R^\anyrg_T(s,t)\in\N$ the number of particles which remain above $\gga_T^\anyrg$ until they get killed by $\ol\gga_T^\anyrg$ at some time $\tau\in[s,t]$: more precisely,\footnote{One can check with standard branching processes theory that $R^\anyrg_T(s,t)$ is a measurable, almost surely finite random variable; and that, with probability 1, two particles do not reach the upper barrier at the same time. For the sake of conciseness we do not develop on that in this paper.}
	\begin{equation}\label{eq:defR}
		R^\anyrg_T(s,t)\,:=\,\left|\bigcup_{\tau\in[s,t]}\Big\{u\in\cN_\tau\,\Big|\, \forall\,r<\tau,\,X_u(r)\in\big(\gga_T^\anyrg(r),\ol\gga_T^\anyrg(r)\big) \,;\, X_u(\tau)=\ol\gga_T^\anyrg(\tau) \Big\}\right|.
	\end{equation}
	
	Recall the definitions of $v(\cdot)$ and $\Psi(\cdot)$ from~\eqref{eq:defv} and~\eqref{def:Psi} respectively, and let $w_{h,T}\in\cC^1([0,1])$ be defined by\footnote{Let us point out that, compared to~\eqref{eq:thm:crit}, we added an $\ga$ in the denominator: this is because it will be more convenient in upcoming computations to express the second order of the critical regime~\eqref{def:gga:crit} in terms of $L(T)\sim\ga T^{1/3}$ instead of $T^{1/3}$.}
	\begin{equation}\label{def:w_T:crit} 
		w_{h,T}(r) \;:=\; -\int_0^{r} \frac{\gs(u)}{\ga^3h^2}\Psi\left(\ga^{3}h^3\frac{\gs'(u)}{\gs(u)}\right) \dd u\;\geq\;0\;,\qquad r\in[0,1]\;.
	\end{equation}
	Let $h>x>0$. Later, we will choose them both to be close to $1$. Then, depending on $L(T)$ and its regime $\anyrg\in\{\supercrit,\subcrit,\crit\}$, we define a pair of barriers by setting, for $t\in[0,T]$,
	\begin{align}\label{def:gga:supercrit}
		\gga_T^\supercrit(t) = \gga_T^{\supercrit,h,x}(t)\;&:=\;v(t/T)\,T + h\,L(T) \int_0^{t/T}(\gs')^-(u)\,\dd u  - x\,\gs(0)  L(T)\;,\\\label{def:gga:subcrit}
		\gga_T^\subcrit(t) = \gga_T^{\subcrit,h,x}(t)\;&:=\;v(t/T)\,T \sqrt{1-\frac{\pi^2}{h^2  L(T)^2}}  - x\,\gs(0) L(T) \;,\\\label{def:gga:crit}
		\gga_T^\crit(t)=
		\gga_T^{\crit,h,x}(t)\;&:=\;v(t/T)\,T - w_{h,T}(t/T)\, L(T) - x\,\gs(0) L(T)\;,
	\end{align}
	and we let $\ol\gga_T^\anyrg(t):=\gga_T^\anyrg(t) + h\gs(t/T) L(T)$ in each regime, so that~\eqref{eq:defxh} holds for $h>x>0$ fixed. One of the core ideas used in the remainder of this paper is that, when started from a single particle, the $N$-BBM is quite similar to a BBM whose particles are killed when reaching the barriers $\gga^\anyrg_T(\cdot)$, $\ol\gga^\anyrg_T(\cdot)$, as soon as their parameters $h>x>0$ are both close to $1$. Those processes are illustrated in each of the three regimes in Figure~\ref{fig:trajectories}.
	
	\begin{figure}[ht]\begin{center}
			\begin{subfigure}[t]{0.55\textwidth}
				\includegraphics[width=\textwidth]{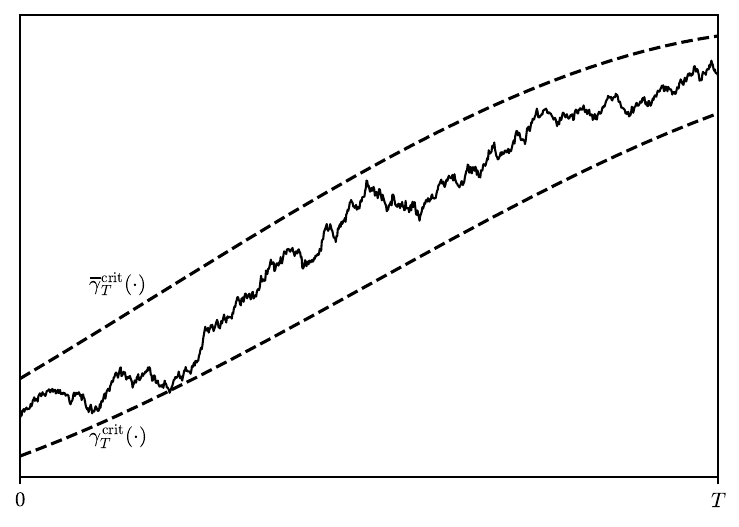}
				\vspace{-9mm}\caption{$L(T)\asymp T^{1/3}$ (crit)}
			\end{subfigure}\vspace{3mm}\\
			\begin{subfigure}[t]{0.45\textwidth}
				\includegraphics[width=\textwidth]{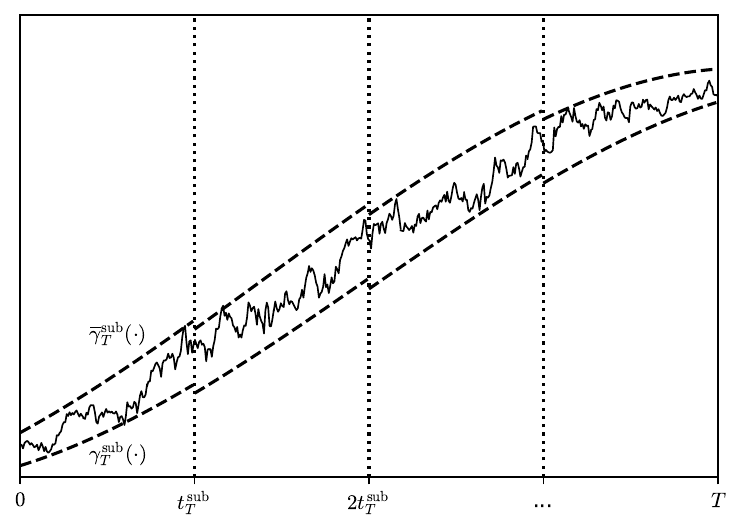}
				\vspace{-7mm}\caption{$L(T)\ll T^{1/3}$ (sub)}
			\end{subfigure}
			\begin{subfigure}[t]{0.45\textwidth}
				\includegraphics[width=\textwidth]{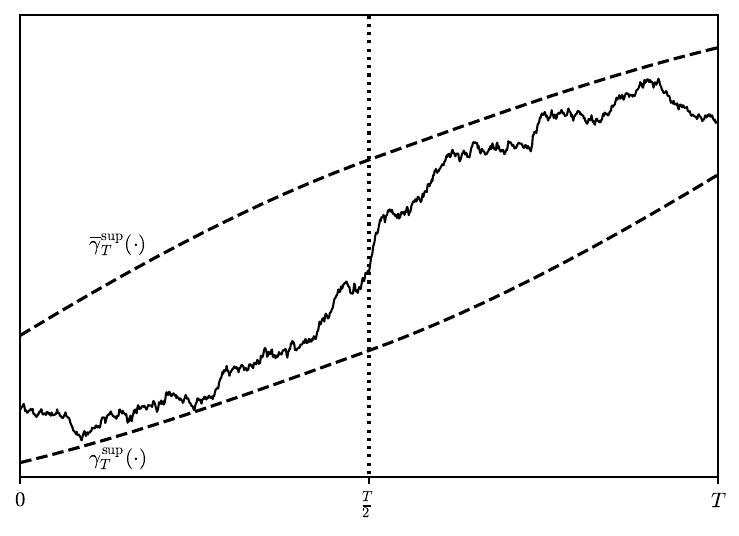}
				\vspace{-7mm}\caption{$L(T)\gg T^{1/3}$ (sup)}
			\end{subfigure}
			\caption{Illustration of the BBM between barriers, which approximates the $N$-BBM. In each regime, we draw the typical trajectory of a single particle that survives until time $T$ (we do not show this rigorously, but this is the heuristic guiding our calculations). (A) In the critical regime, the trajectory resembles that of a Brownian motion constrained to remain between the space-time barriers and, moreover, in a time-inhomogeneous potential of Airy-type. (B) In the sub-critical regime, we compare the $N$-BBM with the BBM between barriers directly only on a time interval of length $t^\subcrit_T$ chosen to be slightly larger than $L(T)^3$, using a block decomposition to recover the process on the full interval $[0,T]$. (C) In the super-critical regime, the surviving trajectories are localized in the vicinity of the lower barrier on each interval where $\gs(\cdot)$ is increasing ($[0,1/2]$ in this picture), and in the vicinity of the upper barrier on each interval where $\gs(\cdot)$ is decreasing ($[1/2,1]$ in this picture).}
			\label{fig:trajectories}
	\end{center}\end{figure}
	
	The following lemma shows that the quantity $m_T^*$, defined in \eqref{eq:defmanyrg}, is indeed well approximated by $\ol\gga^{\anyrg,h,x}_T(T)$ when $h$ and $x$ are close to $1$.
	\begin{lemma}\label{lem:comparison:m:gga}
		Let $\anyrg\in\{\supercrit,\subcrit,\crit\}$. Then,
		\begin{equation}\label{eq:lem:comparison:m:gga}
			\limsup_{\substack{(h,x)\to(1,1),\\ h>x>0}} \, \limsup_{T\to+\infty} \frac1{b^\anyrg_T}\left|m^\anyrg_T - \ol\gga^{\anyrg,h,x}_T(T)\right|\,=\, 0\,.
		\end{equation}
	\end{lemma}
	
	\begin{proof}
		The proof is straightforward in all three regimes. In the super-critical case, one has
		\begin{align*}
			\ol\gga^{\supercrit,h,x}_T(T)\,&=\, v(1)T+hL(T)\int_0^1 \big[(\gs')^+-\gs'\big](u)\,\dd u - x\gs(0)L(T) + h\gs(1)L(T)\\
			&=\, v(1)T + hL(T) \int_0^1 (\gs')^+(u)\,\dd u + (h-x)\gs(0)L(T)\,,
		\end{align*}
		so $\frac1{L(T)}|m^\supercrit_T-\ol\gga^{\supercrit,h,x}_T(T)|\leq \gh^{-1}(|1-h|+|h-x|)$ for all $T\geq0$, which yields the expected result. In the sub-critical regime, one has
		\[
		\ol\gga^{\subcrit,h,x}_T(T)\,=\, v(1)T \sqrt{1-\frac{\pi^2}{h^2  L(T)^2}} + O\left(L(T)\right)\,,
		\]
		where $O(L(T))$ is locally uniform in $h>x>0$. Letting $h$ close to 1 and writing the Taylor expansion $\sqrt{1-y}=1-\frac y2 + o(y)$ as $y\to0$, this yields~\eqref{eq:lem:comparison:m:gga}. Regarding the critical regime, recall that $\Psi$ satisfies $\Psi(-q)=q+\Psi(q)$ for all $q\in\R$. Therefore, $w_{h,T}$ satisfies
		\begin{align*}
			w_{h,T}(1)\,&=\, -\int_0^{1} \frac{\gs(u)}{\ga^3h^2}\Psi\left(\ga^{3}h^3\frac{\gs'(u)}{\gs(u)}\right) \dd v\\
			&=\, -\int_0^{1} \frac{\gs(u)}{\ga^3h^2}\Psi\left(-\ga^{3}h^3\frac{\gs'(u)}{\gs(u)}\right) \dd v + \gs(1)-\gs(0)\,.
		\end{align*}
		Plugging this into~\eqref{def:gga:crit} and recalling that $L(T)\sim\ga T^{1/3}$ in this regime, this straightforwardly concludes the proof.
	\end{proof}
	
	We conclude this subsection by claiming the following useful fact: we may ``tighten'' the barriers on a short time interval (i.e. shorter than $L(T)^3$), by modifying the parameters $h>x>0$. Recall from~\eqref{eq:defllbis} that $\gth(\cdot)$ denotes a function vanishing at $+\infty$.
	
	\begin{lemma}\label{lem:ggachangeparam}
		Let $\anyrg\in\{\supercrit,\subcrit,\crit\}$ and let $t=t(T)$ such that $0\leq t(T)\llbis L(T)^3\wedge T$ for $T$ sufficiently large. Let $h>x>0$ and $h'>x'>0$ such that
		\begin{equation}\label{eq:lem:ggachangeparam:hyp}
			x'\,<\,x\,,\qquad\text{and}\qquad h'-x'\,<\,h-x\,.
		\end{equation}
		Then, there exists $T_0$ such that, for $T\geq T_0$ and $s\in[0,t(T)]$, one has
		\begin{equation}\label{eq:lem:ggachangeparam}
			\gga^{\anyrg,h,x}_T(s)\,\leq\,\gga^{\anyrg,h',x'}_T(s)\,\leq\,\ol\gga^{\anyrg,h',x'}_T(s)\,\leq\,\ol\gga^{\anyrg,h,x}_T(s)\,.
		\end{equation}
		Moreover, $T_0$ is uniform in $\gs\in\gsens$, and locally uniform in $h>x>0$, $h'>x'>0$ which satisfy~\eqref{eq:lem:ggachangeparam:hyp}.
	\end{lemma}
	
	\begin{proof}
		Notice that the assumptions also imply $h'<h$. We prove this claim separately for each $\anyrg\in\{\supercrit,\subcrit,\crit\}$. In the super-critical case, 
		one has for all $s\in[0,T]$,
		\begin{align*}
			\gga^{\supercrit,h',x'}_T(s)-\gga^{\supercrit,h,x}_T(s)\,&=\, (h'-h)\,L(T) \int_0^{t/T}(\gs')^-(u)\,\dd u  - (x'-x)\gs(0)L(T) \\
			&\geq\, -(h-h')\, \gh^{-1} \gth(T)L(T) + (x-x')\gh L(T)\,\geq\,0\,,
		\end{align*}
		where the last inequality holds for $T$ larger than some $T_0$ locally uniform in $x,x',h,h'$. Moreover, one can easily check that
		\begin{equation}\label{eq:olggasup:alt}
			\ol\gga^{\supercrit,h,x}_T(s)\,=\, v(t/T)T + h\,L(T) \int_0^{t/T}(\gs')^+(u)\,\dd u  + (h-x)\,\gs(0)  L(T)\,,
		\end{equation}
		for all $s\in[0,T]$, $h>x>0$. Hence, one has for all $s\in[0,T]$.
		\begin{align*}
			&\ol\gga^{\supercrit,h,x}_T(s)-\ol\gga^{\supercrit,h',x'}_T(s)\\
			&\qquad=\, (h-h')\,L(T) \int_0^{t/T}(\gs')^+(u)\,\dd u  + [(h-x)-(h'-x')]\gs(0)L(T)\,\geq\,0\,,
		\end{align*}
		which concludes the proof in the super-critical regime. 
		
		Regarding the sub-critical case, we have for $s\in[0,T]$,
		\begin{align}\label{eq:lem:ggachangeparam:1}
			&\gga^{\subcrit,h',x'}_T(s)-\gga^{\subcrit,h,x}_T(s)\\
			\notag &\qquad=\, (x-x')\gs(0)L(T) +\left[\sqrt{1-\frac{\pi^2}{h'{}^2L(T)^2}}-\sqrt{1-\frac{\pi^2}{h^2L(T)^2}}\right]v(s/T)T\,,
		\end{align}
		and a direct Taylor expansion gives as $L(T)\to+\infty$,
		\begin{equation}\label{eq:lem:ggachangeparam:2}
			\sqrt{1-\frac{\pi^2}{h'{}^2L(T)^2}}-\sqrt{1-\frac{\pi^2}{h^2L(T)^2}}\,=\,-\frac{\pi^2(h^2-h'{}^2)}{2(hh'L(T))^2}+O(L(T)^{-4})\,.
		\end{equation}
		Recall that $t(T)\leq \gth(T)L(T)^3\leq T$ in the sub-critical regime: thus, one has for $s\in[0,t(T)]$,
		\[v(s/T)T \,=\, T\int_0^{s/T} \gs(u)\,\dd u \,\leq \, \gh^{-1}t(T)\,\leq\, \gh^{-1}\gth(T) L(T)^3\,. \]
		Since $x>x'$, we deduce that for $T$ sufficiently large, the second term in the r.h.s. of~\eqref{eq:lem:ggachangeparam:1} is larger than $-\frac12(x-x')\gh L(T)$, uniformly in $\gs\in\gsens$ and $s\in[0,t(T)]$, locally uniformly in $x,x',h,h'$; which is one of the expected results. On the other hand we have for all $s\in[0,T]$,
		\begin{align*}
			&\ol\gga^{\subcrit,h,x}_T(s)-\ol\gga^{\subcrit,h',x'}_T(s)\\
			&\quad=\, (h-h')\gs(s/T)L(T)+(x'-x)\gs(0)L(T) \\
			&\qquad\qquad\qquad\qquad\qquad\qquad\qquad\qquad +\left[\sqrt{1-\frac{\pi^2}{h^2L(T)^2}}-\sqrt{1-\frac{\pi^2}{h'{}^2L(T)^2}}\right]v(s/T)T\\
			&\quad=\, [(h-x)-(h'-x')]\gs(s/T)L(T)+(x-x')[\gs(s/T)-\gs(0)]L(T) -O(\gth(T)L(T))\,,
		\end{align*}
		where we used~\eqref{eq:lem:ggachangeparam:2}. Recalling~\eqref{eq:gsapprox} and that $t(T)\leq\gth(T)T$, the second term above is larger than $-O(\gth(T)L(T))$ for $T$ large. Therefore, the r.h.s. above is larger than $\tfrac12[(h-x)-(h'-x')]\gh L(T)$ for $T$ sufficiently large: this concludes the proof in the sub-critical regime.
		
		We finally turn to the critical case. Recall~\eqref{def:w_T:crit} and that $\Psi$ is continuous, hence
		\[
		\sup_{r\in[h',h]}\,\sup_{s\in[0,t(T)]} |w_{r,T}(s/T)|\,\leq\, \frac{t(T)}{T} \frac{\gh^{-1}}{\ga^3h'{}^2} \,\sup\Big\{|\Psi(q)|,\,q\in[-\ga^3h'{}^3\gh^{-2},\ga^3h^3\gh^{-2}]\Big\}\,.
		\]
		Since we assumed $t(T)\leq\gth(T) (L(T)^3\wedge T)$ for $T$ large, there exists $C>0$, uniform in $\gs\in\gsens$ and locally uniform in $h,h'$, such that for $T$ sufficiently large, one has $|w_{r,T}(s/T)|\leq C\gth(T)$ for all $r\in[h',h]$ and $s\in[0,t(T)]$. In particular, this implies
		\[
		\gga^{\crit,h',x'}_T(s)-\gga^{\crit,h,x}_T(s)\,\geq\, (x-x')\gs(0)L(T) -2C\gth(T)L(T)\,,
		\]
		and
		\[
		\ol\gga^{\crit,h,x}_T(s)-\ol\gga^{\crit,h',x'}_T(s)\,\geq\, (h-h')\gs(s/T)L(T)+(x'-x)\gs(0)L(T) -2C\gth(T)L(T)\,,
		\]
		for all $s\in[0,t(T)]$. Assuming $T$ is sufficiently large and reproducing the arguments from the sub-critical case (we do not write them again), this completes the proof of the lemma.
	\end{proof}

	\subsection{First and second moment formulae}\label{sec:prelim:toolbox} 
	We now introduce several exact formulae which will be used throughout Section~\ref{sec:moments} to estimate some moments of $|A_{T,I}^\anyrg(\cdot)|$ and $R_{T}^\anyrg(\cdot,\cdot)$.
	
	First, several results from Section~\ref{sec:moments} rely on the first moment formula for branching Markov processes, often called ``Many-to-one lemma'', as well as Girsanov's theorem: we condense them in the following statement. Recall from Section~\ref{sec:coupling:construct} that $\bbE_\mu$, $\bbP_\mu$ denote the expectation and law of the time-inhomogeneous BBM started from some configuration $\mu\in\meas$ and that $\bE_x$, $\bP_x$ denote expectation and law of a single time-inhomogeneous Brownian motion $(B_t)_{t\ge0}$.
	\begin{lemma}[First moment formula]\label{lem:MtOGirs}Let $\anyrg\in\{\supercrit,\subcrit,\crit\}$, $h>x>0$ and $I\subset [0,h]$ an interval.
		Then, one has
		\begin{align}\label{eq:MtOGirs:A}
			&\bbE_{\gd_0}\big[|A_{T,I}^{\anyrg}(t)|\big] = e^{\frac t2}\,
			\bE_{x\gs(0)L(T)}\Bigg[\ind_{\left\{\forall s\leq t,\, \frac{B_s}{\gs(s/T) L(T)}\in[0,h] ; \frac{B_t}{\gs(t/T)L(T)}\in I\right\}} \exp\Bigg(-\frac{\gga_T^\anyrg{}'(t)B_t}{\sigma^2(t/T)} \\
			\notag &\qquad\qquad\quad  
			+ \frac{\gga_T^\anyrg{}'(0)}{\sigma(0)}xL(T)  + \int_0^t \left.\frac{\partial}{\partial u}\left(\frac{\gga_T^\anyrg{}'(u)}{\sigma^2(u/T)}\right)\right|_{u=s} B_s \,\dd s - \int_0^t \frac{(\gga^\anyrg_T{}'(s))^2}{2\gs^2(s/T)}\dd s \Bigg)\Bigg].
		\end{align}
		Moreover, letting $H_0(Y):=\inf\{t\geq0,\,Y_t=0\}$ for $(Y_t)_{t\geq 0}\in\cC^0(\R_+)$, one has
		\begin{align}\label{eq:MtOGirs:R}
			\bbE_{\gd_0}\big[R^\anyrg_T(0,t)\big] &= 
			\bE_{(h-x)\gs(0)L(T)}\Bigg[e^{H_0(B)/2} \, \ind_{\{H_0(B)\leq t\}}\,  \ind_{\left\{\forall s\leq H_0(B),\, \frac{B_s}{\gs(s/T) L(T)}\in[0,h]  \right\}} \\
			\notag &\;\:\times \exp\Bigg(- \frac{\ol\gga_T^\anyrg{}'(0)}{\sigma(0)}(h-x)L(T)  - \int_0^{H_0(B)} \left.\frac{\partial}{\partial u}\left(\frac{\ol\gga_T^\anyrg{}'(u)}{\sigma^2(u/T)}\right)\right|_{u=s} B_s \,\dd s \\
			\notag &\qquad\qquad\qquad\qquad\qquad\qquad\qquad\qquad\qquad - \int_0^{H_0(B)} \frac{(\ol\gga^\anyrg_T{}'(s))^2}{2\gs^2(s/T)}\dd s \Bigg)\Bigg].
		\end{align}
	\end{lemma}
	\begin{remark}
		$(i)$ The terms involving derivatives of the form $\frac{\partial}{\partial u}\left(\cdots\right)$ appearing in the r.h.s. of~(\ref{eq:MtOGirs:A}--\ref{eq:MtOGirs:R}) are to be interpreted as being defined Lebesgue-almost everywhere. This matters only in the super-critical case, where the definition of the barrier functions involve $\sigma^-(u)$, whose derivative may be discontinuous at a finite number of points.
		%
		
		$(ii)$ Let us mention that~\eqref{eq:MtOGirs:R} has an analogous formulation in terms of $\gga^\anyrg_T$ instead of $\ol\gga^\anyrg_T$, which involves the hitting time of the curve $(h\gs(s/T)L(T))_{s\in[0,T]}$. We decided to stick with the expression~\eqref{eq:MtOGirs:R} in the lemma, since it is used more often in this paper.\end{remark}
	
	\begin{proof}
		Let us start with~\eqref{eq:MtOGirs:A}. The Many-to-one lemma for branching Markov processes, see e.g.~\cite[Theorem~4.1]{INW69Part3} or \cite[Theorem 2.1]{Saw76} yields,
		\begin{align}
			\notag \bbE_{\gd_0}\big[|A_{T,I}^\anyrg(t)|\big] \,&=\, \bbE_{\gd_0}\Bigg[\sum_{u\in\cN_t} \ind_{\big\{\forall s\leq t,\,\frac{X_u(s)-\gga^\anyrg_T(s)}{\gs(s/T)L(T)} \in [0,h]\,;\, \frac{X_u(t)-\gga^\anyrg_T(t)}{\gs(t/T)L(T)} \in I\big\}}\Bigg]\\
			\notag  &=\, \bbE_{\gd_0}\big[|\cN_t|\big] \times \bP_0\bigg(\forall s\leq t,\, \frac{B_s-\gga^\anyrg_T(s)}{\gs(s/T)L(T)} \in [0,h]\,;\, \frac{B_t-\gga^\anyrg_T(t)}{\gs(t/T)L(T)} \in I\bigg)\,,
		\end{align}
		and, by Girsanov's theorem,
		\begin{align}\label{eq:MtOGirs:proof:1}
		&\bP_0\bigg(\forall s\leq t,\, \frac{B_s-\gga^\anyrg_T(s)}{\gs(s/T)L(T)} \in [0,h]\,;\, \frac{B_t-\gga^\anyrg_T(t)}{\gs(t/T)L(T)} \in I\bigg)\\
		\notag & = 
		\bE_{-\gga^\anyrg_T(0)}\Bigg[ \ind_{\big\{\forall s\leq t,\,\frac{B_s}{\gs(s/T)L(T)} \in [0,h]\,;\, \frac{B_t}{\gs(t/T)L(T)} \in I\big\}} \,e^{-\int_0^t \frac{\gga^\anyrg_T{}'(s)}{\gs^2(s/T)}\dd B_s-\int_0^t \frac{(\gga^\anyrg_T{}'(s))^2}{2\gs^2(s/T)}\dd s}\Bigg].
		\end{align}
		Recall that $\bbE_{\gd_0}\big[|\cN_t|\big]=e^{\frac t2}$ under our assumptions, and write with an integration by parts, for $t\in[0,T]$,
		\begin{equation}\label{eq:MtOGirs:proof:2}
			\int_0^t \frac{\gga_T^\anyrg{}'(s)}{\sigma^2(s/T)}\dd B_s = \frac{\gga_T^\anyrg{}'(t)B_t}{\sigma^2(t/T)} - \frac{\gga_T^\anyrg{}'(0)B_0}{\sigma^2(0)} - \int_0^t \left.\frac{\partial}{\partial u}\left(\frac{\gga_T^\anyrg{}'(u)}{\sigma^2(u/T)}\right)\right|_{u=s} B_s \,\dd s\,.
		\end{equation}
		Since~\eqref{eq:defxh} implies $\gga^\anyrg_T(0)=-x\gs(0)L(T)$, plugging this into~\eqref{eq:MtOGirs:proof:1} yields~\eqref{eq:MtOGirs:A}.
		
		Regarding~\eqref{eq:MtOGirs:R}, the Many-to-one lemma gives
		\[
		\bbE_{\gd_0}\big[R^\anyrg_T(0,t)\big]\,=\,\bE_{0}\left[e^{H_0(\ol\gga_T^\anyrg - B)/2} \ind_{\{H_0(\ol\gga_T^\anyrg - B)\leq t\}} \ind_{\big\{\forall s\leq H_0(\ol\gga_T^\anyrg - B),\,\frac{\ol\gga_T^\anyrg(s)-B_s}{\gs(s/T)L(T)}\leq h\big\}}\right].
		\]
		Then, applying Girsanov's theorem and recalling that $(B_s)_{s\leq t}$ under $\bbP_x$ as the same law as $(-B_s)_{s\leq t}$ under $\bbP_{-x}$, $x\in\R$, one obtains
		\begin{align*}
		&\bbE_{\gd_0}\big[R^\anyrg_T(0,t)\big] \,=\, \bE_{\ol\gga^\anyrg_T(0)}\bigg[e^{H_0(B)/2} \ind_{\{H_0(B)\leq t\}} \ind_{\big\{\forall s\leq H_0(B),\frac{B_s}{\gs(s/T)L(T)}\leq h\big\}}\\
		&\qquad\qquad\qquad\qquad\qquad\qquad\quad \times e^{\int_0^{H_0(B)} \frac{\ol\gga^\anyrg_T{}'(s)}{\gs^2(s/T)}\dd B_s-\int_0^{H_0(B)} \frac{(\ol\gga^\anyrg_T{}'(s))^2}{2\gs^2(s/T)}\dd s}\bigg].
		\end{align*}
		Recalling that $\ol\gga^\anyrg_T(0)=(h-x)\gs(0)L(T)$ by~\eqref{eq:defxh}, and replacing $\gga^\anyrg_T$ with $\ol\gga^\anyrg_T$ in~\eqref{eq:MtOGirs:proof:2}, this yields~\eqref{eq:MtOGirs:R} and finishes the proof of the lemma.
	\end{proof}
	
	
	We also present the ``Many-to-two lemma'' below, which will be used in upcoming second moment computations. For $y,w\in\R$, $0\leq s\leq t\leq T$, we let
	\begin{align}\label{eq:defG:allregimes}
		&G^\anyrg(y,w,s,t) \,\dd w\\\notag &\quad:=\; \bbE_{\gd_{(s,\gga_T^\anyrg(s)+y\gs(s/T)L(T))}}\left[\#\left\{u\in\cN_t\;\middle|\; \forall\,r\in[s,t]:\,\begin{aligned}&X_u(r)\in \big[\gga_{T}^\anyrg(r), \ol\gga_{T}^\anyrg(r)\big], \\ &\tfrac{X_u(t)-\gga_{T}^\anyrg(t)}{\gs(t/T)L(T)}\in \dd w \end{aligned}\right\}\right],
	\end{align}
	denote the expected number of descendants in the BBM of a single particle at time-space location $(s,\gga^\anyrg_T(s)+y\gs(s/T)L(T))$, whose path remain between the barriers $\gga^\anyrg_T(\cdot)$, $\ol\gga^\anyrg_T(\cdot)$ until time $t$, at which point it reaches an infinitesimal neighborhood of $\gga^\anyrg_T(t)+w\gs(t/T)L(T)$ (notice that it is zero unless $y,w\in[0,h]$). In particular, one has 
	\begin{align}
		\label{eq:A_G}
		\bbE_{\gd_0}[|A_{T,I}^\anyrg(t)|]=\int_I G^\anyrg(x,w,0,t) \dd w.
	\end{align}
	\begin{lemma}[``Many-to-two lemma'', see Theorem~4.15 in \cite{INW69Part3} or Theorem~2.2 in \cite{Saw76}]\label{lem:MtT}
		Let $\anyrg\in\{\supercrit,\subcrit,\crit\}$, $t\leq T$, and $z\in[0,h)$. 
		Then, one has
		\begin{align}
			&\bbE_{\gd_0}\big[|A_{T,z}^\anyrg(t)|^2\big] -  \bbE_{\gd_0}\big[|A_{T,z}^\anyrg(t)|\big] \\
			\notag &\quad= \gb_0\bbE[\xi(\xi-1)] \int_0^t \dd s \int_0^{h} G^\anyrg(x,y,0,s) \bigg(\int_{z}^{h} G^\anyrg(y,w,s,t) \, \dd w\bigg)^2 \dd y \;.
		\end{align}
	\end{lemma}
	
	Finally, in order to use the formulae from the previous lemmas, we rely on a result on the density of a Brownian motion killed outside an interval. Recall that the standard, time-homogeneous Brownian motion is denoted by $(W_s)_{s\geq0}$. The following result can be found for example in e.g.~\cite[Part II.1, Eq.~1.15.8]{BS02} or~\cite[(7.8--7.10)]{Mai16}. 
	
	\begin{lemma}[(7.8--7.10) in~\cite{Mai16}]\label{lem:mai16}
		For $t>0$, $h>0$, and $x,z\in[0,h]$, one has,
		\begin{equation}\label{eq:mai16}\begin{aligned}
				\bP_{x}\big(\forall s\leq t,\,W_s\in[0,h] \,;\, W_{t}\in \dd z\big) \;&=\; \frac2h\sum_{n=1}^\infty \exp\Big(-\frac{\pi^2}{2h^2}n^2 t\Big)\sin\Big(\frac{\pi nx}{h}\Big)  \sin\Big(\frac{\pi nz}{h}\Big) \,\dd z \\
				&=\; \frac2h \exp\Big(-\frac{\pi^2}{2h^2}t\Big)\sin\Big(\frac{\pi x}{h}\Big)  \sin\Big(\frac{\pi z}{h}\Big) \big(1+o(1)\big) \,\dd z,
			\end{aligned}
		\end{equation}
		where $o(1)$ is a term vanishing as $t/h^2\to+\infty$, uniformly in $x,z$ (see~\cite[(7.9)]{Mai16} for an explicit expression).
	\end{lemma}

	
	\section{Moment estimates on the BBM with barriers}\label{sec:moments}
	Recall~\eqref{eq:defA:allregimes} and~\eqref{eq:defR}: in this section we compute first and second moment estimates of $|A_{T,I}^\anyrg(t)|$, as well as first moment estimates of $R^\anyrg_T(0,t)$, in all three regimes (super-critical, sub-critical and critical). These estimates will be used throughout Sections~\ref{sec:LB} and~\ref{sec:UB} below, in order to prove Propositions~\ref{prop:main:LB} and~\ref{prop:main:UB} respectively. Let us warn the reader that the upcoming proofs contain a lot of bookkeeping (especially for the first moment estimates), mainly due to the fact that we need to obtain bounds which are uniform in certain ranges of values for $t$, $\gs(\cdot)$ and other parameters.

	\subsection{Super-critical case}  \label{sec:moments:supercrit}
	In this section we assume $T^{1/3}\ll L(T)\ll T$, and $\gs\in\gsens^\supercrit$ for some $\gh>0$ (recall~(\ref{eq:defgsens}--\ref{eq:defgsens:anyrg})); in particular, there exist $n\ge 1$ and $0=u_0<u_1<\ldots<u_n$ such that $\gs$ is monotonic on $[u_{i-1},u_i]$ and $u_i-u_{i-1} \ge \gh$, $1\leq i\leq n$. Recall that $\gga^\supercrit_T$ is defined in~\eqref{def:gga:supercrit}, and that $\ol\gga^\supercrit_T$ is such that~\eqref{eq:defxh} holds. We begin this section by computing first moment estimates for $A_{T,I}^{\supercrit}(t)$, $t\in[0,T]$, $I\subset[0,h]$ an interval. Finally, recall~\eqref{eq:defllbis}, and that we fixed some arbitrary $\gth:\R_+\to\R^*_+$ that vanishes at $+\infty$ such that~\eqref{eq:gthasymp} holds. 
	
	\begin{proposition}[First moment, Super-critical]\label{prop:1stmom:supercrit}
		Let $h>0$. As $T\to+\infty$, one has
		\begin{equation}\label{eq:prop:1stmom:supercrit:UB} 
			\bbE\big[|A_{T,I}^{\supercrit}(t)|\big] \;\leq\; e^{(x-\inf I)  L(T) + o(L(T))}\,,
		\end{equation}
		uniformly in $\gs\in\gsens^\supercrit$, $x\in[0,h]$ and $I\subset[0,h]$ an interval.
		Moreover for any $h>0$, $T_0\geq0$ fixed, one has as $T\to+\infty$, for every interval $I\subset [0,h]$ with non-empty interior,
		\begin{equation}\label{eq:prop:1stmom:supercrit:LB} 
			\bbE\big[|A_{T,I}^{\supercrit}(t)|\big] \;\geq\; e^{(x-\inf I)  L(T) +o(L(T))}\,,
		\end{equation}
		locally uniformly in $x\in(0,h)$, $\inf I\in[0,h)$ and $\sup I-\inf I\in (0,h]$, uniformly in $\gs\in\gsens^\supercrit$, and uniformly in $t=t(T)\in[0,T]$ which satisfies $L(T)\llbis t(T)\leq T$ for $T\geq T_0$.
	\end{proposition}
	
	\begin{remark}\label{rem:1stmom}
		$(i)$ Here, ``locally uniformly'' means that for any $\eps>0$, there exists $o(L(T))\in\R$ such that $|o(L(T))|\ll L(T)$ as $T\to+\infty$ and~\eqref{eq:prop:1stmom:supercrit:LB} holds uniformly in $x\in[\eps,h-\eps]$, $\inf I\in[0,h-4\eps]$ and $\sup I-\inf I>4\eps$. Moreover, let us stress that ``uniformly'' in $\gs(\cdot)$ and $t(T)$ means that this error term does not depend on those, as long as $\gh$, $\gth(\cdot)$ and $T_0$ are fixed. In the remainder of the paper we will not write $T_0$ in similar statements, but rather ``$L(T)\llbis t(T)\leq T$ for $T$ sufficiently large'', to lighten the phrasing.
		
		$(ii)$ Finally, let us mention that, in~\eqref{eq:prop:1stmom:supercrit:LB}, the assumption $t(T)\ggbis L(T)$ is necessary since a branching process (without selection) started from one individual needs a time $\approx L(T)$ to grow to a population of size $e^{O(L(T))}$.
	\end{remark}
	
	Let us briefly present the idea of the proof: the main difficulty lies in showing the lower bound~\eqref{eq:prop:1stmom:supercrit:LB}. We obtain it by restricting $A_{T,I}^{\supercrit}(t)$ to the subset of particles that follow a specific strategy: when $\gs(\cdot)$ is increasing (resp. decreasing) on an interval $[u_{i-1},u_{i}]$, $i\leq n$, we only consider particles whose trajectories remain close to the lower (resp. upper) barrier on $[u_{i-1}T,u_{i}T]$. More precisely, for some $\tau_T$, $h_T$ which are defined below, we consider particles that stay within a distance $O(h_T)$ of the lower/upper barrier on the time interval $[u_{i-1}T+\tau_T,u_{i}T-\tau_T]$, see Figure~\ref{fig:supercrit_strategy}. The following lemma will be used to bound from below the probability that a single particle follows this strategy on one interval $[u_{i-1}T,u_{i}T]$, $i\leq n$. Its proof is delayed afterwards.
	\begin{figure}[ht]
		\centering
		\includegraphics[width=0.75 \textwidth]{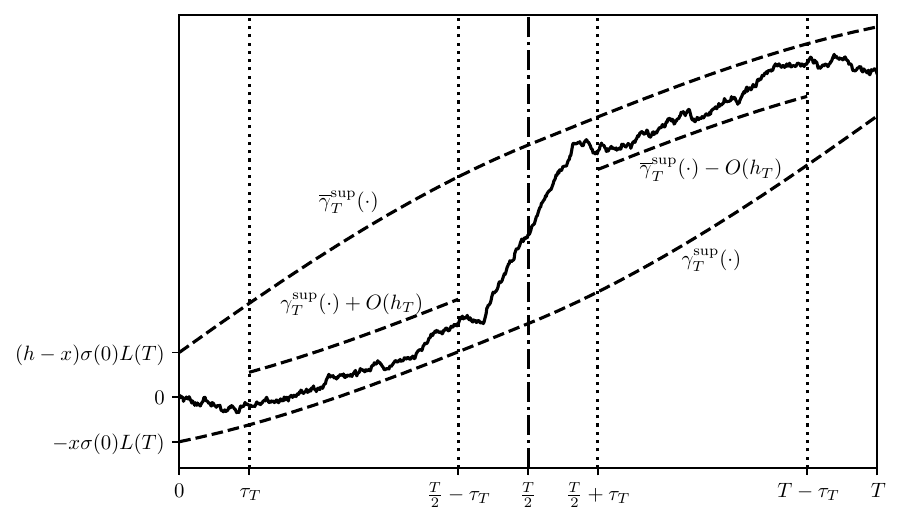}
		\caption{Illustration of the strategy of a particle contributing to the lower bound for the first moment in Proposition~\ref{prop:1stmom:supercrit} (super-critical case). In this figure we consider $t=T$ and a function $\gs(\cdot)$ increasing on $[0,1/2]$, decreasing on $[1/2,1]$. We consider separately each time interval on which $\gs(\cdot)$ is monotone. Let $\tau_T$ and $h_T$ be as in the proof of Lemma~\ref{lem:1stmom:supercrit}. On $[\tau_T,T/2-\tau_T]$, we introduce an additional killing barrier at distance $O(h_T)$ of $\gga_T^\supercrit(\cdot)$, and constrain particles to remain below this barrier. On $[T/2+\tau_T,T-\tau_T]$, the additional barrier is located at distance $O(h_T)$ of $\ol\gga_T^\supercrit(\cdot)$, and particles are constrained to remain above it. 
		}
		\label{fig:supercrit_strategy}
	\end{figure}

	\begin{lemma}\label{lem:1stmom:supercrit}
		Assume $T^{1/3}\ll L(T)\ll T$, and let $h,\gh>0$, $\eps\in(0,h/2)$. Define,
		\begin{align*}
		\gU:=\big\{(x,t,\gs,I)\,\big|\, x\in[\eps,h-\eps]\,;\, t\in[\gth^{-1}(T)L(T),T]\,;\, \gs\in\gsens\,;\,&\\
		I=[a,b]\subset[0,h],\, b-a\geq 2\eps\big\}.&
		\end{align*}
		 Then, as $T\to+\infty$, one has for every $0\le C<\infty$,
		\begin{align}\label{eq:lem:1stmom:supercrit}                    
			\frac1{L(T)}\inf_{(x,t,\gs,I)\in\gU} \log \bE_{x\gs(0)L(T)}\left[
			e^{-\frac C T \int_0^t B_s \dd s}
			\ind_{\big\{\forall s\in[0,t],\,\frac{B_s}{\gs(s/T)L(T)}\in[0,h]\,;\,\frac{B_t}{\gs(t/T)L(T)}\in I\big\}}\right]& \\
			\notag \underset{T\to+\infty}{\longrightarrow}\; 0\,.&
		\end{align}
	\end{lemma}
		
	\begin{proof}[Proof of Proposition~\ref{prop:1stmom:supercrit} assuming Lemma~\ref{lem:1stmom:supercrit}]
		Let $\gs\in\gsens^\supercrit$, with $u_0,\ldots,u_n$ as above. Let us assume $h>x>0$, otherwise the l.h.s. of~\eqref{eq:prop:1stmom:supercrit:UB} is zero and the proof is immediate. Recall Lemma~\ref{lem:MtOGirs}, in particular~\eqref{eq:MtOGirs:A}. Notice that~\eqref{def:gga:supercrit} implies,
		\[
		\gga^\supercrit_T{}'(s)\,=\,\gs(s/T) + h\,\frac{L(T)}{T}(\gs'(s/T))^-\,,\qquad \forall\, s\in[0,T]\,,
		\]
		so one has for $t\in[0,T]$,
		\begin{align}
			\notag \bbE\big[|A_{T,I}^{\supercrit}(t)|\big] &= e^{\frac t2}\,\bE_{x\gs(0)L(T)}\bigg[ 1_{\big\{\forall s\in[0,t],\,\frac{B_s}{\gs(s/T)L(T)} \in [0,h]\;;\; \frac{B_t}{\gs(t/T)L(T)} \in I\big\}} \\
			&\notag\;\;\times\exp\bigg( -\frac{\gga^\supercrit{}'(t) B_t}{\gs^2(t/T)}+\frac{\gga^\supercrit{}'(0)}{\gs(0)}xL(T) + \int_0^t \left.\frac{\partial}{\partial u}\left(\frac{\gga_T^\supercrit{}'(u)}{\sigma^2(u/T)}\right)\right|_{u=s} B_s \,\dd s \\
			&\notag\qquad\qquad\qquad\qquad\qquad\qquad\qquad\qquad\qquad\qquad\qquad- \int_0^t \frac{(\gga^\supercrit{}'(s))^2}{2\gs^2(s/T)}\dd s \bigg)  \bigg]\\
			\label{eq:prf:1stmom:supercrit:gsblock:1} 
			&= e^{xL(T)+O(L(T)^2/T)}\,\bE_{x\gs(0)L(T)}\bigg[ 1_{\big\{\forall s\in[0,t],\,\frac{B_s}{\gs(s/T)L(T)} \in [0,h]\;;\; \frac{B_t}{\gs(t/T)L(T)} \in I\big\}} \\
			&\notag\quad\times\exp\left( -\frac{B_t}{\gs(t/T)} -\frac1T \int_0^t \frac{\gs'(s/T)}{\gs^2(s/T)}B_s \, \dd s - h\frac{L(T)}T \int_0^t \frac{(\gs')^-(s/T)}{\gs(s/T)}\dd s \right)  \bigg].
		\end{align}
		
		\noindent\emph{Proof of~\eqref{eq:prop:1stmom:supercrit:UB} (upper bound).}  
		Using that $\gs'=(\gs')^+-(\gs')^-$, one observes that on the event $\{\forall s\in[0,t], B_s\le h\gs(s/t)L(T)\}$, one has almost surely,
		\[
		-\frac1T \int_0^t \frac{\gs'(s/T)}{\gs^2(s/T)}B_s \dd s \le -\frac1T \int_0^t \frac{(\gs')^+(s/T)}{\gs^2(s/T)}B_s \dd s + h\frac{L(T)}T \int_0^t \frac{(\gs')^-(s/T)}{\gs(s/T)} \dd s \,,
		\]
		so \eqref{eq:prf:1stmom:supercrit:gsblock:1} yields,
		\begin{align*}
			&\bbE\big[|A_{T,I}^{\supercrit}(t)|\big] \;\leq \; e^{(x-\inf(I))L(T) + O(L(T)^2/T)} \\
			&\quad\times\bE_{x\gs(0)L(T)}\bigg[ 1_{\big\{\forall s\in[0,t],\frac{B_s}{\gs(s/T)L(T)} \in [0,h]\,;\, \frac{B_t}{\gs(t/T)L(T)} \in I\big\}}\exp\left( -\frac1T \int_0^t \frac{(\gs')^+(s/T)}{\gs^2(s/T)}B_s \dd s  \right)  \bigg] \,,
		\end{align*}
		and the latter expectation is bounded by 1, which proves \eqref{eq:prop:1stmom:supercrit:UB}.\bigskip
		
		\noindent\emph{Proof of~\eqref{eq:prop:1stmom:supercrit:LB} (lower bound).} 
		Let $\eps>0$ be sufficiently small such that $4\eps \leq \sup(I)-\inf(I)$ and $x\in [\eps,h-\eps]$. Define
		\begin{equation}\label{eq:prf:1stmom:supercrit:gsblock:Igh}
			I_f\,:=\, [\inf(I),\inf(I)+4\eps]\,\subset\, I.
		\end{equation}
		By constraining $(B_s/\gs(s/T)L(T))_{s\geq0}$ to end in $I_{f}$ at time $t$, one deduces from~\eqref{eq:prf:1stmom:supercrit:gsblock:1} that,
		\begin{align}\label{eq:prf:1stmom:supercrit:gsblock:LBdecomp}
			\bbE\big[|A_{T,I}^{\supercrit}(t)|\big] &\geq e^{(x-\inf I +4\eps)L(T)+O(L(T)^2/T)}\\
			&\notag\quad\times\bE_{x\gs(0)L(T)}\bigg[ 1_{\big\{\forall s\in[0,t],\,\frac{B_s}{\gs(s/T)L(T)} \in [0,h]\;;\; \frac{B_t}{\gs(t/T)L(T)} \in I_f\big\}} \\
			&\notag\quad\times\exp\left( -\frac1T \int_0^t \frac{\gs'(s/T)}{\gs^2(s/T)}\,B_s \, \dd s - h\frac{L(T)}T \int_0^t \frac{(\gs'(s/T))^-}{\gs(s/T)}\dd s \right)  \bigg].
		\end{align}
		In order to bound the right-hand side of \eqref{eq:prf:1stmom:supercrit:gsblock:LBdecomp} from below, we wish to decompose the trajectory at the times $u_0T,\ldots,u_nT$ and apply Lemma~\ref{lem:1stmom:supercrit} to each part. We have to be a bit careful in doing this, in order to ensure that the final time $t$ is not too close to one of these times. Also, we have to constrain the trajectory to return to certain intervals at each time. More precisely, set $i_{\max}:=\max\{i\leq n:  u_iT\le t-\gth^{-1}(T)L(T)\}$ and define intervals and time steps as follows:
		\begin{align*}
			&I_i \,:=\, [\eps,h-\eps] \ni x,\ 0 \leq i \le i_{\max}  & &t_i \,:=\,u_iT\wedge (t-\gth^{-1}(T)L(T)),\ 0 \leq i \le i_{\max}\\
			&I_{i_{\max}+1} \,:=\, I_f & &t_{i_{\max}+1} \,:=\,t
		\end{align*}
		Note that $t_{i+1}-t_i \ge \gth^{-1}(T)L(T)$ by construction, for every $i=0,\ldots,i_{\max}$. 
		
		We now set for each $i=0,\ldots,i_{\max}$,
		\begin{align}\label{eq:prf:1stmom:supercrit:gsblock:2bis}
			E_i \,:=\, \inf_{y\in I_i}& \bE_{(t_i,y\gs(t_i/T)L(T))}\Bigg[ 1_{\big\{\forall s\in[t_i,t_{i+1}],\,\frac{B_s}{\gs(s/T)L(T)} \in [0,h]\;;\; \frac{B_{t_{i+1}}}{\gs(t_{i+1}/T)L(T)} \in I_{i+1}\big\}} \\
			&\notag\times \exp\left( -\frac1T \int_{t_i}^{t_{i+1}} \frac{\gs'(s/T)}{\gs^2(s/T)}\,B_s \, \dd s - h\frac{L(T)}T \int_{t_i}^{t_{i+1}} \frac{(\gs'(s/T))^-}{\gs(s/T)}\dd s \right) \Bigg].
		\end{align}
		Equation~\eqref{eq:prf:1stmom:supercrit:gsblock:LBdecomp}  yields
		\begin{align}
			\label{eq:prf:1stmom:supercrit:gsblock:prodEi}
			\bbE\big[|A_{T,I}^{\supercrit}(t)|\big] \geq e^{(x-\inf I +4\eps)L(T)+O(L(T)^2/T)} \times \prod_{i=0}^{i_{\max}} E_i.
		\end{align}
		It remains to show that, uniformly in $\gs\in\gsens^\supercrit$ and $t$ such that $L(T)\llbis t\leq T$, we have
		\begin{align}
			\label{eq:toshowEi}
			E_i = \exp(o(L(T))),\quad i=0,\ldots,i_{\max}.
		\end{align}
		We now want to distinguish between two cases, depending on $i=0,\ldots,i_{\max}$ through the sign of $\gs'(\cdot/T)$ on the interval $[t_i,t_{i+1}]$. Note that by the assumption on $\gs$ and by the definition of the $t_i$'s, $\gs'(\cdot/T)$ indeed does not change sign on $[t_i,t_{i+1}]$, unless $t_i = t - \gth^{-1}(T)L(T)$ (and therefore, $t_{i+1} = t$). In this case, the sign may change, leading to a third case to consider.
		
		\bigskip
		
		\emph{Case 1: $\gs'(\cdot/T)$ is non-negative on $[t_i,t_{i+1}]$}. In this case, we have $(\gs')^- = 0$ and $\gs' = (\gs')^+$. Hence, we get from \eqref{eq:prf:1stmom:supercrit:gsblock:2bis},
		\begin{align}
			\label{eq:case_1_Ei}
			E_i \,=\, \inf_{y\in I_i}& \bE_{(t_i,y\gs(t_i/T)L(T))}\Bigg[ 1_{\big\{\forall s\in[t_i,t_{i+1}],\,\frac{B_s}{\gs(s/T)L(T)} \in [0,h]\;;\; \frac{B_{t_{i+1}}}{\gs(t_{i+1}/T)L(T)} \in I_{i+1}\big\}} \\
			&\notag\times \exp\left( -\frac1T \int_{t_i}^{t_{i+1}} \frac{(\gs')^+(s/T)}{\gs^2(s/T)}\,B_s \, \dd s \right) \Bigg].
		\end{align}
		Note that we obtain a lower bound on $E_i$ by writing $(\gs')^+(s/T)/\gs^2(s/T) \le \eta^{-3}$ for each $s\in [0,T]$ (this follows from the assumptions on $\gs$). Then, replacing $(B_s)_{s\ge0}$ with the time-shifted process $(B_{t_i+s})_{s\ge 0}$ and applying Lemma~\ref{lem:1stmom:supercrit}, 
		this readily concludes the proof of \eqref{eq:toshowEi} in this case.
		\bigskip
		
		\emph{Case 2: $\gs'(\cdot/T)$ is non-positive on $[t_i,t_{i+1}]$}. 
		In this case, we have $\gs'= -(\gs')^-$.
		Hence, we get from \eqref{eq:prf:1stmom:supercrit:gsblock:2bis},
		\begin{align*}
			E_i \,=\, \inf_{y\in I_i}& \bE_{(t_i,y\gs(t_i/T)L(T))}\Bigg[ 1_{\big\{\forall s\in[t_i,t_{i+1}],\,\frac{B_s}{\gs(s/T)L(T)} \in [0,h]\;;\; \frac{B_{t_{i+1}}}{\gs(t_{i+1}/T)L(T)} \in I_{i+1}\big\}} \\
			&\notag\times \exp\left( -\frac1T \int_{t_i}^{t_{i+1}} \frac{(\gs')^-(s/T)}{\gs^2(s/T)}\,(h \gs(s/T)L(T) - B_s) \, \dd s \right) \Bigg].
		\end{align*}
	
	We apply Girsanov's theorem to the shifted process $(B_s-h\gs(s/T)L(T))_{s\leq u_1T}$. Letting $\rho(s):=h\gs(s/T)L(T)$, one has on the event $\{\forall s\in[t_i,t_{i+1}],\,|B_s|\leq h\gh^{-1}L(T)\}$ that,
	\begin{equation}\label{eq:supercrit:girsanov}
		\left|\int_{t_i}^{t_{i+1}} \frac{\rho'(s)}{\gs^2(s/T)} \,\dd B_s \right| \,\leq\, \frac{3h^2\gh^{-4} L(T)^2}{T}\,,\quad\text{and}\quad \left|\int_{t_i}^{t_{i+1}} \frac{(\rho'(s))^2}{2\gs^2(s/T)} \,\dd s \right| \,\leq\, \frac{h^2\gh^{-4} L(T)^2}{2T}\,,
	\end{equation}
	(this follows from an integration by parts and computations similar to those of~\eqref{eq:MtOGirs:A} and~\eqref{eq:prf:1stmom:supercrit:gsblock:1}---we do not detail them again). Both those terms are $o(L(T))$ uniformly in $\gs\in\gsens^\supercrit$; therefore, Girsanov's theorem and the symmetry of Brownian motion yield,
	\begin{align*}
		E_1&= e^{o(L(T))} \,\inf_{y\in I_i}\bE_{(h-y)\gs(t_i/T)L(T)}\Bigg[ 1_{\big\{\forall s\in[t_i,t_{i+1}],\,\frac{B_s}{\gs(s/T)L(T)} \in [0,h]\;;\; \frac{B_{t_{i+1}}}{\gs(t_{i+1}/T)L(T)} \in (h-I_{i+1})\}} \\
		&\hspace{40mm}\times
		\exp\Bigg( -\frac1T \int_{t_i}^{t_{i+1}} \frac{(\gs')^-(s/T)}{\gs^2(s/T)}B_s \, \dd s 
		\Bigg) \Bigg].
	\end{align*}
	We now conclude as in Case 1, proving \eqref{eq:toshowEi} in this case as well.\bigskip
				
	\emph{Case 3: $t_i = t - \gth^{-1}(T)L(T)$}. In this case $i = i_{\max}$ and $t_{i+1} = t$. We brutally bound the term inside the exponential in \eqref{eq:prf:1stmom:supercrit:gsblock:2bis} from below by
	\[
	-\frac1T \int_{t_i}^{t_{i+1}} \frac{|\gs'(s/T)|}{\gs(s/T)} 2hL(T) \, \dd s \ge -2h\eta^{-2} \frac{\gth^{-1}(T)L(T)}{T} L(T) = o(L(T)),
	\]
	since $\gth^{-1}(T)L(T) \ll T$ by \eqref{eq:gthasymp}. We conclude again by a use of Lemma~\ref{lem:1stmom:supercrit} as in Case 1.\bigskip
	
	Recollecting~\eqref{eq:prf:1stmom:supercrit:gsblock:prodEi} and \eqref{eq:toshowEi} and taking $\eps\to0$, this concludes the proof of \eqref{eq:prop:1stmom:supercrit:LB} and finishes the proof of the proposition.
\end{proof}

\begin{proof}[Proof of Lemma~\ref{lem:1stmom:supercrit}]
	Throughout the proof, all statement involving $t$ or $\gs$ are meant to hold uniformly in $t\ggbis L(T)$ and $\gs\in\gsens$.
	
	Since the expectation is bounded by 1, we only have to prove a lower bound. 
	Let $x\in[\eps,h-\eps]$, $t:=t(T)\in[0,T]$, $\gs\in\gsens$ and $I\subset[0,h]$ of length at least $2\eps$. Recall from~\eqref{eq:BM:timechange} the definition of the time-change $\cC^1$-diffeomorphism $J(\cdot)$: in particular, $(W_r)_{r\geq0}$ denotes the standard, time-homogeneous Brownian motion, and $(W_{J(s)})_{s\in[0,T]}$ has the same law as $(B_s)_{s\in[0,T]}$. Hence, for $t\in[0,T]$, one has
	\begin{align*}
		& \bE_{x\gs(0)L(T)}\bigg[ \exp\left(-\frac C T \int_0^t B_s \, \dd s \right) 1_{\big\{\forall s\in[0,t],\,\frac{B_s}{\gs(s/T)L(T)} \in [0,h]\,;\, \frac{B_t}{\gs(t/T)L(T)} \in I\big\}} \bigg]\\
		& = \bE_{x\gs(0)L(T)}\bigg[ \exp\left(-\frac C T \int_0^{J(t)} \frac{1}{\gs^2(J^{-1}(s)/T)} W_s \, \dd s \right) \\
		& \qquad\qquad\qquad\qquad\qquad\qquad\qquad\qquad\times 1_{\big\{\forall s\in[0,J(t)],\,\frac{W_{s}}{\gs(J^{-1}(s)/T)L(T)} \in [0,h]\,;\, \frac{W_{J(t)}}{\gs(t/T)L(T)} \in I\big\}} \bigg]\\
		& \geq  \bE_{x\gs(0)L(T)}\bigg[ \exp\left(-\frac{C'}{T} \int_0^{J(t)} W_s \, \dd s \right) 1_{\big\{\forall s\in[0,J(t)],\,\frac{W_{s}}{\gs(J^{-1}(s)/T)L(T)} \in [0,h]\,;\, \frac{W_{J(t)}}{\gs(t/T)L(T)} \in I\big\}} \bigg],
	\end{align*}
	with $C' = C\eta^{-2}$, using that $\gs\in\gsens$. Recall that $T^{1/3}\llbis L(T)\llbis t\leq T$ by assumption. Fix $(\tau_T)_{T>0}$ and $(h_T)_{T>0}$,  satisfying, as $T\to+\infty$,
	\begin{equation}\label{eq:prop:1stmom:supercrit:param}
	L(T) \ll \tau_T \ll \gth^{-1}(T)L(T),\quad h_T \vee \frac T {h_T^2} \ll L(T),\quad \tau_T \ll L(T)h_T.
	\end{equation}
	For example, we may set $\tau_T = \sqrt{\gth^{-1}(T)}L(T)$ and $h_T = \sqrt{T^{1/3}L(T)}$, as one quickly checks.
		
	Assume $T$ is sufficiently large so that $2\tau_T \le \gth^{-1}(T)L(T)$ and $h_T \le \eta L(T)/3$.
	In order to bound from below the last expectation, we constrain the trajectory $(W_s)_{s\in[0,J(t)]}$ to land in $[h_T,2h_T]$ at times $\tau_T$ and $J(t)-\tau_T$, and to remain below $3h_T$ in-between; then we apply Markov's property at times $\tau_T$ and $J(t)-\tau_T$. We obtain,
	\begin{align}\label{eq:prop:1stmom:supercrit:decomp}
		\bE_{x\gs(0)L(T)}\bigg[ \exp\left(-\frac{C'}{T} \int_0^{J(t)} W_s \, \dd s \right) 1_{\big\{\forall s\in[0,J(t)],\,\frac{W_{s}}{\gs(J^{-1}(s)/T)L(T)} \in [0,h]\,;\, \frac{W_{J(t)}}{\gs(t/T)L(T)} \in I\big\}} \bigg]& \\
		\notag \geq B_1\times B_2\times B_3\,,&
	\end{align}
	where
	\begin{align*}
		B_1&:= \bE_{x\gs(0)L(T)}\bigg[ \exp\left(-\frac{C'}{T} \int_0^{\tau_T} W_s \, \dd s \right) \\
		&\qquad\qquad\qquad\qquad\qquad\qquad\times 1_{\big\{\forall s\leq\tau_T,\,\frac{W_{s}}{\gs(J^{-1}(s)/T)L(T)} \in [0,h]\,;\, W_{\tau_T} \in [h_T,2h_T]\big\}} \bigg]\;,\\
		B_2&:= \inf_{y\in[h_T,2h_T]}\bE_{y}\bigg[ \exp\left(-\frac{C'}{T} \int_0^{J(t)-2\tau_T} W_s \, \dd s \right) \\
		&\qquad\qquad\qquad\qquad\qquad\qquad\times 1_{\big\{ \forall s\leq J(t)-2\tau_T,\,W_s \in [0,3h_T]\,;\, W_{J(t)-2\tau_T} \in [h_T,2h_T]\big\}} \bigg]\;,\\
		B_3&:= \inf_{y\in[h_T,2h_T]}\bE_{y}\bigg[ \exp\left(-\frac{C'}{T} \int_0^{\tau_T} W_s \, \dd s \right) \\
		&\qquad\qquad\qquad\qquad\qquad\qquad\times 1_{\big\{\forall s\leq \tau_T,\,\frac{W_{s}}{\gs(J^{-1}(s+J(t)-\tau_T)/T)L(T)} \in [0,h]\,;\, \frac{W_{\tau_T}}{\gs(t/T)L(T)} \in I\big\}} \bigg]\;.
	\end{align*}
Let us now bound from below $B_1$, $B_2$ and $B_3$ separately.
	
	We start with $B_2$. On the event $\{\forall s\in[0,J(t)-2\tau_T],\,W_s \in [0,3h_T]\}$,~\eqref{eq:BM:timechange} and~\eqref{eq:prop:1stmom:supercrit:param} imply that,
	\[ \frac1T\int_0^{J(t)-2\tau_T} W_s \, \dd s  \;\leq\; \frac{3h_T(J(t)-2\tau_T)}T\;\leq\; 3 \gh^{-2} h_T \;\ll\; L(T)\;, \]
	as $T\to+\infty$. Hence one has,
	\begin{equation}\label{eq:prop:1stmom:supercrit:B2}
		B_2\;\geq\; e^{o(L(T))} \inf_{y\in[h_T,2h_T]}\bP_{y}\Big(\forall s\leq J(t)-2\tau_T,\,W_s \in [0,3h_T]\,;\, W_{J(t)-2\tau_T} \in [h_T,2h_T] \Big),
	\end{equation}
	as $T\to+\infty$. Recalling Lemma~\ref{lem:mai16}, notice that there exists a constant $K>0$ such that,
	\[
	\forall\,t'\geq K\,,\qquad \inf_{y\in[1,2]}\bP_y\big(\forall s\leq t',\,W_s\in[0,3]\,;\, W_{t'}\in [1,2]\big)\;\geq\; \frac{\sin^2(\pi/3)}3 \exp\left(-\frac{\pi^2}{18}t'\right)\,.
	\]
	Furthermore, the left-hand side of the last display is positive and continuous in $t'$ for $t'\in(0,K]$ since the density of Brownian motion killed outside the interval $[0,3]$ solves the heat equation with Dirichlet boundary condition at $0$ and $3$. Furthermore, the limit as $t'\to 0$ of the left-hand side of the last display equals $1/2 > 0$, as one readily checks. Thus, we have for some $c_0>0$,
	\[
	\forall\,t'\geq 0\,,\qquad \inf_{y\in[1,2]}\bP_y\big(\forall s\leq t',\,W_s\in[0,3]\,;\, W_{t'}\in [1,2]\big)\;\geq\; c_0 \exp\left(-\frac{\pi^2}{18}t'\right)\,.
	\]
	In particular, 
	\eqref{eq:prop:1stmom:supercrit:B2} and the Brownian scaling property yield
	\begin{equation}\label{eq:prop:1stmom:supercrit:B2:1}
		B_2\;\geq\; e^{o(L(T))}  \exp\left(-\frac{\pi^2}{18}\frac{J(t)-2\tau_T}{h_T^2}\right)\;\geq\; e^{o(L(T))}\;, 
	\end{equation}
since $\frac{J(t)-2\tau_T}{h_T^2} =O(T/h_T^2) =o(L(T))$ by~\eqref{eq:BM:timechange} and~\eqref{eq:prop:1stmom:supercrit:param}.

	Let us turn to $B_1$. 
	One has on the event $\{\forall s\leq \tau_T,\,W_s\leq h\gs(J^{-1}(s)/T)L(T)\}$ that,
	\begin{equation}\label{eq:prop:1stmom:supercrit:B1:area}
		\frac1T\int_0^{\tau_T} W_s \, \dd s  \;\leq\; \frac{h\gh^{-1} L(T) \tau_T}T\;\ll\; L(T)\,,
	\end{equation}
	as $T\to+\infty$, since $\tau_T\ll T$. Moreover, since $\tau_T \ll T$, we have for $T$ large enough,
	\[
	\left|\gs\left(\frac{J^{-1}(s)}T\right) - \gs(0)\right| \le \frac{\eps}{2},\quad \forall s\le \tau_T.
	\]
	Hence, for $T$ large enough,
	\begin{equation}\label{eq:prop:1stmom:supercrit:B1}
		B_1\;\geq\; e^{o(L(T))} \bP_{x\gs(0)L(T)}\Big( \forall s\leq\tau_T,\,W_s \in [0,(h-\eps/2)\gs(0)L(T)]\,;\, W_{\tau_T} \in [h_T,2h_T] \Big),
	\end{equation}
	The Brownian reflection principle (see e.g.~\cite[Part~1, Ch.~4]{BS02}) yields, for every $y\in [h_T,2h_T]$,
	\begin{align*}
		&\bP_{x\gs(0)L(T)}\big( W_{\tau_T} \in dy \big) \,=\, \frac{1}{\sqrt{2\pi \tau_T}} \exp\left(-\tfrac{(x\gs(0)L(T)-y)^2}{2\tau_T}\right)\,dy,\\
		&\bP_{x\gs(0)L(T)}\big( W_{\tau_T} \in dy; \inf_{s\leq\tau_T} W_s \leq 0 \big) \,=\, \frac{1}{\sqrt{2\pi \tau_T}} \exp\left(-\tfrac{(x\gs(0)L(T)+y)^2}{2\tau_T}\right)\,dy,\\
		\text{and}\quad &\bP_{x\gs(0)L(T)}\big( W_{\tau_T} \in dy; \sup_{s\leq\tau_T} W_s \geq (h-\eps/2)\gs(0)L(T) \big) \\
		&\qquad\qquad\qquad\qquad\qquad\qquad=\, \frac{1}{\sqrt{2\pi \tau_T}} \exp\left(-\tfrac{((2h-x-\eps)\gs(0)L(T)-y)^2}{2\tau_T}\right)\,dy.
	\end{align*}
	Now note that for $y\in [h_T,2h_T],$ by \eqref{eq:prop:1stmom:supercrit:param},
	\[
	\frac{(x\gs(0)L(T)+y)^2}{2\tau_T} - \frac{(x\gs(0)L(T)-y)^2}{2\tau_T} \ge \frac{\eps\eta L(T) h_T}{\tau_T} \gg 1.
	\]
	Furthermore, using that $2h-x-\eps \ge x+\eps$, we have for $y\in [h_T,2h_T],$ again by \eqref{eq:prop:1stmom:supercrit:param},
	\[
	\frac{((2h-x-\eps)\gs(0)L(T)-y)^2}{2\tau_T} - 	\frac{(x\gs(0)L(T)-y)^2}{2\tau_T} \ge \frac{\eps\eta L(T)}{\tau_T} \gg 1.
	\]
	It follows from the above estimates, that for $T$ large enough,
	\begin{align}\label{eq:prop:1stmom:supercrit:B1:1}
		\notag&\bP_{x\gs(0)L(T)}\Big(\forall s\leq\tau_T,\, W_s \in [0,(h-\eps/2)\gs(0)L(t)]\,;\, W_{\tau_T} \in [h_T,2h_T] \Big)\\
		&\qquad\geq\; \frac{h_T}{\sqrt{2\pi \tau_T}} \exp\left(-\tfrac{(x\gs(0)L(T)-2h_T)^2}{2\tau_T}\right) (1-o(1))\;\geq\; e^{o(L(T))}\;,
	\end{align}
	 by~\eqref{eq:prop:1stmom:supercrit:param}. Equations~\eqref{eq:prop:1stmom:supercrit:B1} and \eqref{eq:prop:1stmom:supercrit:B1:1} imply that $B_1\geq e^{o(L(T))}$ as $T\to+\infty$. The term $B_3$ is handled similarly. Recollecting~\eqref{eq:prop:1stmom:supercrit:decomp} and \eqref{eq:prop:1stmom:supercrit:B2:1}, this finishes the proof of the lemma.
\end{proof}

We now provide an upper bound on the second moment of $|A^{\supercrit}_{T,I}(t)|$ when $I=[z,h]$ for some $z\in[0,h]$ and $t\in[0,T]$.

\begin{proposition}[Second moment, Super-critical]\label{prop:2ndmom:supercrit}
	Let $h>0$. One has as $T\to+\infty$,
	\begin{equation}\label{eq:prop:1stmom:supercrit:ULB} 
		\bbE\big[|A_{T,z}^{\supercrit}(t)|^2\big] \;\leq\; e^{(x+h-2z)L(T)+o(L(T))}\,,
	\end{equation}
	uniformly in $t\in[0,T]$, $\gs\in\gsens$, $x\in[0,h]$ and $z\in[0,h]$.
\end{proposition}

\begin{remark}\label{rem:2ndmom}
	The proof of Proposition~\ref{prop:2ndmom:supercrit} relies mostly on the Many-to-two lemma (Lemma~\ref{lem:MtT}) and the upper bound~\eqref{eq:prop:1stmom:supercrit:UB} from Proposition~\ref{prop:1stmom:supercrit}. Noticeably, the second moment estimates in the sub-critical and critical cases will be obtained very similarly, by using respectively the first moment upper bounds from Propositions~\ref{prop:1stmom:subcrit} and~\ref{prop:1stmom:crit} below.
\end{remark}

\begin{proof}
	The Many-to-two lemma (Lemma~\ref{lem:MtT}) states that
	\begin{align*}&\bbE\big[|A_{T,z}^\supercrit(t)|^2\big] - \bbE\big[|A_{T,z}^\supercrit(t)|\big] \\
	&\quad= \gb_0\bbE[\xi(\xi-1)]\int_0^t \dd s \int_0^{h} G^\supercrit(x,y,0,s) \bigg(\int_{z}^{h} G^\supercrit(y,w,s,t) \, \dd w\bigg)^2 \dd y \;.
	\end{align*}
	Similarly to~\eqref{eq:A_G}, one has $\int_z^h G^\supercrit(y,w,s,t) \dd w = \bbE[|\tilde A_{T,z}^\supercrit(t-s)|]$, where $\tilde A_{T,z}^\supercrit$ is defined similarly to $A_{T,z}^\supercrit$ by replacing $\gs(\cdot)$ with $\gs(\cdot-s/T)$. Recall also the upper bound~\eqref{eq:prop:1stmom:supercrit:UB} from Proposition~\ref{prop:1stmom:supercrit}, which is uniform in $x\in[0,h]$, $I\subset[0,h]$, $\gs\in\gsens$ and $t\in[0,T]$. Therefore, we have for any $s\in[0,t]$,
	\begin{align*}&\int_0^h G^\supercrit(x,y,0,s) \,\dd y\,  \bigg(\int_z^h G^\supercrit(y,w,s,t) \, \dd w\bigg)^2  \\
	&\quad\leq\; e^{o(L(T))} \int_{0}^{h} G^\supercrit(x,y,0,s) \,e^{2(y-z)L(T)} \,\dd y\,, 
	\end{align*}
	where we used that the error term from~\eqref{eq:prop:1stmom:supercrit:UB} is uniform in $y, z, s, t$. 
	Let $K>0$ a large constant, and split $[0,h]$ into $K$ intervals of length $\frac hK$. For any $0\leq i<K$, one deduces from Proposition~\ref{prop:1stmom:supercrit},
	\[\begin{aligned} \int_{i\frac hK}^{(i+1)\frac hK} G^\supercrit(x,y,0,s) \,e^{2(y-z)L(T)} \,\dd y \;&\leq\; e^{2((i+1)\frac hK-z)L(T)} \times e^{(x-i\frac hK)L(T)+o(L(T))} \\
		&\leq \; e^{((2+i)\frac hK +x -2z)L(T)+o(L(T))}\\
		&\leq \; e^{2\frac hK L(T) + (x+h-2z)L(T) + o(L(T))}.
	\end{aligned} \]
	Therefore, summing over $0\leq i<K$ and integrating over $s\in[0,t]$, one obtains 
	\[
	\bbE\big[|A_{T,z}^\supercrit(t)|^2\big] - \bbE\big[|A_{T,z}^\supercrit(t)|\big] \;\leq\; e^{o(L(T))}\times e^{2\frac hKL(T)}\times t\times K\times e^{(x+h-2z)L(T)}.
	\]
	Taking $K\to+\infty$, and recalling that $t\leq T\leq L(T)^3$ and $\bbE[|A_{T,z}^\supercrit(t)|]\leq e^{(x-z)L(T)+o(L(T))}$, this yields the expected upper bound uniformly in $t$, $\gs(\cdot)$, $x$ and $z$.
\end{proof}

We conclude this section with an estimate on the number of particles killed at the upper barrier. Recall the definition of $R^\supercrit_T(s,t)$, $0\leq s\leq t\leq T$ from~\eqref{eq:defR}.

\begin{proposition}[Killed particles, Super-critical] \label{prop:killedpart:supercrit} Let $h>0$. Then as $T\to+\infty$, one has
	\begin{equation}
		\bbE\left[R^\supercrit_T(0,T)\right]\,\leq\, e^{-(h-x)L(T) + o(L(T))}\;,
	\end{equation} 
	uniformly in $x\in[0,h]$ and $\gs\in\gsens$. 
\end{proposition}

\begin{proof}
	Recall that we may write $\ol\gga^\supercrit_T$ as in~\eqref{eq:olggasup:alt}: in particular, one has for $s\in[0,T]$,
	\[
	(\ol\gga_T^\supercrit)'(s)=\gs(s/T)+h(\gs')^+(s/T)\frac{L(T)}T\,.
	\]
	Recollect~\eqref{eq:MtOGirs:R} from Lemma~\ref{lem:MtOGirs}. On the one hand, one has for $T$ sufficiently large, uniformly in $s\in[0,T]$ and $\gs\in\gsens$, that
	\[
	\left.\frac{\partial}{\partial u}\left(\frac{\ol\gga_T^\supercrit{}'(u)}{\sigma^2(u/T)}\right)\right|_{u=s} \;\geq\; -\frac1T\frac{\gs'(s/T)}{\gs^2(s/T)} - \frac{h\gh^{-7}L(T)}{T^2} \;,
	\]
	so, on the event $\{\forall s\leq H_0(B),\,B_s\leq h\gs(s/T)L(T)\}$, one obtains
	\[
	- \int_0^{H_0(B)} \left.\frac{\partial}{\partial u}\left(\frac{\ol\gga_T^\supercrit{}'(u)}{\sigma^2(u/T)}\right)\right|_{u=s} B_s \,\dd s\;\leq\;
	\frac1T \int_0^{H_0(B)} \frac{\gs'(s/T)}{\gs^2(s/T)} B_s \,\dd s + h^2\gh^{-8}\frac{L(T)^2}{T}.
	\]
	On the other hand, one has
	\[
	\int_0^{H_0(B)} \frac{(\ol\gga^\supercrit_T{}'(s))^2}{2\gs^2(s/T)}\dd s\;\geq\; \frac{H_0(B)}2 + h\frac{L(T)}T\int_0^{H_0(T)} \frac{(\gs')^+(s/T)}{\gs(s/T)}\,\dd s
	\,.\]
	Therefore,~\eqref{eq:MtOGirs:R} eventually yields
	\begin{align*}
		\bbE\left[R^\supercrit_T(0,T)\right]\;&\leq\; e^{-(h-x)L(T)+o(L(T))} \\
		&\quad\times \bE_{(h-x)\gs(0)L(T)}\bigg[ \ind_{\{H_0(B)\leq T\}}\ind_{\{\forall s\leq H_0(B),\,\frac{B_s}{\gs(s/T) L(T)}\in[0,h]\}}\\
		&\quad \times \exp\bigg( \frac1T \int_0^{H_0(B)} \frac{\gs'(s/T)}{\gs^2(s/T)} B_s \,\dd s - h\frac{L(T)}T\int_0^{H_0(T)} \frac{(\gs')^+(s/T)}{\gs(s/T)}\,\dd s \bigg) \bigg].
	\end{align*}
	Recall that $\gs'(u)=(\gs')^+(u)-(\gs')^-(u)$ for all $u\in[0,1]$. Thus one obtains on the event $\{\forall\, s\leq H_0(B),\,B_s\leq h\gs(s/T)L(T) \}$ that,
	\begin{align*}
		\bbE\left[R^\supercrit_T(0,T)\right]\;&\leq\; e^{-(h-x)L(T)+o(L(T))} \\
		&\qquad \times \bE_{(h-x)\gs(0)L(T)}\bigg[ \ind_{\{H_0(B)\leq T\}}\ind_{\{\forall s\leq H_0(B),\,\frac{B_s}{\gs(s/T) L(T)}\in[0,h]\}}\\
		&\qquad \times \exp\bigg( - \frac1T \int_0^{H_0(B)} \frac{(\gs')^-(s/T)}{\gs^2(s/T)} B_s \,\dd s \bigg) \bigg],
	\end{align*}
	and the latter expectation is bounded by 1, which concludes the proof.
\end{proof}

\subsection{Sub-critical case} \label{sec:moments:subcrit}
In this section we assume $\anyrg=\subcrit$, that is $1\ll L(T)\ll T^{1/3}$, and recall from~\eqref{def:gga:subcrit} and~\eqref{eq:defxh} that we defined the lower and upper barriers, for $t\in[0,T]$, by
\[
\gga_T^\subcrit(t)\;:=\;\sqrt{1-\frac{\pi^2}{h^2  L(T)^2}}\times v(t/T)\,T - x\,\gs(0) L(T)\;,
\]
and $\ol\gga_T^\subcrit(t):=\gga_T^\subcrit(t) +h\gs(t/T) L(T)$ for some $h>x>0$. Recall~\eqref{eq:defA:allregimes}: in particular, the set of descendants remaining between the barriers and reaching an interval $I\subset[0,h]$ at time $t$ is denoted by $A_{T,I}^\subcrit(t)$. 
Recall also~\eqref{eq:defgsens} and~\eqref{eq:defllbis}, where we fixed some $\gh>0$ and $\gth:\R_+\to\R^*_+$ such that $\gth(T)\to0$ as $T\to+\infty$ and~\eqref{eq:gthasymp} holds. 

\begin{proposition}[First moment, Sub-critical] \label{prop:1stmom:subcrit} 
	Let $h>0$. As $T\to+\infty$, one has
	\begin{equation}\label{eq:prop:1stmom:subcrit:general}
		\bbE\big[|A_{T,I}^\subcrit(t)|\big] \,\leq\, e^{(x-\inf I) L(T)+o(L(T))}\,,
	\end{equation}
	uniformly in $x\in[0,h]$, $I\subset[0,h]$ an interval, $\gs\in\gsens$, and $t=t(T)$ such that $0\leq t(T)\llbis \sqrt{TL(T)^3}$ for $T$ sufficiently large. Moreover, as $T\to+\infty$ one also has for every interval $I\subset [0,h]$ with non-empty interior,
	\begin{equation}\label{eq:prop:1stmom:subcrit:general:LB}
		\bbE\big[|A_{T,I}^\subcrit(t)|\big] \,\geq\, e^{(x-\inf I) L(T)+o(L(T))}\,,
	\end{equation}
	locally uniformly in $x\in(0,h)$, $\inf I\in[0,h)$ and $\sup I-\inf I\in(0,h]$; and uniformly in $\gs\in\gsens$ and $t=t(T)$ such that $L(T) \llbis t(T)\llbis \sqrt{TL(T)^3}$ for $T$ sufficiently large.
\end{proposition}

This proposition should be compared with Proposition~\ref{prop:1stmom:supercrit} for the super-critical regime. The definitions of ``uniformly'' and ``locally uniformly'' are the same as in Remark~\ref{rem:1stmom}. 
\begin{remark} The upper bound on $t$ in the hypotheses of Proposition~\ref{prop:1stmom:subcrit} can be improved. However, in contrast to the analogous results in the super-critical and critical regimes, in general we do not expect the estimates from Proposition~\ref{prop:1stmom:subcrit} to hold for all $t\le T$ anyway, at least not if $L(T)$ is sufficiently small. For this reason, the proofs of Propositions~\ref{prop:main:LB} and~\ref{prop:main:UB} in Sections~\ref{sec:LB} and~\ref{sec:UB} below will require additional arguments in the sub-critical regime.
\end{remark}

\begin{proof} 
	Recall Lemma~\ref{lem:MtOGirs}, in particular~\eqref{eq:MtOGirs:A}. Notice that~\eqref{def:gga:subcrit} implies that for all $s\in[0,T]$, $T\geq0$,
	\[
	\frac{\gga^\subcrit_T{}'(s)}{\gs(s/T)}\,=\, \sqrt{1-\frac{\pi^2}{h^2  L(T)^2}}\,,
	\]
	which does not depend on $s$. On the one hand, this yields for all $s\in[0,T]$,
	\begin{equation}\label{eq:subcrit:1stmom:gga1}
		\frac{\big(\gga_T^\subcrit{}'(s)\big)^2}{2\,\gs^2(s/T)}
		= \frac{1}{2} - \frac{\pi^2}{2h^2L(T)^2} \,.
	\end{equation}
	On the other hand, on the event $\{\frac{B_s}{\gs(s/T)L(T)}\in[0,h],\forall s\leq t\}$, one has
	\begin{equation}\label{eq:subcrit:1stmom:gga2}
		\int_0^t \left.\frac{\partial}{\partial u}\left(\frac{\gamma_T^\subcrit{}'(u)}{\sigma^2(u/T)}\right)\right|_{u=s} B_s\, \mathrm{d}s\,\leq\, \frac{\gh^{-3}}T\times  h\gh^{-1}L(T) \times \gth(T) \sqrt{TL(T)^3}\,=\,o(L(T))\,,
	\end{equation}
	uniformly in $t\llbis\sqrt{TL(T)^3}$ and $\gs\in\gsens$. Therefore,~\eqref{eq:MtOGirs:A} becomes
	\begin{align}\label{eq:subcrit:1stmom:2}
		&\bbE\big[|A_{T,I}^\subcrit(t)|\big] = \exp\left(\frac{\pi^2}{2h^2}\frac{t}{L(T)^2}\,+\,o(L(T))\right) \\
		\notag &\qquad\qquad \times \bE_{x\gs(0) L(T)}\Bigg[\exp\left(-\frac{B_t}{\gs(t/T)}+x L(T)\right)\ind_{\left\{ \forall s\leq t,\, \frac{B_s}{\gs(s/T) L(T)}\in[0,h] \,;\, \frac{B_t}{\gs(t/T)L(T)}\in I\right\}}\Bigg].
	\end{align}
	We now focus on the latter expectation. It can be bounded from above by
	\begin{equation}\label{eq:subcrit:1stmom:2:UB}
		e^{(x-\inf I) L(T)} \, \bP_{x\gs(0) L(T)}\bigg(\forall s\leq t,\, \tfrac{B_s}{\gs(s/T) L(T)}\in[0,h]\bigg) ;
	\end{equation}
	and, letting any $\eps_T\to0$ as $T\to+\infty$ and adding the constraint $\frac{B_t}{\gs(t/T)L(T)}\leq \inf I+\eps_T$, it can be bounded from below by
	\begin{equation}\label{eq:subcrit:1stmom:2:LB}
		e^{(x-\eps_T-\inf I) L(T)} \, \bP_{x\gs(0) L(T)}\bigg( \forall s\leq t,\,\tfrac{B_s}{\gs(s/T) L(T)}\in[0,h] \,;\, \tfrac{B_t}{\gs(t/T)L(T)}\in (\inf I,\inf I+\eps_T)\bigg) .
	\end{equation}
	Therefore, writing $z:=\inf I$ to lighten notation, both statements of the proposition are obtained by showing that~(\ref{eq:subcrit:1stmom:2:UB}--\ref{eq:subcrit:1stmom:2:LB}) are of order $\exp((x-z)L(T)-\frac{\pi^2}{2h^2}\frac t{L(T)^2}+o(L(T)))$. Once again, we achieve this via a comparison with the standard, time-homogeneous Brownian motion $(W_s)_{s\geq0}$. In the following, recall Lemma~\ref{lem:mai16}.
	\smallskip
	
	\noindent\emph{Upper bound.} 
	Using~\eqref{eq:gsapprox}, the Brownian scaling property and a time change, we have
	\[\begin{aligned}&\bP_{x\gs(0)L(T)}\Big(B_s\in[0,h\,\gs(s/T) L(T)], \,\forall s\leq t\Big)\\&\qquad\leq \bP_{x\gs(0)L(T)}\Big(B_s\in[0,h\,(\gs(0)+\gh^{-2}t/T) L(T)], \,\forall s\leq t\Big) \\
		&\qquad= \bP_{x\frac{\gs(0)}{\gs(0)+\gh^{-2}t/T}}\Big(B_s\in[0,h], \,\forall s\leq \tfrac{t}{L(T)^2(\gs(0)+\gh^{-2}t/T)^2}\Big) \\
		&\qquad= \bP_{x\frac{\gs(0)}{\gs(0)+\gh^{-2}t/T}}\Big(W_s\in[0,h], \,\forall s\leq \tfrac{t}{L(T)^2(\gs(0)+\gh^{-2}t/T)^2} \smallint_0^{t} \gs^2(u/T)\dd u\Big)\\
		&\qquad\leq \bP_{x\frac{\gs(0)}{\gs(0)+\gh^{-2}t/T}}\Big(W_s\in[0,h], \,\forall s\leq \tfrac{t}{L(T)^2}\tfrac{(\gs(0)-\gh^{-2}t/T)^2}{(\gs(0)+\gh^{-2}t/T)^2}\Big)\,. \end{aligned}\]
	Then,~\eqref{eq:gsapprox} implies that, for $T$ sufficiently large, $\tfrac{(\gs(0)-\gh^{-2}t/T)^2}{(\gs(0)+\gh^{-2}t/T)^2}\geq\tfrac12$ uniformly in $\gs\in\gsens$ and $t\llbis\sqrt{TL(T)^3}$. Moreover
	Lemma~\ref{lem:mai16} implies that there exists a constant $K>0$ such that,
	\[
	\forall \, t'\geq K,\qquad \sup_{y\in[x/2,x]} \bP_y(\forall s\leq t',\,W_s\in[0,h])\,\leq\, \frac8\pi \exp\left( -\frac{\pi^2}{2h^2} t' \right)\,,
	\]
	where we also used that $\int_0^h \sin(\pi z/h)\dd z = 2h/\pi$.
	For any $t$ such that $t\geq 2KL(T)^2$, this yields
	\begin{align}\label{eq:subcrit:1stmom:3}
		\notag & \bP_{x}\Big(\forall s\leq t /L(T)^2,\,W_s\in[0,h\,\gs(t/T)/\gs(0)]\Big)\\
		&\qquad \leq\frac{8}{\pi} \exp\bigg(-\frac{\pi^2}{2h^2}\frac{t}{L(T)^2}\tfrac{(\gs(0)-\gh^{-2}t/T)^2}{(\gs(0)+\gh^{-2}t/T)^2}\bigg)\\
		&\qquad \leq \frac{8}{\pi} \exp\bigg(-\frac{\pi^2}{2h^2}\frac{t}{L(T)^2} +\frac{2\pi^2\gh^{-3}}{h^2}\frac{t^2}{L(T)^2T} \bigg)
		\,,\end{align}
	and, by assumption, we have $\frac{t^2}{L(T)^2T}\leq \gth(T)^{2}L(T)=o(L(T))$ for $T$ large. On the other hand if $t\leq 2KL(T)^2$, then one may write $|\frac{\pi^2}{2h^2}\frac{t}{L(T)^2}|\leq \frac{K\pi^2}{h^2}$ in~\eqref{eq:subcrit:1stmom:2}, and the probability in~\eqref{eq:subcrit:1stmom:2:UB} is bounded by 1. Finally, taking the maximum of this and~\eqref{eq:subcrit:1stmom:3}, we obtain an upper bound which is uniform in $\gs\in\gsens$ and $t(T)$ such that $0\leq t\llbis \sqrt{TL(T)^3}$.
	\smallskip
	
	\noindent\emph{Lower bound.}  
	Similarly to the upper bound, we have to distinguish the cases $t$ smaller or larger than $KL(T)^2$, for some constant $K$ which is determined in~\eqref{eq:subcrit:1stmom:4quad} below. 
	\smallskip
	
	\noindent\emph{Case $t(T)\geq KL(T)^2$.} 
	Recall that we want to bound from below the expression in \eqref{eq:subcrit:1stmom:2:LB}.
	Let $\eps_T$ such that $1\gg \eps_T\gg \max(\gth(T)\sqrt{L(T)^3T^{-1}},L(T)^{-1})$ as $T\to+\infty$. For $T$ sufficiently large, this and~\eqref{eq:gsapprox} imply $z+\eps_T<h$ and $[z+\eps_T/3, z+2\eps_T/3]\gs(0) \subset [z,z+\eps_T]\gs(t/T)$ uniformly in $\gs\in\gsens$ and locally uniformly in $z$. Thus we have the lower bound
	\begin{equation}\label{eq:subcrit:1stmom:4}
		\begin{aligned}
			&\bP_{x\gs(0)L(T)}\bigg(\forall s\leq t,\,B_s\in[0,h\,\gs(s/T)L(T)]\,;\, B_t\in [z,z+\eps_T]\gs(t/T)L(T)\bigg)\\
			&\quad\geq\; \bP_{x\gs(0)L(T)}\bigg(\forall s\leq t,\,B_s\in[0,h\,\gs(0)L(T)]\,;\, B_t\in [z+\eps_T/3,z+2\eps_T/3]\gs(0)L(T)\bigg).\\
		\end{aligned}
	\end{equation}
	Define
	\begin{equation}\label{eq:subcrit:1stmom:4bis}
		\tau=\tau(t):=\frac1{\gs^2(0)L(T)^2}\int_0^{t}\gs^2(u/T)\,\dd u\,,
	\end{equation}
	in particular~\eqref{eq:gsapprox} and $t\llbis\sqrt{TL(T)^3}$ imply that
	\begin{align*}
		\Big|\tau-\frac{t}{L(T)^2}\Big|
		\,=\, \frac{T}{L(T)^2}\left|\int_0^{t/T}\left(\frac{\gs(u)^2}{\gs(0)^2}-1\right)\,\dd u\right| \,\le \, \frac{T}{L(T)^2} \int_0^{t/T} u\times  2\sigma'(u)\sigma(u)\,du \,&\\
		\,\le\, \frac{T}{L(T)^2}\eta^{-2}\left(\frac{t}{T}\right)^2=\, O\left(\gth(T)^2L(T)\right),&
	\end{align*}
	so that we have
	\begin{equation}
	    \label{eq:subcrit:1stmom:4bis:ineq}
	   	\Big|\tau-\frac{t}{L(T)^2}\Big| = o(L(T)), 
	\end{equation}
	uniformly in $\gs\in\gsens$ and $t$. Using the Brownian scaling property and a time change, we have
	\begin{align}\label{eq:subcrit:1stmom:4ter}
		\notag &\bP_{x\gs(0)L(T)}\bigg(\forall s\leq t,\,B_s\in[0,h\,\gs(0)L(T)]\,;\, B_t\in [z+\eps_T/2,z+\eps_T]\gs(0)L(T)\bigg)\\
		&\qquad=\;\bP_{x}\Big(\forall s\leq \tau,\,W_s\in[0,h]\,;\, W_{\tau}\in [z+\eps_T/3,z+2\eps_T/3]\Big).
	\end{align}
	By Lemma~\ref{lem:mai16}, there exists a large constant $K>0$ such that, for $T$ sufficiently large, if $t\geq KL(T)^2$ then $\tau\geq K/2$, and,
	\begin{align}\label{eq:subcrit:1stmom:4quad}
		\notag&\bP_{x}\Big(\forall s\leq \tau,\,W_s\in[0,h]\,;\, W_{\tau}\in [z+\eps_T/3,z+2\eps_T/3]\Big) \\
		&\qquad\geq\; \frac1h \exp\Big(-\frac{\pi^2}{2h^2}\tau\Big)\sin\Big(\frac{\pi x}{h}\Big) \int_{z+\eps_T/3}^{z+2\eps_T/3} \sin\Big(\frac{\pi y}{h}\Big) \dd y\,.
	\end{align}
	Recalling $\eps_T\gg L(T)^{-1}$ and that $\sin(\pi w)\geq \frac12 (w\wedge 1-w)$, $\forall\,w\in[0,1]$, we have as soon as $T$ is sufficiently large,
	\begin{equation}\label{eq:subcrit:1stmom:4quad:sinus}
		\int_{z+\eps_T/3}^{z+2\eps_T/3} \sin\Big(\frac{\pi y}{h}\Big) \dd y \geq \frac{\eps_T}6\left(\frac{z+\eps_T/3}h\wedge \Big(1-\frac{z+2\eps_T/3}h\Big)\right) \geq \frac{\eps_T^2}{18h} = \exp(o(L(T))).
	\end{equation}
	Plugging this and~\eqref{eq:subcrit:1stmom:4bis:ineq} into~\eqref{eq:subcrit:1stmom:4quad}, this finally yields the lower bound uniformly in $t\geq KL(T)^2$ and $\gs\in\gsens$.
	\smallskip
	
	\noindent\emph{Case $t(T)\leq KL(T)^2$.} 
	Let $\eps_T$ which satisfies $L(T)^{-1}\ll\eps_T\ll L(T)^{-\frac12}$. Reproducing the computation~(\ref{eq:subcrit:1stmom:4}--\ref{eq:subcrit:1stmom:4ter}) with that choice of $\eps_T$,
	we then only need to prove that \eqref{eq:subcrit:1stmom:4ter} is larger than $\exp(o(L(T)))$ (since one has $\frac{\pi^2}{2h^2}\frac{t}{L(T)^2}\geq0$ in~\eqref{eq:subcrit:1stmom:2}). Using standard computations on the Gaussian density, one has
	\begin{equation}\label{eq:subcrit:1stmom:6}
		\bP_x\big(W_{\tau}\in [z+\eps_T/3,z+2\eps_T/3]\big) \,\geq\, \frac{\eps_T}{3}\frac{1}{\sqrt{2\pi \tau}} \exp\left(-\frac{(|z-x|+\eps_T)^2}{2\tau}\right)\geq \exp(o(L(T)))\,,
	\end{equation}
	where the second inequality follows from the observation that~\eqref{eq:gsapprox},~\eqref{eq:subcrit:1stmom:4bis} and $t\ggbis L(T)$ imply $\tau^{-1}\leq 2\gth(T)L(T)=o(L(T))$ for $T$ sufficiently large. Moreover, we claim that
	\begin{equation}\label{eq:subcrit:1stmom:6:claim}
		\liminf_{T\to+\infty} \bP_x\Big(\forall s\leq \tau,\,W_s\in[0,h]\,\Big|\, W_{\tau}\in [z+\eps_T/3,z+2\eps_T/3]\Big) \;>\;0\,,
	\end{equation}
	locally uniformly in $x,z$. With this, one only needs to take the minimum between~\eqref{eq:subcrit:1stmom:4quad} and~(\ref{eq:subcrit:1stmom:6}--\ref{eq:subcrit:1stmom:6:claim}), then plug it into~\eqref{eq:subcrit:1stmom:2:LB}, to obtain the announced lower bound uniformly in $t(T)$ and $\gs(\cdot)$.
	
	Let us now prove~\eqref{eq:subcrit:1stmom:6:claim}. We write that $t\leq KL(T)^2$ implies $\tau\leq 2K$ for $T$ large, and
	\begin{equation}\label{eq:subcrit:1stmom:7}\begin{aligned}
			&\bP_x\Big(\forall s\leq \tau,\,W_s\in[0,h]\,\Big|\, W_{\tau}\in [z+\tfrac{\eps_T}3,z+\tfrac{2\eps_T}3]\Big)\\
			&\;\geq \bP_x\Big(\forall u\leq 1,\,\tau^{-\frac12} \big(W_{u\tau}-uz-(1-u)x\big)\in[-\tfrac{x\wedge z}{\sqrt{2K}},\tfrac{h-x\vee z}{\sqrt{2K}}]\,\Big|\, W_{\tau}\in [z+\tfrac{\eps_T}3,z+\tfrac{2\eps_T}3]\Big)\,.
		\end{aligned}
	\end{equation}
	Moreover,~\eqref{eq:subcrit:1stmom:4bis} implies $\tau\geq (2\gth(T)L(T))^{-1}$ for $T$ large, so $\tau^{-\frac12}\eps_T \to0$ as $T\to+\infty$. Therefore, the Gaussian process $(\tau^{-\frac12} (W_{u\tau}-uz-(1-u)x))_{u\in[0,1]}$ conditioned to $W_0-x=0$, $W_{\tau}-z\in [\eps_T/3,2\eps_T/3]$ converges in law to a standard Brownian bridge $(X^{0,0}_u)_{u\in[0,1]}$ as $T\to+\infty$ (see~\cite[Sect. 9]{Bil99}); more precisely, the r.h.s of~\eqref{eq:subcrit:1stmom:7} converges to $\bP(X^{0,0}_u\in[-\tfrac{x\wedge z}{\sqrt{2K}},\tfrac{h-x\vee z}{\sqrt{2K}}], \forall u\leq 1)>0$ as $T\to+\infty$, and this convergence is uniform in $L(T)\llbis t\leq KL(T)^2$. Since the latter probability is positive, this concludes the proof.
\end{proof}

We now provide an upper bound on the second moment of $|A_{T,I}^\subcrit(t)|$ when $I=[z,h]$ for some $z\in[0,h]$. 
\begin{remark}Let us point out that, in contrast to the super-critical case (recall Proposition~\ref{prop:2ndmom:supercrit}), the statement below involves an error factor $O(t)$: in general one cannot guarantee for all $t\in[0,\sqrt{TL(T)^3}]$ that $t\leq e^{o(L(T))}$ in the sub-critical regime, especially when $L(T)$ grows very slowly in $T$. However this will not be an issue in this paper, since in Sections~\ref{sec:LB}--\ref{sec:UB} we shall consider values $t=t(T)$ which are at most polynomial in $L(T)$.
\end{remark}

\begin{proposition}[Second moment, Sub-critical] \label{prop:2ndmom:subcrit} Let $h>0$. As $T\to+\infty$, one has
	\begin{equation}
		\bbE\big[|A_{T,z}^\subcrit(t)|^2\big] \;\leq\;e^{(x+h-2z) L(T) +o(L(T)) } \times O(t)\,,
	\end{equation}
	uniformly in $x\in[0,h]$, $z\in[0,h]$, $\gs\in\gsens$ and $t=t(T)$ such that $0\leq t(T)\llbis \sqrt{TL(T)^3}$ for $T$ sufficiently large.
\end{proposition}

\begin{proof}
	This proposition is very similar to Proposition~\ref{prop:2ndmom:supercrit}. Reproducing all arguments from its proof but using~\eqref{eq:prop:1stmom:subcrit:general} instead of~\eqref{eq:prop:1stmom:supercrit:UB}, one obtains for $K\in\N$ and $T$ sufficiently large,
	\[
	\bbE\big[|A_{T,z}^\subcrit(t)|^2\big] - \bbE\big[|A_{T,z}^\subcrit(t)|\big] \;\leq\; e^{2\frac hKL(T) + o(L(T))}\times t\times K\times e^{(x+h-2z)L(T)},
	\]
	uniformly in $x,z\in[0,h]$, $\gs\in\gsens$ and $0\leq t\llbis \sqrt{TL(T)^3}$. Taking $K$ large, this yields the expected result.
\end{proof}

We conclude this section with an estimate on the number of particles killed at the upper barrier. Recall the definition of $R^\subcrit_T(s,t)$, $0\leq s\leq t\leq T$ from~\eqref{eq:defR}.

\begin{proposition}[Killed particles, Sub-critical] \label{prop:killedpart:subcrit} Let $h>0$. Then as $T\to+\infty$, one has
	\begin{equation}
		\bbE\left[R^\subcrit_T(0,t)\right]\,\leq\, e^{-(h-x)L(T) + o(L(T))}\times O\big(tL(T)^{-2}\vee 1\big)\;,
	\end{equation} 
	uniformly in $x\in[0,h]$, $\gs\in\gsens$, and $t=t(T)$ such that $0\leq t\llbis \sqrt{TL(T)^3}$ for $T$ sufficiently large. 
\end{proposition}

Notice that, similarly to Proposition~\ref{prop:2ndmom:subcrit}, we have an additional error term $O(tL(T)^{-2}\vee 1)$, which vanishes as soon as $t(T)$ is at most polynomial in $L(T)$.

\begin{proof} 
	This proof relies on a comparison with the time-homogeneous case, which has already been studied in~\cite{MaiPhD, Mai16}. 
	Recollect~\eqref{eq:MtOGirs:R} from Lemma~\ref{lem:MtOGirs}. Notice that $(\ol\gga_T^\subcrit)'(s)=(\gga_T^\subcrit)'(s) + O(L(T)/T)$ uniformly in $s\in[0,T]$, $\gs\in\gsens$. Combining this with (\ref{eq:subcrit:1stmom:gga1}--\ref{eq:subcrit:1stmom:gga2}), we deduce from~\eqref{eq:MtOGirs:R} that,
	\begin{align*}
	\bbE[R^\subcrit_T(0,t)]&= e^{-(h-x)L(T) + o(L(T))} \\
	&\;\;\;\times \bE_{(h-x)\gs(0)L(T)}\Bigg[ \exp\left( \frac{\pi^2}{2h^2}\frac{H_0(B)}{L(T)^2} \right) \ind_{\{H_0(B)\leq t\}} \ind_{\{\forall s\leq H_0(B),\,\frac{B_s}{\gs(s/T)L(T)}\leq h\}}\Bigg].
	\end{align*}
	Therefore, it only remains to bound from above the latter expectation. Using the Brownian scaling property, we have
	\[\begin{aligned} &\bE_{(h-x)\gs(0)L(T)}\Bigg[ \exp\left( \frac{\pi^2}{2h^2}\frac{H_0(B)}{L(T)^2} \right) \ind_{\{H_0(B)\leq t\}} \ind_{\{\forall s\leq H_0(B),\,\frac{B_s}{\gs(s/T)L(T)}\in[0,h]\}}\Bigg] \\
		&\, = \bE_{(h-x)}\Bigg[ \exp\left( \frac{\pi^2\gs^2(0)}{2h^2}H_0(B) \right) \ind_{\{H_0(B)\leq t/(L(T)^2\gs^2(0))\}} \ind_{\{\forall s\leq H_0(B),B_s\in[0,h\gs(s/T)/\gs(0)]\}}\Bigg].
	\end{aligned}\]
	Let $(W_s)_{s\geq0}$ denote the standard, time-homogeneous Brownian motion, recall the definition of $J(s)$ from~\eqref{eq:BM:timechange}, and recall~\eqref{eq:gsapprox}. In particular, we have for $s\in[0,T]$,
	\[
	s\,\gs^2(0)\left(1-\gh^{-2}\tfrac sT\right)^2\,\leq\, J(s)\,\leq\, s\,\gs^2(0)\left(1+\gh^{-2}\tfrac sT\right)^2.
	\]
	Thus, applying the time change~\eqref{eq:BM:timechange} gives,
	\[\begin{aligned} &\bE_{(h-x)}\Bigg[ \exp\left( \frac{\pi^2}{2h^2}\gs^2(0)H_0(B) \right) \ind_{\{H_0(B)\leq t/(L(T)^2\gs^2(0))\}} \ind_{\{\forall s\leq H_0(B),\,B_s\in[0,h\,\gs(s/T)/\gs(0)]\}}\Bigg] \\
		&\; =\; \bE_{(h-x)}\Bigg[ \exp\left( \frac{\pi^2}{2h^2}\gs^2(0)J^{-1}(H_0(W)) \right) \ind_{\{H_0(W)\leq J(t/[L(T)^2\gs^2(0)])\}} \\
		&\qquad\qquad\qquad\qquad\qquad\qquad\qquad\qquad\qquad\qquad\qquad \times \ind_{\{\forall s\leq H_0(W),\,W_s\in[0,h\,\gs(s/T)/\gs(0)]\}}\Bigg] \\
		&\; \leq\; \bE_{(h-x)}\Bigg[ \exp\left( \frac{\pi^2}{2h^2}\big(1-\gh^{-2}\tfrac tT\big)^{-2} H_0(W) \right) \ind_{\big\{H_0(W)\leq \frac{t}{L(T)^2}(1+\gh^{-2}t/T)^2\big\}} \\
		&\qquad\qquad\qquad\qquad\qquad\qquad\qquad\qquad\qquad\qquad\qquad \times \ind_{\{\forall s\leq H_0(W),\,W_s\in[0,h\,(1+\gh^{-2}t/T)]\}}\Bigg].
	\end{aligned}\]
	Moreover, on the event $H_0(W)\leq \frac{2t}{L(T)^2}$, one has,
	\[\big(1-\gh^{-2}\tfrac tT\big)^{-2} H_0(W)=\big(1+\gh^{-2}\tfrac tT\big)^{-2} H_0(W) + o(L(T)),\]
(recall that $t\llbis\sqrt{TL(T)^3}$). Then, the expectation above has already been estimated in~\cite[Lemma~2.2.1]{MaiPhD} (see also~\cite[Lemma~7.1]{Mai16}): therefore, we have for some universal $C>0$,
	\[\begin{aligned} &\bE_{(h-x)}\Bigg[ \exp\left( \frac{\pi^2}{2h^2}\big(1+\gh^{-2}\tfrac tT\big)^{-2} H_0(W) \right) \ind_{\big\{H_0(W)\leq \frac{t}{L(T)^2}(1+\gh^{-2}t/T)^2\big\}} \\
		&\qquad\qquad\qquad\qquad\qquad\qquad\qquad\qquad\qquad\qquad\qquad \times \ind_{\{\forall s\leq H_0(W),\,W_s\in[0,h\,(1+\gh^{-2}t/T)]\}}\Bigg]\\
		&\qquad \leq \; \pi \tfrac{t}{h^2L(T)^2} \sin\Big(\pi \tfrac{h-x}h \big(1+\gh^{-2}\tfrac tT\big)^{-1}\Big) +C\tfrac{h-x}h\big(1+\gh^{-2}\tfrac tT\big)^{-1} =O\big( tL(T)^{-2}\vee 1\big),
	\end{aligned}\]
	uniformly in $\gs\in\gsens$ and $0\leq t(T)\llbis\sqrt{TL(T)^3}$, which concludes the proof.
\end{proof}

\subsection{Critical case} \label{sec:moments:crit}
    In this section we assume there exists $\ga>0$ such that $L(T)\sim\ga T^{1/3}$ as $T\to+\infty$. Recall the definition and properties of $\Psi$ from~\eqref{def:Psi}. We recollect the following result from \cite{Mal15}.
\begin{lemma}[{\cite[Lemma 2.4]{Mal15}}]\label{lem:mallein15}
	Let $W$ a standard time-homogeneous Brownian motion, $q\in\R$, $0<a<b<1$ and $0<a'<b'<1$. Then,
	\begin{equation}\begin{aligned}
			&\lim_{t\to\infty}\frac1t\sup_{x\in[0,1]}\log\bE_x\bigg[e^{-q\int_0^t W_s\dd s}\ind_{\{\forall s\leq t,\,W_s\in[0,1]\}}\bigg]\\
			&\qquad=\; \lim_{t\to\infty}\frac1t\inf_{x\in[a,b]}\log\bE_x\bigg[e^{-q\int_0^t W_s\dd s}\ind_{\{\forall s\leq t,\,W_s\in[0,1]\,;\, \,W_t\in[a',b']\}}\bigg]\;=\; \Psi(q)\;.
	\end{aligned}\end{equation}
\end{lemma}

In the critical regime, recall from~\eqref{def:gga:crit} and~\eqref{eq:defxh} that the lower and upper barriers are defined, for $t\in[0,T]$, by
\[
\gga_T^\crit(t)=v(t/T)\,T - w_{h,T}(t/T)\, L(T) - x\,\gs(0)\, L(T)\;,
\]
and $\ol\gga_T^\crit(t)=\gga_T^\crit(t) +h\gs(t/T) L(T)$ for some $h>x>0$, and where $w_{h,T}\in\cC^1([0,1])$ is defined in~\eqref{def:w_T:crit}. 

\begin{remark}
	Let us point out that $w_{h,T}(u)\sim \frac{\pi^2}{2\ga^3h^2}v(u)$ as $\ga\to0$, so that choice of upper barrier matches the sub-critical asymptotics of the $N$-BBM when $\ga$ is small (recall~\eqref{eq:defmanyrg}, Lemma~\ref{lem:comparison:m:gga} and that $L(T)=o(T/L(T)^2)$ in that case). Moreover, the relation $\Psi(-q)=q+\Psi(q)$, $q\in\R$ yields \[
	\ga\left(w_{h,T}(u)-h\int_0^u(\gs')^-(s)\,\dd s\right)\,\underset{\ga\to+\infty}{\longrightarrow}\, \frac{\mathrm{a}_1}{2^{1/3}}\int_0^u \gs(v)^{1/3}|\gs'(v)|^{2/3}\dd v\,,
	\] for all $u\in(0,1]$. Thus, multiplying these terms by $T^{1/3}\sim \ga^{-1} L(T)$, we see that the upper barrier matches the super-critical asymptotics of the $N$-BBM when $\ga$ is large; and when $\gs$ is decreasing, one also recovers the asymptotics of Proposition~\ref{prop:main:gsdecreasing}.
\end{remark}

Recalling~\eqref{eq:defA:allregimes}, the set of descendants remaining between the barriers and ending in a (rescaled) interval $I\subset[0,h]$ at time $t$ is denoted by $A_{T,I}^\crit(t)$.

\begin{proposition}[First moment, Critical] \label{prop:1stmom:crit} 
	Let $h>0$. One has, as $T\to+\infty$,
	\begin{equation}\label{eq:prop:1stmom:crit:UB}
		\bbE\big[|A_{T,I}^\crit(t)|\big] \,\leq\, e^{(x-\inf I)L(T)+o(T^{1/3})}\,,
	\end{equation}
	uniformly in $t\in[0,T]$, $\gs\in\gsens$, $x\in[0,h]$, and $I\subset[0,h]$ an interval. Moreover, one also has, as $T\to+\infty$, for every interval $I\subset [0,h]$ with non-empty interior,
	\begin{equation}\label{eq:prop:1stmom:crit:LB}
		\bbE\big[|A_{T,I}^\crit(t)|\big] \,\geq\, e^{(x-\inf I)L(T)+o(T^{1/3})}\,,
	\end{equation}
	locally uniformly in $x\in(0,h)$, $\inf I\in[0,h)$ and $\sup I-\inf I\in(0,h]$; and uniformly in $\gs\in\gsens$ and $t=t(T)$ such that $L(T)\llbis t(T)\leq T$ for $T$ sufficiently large.
\end{proposition}

Here again, this proposition should be compared with Propositions~\ref{prop:1stmom:supercrit} and~\ref{prop:1stmom:subcrit} for the super- and sub-critical regimes; and the definitions of ``locally uniformly'' and ``uniformly'' are the same as in Remark~\ref{rem:1stmom}.

\begin{proof}
We follow a similar strategy as in \cite{Mal15}, approximating the time-inhomogeneous variance by constants on suitable time intervals.
	Recall~\eqref{eq:MtOGirs:A} from Lemma~\ref{lem:MtOGirs}. Notice that~\eqref{def:gga:crit} implies
	\begin{equation}\label{eq:crit:1stmom:0}
		\gga^\crit_T{}'(s)\,=\, \gs(s/T) + w'_{h,T}(s/T)\frac{L(T)}{T}\,=\, \gs(s/T) + O(T^{-2/3})\,,
	\end{equation}
	uniformly in $s\in[0,T]$ and $\gs\in\gsens$; more precisely, one can check that the function $\Psi(\cdot)$ is bounded on $[-\ga^3h^3\gh^{-2},\ga^3h^3\gh^{-2}]$, as well as $w_{h,T}(\cdot)$ and $w_{h,T}'(\cdot)$ on $[0,1]$, uniformly in $\gs\in\gsens$. In the following, we let
	\begin{align*}
	&\ol\Psi:=\sup\{|\Psi(q)|,q\in[-\ga^3h^3\gh^{-2},\ga^3h^3\gh^{-2}]\}\,,\\
	\text{and}\qquad &\ol\Psi':=\sup\{|\Psi'(q)|,q\in[-\ga^3h^3\gh^{-2},\ga^3h^3\gh^{-2}]\}\,,
	\end{align*}
	which are well defined since $\Psi$ is convex on $\R$. Therefore, on the event $\{\frac{B_s}{\gs(s/T)L(T)}\in[0,h],\forall s\leq t\}$, one deduces from~\eqref{eq:MtOGirs:A} and a straightforward computation that,
	\begin{align}\label{eq:crit:1stmom:1}
			\bbE\big[|A_{T,I}^\crit(t)|\big]\,&=\, \exp\left(xL(T)+\frac{L(T)}T\int_0^t\frac{w_{h,T}'(s/T)}{\gs(s/T)}\dd s +O(T^{-1/3})\right) \\
			\notag &\quad \times\bE_{x\gs(0)L(T)}\bigg[\exp\left(-\frac{B_t}{\gs(t/T)}-\frac1T\int_0^t\frac{\gs'(s/T)}{\gs^2(s/T)}B_s\dd s\right) \\
			\notag &\qquad\qquad\qquad\qquad\quad \times \ind_{\left\{\forall s\leq t,\, \frac{B_s}{\gs(s/T) L(T)}\in[0,h] \,;\, \frac{B_t}{\gs(t/T)L(T)}\in I\right\}}\bigg].
	\end{align}
	
	\noindent\emph{Upper bound.} Let $z:=\inf I$. Let $t\in[0,T]$, and assume first that $(t,T)$ satisfies $t\leq T^{5/6}$. Then,~\eqref{eq:crit:1stmom:1} gives
	\begin{align}\label{eq:crit:1stmom:1:smallt}
		\notag\bbE\big[|A_{T,I}^\crit(t)|\big]\,&\leq\, e^{(x-z)L(T)+O(T^{1/6})} \bP_{x\gs(0)L(T)}\left(\forall s\leq t,\,\tfrac{B_s}{\gs(s/T)L(T)}\in[0,h]\;; \tfrac{B_t}{\gs(t/T)L(T)}\in I\right) \\&\leq\, e^{(x-z)L(T)+O(T^{1/6})}\,,
	\end{align}
	uniformly in $\gs\in\gsens$ and $t\in[0,T^{5/6}]$.
	
	For $(t,T)$ such that $t\geq T^{5/6}$, let us split $[0,t]$ into intervals of length $KL(T)^{2}$, where $K$ is a large constant which is determined below. Writing $i_{\max}:=\lfloor t/(KL(T)^{2})\rfloor$, we let $t_i:=iKL(T)^{2}$ for $0\leq i< i_{\max}$, and $t_{i_{\max}}:=t$ (notice that $(t_{i_{\max}}-t_{i_{\max}-1})\in[KL(T)^2,2KL(T)^2]$). Define for $1\leq i\leq i_{\max}$,
	\begin{equation}\label{eq:crit:1stmom:1:timesplit}\begin{aligned}
			&\ol\gs_i:=\sup_{[t_{i-1},t_i]} \gs\,,\quad\ul\gs_i:=\inf_{[t_{i-1},t_i]} \gs\,,\\\text{and}\qquad &\ol\gs'_i:=\sup_{[t_{i-1},t_i]} \gs'\,,\quad\ul\gs'_i:=\inf_{[t_{i-1},t_i]} \gs',
		\end{aligned}
	\end{equation}
	and $\ol\gs_0=\ul\gs_0:=\gs(0)$. Using the Markov property at times $t_i$, $1\leq i<i_{\max}$, one has
	\begin{align}\label{eq:crit:1stmom:2}
		&\bE_{x\gs(0)L(T)}\left[\exp\left(-\frac{B_t}{\gs(t/T)}-\frac1T\int_0^tB_s\frac{\gs'(s/T)}{\gs^2(s/T)}\dd s\right)\ind_{\left\{ \forall s\leq t,\,\frac{B_s}{\gs(s/T) L(T)}\in[0,h] \,;\, \frac{B_t}{\gs(t/T)L(T)}\in I\right\}}\right]\\
		\notag &\leq e^{-zL(T)}\prod_{i=1}^{i_{\max}}\!\sup_{y\in[0,h\ol\gs_iL(T)]}\!\bE_{(t_{i-1},y)}\!\left[ 
		\exp\left(-\frac1T\int_{t_{i-1}}^{t_i} B_s\frac{\ul\gs'_i}{\ol\gs_i^2}\dd s\right)
		\ind_{\{\forall s\in[t_{i-1},t_i],\,B_s\in[0,h\ol\gs_iL(T)]\}} \right]\!\!.
	\end{align}
	To lighten notation, let us focus on the factor $i=1$, but the following proof holds for other blocks as well (including $[t_{i_{\max}-1},t_{i_{\max}}]$). Let $\eps>0$. Recalling the time-change $J(s):=\int_0^s \gs^2(r/T)\dd r$, $s\leq T$ from~\eqref{eq:BM:timechange}, and using the Brownian scaling property, we have for $y\in[0,h\ol\gs_1L(T)]$,
	\begin{align}\label{eq:crit:1stmom:2.5}
		&\bE_{y}\left[ \exp\left(-\frac1T \frac{\ul\gs'_1}{\ol\gs_1^2}\int_{0}^{t_1} B_s\dd s\right)\ind_{\{B_s\in[0,h\ol\gs_1L(T)],\,\forall s\leq t_1\}} \right]\\
		\notag &\quad=\;\bE_{y}\left[ \exp\left(-\frac1T \frac{\ul\gs'_1}{\ol\gs_1^2}\int_{0}^{J(t_1)} W_s \,(J'(J^{-1}(s)))^{-1} \dd s\right)\ind_{\{W_s\in[0,h\ol\gs_1L(T)],\,\forall s\leq J(t_1)\}} \right]\\
		\notag &\quad\leq\; \bE_{y}\left[ \exp\left(- \frac1T\frac{\ul\gs'_1}{\ol\gs_1^4}\int_{0}^{\ul\gs_1^2 t_1} W_s  \dd s\right)\ind_{\{W_s\in[0,h\ol\gs_1L(T)],\,\forall s\leq \ul\gs_1^2 t_1\}} \right] \\
		\notag &\quad\leq\; \bE_{y/(h\ol\gs_1L(T))}\bigg[ \exp\left(-\frac{1}{T}\frac{\ul\gs'_1}{\ol\gs_1^4}(h\ol\gs_1L(T))^3\int_{0}^{\ul\gs_1^2 t_1/(h\ol\gs_1L(T))^2} W_s  \dd s\right) \\
		\notag &\qquad\qquad\qquad\qquad\qquad\qquad\qquad\qquad\qquad\qquad\times \ind_{\{W_s\in[0,1],\,\forall s\leq \ul\gs_1^2 t_1/(h\ol\gs_1L(T))^2\}} \bigg] \\
		\notag &\quad\leq\; \bE_{y/(h\ol\gs_1L(T))}\bigg[ \exp\left(-\frac{\ul\gs'_1}{\ol\gs_1}\ga^3h^3(1-\eps)\int_{0}^{\ul\gs_1^2 t_1/(h\ol\gs_1L(T))^2} W_s  \dd s\right) \\
		\notag &\qquad\qquad\qquad\qquad\qquad\qquad\qquad\qquad\qquad\qquad\times \ind_{\{W_s\in[0,1],\,\forall s\leq \ul\gs_1^2 t_1/(h\ol\gs_1L(T))^2\}} \bigg] ,
	\end{align}
	where the last inequality holds for $T$ sufficiently large, by recalling~\eqref{eq:gsapprox} and that $L(T)\sim\ga T^{1/3}$. Let $\eps>0$, and recall that $t_i-t_{i-1}\geq KL(T)^2$ for all $i$. Assume w.l.o.g that $L(T)>0$ for all $T$: then $t_1\to+\infty$ as $K\to+\infty$ uniformly in $T$: applying Lemma~\ref{lem:mallein15}, there exists $K_\eps>0$ such that for $K>K_\eps$, one has
	\begin{align*}
		&\sup_{y\in[0,1]} \bE_{y}\left[ \exp\left(-\frac{\ul\gs'_1}{\ol\gs_1}\ga^3h^3(1-\eps)\int_{0}^{\ul\gs_1^2 t_1/(h\ol\gs_1L(T))^2} W_s  \dd s\right)\ind_{\{W_s\in[0,1],\,\forall s\leq \ul\gs_1^2 t_1/(h\ol\gs_1L(T))^2\}} \right]\\
		&\qquad\leq\; \exp\left(\frac{t_1}{h^2L(T)^2}\frac{\ul\gs_1^2}{\ol\gs_1^2} \left[\Psi\left(\frac{\ul\gs'_1}{\ol\gs_1}\ga^3h^3(1-\eps)\right)+\eps\right] \right),
	\end{align*}
	for all $T\geq0$. Then, for $T$ sufficiently large,~\eqref{eq:gsapprox} and the definitions of $\ol\Psi$, $\ol\Psi'$ yield,
	\[
	\frac{\ul\gs_1^2}{\ol\gs_1^2} \left[\Psi\left(\frac{\ul\gs'_1}{\ol\gs_1}\ga^3h^3(1-\eps)\right)+\eps\right]\,\leq\, \Psi\left(\frac{\ul\gs'_1}{\ol\gs_1}\ga^3h^3\right) + c_1\ol\Psi \frac{KL(T)^2}{T} + \eps c_1(\Psi'+1)  \,,
	\]
	for some constant $c_1>0$. Therefore, there exists $T_0(K,\eps)>0$ such that, for $K$ sufficiently large and $T\geq T_0(K,\eps)$,~\eqref{eq:crit:1stmom:2} becomes
	\begin{align*}
		\notag &\bE_{x\gs(0)L(T)}\left[\exp\left(-\frac{B_t}{\gs(t/T)}-\frac1T\int_0^tB_s\frac{\gs'(s/T)}{\gs^2(s/T)}\dd s\right)\ind_{\left\{\forall s\leq t,\, \frac{B_s}{\gs(s/T) L(T)}\in[0,h] \,;\, \frac{B_t}{\gs(t/T)L(T)}\in I\right\}}\right]\\
		&\qquad\leq\; e^{-zL(T)}\exp\left(\sum_{i=1}^{i_{\max}}\frac{(t_i-t_{i-1})}{h^2L(T)^2} \left[\Psi\left(\frac{\ul\gs'_i}{\ol\gs_i}\ga^3h^3\right) + c_1\ol\Psi \frac{KL(T)^2}{T} + \eps c_1(\Psi'+1)\right] \right).
	\end{align*}
	Using a Riemann sum approximation, there exists $T_0(K,\eps)>0$ such that, for $K$ large and $T\geq T_0(K,\eps)$, and for $t\geq T^{5/6}$, one has
	\begin{equation}\label{eq:crit:1stmom:2.5:Riemann}
		\sum_{i=1}^{i_{\max}} \frac{(t_i-t_{i-1})}{L(T)^2}\,\Psi\left(\frac{\ul\gs'_i}{\ol\gs_i}\ga^3h^3\right) \leq \frac1{L(T)^{2}} \int_0^t \Psi\left(\frac{\gs'(r/T)}{\gs(r/T)} \ga^3 h^3\right) \dd r + \eps\frac t{L(T)^2}\;.
	\end{equation}
	Moreover, one has
	\begin{equation}
		\sum_{i=1}^{i_{\max}}  \frac{(t_i-t_{i-1})}{h^2L(T)^2} \left[\ol\Psi \frac{KL(T)^2}{T} + \eps (\Psi'+1)\right]
		\,=\, O(1) + \eps \, O(T^{1/3})\,.
	\end{equation}
	Recollect~\eqref{eq:crit:1stmom:1} and recall that $L(T)\sim\ga T^{1/3}$. 
	Recalling the definition of $w_{h,T}$ from~\eqref{def:w_T:crit} and that $\frac1{L(T)^2}\sim \ga^{-3}\frac{L(T)}{T}$, we finally obtain that there exist some $C,C'>0$ (depending only on $\gh, \eps$ and $\ga$) such that for $T$ large enough, one has
	\begin{equation}\label{eq:crit:1stmom:2.5:ccl}
		\bbE\big[|A_{T,z}^\crit(t)|\big]\,\leq\,\exp\left((x-z)L(T) +C+C'\eps T^{1/3}\right),
	\end{equation}
	for all $t\geq T^{5/6}$. Taking the maximum of~\eqref{eq:crit:1stmom:1:smallt} and~\eqref{eq:crit:1stmom:2.5:ccl}, and letting $\eps\to0$, this finally yields the expected upper bound uniformly in $\gs(\cdot)$ and $t\in[0,T]$.\smallskip
	
	\noindent\emph{Lower bound.} 
	Similarly to the upper bound, we distinguish the cases $t$ smaller or larger than $T^{5/6}$. We first consider the case $t\geq T^{5/6}$, using the same time split with intervals of length $KL(T)^{2}$ and notation as above (recall~\eqref{eq:crit:1stmom:1:timesplit}). 
	Let $0<a<b<h$ such that $x\in(a,b)$, and $\eps>0$ such that $z+\eps<h$. We bound the expectation in~\eqref{eq:crit:1stmom:1} from below by constraining the trajectory to pass through the intervals $[a,b]\cdot \ul\gs_iL(T)$ at times $t_i$, $1\leq i\leq i_{\max}-1$. Hence, the Markov property gives
	\begin{align}\label{eq:crit:1stmom:3}
		&\bE_{x\gs(0)T^{1/3}}\left[\exp\left(-\frac{B_t}{\gs(t/T)}-\frac1T\int_0^tB_s\frac{\gs'(s/T)}{\gs^2(s/T)}\dd s\right)\ind_{\left\{\substack{\forall s\leq t,\,B_s\in[0,h\gs(s/T)L(T)]\,;\\ B_t\geq z\gs(t/T)L(T)}\right\}}\right]\\
		\notag &\geq\; e^{-(z+\eps)L(T)} \\
		\notag &\;\;\;\times \prod_{i=1}^{i_{\max}-1}\!\inf_{y\in [a,b]\cdot\ul\gs_{i-1}L(T)}\bE_{(t_{i-1},y)}\left[ \exp\left(-\frac1T\frac{\ol\gs'_i}{\ul\gs_i^2}\int_{t_{i-1}}^{t_i} \! B_s\dd s\right)\ind_{\left\{\substack{\forall s\in[t_{i-1},t_i],\,B_s\in[0,h\ul\gs_iL(T)]\,;\\ B_{t_i}\in [a\ul\gs_iL(T),b\ul\gs_iL(T)]}\right\}} \right]\\
		\notag &\;\;\;\times \inf_{y\in[a,b]\cdot\ul\gs_{i_{\max}-1}L(T)}\bE_{(t_{i_{\max}-1},y)}\bigg[ \exp\bigg(-\frac1T\frac{\ol\gs'_{i_{\max}}}{\ul\gs_{i_{\max}}^2}\int_{t_{i_{\max}-1}}^{t_{i_{\max}}} \! B_s\dd s\bigg)\\
		\notag &\qquad\qquad\qquad\qquad\qquad\qquad\qquad\qquad\qquad\qquad\quad\times \ind_{\left\{\substack{\forall s\in[t_{i_{\max}-1},t],\,B_s\in[0,h\ul\gs_{i_{\max}}L(T)]\,;\\ B_{t}\in [z\gs(t/T)L(T),(z+\eps)\gs(t/T)L(T)]}\right\}} \bigg].
	\end{align}
	Let us focus on the factor $i=1$ here again (others are handled similarly, including the last one). Reproducing the time-change argument from~\eqref{eq:crit:1stmom:2.5}, one has for $y\in[a,b]\cdot \gs(0)L(T)$,
	\begin{align*}
		&\bE_{y}\left[ \exp\left(-\frac1T\frac{\ol\gs'_1}{\ul\gs_1^2}\int_{0}^{t_1}  B_s\dd s\right)\ind_{\left\{\substack{\forall s\leq t_1,\,B_s\in[0,h\ul\gs_1L(T)]\,;\\ B_{t_1}\in [a\ul\gs_1L(T),b\ul\gs_1L(T)]}\right\}} \right]\\
		&\quad\geq\bE_{y}\left[ \exp\left(-\frac1T\frac{\ol\gs'_1}{\ul\gs_1^4}\int_{0}^{J(t_1)}  W_s\dd s\right)\ind_{\left\{\substack{\forall s\leq J(t_1),\,W_s\in[0,h\ul\gs_1L(T)]\,;\\ W_{J(t_1)}\in [a\ul\gs_1L(T),b\ul\gs_1L(T)]}\right\}} \right].
	\end{align*}
	Recall~\eqref{eq:gsapprox}. Applying Lemma~\ref{lem:mallein15} and the Brownian scaling property, one obtains similarly to the upper bound (we do not write the details again),
	\[\begin{aligned}
		&\inf_{y\in[a,b]\cdot \gs(0)L(T)} \bE_{y}\left[ \exp\left(-\frac1T\frac{\ol\gs'_1}{\ul\gs_1^4}\int_{0}^{J(t_1)}  W_s\dd s\right)\ind_{\left\{\substack{\forall s\leq J(t_1),\,W_s\in[0,h\ul\gs_1L(T)]\,;\\ W_{J(t_1)}\in [a\ul\gs_1L(T),b\ul\gs_1L(T)]}\right\}} \right]\\
		&\qquad\geq\; \exp\left[\frac{t_1}{h^2L(T)^2}\Psi\left(\frac{\ol\gs'_1}{\ul\gs_1}\ga^3h^3\right)  - O(L(T)^2/T) - \eps\, O(1)\right].\end{aligned}\]
	for some $K_\eps>0$, $K\geq K_\eps$ and $T_0(K,\eps)>0$, $T\geq T_0(K,\eps)$. Reproducing this lower bound for all factors of~\eqref{eq:crit:1stmom:3}, using a Riemann sum approximation similar to~\eqref{eq:crit:1stmom:2.5:Riemann} and recollecting~\eqref{eq:crit:1stmom:1}, we may conclude similarly to the upper bound in the case $t\geq T^{5/6}$. 
	
	Let us now consider the case $t\leq T^{5/6}$. Recall the proof of the first moment lower bound in the sub-critical case (Proposition~\ref{prop:1stmom:subcrit}), that is (\ref{eq:subcrit:1stmom:2:LB}, \ref{eq:subcrit:1stmom:4}--\ref{eq:subcrit:1stmom:4ter}). We define as in \eqref{eq:subcrit:1stmom:4bis},
	\[
	\tau=\tau(T)\,:=\, \frac1{\gs^2(0)L(T)^2}\int_0^t\gs^2(u/T)\,\dd u\,,
	\]
	and $(\eps_T)_{T\geq0}$ such that $T^{-1/3}\ll \eps_T\ll T^{-1/6}$ as $T\to+\infty$. Then, the Brownian scaling property and a time change yield, similarly to the sub-critical regime,
	\begin{align}\label{eq:crit:1stmom:4:smallt}
		&\bbE\big[|A_{T,I}^\crit(t)|\big]\\
		\notag &\qquad\geq\, e^{(x-z - \eps_T)L(T) + o(T^{1/3})} \,\bP_{x}\Big(\forall\, s\leq \tau,\,W_s\in[0,h]\,;\, W_{\tau}\in [z+\eps_T/3,z+2\eps_T/3]\Big)\,.
	\end{align}
	It remains to prove that the probability above is larger than $e^{o(T^{1/3})}$. Recall Lemma~\ref{lem:mai16}: proceeding similarly to~(\ref{eq:subcrit:1stmom:4quad}--\ref{eq:subcrit:1stmom:4quad:sinus}), there exists $R>0$ such that, for all $t'\ge R$,
	\[
	\bP_x\big(\forall s\leq t',\,W_s\in[0,h]\,;\, W_{t'}\in [z+\eps_T/3,z+2\eps_T/3]\big)\;\geq\; \frac1{h} \frac{\eps_T^2}{18h} \sin\left(\frac{\pi x}{h}\right) \exp\left(-\frac{\pi^2}{2h^2}t'\right).
	\]
	If $t,T$ satisfy $2RL(T)^2\leq t\leq T^{5/6}$, then~\eqref{eq:gsapprox} implies $R\leq\tau\leq 2\ga^{-2}T^{1/6}$ for $T$ sufficiently large, uniformly in $\gs\in\gsens$. In particular the r.h.s. above, evaluated at $t'=\tau$, is larger than $e^{O(T^{1/6})}$, uniformly in $\gs\in\gsens$ and locally uniformly in $x,z$. On the other hand, if $t,T$ satisfy $t\leq 2RL(T)^2$, then $\tau\leq 4R$ for $T$ sufficiently large, uniformly in $\gs\in\gsens$. Moreover, one notices that $t\ggbis L(T)$ and $\eps_T\ll T^{-1/6}$ imply that $\tau^{-1/2}\eps_T\to0$ as $T\to+\infty$. Recalling~(\ref{eq:subcrit:1stmom:6}--\ref{eq:subcrit:1stmom:7}) from the sub-critical case, we have already proven by introducing a Brownian bridge that, under this assumption, the probability in~\eqref{eq:crit:1stmom:4:smallt} is larger than $e^{o(L(T))}$ for $T$ large: we do not replicate the details of the proof here.
	
	Therefore, we derived three lower bounds which hold respectively in the cases $t\geq T^{5/6}$, $2RL(T)^2\leq t\leq T^{5/6}$ and $L(T)\llbis t\leq 2RL(T)^2$: taking the minimum of these, we obtain a lower bound that holds uniformly in $L(T)\llbis t\leq T$ and $\gs\in\gsens$, which fully concludes the proof of the proposition.
\end{proof}

\begin{proposition}[Second moment, Critical] \label{prop:2ndmom:crit} Let $h>0$. As $T\to+\infty$, one has
	\begin{equation}
		\bbE\big[|A_{T,z}^\crit(t)|^2\big] \;\leq\; e^{(x+h-2z) L(T) +o(T^{1/3})}\,.
	\end{equation}
	uniformly in $x\in[0,h]$, $z\in[0,h]$, $\gs\in\gsens$ and $t\in[0,T]$.
\end{proposition}
This statement is analogous to Propositions~\ref{prop:2ndmom:supercrit} and~\ref{prop:2ndmom:subcrit} in the super- and sub-critical cases respectively. Since its proof is identical (using the upper bound~\eqref{eq:prop:1stmom:crit:UB}, and observing that $t\leq T=e^{o(T^{1/3})}$ in the critical regime), we do not reproduce it here.

We now state the analogue of Proposition~\ref{prop:killedpart:subcrit} regarding the number of particles killed by the upper barrier. Recall from~\eqref{eq:defR} that $R^\crit_T(s,t)$, $0\leq s\leq t\leq T$ denotes the expected number of particles killed by the upper barrier on the time interval $[s,t]$.
\begin{proposition}[Killed particles, Critical] \label{prop:killedpart:crit}
	Let $h>0$. Then as $T\to+\infty$, one has
	\begin{equation}
		\bbE[R^\crit_T(0,T)]\,\leq\, e^{-(h-x)L(T) + o(T^{1/3})}\;,
	\end{equation}
	uniformly in $x\in[0,h]$ and $\gs\in\gsens$.
\end{proposition}

\begin{proof} 
	Recollect~\eqref{eq:MtOGirs:R} from Lemma~\ref{lem:MtOGirs}. Notice that we have $(\ol\gga_T^\crit)'(\cdot)=(\gga_T^\crit)'(\cdot)+h\gs'(\cdot/T)\frac{L(T)}T$, and recall that $\ol\gga_T^\crit(0)=(h-x)\gs(0)L(T)$. Combining this with \eqref{eq:crit:1stmom:0}, we deduce from~\eqref{eq:MtOGirs:R} with a straightforward computation that
	\begin{align*}
		&\bbE[R^\crit_T(0,T)]\,=\, e^{-(h-x)L(T) +O(T^{-1/3})} \\
		&\qquad\times \bE_{(h-x)\gs(0)L(T)} \Bigg[ \ind_{\{H_0(B)\leq T\}} \ind_{\left\{\forall s\leq H_0(B),\,\frac{B_s}{\gs(s/T)L(T)}\in[0,h]\right\}} \\
		&\qquad\times \exp\left( \frac{L(T)}{T}\int_{0}^{H_0(B)}\frac{w'_{h,T}(s/T)-h\gs'(s/T)}{\gs(s/T)}\dd s+\frac1T \int_0^{H_0(B)}\frac{\gs'(s/T)}{\gs^2(s/T)}B_s\dd s\right) \Bigg],
	\end{align*}
	and it remains to show that the latter expectation is of order $\exp(o(T^{1/3}))$. Let us apply Girsanov's theorem and the Brownian symmetry property to the process $(h\gs(s/T)L(T)-B_s)_{s\geq0}$: recalling estimates from~\eqref{eq:supercrit:girsanov} and setting $\tilde H(B):=H_0(B-h\gs(\cdot/T)L(T))$ to lighten notation, this yields,
	\begin{align*}
		&\bE_{(h-x)\gs(0)L(T)} \Bigg[ \ind_{\{H_0(B)\leq T\}} \ind_{\left\{\forall s\leq H_0(B),\,\frac{B_s}{\gs(s/T)L(T)}\in[0,h]\right\}} \\
		&\qquad\qquad\qquad\qquad\qquad\qquad\;\times e^{\frac{L(T)}{T}\int_{0}^{H_0(B)}\frac{w'_{h,T}(s/T)-h\gs'(s/T)}{\gs(s/T)}\dd s+\frac1T \int_0^{H_0(B)}\frac{\gs'(s/T)}{\gs^2(s/T)}B_s\dd s} \Bigg]\\
		&\quad = e^{O(T^{-1/3})}\,\bE_{x\gs(0)L(T)} \Bigg[ \ind_{\{\tilde H(B)\leq T\}} \ind_{\left\{\forall s\leq \tilde H(B),\,\frac{B_s}{\gs(s/T)L(T)}\in[0,h]\right\}} \\
		&\qquad\qquad\qquad\qquad\qquad\qquad\qquad\qquad\times e^{\frac{L(T)}{T}\int_{0}^{\tilde H(B)}\frac{w'_{h,T}(s/T)}{\gs(s/T)}\dd s-\frac1T \int_0^{\tilde H(B)}\frac{\gs'(s/T)}{\gs^2(s/T)}B_s\dd s} \Bigg].
	\end{align*}
	Notice that the latter formula is equivalent to rewriting~\eqref{eq:MtOGirs:R} in terms of $\gga^\anyrg_T$ instead of $\ol\gga^\anyrg_T$ (details are left to the reader). Then we split $[0,T]$ into intervals of length $T^{5/6}$, setting $t_i:=iT^{5/6}$ for $0\leq i< i_{\max}:=\lfloor T^{1/6}\rfloor$ and $t_{i_{\max}}:=T$. Thus we have,
	\[\begin{aligned}
		&\bE_{x\gs(0)L(T)} \! \Bigg[ \ind_{\{\tilde H(B)\leq T\}} \ind_{\left\{\forall s\leq \tilde H(B),\,\frac{B_s}{\gs(s/T)L(T)}\in[0,h]\right\}} e^{\frac{L(T)}{T}\int_{0}^{\tilde H(B)}\frac{w'_{h,T}(s/T)}{\gs(s/T)}\dd s-\frac1T \int_0^{\tilde H(B)}\frac{\gs'(s/T)}{\gs^2(s/T)}B_s\dd s} \Bigg]\\
		&\;\;\leq \sum_{i=0}^{i_{\max}-1} \bE_{x\gs(0)L(T)}\bigg[\ind_{\{t_i<\tilde H(B)\leq t_{i+1}\}} \ind_{\big\{\forall s\leq \tilde H(B), \,\frac{B_s}{\gs(s/T)L(T)}\in [0,h]\big\}} \\
		&\qquad\qquad\qquad\qquad\qquad\qquad\qquad\qquad\qquad\;\times e^{\frac{L(T)}{T}\int_{0}^{\tilde H(B)}\frac{w'_{h,T}(s/T)}{\gs(s/T)}\dd s-\frac1T \int_0^{\tilde H(B)}\frac{\gs'(s/T)}{\gs^2(s/T)}B_s\dd s} \bigg]\\
		&\;\;\leq \sum_{i=0}^{i_{\max}-1} e^{O(T^{1/6})} \,\bE_{x\gs(0)L(T)}\bigg[\ind_{\big\{\forall s\leq iT^{5/6},\,\frac{B_s}{\gs(s/T)L(T)}\in[0,h]\big\}} \\
		&\qquad\qquad\qquad\qquad\qquad\qquad\qquad\qquad\qquad\times e^{\frac{L(T)}{T}\int_{0}^{iT^{5/6}}\frac{w'_{h,T}(s/T)}{\gs(s/T)}\dd s-\frac1T \int_0^{iT^{5/6}}\frac{\gs'(s/T)}{\gs^2(s/T)}B_s\dd s} \bigg],
	\end{aligned}\]
	uniformly in $\gs\in\gsens$. Recalling~(\ref{eq:crit:1stmom:2}--\ref{eq:crit:1stmom:2.5:Riemann}) from the proof of Proposition~\ref{prop:1stmom:crit}, we have already proven that the latter expectation is of order $e^{o(T^{1/3})}$ as $T\to+\infty$ uniformly in $1\leq i\leq i_{\max} =O(T^{1/6})$ (we do not reproduce the details). This concludes the proof of the proposition.
\end{proof}

\section{Lower bound on the maximum of the \texorpdfstring{$N$}{N}-BBM}\label{sec:LB}
In this section we prove Proposition~\ref{prop:main:LB} by considering a BBM between well chosen barriers and showing that it is equal to an $N^-$-BBM with large probability. Then, the lower bound is obtained by showing that the maximum of the BBM between barriers is close to the upper barrier via a moment method. The following results hold for all three regimes $\anyrg\in\{\subcrit, \supercrit, \crit\}$, where we use our notation from Section~\ref{sec:prelim}. 

\subsection{Estimates for the BBM between barriers}
We first focus on the case $\gk\in(0,1)$, and consider a BBM between barriers starting with $N^\gk$ particles. In order to lighten upcoming formulae, we assume that they are initially located at the origin: then, results on the process started from the initial measure $N^\gk\gd_{-\gk\gs(0)L(T)}$ are obtained through a direct shift of upcoming estimates. Let us recall that we omit the integer parts when writing $N^\gk$, $\gk\ge0$ (i.e. we assume that $N^\gk\in\N$).

Since the moment estimates from Section~\ref{sec:moments:subcrit} do not hold on the whole interval $[0,T]$ in the sub-critical case, let us define
\begin{equation}\label{eq:defta}
	t^\anyrg_T\;:=\; 
	\left\{
	\begin{aligned}
		\min\big(L(T)^4,L(T)^2T^{1/3}\big) & \quad\text{if}\quad \anyrg=\subcrit\,,\\
		T & \quad\text{if}\quad \anyrg\in\{\supercrit,\crit\}\,,
	\end{aligned}
	\right.
\end{equation}
which is the time horizon we consider below: we compute a lower bound on the $N$-BBM at time $t^\anyrg_T$, $\anyrg\in\{\subcrit, \supercrit, \crit\}$; then, in the sub-critical regime, we extend the lower bound up to time $T$ with additional arguments.

Let $\gh>0$, and consider a vanishing function $\gth(T)\to0$ as $T\to+\infty$ (depending on $L(T)$), such that for $T$ sufficiently large, one has~\eqref{eq:gthasymp}, as well as
\begin{equation}\label{eq:defta:gthapprox}
	L(T)^3 \llbis t^\subcrit_T \leq (L(T)^3)^{2/3}T^{1/3}  \llbis \sqrt{TL(T)^3}\,,\qquad\text{if }\anyrg=\subcrit\,.
\end{equation}
Let $0<\eps<1$ be small, and fix
\begin{equation}\label{eq:LB:parameters}
	h=1-\eps\,, \qquad x=(1-\gk)(1-\eps) \;\in\;(0,h)\;.
\end{equation}
Then, define the barriers $\gga^\anyrg_T$, $\ol\gga^\anyrg_T$ for each regime $\anyrg\in\{\supercrit,\subcrit,\crit\}$ according to \eqref{def:gga:supercrit},~\eqref{def:gga:subcrit} and~\eqref{def:gga:crit} respectively, with these parameters $h>x>0$; in particular they satisfy~\eqref{eq:defxh}. 

Recall the definitions of $A^\anyrg_{T,I}$, $R^\anyrg_T(s,t)$ from~\eqref{eq:defA:allregimes},~\eqref{eq:defR}, and that $N=N(T)=e^{L(T)}$. Then, moment estimates from Propositions~\ref{prop:1stmom:supercrit}, \ref{prop:2ndmom:supercrit}, \ref{prop:1stmom:subcrit}, \ref{prop:2ndmom:subcrit}, \ref{prop:1stmom:crit} and~\ref{prop:2ndmom:crit} yield for all regimes $\anyrg\in\{\supercrit, \subcrit, \crit\}$, 
\begin{align}\label{eq:allmoments:1stUB}
	\bbE_{\gd_0}\big[|A^\anyrg_{T,I}(t)|\big]\,&\leq\, N^{x-\inf I+o(1)}\,, 
	\\\label{eq:allmoments:1st}
	\bbE_{\gd_0}\big[|A^\anyrg_{T,I}(t)|\big]\,&=\, N^{x-\inf I+o(1)}  \qquad\text{if, additionally,}\qquad t\ggbis L(T)\,,
	\\\label{eq:allmoments:2nd}
	\bbE_{\gd_0}\big[|A^\anyrg_{T,z}(t)|^2\big]\,&\leq\, N^{x+h-2z+o(1)} \,, 
\end{align}
for some vanishing terms $o(1)$ as $T\to+\infty$: in~(\ref{eq:allmoments:1stUB},~\ref{eq:allmoments:2nd}), these error terms are uniform in $\gs\in\gsensrg$, $0\leq t\leq t^\anyrg_T$ and $x,z\in[0,h]$, $I\subset[0,h]$ a non-trivial sub-interval; and in~\eqref{eq:allmoments:1st}, it is uniform in $\gs\in\gsensrg$, $L(T)\llbis t\leq t^\anyrg_T$ and locally uniform in $x,I$. Let us point out that the additional error factors from Propositions~\ref{prop:2ndmom:subcrit} and~\ref{prop:killedpart:subcrit} in the sub-critical case are absorbed into the $N^{o(1)}$, since $t^\subcrit_T\leq L(T)^4=e^{o(L(T))}$.

With those parameters and initial condition, let us first prove that the BBM between the barriers $\gga^\anyrg_T$, $\ol\gga^\anyrg_T$ does not contain more than $N$ particles at any time in $[0,t^\anyrg_T]$ with high probability.
\begin{proposition}\label{prop:allregimes:LB:npart}
	Let $\eps$, $h$, $x$ as in~\eqref{eq:LB:parameters}, and $\gk\in(0,1)$. Then there exist constants $c_1, c_2>0$ such that, for $T$ sufficiently large, one has
	\begin{equation}
		\bbP_{N^{\gk}\gd_0}\big(\exists s\leq t^\anyrg_T:|A_{T}^\anyrg(s)| > N\big) \;\leq\; c_1 N^{-c_2}  \,,
	\end{equation}
	where $c_1,c_2$ are uniformly bounded away from $0$ and $\infty$ in $\gs\in\gsensrg$, and locally uniformly in $\eps,\gk\in(0,1)$.
\end{proposition}
\begin{proof}
	Define $\gl:=1-\frac\eps2(1-\gk)\in(x+\gk,1)$. With a union bound, we write
	\begin{align}\label{eq:prop:allregimes:LB:npart}
		\notag&\bbP_{N^\gk\gd_0}\big(\exists s\leq t^\anyrg_T:|A_{T}^\anyrg(s)| > N\big) \;\leq\; \sum_{k=0}^{t^\anyrg_T-1} \bbP_{N^\gk\gd_0}\big(\exists s\in[k,k+1]:|A_{T}^\anyrg(s)| > N\big) \\&\;\leq\, \sum_{k=0}^{t^\anyrg_T-1} \bbP_{N^\gk\gd_0}\big(|A_{T}^\anyrg(k)| > N^{\gl}\big) + \bbP_{N^\gk\gd_0}\big(|A_{T}^\anyrg(k)| \leq N^{\gl}\,;\, \exists\,s\in[k,k+1]:|A_{T}^\anyrg(s)| > N\big)\;,
	\end{align}
	where we wrote $t^\anyrg_T-1$ instead of $\lceil t^\anyrg_T-1\rceil$ to lighten notation. On the one hand, the additivity of $|A_{T}^\anyrg(k)|$ in the initial measure implies that
	\[\bbE_{N^\gk\gd_0}[|A_{T}^\anyrg(s)|]\,=\,N^\gk\,\bbE_{\gd_0}[|A_{T}^\anyrg(s)|].\]
	One deduces from Markov's inequality and \eqref{eq:allmoments:1stUB} that for all $0\leq k\leq t^\anyrg_T-1$,
	\begin{equation}\label{eq:prop:allregimes:LB:npart:2}
		\bbP_{N^\gk\gd_0}\big(|A_{T}^\anyrg(k)| > N^{\gl}\big) \,\leq\, N^\gk N^{-\gl}\, \bbE_{\gd_0} \big[|A_{T}^\anyrg(k)|\big]\,=\, N^{-(\gl-x-\gk)+o(1)}\;,
	\end{equation}
	for $T$ large, where $o(1)$ does not depend on $k$ or $\gs\in\gsensrg$.
	
	On the other hand, let $(Z_t)_{t\geq0}$ denote the population size of a BBM without selection (in particular its population size is non-decreasing in time), and let $\cA$ denote the set of counting measures on $\R$ with mass at most $N^{\gl}$. Using Markov's property and a straightforward coupling argument, one has for $0\leq k\leq t^\anyrg_T-1$,
	\[
	\bbP_{N^\gk\gd_0}\big(|A_{T}^\anyrg(k)| \leq N^{\gl}\,;\, \exists\,s\in[k,k+1]:|A_{T}^\anyrg(s)| > N\big)\,\leq\, \sup_{\mu\in\cA}\,\bbP_{\mu}(Z_1 >N)\,.
	\]
	Since $\bbE_{\mu}[Z_1]=e^{1/2}\mu(\R)$ for any $\mu\in\cA$, one has by Markov's inequality,
	\[
	\bbP_{\mu}(Z_1 >N)\,\leq\, N^{-(1-\gl)} e^{1/2}\,.
	\]
	Recollecting~(\ref{eq:prop:allregimes:LB:npart},~\ref{eq:prop:allregimes:LB:npart:2}) and~\eqref{eq:defta}, we finally obtain in all regimes,
	\[
	\bbP_{N^\gk\gd_0}\big(\exists s\leq t^\anyrg_T:|A_{T}^\anyrg(s)| > N\big)\,\leq\, L(T)^4\times N^{-\frac\eps2(1-\gk)+o(1)}\,=\, N^{-\frac\eps2(1-\gk)+o(1)}` \,,
	\]
	which concludes the proof.
\end{proof}

We now bound from below the number of particles located at a given height, at a time $t$ not too small.
\begin{proposition}\label{prop:allregimes:LB:endpoint}
	There exist $c_1,c_2>0$ such that, for $T$ sufficiently large, one has
	\begin{equation}
		\bbP_{N^\gk\gd_0}\big(|A_{T,[z,z+\eps]}^\anyrg(t)|\geq N^{1-\eps-z}\big)\;\geq\; 1-c_1 N^{-c_2} \,,
	\end{equation}
	where $c_1,c_2$ are uniformly bounded away from $0$ and $\infty$ in $\gs\in\gsensrg$, $t\in[\gth(T)^{-1}L(T),t^\anyrg_T]$ and $z\in[0,1-2\eps]$, and locally uniformly in $\eps,\gk\in(0,1)$.
\end{proposition}
Notice that, taking $z=1-2\eps$, this proposition implies that, with large probability, there are at least $N^\eps$ particles located in the vicinity of the upper barrier $\ol\gga^\anyrg_T$ at time $t^\anyrg_T$.

\begin{proof}
	Applying Paley-Zygmund's inequality, we have
	\begin{equation}\label{eq:PZ:allregimes:LB}
		\bbP_{N^\gk\gd_0}\big(|A_{T,[z,z+\eps]}^\anyrg(t)| \geq N^{1-\eps-z}\big) \geq 
		\bigg(1-\frac{N^{1-\eps-z}}{\bbE_{N^\gk\gd_0}[|A_{T,[z,z+\eps]}^\anyrg(t)|]}\bigg)^{\!\!2} \; \frac{\bbE_{N^\gk\gd_0}[|A_{T,[z,z+\eps]}^\anyrg(t)|]^2}{\bbE_{N^\gk\gd_0}[|A_{T,[z,z+\eps]}^\anyrg(t)|^2]}\,.
	\end{equation}
	For all $M\in\N$, $I\subset[0,h]$ and $t\geq0$, recall that $\bbE_{M\gd_0}[|A_{T,I}^\anyrg(t)|]=M\bbE_{\gd_0}[|A_{T,I}^\anyrg(t)|]$. Regarding the second moment, by splitting the sum over pairs of particles $u,v\in\cN_t$ depending on whether they come from the same ancestor or two distinct ancestors at time 0, we obtain for any $I\subset[0,h]$, $t\geq0$,
	\begin{equation}\label{eq:PZ:2ndmoment}
		\bbE_{M\gd_0}\big[|A_{T,I}^\anyrg(t)|^2\big] = M \bbE_{\gd_0}\big[|A_{T,I}^\anyrg(t)|^2\big] + M(M-1) \bbE_{\gd_0}[|A_{T,I}^\anyrg(t)|]^2.
	\end{equation}
	Recalling~(\ref{eq:allmoments:1stUB}--\ref{eq:allmoments:2nd}) and the definitions of $x,h$ from~\eqref{eq:LB:parameters}, we obtain
	\begin{align*}
		&\bbP_{N^\gk\gd_0}\big(|A_{T,[z,z+\eps]}^\anyrg(t)| \geq 1\big) \\
		&\qquad\geq\,\left(1-N^{(1-\eps-z)-\gk-(x-z)+o(1)}\right)^2\left(1+N^{x+h-2z-\gk-2(x-z)+o(1)}\right)^{-1} \\
		&\qquad=\, \left(1-N^{-\eps\gk+o(1)}\right)^2\left(1+N^{-\eps\gk+o(1)}\right)^{-1}\, \geq 1-N^{-\eps\gk+o(1)}\,,
	\end{align*}
	as $T\to+\infty$, uniformly in $\gs\in\gsensrg$ and $t\in[\gth(T)^{-1}L(T),t^\anyrg_T]$, which concludes the proof.
\end{proof}

\subsection{Proof of Proposition~\ref{prop:main:LB}}\label{sec:LB:coreprf}
We finally prove Proposition~\ref{prop:main:LB}. This is achieved by coupling the BBM between barriers with an $N$-BBM, through the introduction of an $N^-$-BBM and the application of Lemma~\ref{lem:N-N+couping}. Recall that the point measure of the $N$-BBM throughout time is denoted by $(\cX_{t}^{N})_{t\geq0}$. We first claim the following, which is a consequence of Propositions~\ref{prop:allregimes:LB:npart} and~\ref{prop:allregimes:LB:endpoint}.
\begin{lemma}\label{lem:main:LB:coupling}
	Let $\anyrg\in\{\subcrit,\supercrit,\crit\}$. There exists $c_1,c_2>0$ such that for $T$ sufficiently large, one has
	\begin{equation}\label{eq:LB:mainineq:subcrit}
		\bbP_{N^\gk\gd_0}\Big( \cX_{t}^{N}\big[ \gga^\anyrg_T(t) + (1-2\eps)\gs(t/T)L(T)\,,\,+\infty\big) < N^\eps \Big) \;\leq\; c_1 N^{-c_2}\,,
	\end{equation}
	and
	\begin{equation}\label{eq:LB:mainineq}
		\bbP_{N^\gk\gd_{-\gk\gs(0)L(T)}}\Big(\max(\cX_{t^\anyrg_T}^{N}) \leq \gga^\anyrg_T(t^\anyrg_T) + (1-2\eps)\gs(t^\anyrg_T/T)L(T)-\gk\gs(0)L(T)\Big) \;\leq\; c_1 N^{-c_2}\,,
	\end{equation}
	where $c_1,c_2$ are uniformly bounded away from $0$ and $\infty$ in $\gs\in\gsensrg$, $t\in[\gth(T)^{-1}L(T),t^\anyrg_T]$, and locally uniformly in $\eps,\gk\in(0,1)$.
\end{lemma}

\begin{proof}[Proof of Lemma~\ref{lem:main:LB:coupling}]
	Let us first prove~\eqref{eq:LB:mainineq:subcrit}. 
	Denote by $(\cX_{t}^{N-})_{t\geq0}$ the point process of a BBM with the following selection mechanism: particles are killed whenever they are not among the $N$ highest \emph{or} when they hit one of the barriers $\gga^\anyrg_T$, $\ol\gga^\anyrg_T$. In particular it is an $N^-$-BBM (recall Definition~\ref{def:N-N+BBM}), therefore Lemma~\ref{lem:N-N+couping} yields that for any $y\in\R$,
		\begin{equation}\label{eq:LBmainineq:coupl1}
			\bbP_{N^\gk\gd_0}\big(\cX_{t^\anyrg_T}^{N}([y,+\infty))< N^\eps \big) \;\leq\; \bbP_{N^\gk\gd_0}\big(\cX_{t^\anyrg_T}^{N-}([y,+\infty))< N^\eps \big) \,.
		\end{equation}
	
	Let $(\cX_{t}^{B})_{t\geq0}$ be the point process of a BBM killed at the barriers $\gga^\anyrg_T$, $\ol\gga^\anyrg_T$ but without any other selection. By Proposition~\ref{prop:allregimes:LB:npart}, it has a large probability under $\bbP_{N^\gk\gd_0}$ to contain fewer than $N$ particles at all time on $[0,t^\anyrg_T]$, in which case its trajectory is equal to that of the process $(\cX_{t}^{N-})_{t\in[0,t^\anyrg_T]}$. Therefore, for $T$ sufficiently large and $y\in\R$,
	\begin{align}\label{eq:LBmainineq:coupl2}
		&\bbP_{N^\gk\gd_0}\big(\cX_{t^\anyrg_T}^{N-}([y,+\infty))< N^\eps\big) \\
		\notag &\quad \leq\, \bbP_{N^\gk\gd_0}\big(\cX_{t^\anyrg_T}^{B}([y,+\infty))< N^\eps\big) + c_1 N^{-c_2}\,.
	\end{align}
	Finally, letting $t\in[\gth(T)^{-1}L(T),t^\anyrg_T]$ and $y=\gga^\anyrg_T(t) + (1-2\eps)\gs(t/T)L(T)$, and applying Proposition~\ref{prop:allregimes:LB:endpoint} with $z=1-2\eps$, this yields
	\begin{align*}
	&\bbP_{N^\gk\gd_0}\Big(\cX_{t}^{N}\big[ \gga^\anyrg_T(t) + (1-2\eps)\gs(t/T)L(T)\,,\,+\infty\big) < N^\eps \Big) \\
	&\quad \leq \bbP_{N^\gk\gd_0}\big(|A_{T,1-2\eps}^\anyrg(t^\anyrg_T)|<N^\eps\big)
	+ c_1 N^{-c_2}\leq 2c_1 N^{-c_2},
	\end{align*}
	uniformly in $\gs$, $t$, from which~\eqref{eq:LB:mainineq:subcrit} follows.
	
	Regarding~\eqref{eq:LB:mainineq}, one deduces straightforwardly from~\eqref{eq:LB:mainineq:subcrit} that
	\[
	\bbP_{N^\gk\gd_{0}}\Big(\max(\cX_{t^\anyrg_T}^{N}) \leq \gga^\anyrg_T(t^\anyrg_T) + (1-2\eps)\gs(t^\anyrg_T/T)L(T)\Big) \;\leq\; c_1 N^{-c_2}\,,
	\]
	and shifting this estimate by $-\gk\gs(0)L(T)$ concludes the proof.
\end{proof}

With this at hand, we resume the proof of Proposition~\ref{prop:main:LB}.
To that end, we first claim that it suffices to prove the following, slightly weaker statement in which the uniformity in $\gk\in[0,1]$ is replaced by local uniformity in $\gk\in(0,1)$. Recall~(\ref{eq:defbanyrg},~\ref{eq:defmanyrg}).
\begin{lemma}\label{lem:main:LB} 
	Let $\anyrg\in\{\supercrit,\subcrit,\crit\}$. 
	Let $\gl>0$, $\gk\in(0,1)$ and $\gh>0$. Then as $T\to+\infty$, one has
	\begin{equation}\label{eq:lem:main:LB}
		\bbP_{N^{\gk}\gd_{-\gk\gs(0)L(T)}}\left(\frac{1}{b^\anyrg_T}\left(\max(\cX_T^{N}) - m^\anyrg_T\right)\leq -\gl \right) \;\longrightarrow\; 0\,,
	\end{equation}
	and the convergence is uniform in $\gs\in\gsensrg$, and locally uniform in $\gk\in(0,1)$.
\end{lemma}

\begin{proof}[Proof of Proposition~\ref{prop:main:LB} subject to Lemma~\ref{lem:main:LB}]
	We start with the case $\gk$ close to $1$. Let $\anyrg\in\{\supercrit,\subcrit,\crit\}$ and $\gl>0$. Notice that there exists $\eps>0$ such that
	\begin{equation}\label{eq:prf:corol:main:1}
		\limsup_{T\to+\infty} \frac{\eps\gs(0)L(T)}{b^\anyrg_T}\,\leq\, \frac\gl3\,.
	\end{equation}
	Indeed, in the sub-critical regime this follows from the fact that $L(T)\ll b^\subcrit_T$, and in the other cases this holds as soon as $\eps\leq (\gl\gh/3) \lim_{T\to+\infty}(b^\anyrg_T/L(T))$ (the latter limit being equal to $1$, resp. $1/\ga$, in the super-critical, resp. critical regime). Moreover, one has $N^\gk\gd_{-\gk\gs(0)L(T)}\succ N^{1-\eps}\gd_{-\gs(0)L(T)}$ for $\gk\in[1-\eps,1]$, so Corollary~\ref{cor:coupling:N} implies that,
	\begin{align*}
		&\bbP_{N^{\gk}\gd_{-\gk\gs(0)L(T)}}\left(\frac{1}{b^\anyrg_T}\left(\max(\cX_T^{N}) - m^\anyrg_T\right)\leq -\gl \right) \\
		&\quad\leq\,\bbP_{N^{1-\eps}\gd_{-\gs(0)L(T)}}\left(\frac{1}{b^\anyrg_T}\left(\max(\cX_T^{N}) - m^\anyrg_T\right)\leq -\gl \right)\\
		&\quad\leq\,\bbP_{N^{1-\eps}\gd_{-(1-\eps)\gs(0)L(T)}}\left(\frac{1}{b^\anyrg_T}\left(\max(\cX_T^{N}) - m^\anyrg_T\right)\leq -\frac\gl3 \right)\,,
	\end{align*}
	where the last inequality is obtained for $T$ sufficiently large by shifting the process upward by $\eps\gs(0)L(T)$, and by recalling~\eqref{eq:prf:corol:main:1}. Applying Lemma~\ref{lem:main:LB} with $\gk = 1-\eps$, this proves that~\eqref{eq:lem:main:LB} holds uniformly in $\gk\in[1-\eps,1]$.
	
	Regarding the case $\gk$ small, let $\eps>0$, and recall that $(\cN_t)_{t\geq0}$ denotes the set of particles in the BBM throughout time; and let $Z_t:=|\cN_T|$, $t\geq0$. We claim the following.
	\begin{lemma}\label{lem:LB:standardresults} Let $\eps,\eps'>0$. One has for $T$ sufficiently large:
		
		$(i)$ $\bbP_{\gd_0}(Z_{\eps L(T)} \leq N^{\eps/4})\leq \eps'$,
		
		$(ii)$ $\bbP_{\gd_0}(\exists\,s\leq \eps L(T): Z_s \geq N)\leq \eps'$,
		
		$(iii)$ $\bbP_{\gd_0}(\exists\,u\in\cN_{\eps L(T)}: X_u(\eps L(T)) \leq -2\eps\gh^{-1}L(T))\,\leq\, \eps'$.
	\end{lemma}
	Those statements follow from classical results on the BBM and birth processes, we postpone their proof for now. For $\eps>0$ and $\gk\in[0,\eps]$, notice that $N^{\gk}\gd_{-\gk\gs(0)L(T)} \succ \gd_{-\eps\gs(0)L(T)}$; so we deduce from Corollary~\ref{cor:coupling:N} and a shift that, for $\eps>0$ and $\gk\in[0,\eps]$,
	\begin{align*}
	&\bbP_{N^{\gk}\gd_{-\gk\gs(0)L(T)}}\left(\frac{1}{b^\anyrg_T}\left(\max(\cX_T^{N}) - m^\anyrg_T\right)\leq -\gl \right) \\
	&\quad \leq\, \bbP_{\gd_{0}}\left(\frac{1}{b^\anyrg_T}\left(\max(\cX_T^{N}) -\eps\gs(0)L(T) - m^\anyrg_T\right)\leq -\gl \right).
	\end{align*}
	Let $\eps'>0$. We apply the Markov property at time $\eps L(T)$, noticing that Lemma~\ref{lem:LB:standardresults} implies for $T$ sufficiently large,
	\[
	\bbP_{\gd_0}\left( \begin{aligned}
		&\cX_{\eps L(T)}^{N} \Big(\big(-\infty, -2\eps\gh^{-1}L(T)\big]\Big) = 0\,;\\
		&\cX_{\eps L(T)}^{N} \Big(\big[-2\eps\gh^{-1}L(T),+\infty\big)\Big) \geq N^{\eps/4}
	\end{aligned} \right)\;\geq\; 1-3\eps'\,,
	\]
	(indeed, on the event $\{\forall s\leq \eps L(T),\,Z_s<N\}$, the particle configurations of the BBM and $N$-BBM at time $\eps L(T)$ are the same). 
	In particular, on that event we have $\cX_{\eps L(T)}^{N} \succ N^{\eps/4}\gd_{-2\eps\gh^{-1}L(T)}$. Applying Markov's property at time $\eps L(T)$, then Corollary~\ref{cor:coupling:N} again and a shift, we obtain
	\begin{align*}
		&\bbP_{N^{\gk}\gd_{-\gk\gs(0)L(T)}}\left(\frac{1}{b^\anyrg_T}\left(\max(\cX_T^{N}) - m^\anyrg_T\right)\leq -\gl \right)\\
		&\leq 3\eps'+ \bbP_{N^{\eps/4}\gd_{(\eps L(T),-\frac\eps4\gs(0)L(T))}}\!\bigg(\frac{1}{b^\anyrg_T}\left(\max(\cX_T^{N}) -\eps\big(3\gs(0)/4+2\gh^{-1}\big)L(T) - m^\anyrg_T\right)\leq -\gl \bigg).
	\end{align*}
	Recall that $(3\gs(0)/4+2\gh^{-1})L(T)=O(b^\anyrg_T)$; hence, choosing $\eps'$ and $\eps=\eps(\gl,\gh)$ small enough, we conclude the proof of Proposition~\ref{prop:main:LB} by applying Lemma~\ref{lem:main:LB} at time $T-\eps L(T)$ with $N^{\eps/4}$ initial particles.
\end{proof}

We now prove Lemma~\ref{lem:LB:standardresults}. Let us first recall the following classical result on pure birth processes (see e.g. \cite[Ch. III]{AN72}). In the following we consider a pure birth process $(Z_t)_{t\geq0}$ with same rate $\gb_0$ and offspring distribution $\xi$ as our time-inhomogeneous BBM $(\cX_t)_{t\in[0,T]}$, and with $Z_0:=1$; in particular, when restricted to $[0,T]$, it has the same distribution as the BBM's population size $(|\cN_t|)_{t\in[0,T]}$ under $\bbP_{\gd_0}$.
\begin{proposition}\label{prop:LB:standardresults:quote}
	\cite[Theorems III.7.1--2]{AN72} 
	Let $\cF_t:=\gs(Z_s,s\leq t)$, $t\geq0$. Then, under $\bbP_{\gd_0}$, the process $(e^{-t/2}Z_t)_{t\geq0}$ is a $(\cF_t)_{t\geq0}$-martingale, is non-negative and converges a.s. and in $L^1$ to some random variable $W$. Moreover, $\bbE[W]=1$ and $\bbP_{\gd_0}(W>0)=1$.
\end{proposition}

\begin{proof}[Proof of Lemma~\ref{lem:LB:standardresults}]
	Let us start with Lemma~\ref{lem:LB:standardresults}.(i--ii). We prove them for a birth process $(Z_t)_{t\geq0}$ by using Proposition~\ref{prop:LB:standardresults:quote}; then these statements also hold for $(|\cN_t|)_{t\in[0,T]}$, assuming $T$ was taken sufficiently large. On the one hand, for any $x>0$, one has for $t$ sufficiently large
	\[
	\bbP(Z_t\leq e^{t/4})\leq \bbP(e^{-t/2}Z_t \leq x)\leq \bbP(W\leq 2x)+o(1)\,,
	\]
	where the second inequality comes from the $L^1$ convergence of the martingale, i.e. $\bbP(|W-e^{-t/2}Z_t|>x)\to0$ as $t\to+\infty$. Assuming $x$ was chosen sufficiently small and replacing $t$ with $\eps L(T)$, this proves Lemma~\ref{lem:LB:standardresults}.(i) for $T$ sufficiently large. On the other hand, we have by Doob's martingale inequality that, for $t\geq0$,
	\[
	\bbP(\exists \, s\leq t: e^{-s/2}Z_s \geq e^{t/2}) \,\leq\, e^{-t/2}\,.
	\]
	Again, letting $t=\eps L(T)$ and assuming $T$ large enough, this implies Lemma~\ref{lem:LB:standardresults}.(ii).\footnote{Actually this proves a much stronger result, but we do not need it in this paper.}
	
	We now turn to Lemma~\ref{lem:LB:standardresults}.(iii), where we let $Z_t:=|\cN_t|$ for $t\in[0,T]$. Let $(Y^i_t)_{t\in[0,T]}$, $i\in\N$ denote a sequence of i.i.d. time-inhomogeneous Brownian motions with infinitesimal variance $\gs(\cdot/T)$. As a consequence of Slepian's lemma~\cite{SLep62}, one has for $t\in[0,T]$ and $k\in\N$,
	\begin{align*}
	&\bbP_{\gd_0}\big(\exists\,u\in\cN_{t}: X_u(t) \leq -2\gh^{-1}t\,\big|\,Z_t=k\big)\\
	&\quad \leq\, \bP_0\big(\exists i\leq k: Y^i_t \geq 2\gh^{-1}t\big)\,=\, 1-\left[1-\bP_0(Y^1_t\geq 2\gh^{-1}t)\right]^{k},
	\end{align*}
	where we also used the Brownian symmetry property. Recall that for $i\in\N$,
	\[\bE_0\big[(Y^i_t)^2\big]\,=\,\int_0^{t}\gs^2(s/T)\,\dd s\,\leq\, \gh^{-2}t,\]
	and recall the standard Gaussian tail estimate, $\bP_0(W_1\geq x)\leq \frac1{x\sqrt{2\pi}} e^{-x^2/2}$ for $x>1$. Thus, for $t$ sufficiently large, one has
	\[
	\bP_0(Y^1_t\geq 2\gh^{-1}t)\,\leq\, \bP_0(W_1\geq 2\sqrt{t})\,\leq\, \frac1{2\sqrt{2\pi t}} e^{-2t}\,.
	\]
	Let $t=\eps L(T)$ in the above, and recall from Proposition~\ref{prop:LB:standardresults:quote} that $\bbP_{\gd_0}(Z_{\eps L(T)}\geq N(T)^{3\eps/2})\to0$ as $T\to+\infty$ (this is similar to Lemma~\ref{lem:LB:standardresults}, we do not write the details again). Therefore, we deduce for $T$ large,
	\begin{align*}
	&\bbP_{\gd_0}\Big(\exists\,u\in\cN_{\eps L(T)}: X_u(\eps L(T)) \leq -2\eps\gh^{-1}L(T)\Big)\\
	&\quad \leq\, o(1) + 1-\left[1-\frac1{2\sqrt{2\pi \eps L(T)}}{N(T)}^{-2\eps}\right]^{N(T)^{3\eps/2}}\,=\, o(1)\,,
	\end{align*}
	which concludes the proof.
\end{proof}

It only remains to prove Lemma~\ref{lem:main:LB}. We proceed differently depending on the regime satisfied by $L(T)$.

\begin{proof}[Proof of Lemma~\ref{lem:main:LB}, super-critical and critical regimes]
	The result follows immediately from Lemmata~\ref{lem:main:LB:coupling} and~\ref{lem:comparison:m:gga} in the super-critical and critical regimes. Indeed, in both cases one has $t^\anyrg_T=T$; recalling \eqref{eq:LB:parameters} and the notation $\gga^{\anyrg,h,x}_T(\cdot)$ from Section~\ref{sec:prelim:notation}, one has
	\begin{align}\label{eq:prfmainLB:rewriting:gga}
		\notag&\gga^{\anyrg,1-\eps,(1-\eps)(1-\gk)}_T(T) + (1-2\eps)\gs(1)L(T)-\gk\gs(0)L(T)\\
		&\quad=\, \ol\gga^{\anyrg,1-\eps,1-\eps(1-\gk)}_T(T)-\eps\gs(1)L(T)\,.
	\end{align}
	Letting $\eps$ be arbitrarily small and applying Lemmata~\ref{lem:comparison:m:gga} and~\ref{lem:main:LB:coupling} (more precisely~\eqref{eq:LB:mainineq}), this gives exactly Lemma~\ref{lem:main:LB} in both regimes.
\end{proof}

We now turn to the proof of Lemma~\ref{lem:main:LB} in the sub-critical regime ($L(T)\ll T^{1/3}$), which requires more care. Since our estimates do not directly hold at time $T$ in that regime, we split the interval $[0,T]$ into blocks whose lengths are of order $t^\subcrit_T$, apply our estimates on each of those, then use them to reconstruct a process on $[0,T]$ which is dominated by the $N$-BBM.

First, we may always bound the $N$-BBM from below by having it start with fewer particles, recall Corollary~\ref{cor:coupling:N}: hence, in~\eqref{eq:LB:mainineq:subcrit} and in the following, we assume that $\eps$ is small and that $\gk=\eps$ without loss of generality. Moreover, it will be very convenient to work with quantiles instead of the maximal displacement: recalling~\eqref{def:quantile:M}, write for $\mu\in\meas$,
\begin{equation}
	q_{\kappa}(\mu)\,:=\,q_{N^\kappa}(\mu)\,=\, \inf\{x\in\R,\, \mu([x,+\infty))< N^{\kappa}\}\,,
\end{equation}
and recall that for any $\mu\prec\nu$, one has by definition $q_{\kappa}(\mu)\leq q_{\kappa}(\nu)$.
%
We now introduce an auxiliary process $(\ol \cX^{N}_t)_{0\le t\le T}$. Let us define $K:=\lfloor 2T/t^\subcrit_T\rfloor$, and for $0\leq k\leq K-1$, let $t_k:= \frac k2t^\subcrit_T$, and $t_{K}=T$ (so $t_{K}-t_{K-1}\in[\frac12t^\subcrit_T,t^\subcrit_T]$). The process $(\ol \cX^{N}_t)_{0\le t\le T}$ is defined as follows: starting from $\lfloor N^\kappa\rfloor \delta_0 $ (we omit the integer part in the following), it evolves between times $t_k$ and $t_{k+1}$ as the process $\cX^N$. Then, at every time $t_k$, $1\le k< K$, all but the $N^\kappa$ top-most particles are removed, and the remaining particles are all set to the lowest among their positions. In other words, a configuration $\mu\in \meas_N$ is replaced by the configuration $N^\kappa \delta_{q_\kappa(\mu)}$. 
Applying Corollary~\ref{cor:coupling:N} inductively, 
one obtains straightforwardly a coupling such that $\ol \cX^{N}_T\,\prec\,\cX^{N}_T$ with probability 1. In particular the law of $q_\kappa(\ol \cX^{N}_T)$ stochastically bounds from below the law of $q_\kappa(\cX^N_T)$, which is lower than $\max(\cX^N_T)$; hence it suffices to prove~\eqref{eq:lem:main:LB} with $\max(\cX^N_T)$ replaced by $q_\kappa(\ol \cX^{N}_T)$.

For $1\le k\le K$, write
\[
X_k = q_\kappa\big(\ol \cX^{N}_{t_k}\big) - q_\kappa\big(\ol \cX^{N}_{t_{k-1}}\big),
\]
so that $q_\kappa(\ol \cX^{N}_T) = X_1+\cdots+X_{K}$ (recall that $\ol \cX^{N}_0=N^\kappa\gd_0$). By the definition of the process $\ol \cX^{N}$ and the fact that a translation of the $N$-BBM is again an $N$-BBM started from a translated initial configuration, the random variables $X_1,\ldots,X_{K}$ are independent.

Using that notation, we may finally prove Lemma~\ref{lem:main:LB} in the sub-critical regime. However, the proof relies on two different methods, depending on the speed at which $N(T)$ diverges.

\begin{proof}[Proof of Lemma~\ref{lem:main:LB}, sub-critical regime, $N(T)$ growing fast]
	Assume that $N(T)$ grows super-polynomial in $T$, i.e. $L(T)\gg \log(T)$. Recall~\eqref{eq:LB:mainineq:subcrit}, which implies for $1\leq k\leq K$ that
	\begin{equation}\label{eq:LB:main:subcrit:goodblock}
		\bbP\big(X_k\geq \gga^\subcrit_T(t_{k})-\gga^\subcrit_T(t_{k-1})-\gh^{-1}L(T)\big)\;\geq\; 1-c_1N^{-c_2}\,,
	\end{equation}
	for some $c_1,c_2$ locally uniform in $\eps=\gk\in(0,1)$. Recalling the definition of the process $(\ol \cX^{N}_t)_{0\le t\le T}$, that the $X_k$, $1\leq k\leq K$ are independent and that $q_\kappa(\ol \cX^{N}_T) = X_1+\cdots+X_{K}$, one deduces through a direct induction that
	\[
	\bbP_{N^\kappa\gd_{u_0}}\Big(q_\kappa(\ol \cX^{N}_T) \geq \gga^\subcrit_T(T) - K\gh^{-1}L(T) \Big)\;\geq\; (1-c_1N^{-c_2})^{K}\,.
	\]
	Notice that $K N^{-c_2}\leq \frac{\gth(T)T}{L(T)^3}e^{-c_2L(T)}$; so, under the assumption $L(T)\gg\log(T)$, the latter probability goes to 1 as $T\to+\infty$, locally uniformly in $\kappa\in(0,1)$. Recalling that $q_\kappa(\ol \cX^{N}_T)$ stochastically bounds from below $\max(\cX^{N}_{T})$, this finally implies
	\[
	\bbP_{N^\kappa\gd_{u_0}}\Big(\max(\cX^{N}_{T})\leq \gga^\subcrit_T(T) - K\gh^{-1}L(T) \Big)\,\longrightarrow\,0
	\]
	as $T\to+\infty$, locally uniformly in $\kappa$. Moreover, $K\gh^{-1}L(T)\leq \gh^{-1}\frac{\gth(T)T}{L(T)^2} =o(T/L(T)^2)$ and $L(T)=o(T/L(T)^2)$. Hence, applying a shift $-\gk\gs(0)L(T)$ to the estimate above, we deduce from Lemma~\ref{lem:comparison:m:gga} and the same computation as~\eqref{eq:prfmainLB:rewriting:gga} that, with $\eps$ arbitrarily small, the lower bound~\eqref{eq:lem:main:LB} holds in the case $N(T)$ sub-critical, super-polynomial.
\end{proof}

If $L(T)$ is too small, the decomposition above involves so many blocks that, with large probability, the auxiliary process $\ol \cX^{N}$ does not satisfy the event~\eqref{eq:LB:main:subcrit:goodblock} on some of them. However, having a small $L(T)$ allows us to use a second moment method, relying only on very crude bounds on the second moments of $X_k$, $1\leq k\leq K$.

\begin{proof}[Proof of Lemma~\ref{lem:main:LB}, sub-critical regime, $N(T)$ growing slowly]
	In the following, we assume that $L(T)$ grows very slowly with $T$, in fact, $L(T)\ll T^{1/8}$ is enough. In particular, we are still in the sub-critical regime, and~\eqref{eq:defta} yields $t^\subcrit_T=L(T)^4$. We claim the following: there exist $c_1,c_2>0$ such that for $T$ sufficiently large, for every $k\in\{1,\ldots,K\}$, we have
	\begin{align}
		\label{X_k_expectation}
		\bbE[X_k]\,&\geq\, (\gga^\subcrit_T(t_{k}) - \gga^\subcrit_T(t_{k-1})) (1-c_1N^{-c_2}) - c_1 L(T),\\
		\label{X_k_variance}
		\Var\big(X_k\big) \,&\leq\, c_1 (t^\subcrit_T)^2.
	\end{align}
	Let us see how equations \eqref{X_k_expectation} and \eqref{X_k_variance}  imply the lemma. First note that using \eqref{X_k_expectation}, we have
	\[
	\bbE[X_1+\ldots+X_{K}] \,\ge\, \gga^\subcrit_T(T)(1-c_1N^{-c_2}) - c_1 L(T)\times K,
	\]
	and, using that $K \le 2T/t^\subcrit_T=2T/L(T)^4$, we get that
	\[
	\bbE[q_\kappa(\ol \cX^{N}_T)] = \bbE[X_1+\ldots+X_K] \,\ge\, \gga^\subcrit_T(T) + o\left(\frac T {L(T)^2}\right).
	\]
	Furthermore, by the independence of the random variables $X_1,\ldots,X_k$, we get using $\eqref{X_k_variance}$ that
	\[
	\Var(q_\kappa(\ol \cX^{N}_T)) = \sum_{k=1}^K \Var(X_k) \,\le\, 2\,c_1 t^\subcrit_T\, T \,=\, o\left(\left(\frac T {L(T)^2}\right)^2\right),
	\]
	where for the last equality we used that $t^\subcrit_T = L(T)^4$ and $L(T) \ll T^{1/8}$. Using the Bienaymé-Chebychev inequality, this yields that
	\[
	q_\kappa(\ol \cX^{N}_T) \,\geq\, \gga^\subcrit_T(T) + o_{\bbP}\left(\frac T {L(T)^2}\right),
	\]
	with large probability, where we recall the notation $o_{\bbP}(\cdot)$ from Theorem~\ref{thm:main}.. 
	Here again, applying a shift $-\gk\gs(0)L(T)=o(T/L(T)^2)$ to the estimate above and letting $\gk,\eps$ be small, we deduce from Lemma~\ref{lem:comparison:m:gga} and the same computation as~\eqref{eq:prfmainLB:rewriting:gga} that the lower bound~\eqref{eq:lem:main:LB} holds in the case $L(T)\ll T^{1/8}$.
	
	It remains to prove \eqref{X_k_expectation} and \eqref{X_k_variance}. We start with \eqref{X_k_expectation}. Noting that the process evolves between times $t_k$ and $t_{k+1}$ as a $N$-BBM with variance profile $\sigma(t_k/T + \cdot)$, it is enough to show that 
	\begin{equation}\label{eq:lem:LB:smallL:1st:1}
		\bbE_{N^\kappa\gd_0}\big[q_\kappa(\cX^N_{t})\big]\,\geq\, \gga^\subcrit_T(t)\,(1-c_1N^{-c_2}).
	\end{equation}
	where $c_1,c_2$ are uniform in $\gs\in\gsensrg$ and $t\in[\frac12 t^\subcrit_T,t^\subcrit_T]$. Most hard work has been done in Lemma~\ref{lem:main:LB:coupling}, which yields that, with the same notation,
	\begin{equation}
		\label{eq:hard_work}
		\bbP_{N^\kappa\gd_0}\big(q_\kappa(\cX^N_{t}) < \gga^\subcrit_T(t)\big) \le c_1 N^{-c_2}.
	\end{equation}
	Assuming from now on that $T$ is large enough, so that $\gga^\subcrit_T(t) \ge c t^\subcrit_T \ge 0$ for some $c>0$ and all $t\in[\frac12 t^\subcrit_T,t^\subcrit_T]$, we deduce from \eqref{eq:hard_work} that
	\begin{align}\label{eq:lem:LB:smallL:1st:2a}
		&\bbE_{N^\kappa\gd_0}\Big[q_\kappa(\cX^N_{t})\,\ind_{\{q_\kappa(\cX^N_{t})\geq 0\}}\Big] \\
		\notag &\qquad \geq\,\gga^\subcrit_T(t)\,\bbP_{N^\kappa\gd_0}\big(q_\kappa(\cX^N_{t})\geq \gga^\subcrit_T(t)\big)\,\geq\, \gga^\subcrit_T(t)\,(1-c_1N^{-c_2})\,,
	\end{align}
	uniformly in $t\in[\frac12 t^\subcrit_T,t^\subcrit_T]$. On the other hand, recalling from Proposition~\ref{prop:coupling:BBM}.$(i)$ that we can couple $\cX^N$ and a BBM $\cX$ in such a way that $\cX_t^N\subset\cX_t$ for all $t$, we have
	\begin{equation*}
		\bbE_{N^\kappa\gd_0}\Big[q_\kappa(\cX^N_{t})\,\ind_{\{q_\kappa(\cX^N_{t}) \leq 0\}}\Big] \ge - \bbE_{N^\kappa\gd_0}\Big[(\min(\cX_t))_-\,\ind_{\{q_\kappa(\cX^N_{t}) \leq 0\}}\Big],
	\end{equation*}
	and, using the Cauchy-Schwarz inequality and the symmetry of the Gaussian distribution, we get
	\begin{equation}
		\label{eq:lem:LB:smallL:1st:2b}
		\bbE_{N^\kappa\gd_0}\Big[q_\kappa(\cX^N_{t})\,\ind_{\{q_\kappa(\cX^N_{t}) \leq 0\}}\Big] \ge - \sqrt{\bbE_{N^\kappa\gd_0}\Big[(\max(\cX_t))_+^2\Big]\times \bbP_{N^\kappa\gd_0}\left(q_\kappa(\cX^N_{t}) \leq 0\right)}.
	\end{equation}

	Let us recall the following standard result on Gaussian random variables. Let us mention that we are not aiming for optimal constants or bounds in this statement. The proof is postponed to the end of this section.
	\begin{lemma}\label{lem:gaussmax}
		Let $t\in(0,T]$, $M\ge 1$ and $(g_i)_{1\le i\le M}$ a centered Gaussian vector, such that each $g_i$ has variance $\rho^2\ge 0$. Then, for $M\geq 2$, one has
		\begin{equation}
			\bE\big[(\max\{g_i,\,1\leq i\leq M\})_+^2\big]\leq 4 \rho^2 \log M\,.
		\end{equation}
	\end{lemma}
	
	We wish to apply Lemma~\ref{lem:gaussmax} to bound $\bbE_{N^\kappa\gd_0}\Big[(\max(\cX_t))_+^2\Big]$. To do this, condition on the branching times and denote by $Z_t$ the number of particles at time $t$. Then apply the lemma with $M= Z_t$ and $\rho^2$ = $\sigma^2(t/T)\times t$ and recall that  $\gs\in\gsens$. This gives
	\[
	\bbE_{N^\kappa\gd_0}\Big[(\max(\cX_t))_+^2\Big] \le 4 \eta^{-2}\times  t  \times \bbE_{N^\kappa\gd_0}[\log Z_t].
	\]
	But we have $\bbE_{N^\kappa\gd_0}[Z_t] = N^\kappa e^{t/2}$ and $Z_t \ne 0$ almost surely, so that by Jensen's inequality,
	\[
	\sqrt{\bbE_{N^\kappa\gd_0}\Big[(\max(\cX_t))_+^2\Big]} \le \sqrt{ 4 \eta^{-2}\times  t\times (\kappa \log N + t/2)} \le C t^\subcrit_T,
	\]
	for some $C>0$, using that
	$t\leq t^\subcrit_T$ and $\log N = L(T) = o(t^\subcrit_T)$. Plugging this into \eqref{eq:lem:LB:smallL:1st:2b} and using \eqref{eq:hard_work}, we get, possibly modifying the values of $c_1$ and $c_2$,
	\begin{equation}
		\label{almost_done}
		\bbE_{N^\kappa\gd_0}\Big[q_\kappa(\cX^N_{t})\,\ind_{\{q_\kappa(\cX^N_{t}) \leq 0\}}\Big] \ge -\gga^\subcrit_T(t)\,c_1N^{-c_2}.
	\end{equation}
	Combining \eqref{eq:lem:LB:smallL:1st:2a} and \eqref{almost_done} yields \eqref{eq:lem:LB:smallL:1st:1} and therefore \eqref{X_k_expectation}.
	
	The proof of \eqref{X_k_variance} is much simpler: it suffices to prove $\bbE[X_k^2]\leq c_1 (t^\subcrit_T)^2$ for some $c_1>0$, which is obtained via the embedding of the $N$-BBM into a BBM without selection used above, together with Lemma~\ref{lem:gaussmax}. Details are omitted.
\end{proof}

\begin{proof}[Proof of Lemma~\ref{lem:gaussmax}]
	This result is standard, but we provide a proof for completeness. For $u_0>0$, bound
	\begin{align*}
		\bE\big[(\max\{g_i,\,1\leq i\leq M\})_+^2\big]\,&=\, \int_0^{+\infty} 2u\,\bP(\max\{g_i,\,1\leq i\leq M\}\geq u)\,\dd u\\
		&\leq u_0^2 + \int_{u_0}^{+\infty} 2u\,\bP(\max\{g_i,\,1\leq i\leq M\}\geq u)\,\dd u\,.
	\end{align*}
	A union bound and a standard estimation on the Gaussian tail imply that, for $u>\rho$,
	\begin{equation}\label{eq:lem:gaussmax:tailmax}
		\bP(\max\{g_i,\,1\leq i\leq M\}\geq u)\,\leq\, M\sqrt{\frac{2}{\pi}} \frac{\rho}{u}\, \exp(-u^2/2\rho^2)\,\leq\, M\sqrt{\frac{2}{\pi}}\, \exp(-u^2/2\rho^2)\,.
	\end{equation}
	Using \eqref{eq:lem:gaussmax:tailmax}, one has for $u_0>\rho$,
	\begin{align*}
		\bE\big[\max\{g_i,\,1\leq i\leq M\}^2\big]\,&\leq\, u_0^2 + 2\sqrt{\frac2\pi}M\rho^2 \exp(-u_0^2/2\rho^2)\,.
	\end{align*}
	Letting $u_0=\rho\sqrt{2\log M}$, and $M \ge 2$, this concludes the proof.
\end{proof}

\begin{remark}
	With this, the proof of Lemma~\ref{lem:main:LB} is finally completed in every regime for $L(T)$. Indeed, the only case we have not treated is when the sub-critical regime alternates between the sub-cases $N(T)$ super-polynomial and $L(T)$ smaller than $T^{1/8}$. This situation can be handled e.g. by letting $L_1(T):=L(T)\wedge (\log T)^2$, $L_2(T):=L(T)\vee (\log T)^2$, then applying~\eqref{eq:lem:main:LB} to both cases (for which we have displayed complete proofs). Then the l.h.s. in~\eqref{eq:lem:main:LB} for the ``oscillating'' $L(T)$ is bounded from above by the maximum of two vanishing sequences: we leave the details to the reader.
\end{remark}

\section{Upper bound on the maximum of the \texorpdfstring{$N$}{N}-BBM}\label{sec:UB}
Let $\anyrg\in\{\subcrit,\supercrit,\crit\}$. In this section we present the proof of Proposition~\ref{prop:main:UB}, which is analogous to that of Proposition~\ref{prop:main:LB}: we show that the trajectory of a BBM between well-chosen barriers is equal to that of an $N^+$-BBM with large probability, and deduce an upper bound on the $N$-BBM with a coupling argument (recall Lemma~\ref{lem:N-N+couping}). In particular, for the comparison to hold, we must tune the barrier parameters $(h,x)$ so that the BBM between barriers counts approximately $N^{1+\gd}$ particles at time $t^\anyrg_T$ for some $\gd>0$ (recall~\eqref{eq:defta}): this implies with large probability that there are at least $N$ living particles at \emph{all} time $t\in[0,t^\anyrg_T]$. But then, we face the additional complication that whenever a particle is killed by the upper barrier $\ol\gga^\anyrg_T$, this breaks the monotone coupling between the BBM with barriers and the $N$-BBM (this was not an issue in Section~\ref{sec:LB}). Furthermore, shifting the upper barrier upward \emph{does not diminish} the number of particles it kills. We will see below (in the sub-critical and critical regimes) that there exists no choice of parameters $(h,x)$ which ensures us both that $N$ particles are alive at all time, and that no particle hits the upper barrier (see Remark~\ref{rem:coloring} for the details).

This issue is related to the following crucial idea: in a branching process with a lower killing barrier, the evolution of the population size throughout time is closely related to the behavior of \emph{``peaking''} particles. More precisely, whenever an individual rises very far from the lower barrier, it has the opportunity to generate a very large number of offspring in the time span before it gets close to the lower barrier again. This was observed in particular in~\cite{BBS13} when studying a (homogeneous) BBM with absorption at 0 and near-critical drift: the authors tune the drift so that the population size remains roughly $e^L$ for a time $L^3$, $L>0$; and in that regime, they observe that removing peaking particles (i.e. introducing a second killing barrier around height $L$) strongly affects the survival probability of the whole process on the time interval $[0,L^3]$ ---even though that second barrier actually kills much fewer than $e^L$ particles. In our setting, this translates into the following heuristics: if there are at least $N(T)$ living particles at all time $t\in[0,t^\anyrg_T]$, and if $t^\anyrg_T$ is too large (that is $t^\anyrg_T\gtrsim L(T)^3$, which happens in the regimes $\anyrg\in\{\subcrit,\crit\}$), then with large probability some particles rise to a height of order $L(T)$ above the lower barrier before time $t^\anyrg_T$; and, if they are not killed, they reproduce a lot, resulting in a large increase in the total population size and allowing many descendants to rise even higher.

Therefore, we shall modify the BBM between barriers by \emph{coloring} instead of killing all peaking particles (i.e. individuals that reach the upper barrier $\ol\gga^\anyrg_T$) and controlling their offspring with a \emph{slightly stronger} selection of their descendants (i.e. we kill the descendants of a colored particle at a shifted-upward lower barrier). This defines a \emph{multi-type} branching process, and we shall tailor it so that, with large probability, $(i)$ there are still $N$ living particles at all time $t\in[0,t^\anyrg_T]$, and $(ii)$ the total population of the process does not blow up due to peaking particles. In particular, with large probability, no particle of this multi-type branching process will be able to overtake the upper barrier $\ol\gga^\anyrg_T$ by too much. An illustration of this multi-type process is given in Figure~\ref{fig:upperbound_multitype}.

\begin{figure}[ht]
		\centering
		\includegraphics[width=0.65 \textwidth]{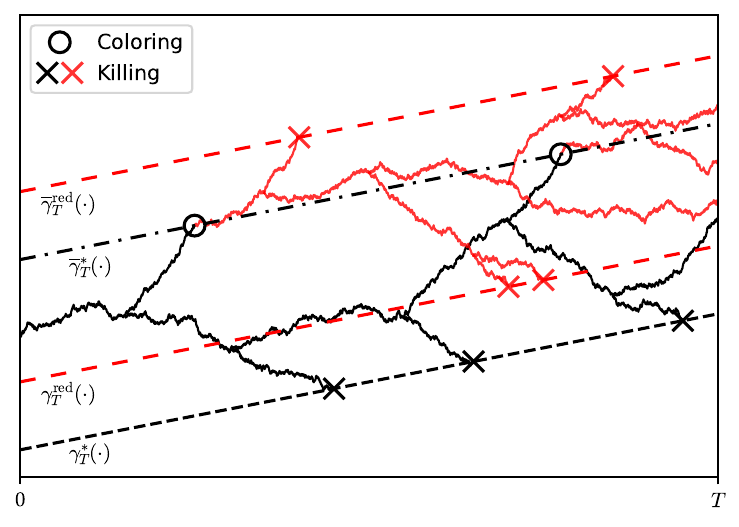}
		\caption{Illustration of the multi-type process. As in the original BBM between barriers, an uncolored individual is killed whenever it reaches $\gga^\anyrg_T(\cdot)$. However, when it reaches $\overline\gga^\anyrg_T(\cdot)$ it is not killed but colored in red instead. Thereafter all its descendants are also colored in red, and they are killed whenever they reach one of two new barriers, $\gga^\shifted_T(\cdot)$ and $\overline\gga^\shifted_T(\cdot)$, which are shifted-upward versions of $\gga^\anyrg_T(\cdot)$ and $\overline\gga^\anyrg_T(\cdot)$. The barriers will be chosen so that, with high probability, no particle reaches $\overline\gga^\shifted_T(\cdot)$, and there are at least $N$ particles above $\gga^\shifted_T(\cdot)$ at all time.
		\label{fig:upperbound_multitype}}
	\end{figure}

\begin{remark}
    In their study of the homogeneous $N$-BRW, Bérard and Gouéré~\cite{BG10} use another argument for the upper bound of the speed (Section 5 in their article). Their argument relies on a deterministic lemma, which ensures that a walk with bounded steps and which reaches a certain height has to have a portion of its trajectory that stays above a certain linear barrier for a certain amount of time. Then they use a result by Gantert, Hu and Shi~\cite{GHS11} on the survival probability of a branching random walk killed below a linear barrier. This argument requires to consider a time horizon which is significantly larger than $L(T)^3$. It could be adapted to our setting only in a certain part of the subcritical regime, but not if $L(T)$ is only slightly smaller than $T^{1/3}$ and certainly not in the critical regime. Our argument circumvents this issue by relying instead on a one-step iteration of more straightforward barrier estimates.
\end{remark}

Let $\eps\in(0,1)$ a small parameter, and fix
\begin{equation}\label{eq:UB:parameters}
	h=1+\eps\,, \qquad x=1+\frac{2}3\eps\;.
\end{equation}
Then, define once again the barriers $\gga^\anyrg_T$, $\ol\gga^\anyrg_T$ for each regime $\anyrg\in\{\supercrit,\subcrit,\crit\}$ according to \eqref{def:gga:supercrit}, \eqref{def:gga:subcrit} and~\eqref{def:gga:crit} respectively, with those $h,x$; in particular they satisfy~\eqref{eq:defxh}. 

We define the ``red'' barriers for $s\in[0,T]$ by
\begin{equation}\label{eq:def:ggashifted}
	\gga^\shifted_T(s):=\gga^\anyrg_T(s)+\frac{\eps}3\gs(s/T)L(T)\;,\quad\text{and}\quad \ol\gga^\shifted_T(s):=\ol\gga^\anyrg_T(s)+\frac{\eps}3\gs(s/T)L(T)\;,
\end{equation}
which are slightly shifted versions of $\gga^\anyrg_T$, $\ol\gga^\anyrg_T$, so that $\gga^\anyrg_T(s)\leq \gga^\shifted_T(s)\leq \ol\gga^\anyrg_T(s)\leq \ol\gga^\shifted_T(s)$ for all $s\in[0,T]$. We also define the sets of ``white'' and ``red'' particles by
\begin{equation}\label{eq:defNwhite}
	\cN^\twhite_s:=\big\{u\in\cN_s\,\big|\, \forall\, r\leq s,\,X_u(r)\in\big(\gga^\anyrg_T(r),\ol\gga^\anyrg_T(r)\big)\big\},
\end{equation}
and
\begin{equation}\label{eq:defNred}
	\cN^\tred_s:=\left\{u\in\cN_s\,\middle|\, \exists\, \tau\leq s:\begin{aligned}& X_u(\tau)=\ol\gga^\anyrg_T(\tau),\\
		&\forall\, r\leq \tau,\,X_u(\tau)\in\big(\gga^\anyrg_T(r),\ol\gga^\anyrg_T(r)\big),\\
		&\forall\, r\in[\tau,s],\, X_u(r)\in\big(\gga^\shifted_T(r),\ol\gga^\shifted_T(r)\big]\end{aligned}\right\}.
\end{equation}
In other words, we start the process with only white particles, which are killed at $\gga^\anyrg_T$. When they reach $\ol\gga^\anyrg_T$, they are ``colored'' red (instead of being killed) and keep evolving; however, red particles and their offspring are thereafter killed at the barriers $\gga^\shifted_T$ and $\ol\gga^\shifted_T$.

Recall $t^\anyrg_T$ and $\gth(\cdot)$ from~(\ref{eq:defta}--\ref{eq:defta:gthapprox}). 
Recall the definitions of $A^\anyrg_{T,I}$, $R^\anyrg_T(s,t)$ from~\eqref{eq:defA:allregimes},~\eqref{eq:defR}, which are expressed in terms of $\gga^\anyrg_T$ and $\ol\gga^\anyrg_T$. Let us rewrite all estimates from Propositions~\ref{prop:1stmom:supercrit}, \ref{prop:2ndmom:supercrit}, \ref{prop:killedpart:supercrit}, \ref{prop:1stmom:subcrit}, \ref{prop:2ndmom:subcrit}, \ref{prop:killedpart:subcrit}, \ref{prop:1stmom:crit}, \ref{prop:2ndmom:crit} and~\ref{prop:killedpart:crit} for an initial condition (or, equivalently, both barriers) shifted by $-y\gs(0)L(T)$, $y\in[0,x)$ (respectively shifted by $+y\gs(0)L(T)$): this formulation is more convenient to handle all atoms from the initial distribution $\mu_\eps$ (recall~\eqref{eq:defmuK:sec2}). For all regimes $\anyrg\in\{\supercrit, \subcrit, \crit\}$, one has
\begin{align}\label{eq:UB:allmoments:1stUB}
	\bbE_{\gd_{-y\gs(0)L(T)}}\big[|A^\anyrg_{T,I}(t)|\big]\,&\leq\, N^{x-y-\inf I+o(1)}\,, 
	\\\label{eq:UB:allmoments:1st}
	\bbE_{\gd_{-y\gs(0)L(T)}}\big[|A^\anyrg_{T,I}(t)|\big]\,&=\, N^{x-y-\inf I+o(1)}\,  \qquad\text{if, additionally,}\qquad t\ggbis L(T)\,,
	\\\label{eq:UB:allmoments:2nd}
	\bbE_{\gd_{-y\gs(0)L(T)}}\big[|A^\anyrg_{T,z}(t)|^2\big]\,&\leq\, N^{x+h-y-2z+o(1)} \,, 
	\\\label{eq:UB:allmoments:R}
	\bbE_{\gd_{-y\gs(0)L(T)}}\big[R^\anyrg_T(0,t^\anyrg_T)\big]\,&\leq\, N^{-(h+y-x)+o(1)} \,,
\end{align}
for some vanishing terms $o(1)$ as $T\to+\infty$: more precisely, in~(\ref{eq:UB:allmoments:1stUB},~\ref{eq:UB:allmoments:2nd},~\ref{eq:UB:allmoments:R}), these error terms are uniform in $\gs\in\gsensrg$, $0\leq t\leq t^\anyrg_T$ and $x,z\in[0,h]$, $I\subset[0,h]$ a non-trivial sub-interval; and in~\eqref{eq:UB:allmoments:1st}, it is uniform in $\gs\in\gsensrg$, $L(T)\llbis t\leq t^\anyrg_T$ and locally uniform in $x,I$.

For $\eps>0$, $s\in[0,T]$, let us generalize the counting measure $\mu_\eps$ defined in~\eqref{eq:defmuK:sec2}, by letting
\begin{equation}\label{eq:defmuK}
	\mu_{\eps,s}\;:=\; \sum_{k=0}^{\lceil\eps^{-1}\rceil} \Big\lceil N^{k\eps+\frac\eps2} \Big\rceil \gd_{-k\eps\gs(s/T)L(T)}\;,\qquad\text{and}\quad \mu_\eps\;:=\; \mu_{\eps,0}\;.
\end{equation}
Recall that $\mu_\eps$ is supported on $[-\gs(0)L(T),0]\subset (\gga^\shifted_T(0),\ol\gga^\anyrg_T(0))$. Moreover, notice that $\mu_\eps([-\gs(0)L(T),0])\geq N^{1+\eps}$ and $\gd_0\prec\mu_\eps$. In the remainder of this section we will assume that $\eps^{-1}\in\N$ and omit all integer parts, in order to lighten all formulae.

\subsection{Estimates for the BBM between barriers}
We start the multi-type branching-selection process with only white particles distributed according to $\mueps$, and prove that, with large probability, the following three claims hold:

(C-1) there are at least $N$ (white) particles above $\gga^\shifted_T$ at all time $t\in[0,t^\anyrg_T]$,

(C-2) no particle reaches $\ol\gga^\shifted_T$ throughout $[0,t^\anyrg_T]$,

(C-3) the distribution of particles between the barriers at time $t\lesssim t^\anyrg_T$ is close to $\mu_{\eps,t}$.

\begin{remark}
	Let us mention that the third claim (C-3) is actually not needed to prove Theorem~\ref{thm:main}; however it is needed for Proposition~\ref{prop:endtimedistribution}, and it relies on moment methods similar to (C-1) and (C-2), so we included its proof in this section.\end{remark}

\subsubsection*{Lower bound on the number of living particles}
Let us first prove (C-1), which ensures us that no particle from $(\cN^\twhite_s\cup\cN^\tred_s)_{s\in[0,t^\anyrg_T]}$ killed either by $\gga_T^\anyrg$ or $\gga_T^\shifted$ is among the $N$ highest of the process at any time. 
\begin{proposition}\label{prop:allregimes:UB:npart}
	Let $\eps\in(0,1)$ and $h$, $x$ as in~\eqref{eq:UB:parameters}. Then there exist constants $c_1,c_2>0$ such that, for $T$ sufficiently large, one has
	\begin{equation}\label{eq:prop:allregimes:UB:npart}
		\bbP_{\mueps}\big(\exists s\leq t^\anyrg_T:|A_{T,\frac\eps3}^\anyrg(s)| < N\big) \;\leq\;  c_1N^{-c_2}  \,,
	\end{equation}
	where $c_1,c_2$ are uniformly bounded from $0$ and $\infty$ in $\gs\in\gsensrg$, and locally uniformly in $\eps\in(0,1)$.
\end{proposition}

Let us discuss the idea of the proof. Similarly to Proposition~\ref{prop:allregimes:LB:npart} for the lower bound, we write a union bound on the probability by splitting $[0,t^\anyrg_T]$ into intervals of length 1; then we use a moment method to control the size of the population at each integer $k$, and a coupling argument to compare the BBM with a simpler process on each time interval $[k,k+1]$. However in this case the required moment method is a Paley-Zygmund inequality (as opposed to Markov's inequality in Proposition~\ref{prop:allregimes:LB:npart}) that follows from~\eqref{eq:UB:allmoments:1st}, in particular it is not valid for $k$ smaller than $L(T)$. To circumvent this we use the following lemma, which bounds from below the survival probability of a single particle's offspring up to a time $t'(T)\llbis L(T)^3$.
Recall that $\gth(T)$ is defined in~(\ref{eq:defta}--\ref{eq:defta:gthapprox}).
\begin{lemma}\label{lem:prop:allregimes:UB:npart}
	Let $t'(T):=\gth(T)(L(T)^3\wedge T)$. Then
	\[
	\bbP_{\gd_{-\gs(0)L(T)}}\Big(\exists\, u\in\cN_{t'}:\forall s\leq t',\, X_u(s)\in[\gga^\shifted_T(s),\ol\gga^\anyrg_T(s)] \Big)\,\geq\,N^{o(1)}\,,
	\]
	as $T\to+\infty$, uniformly in $\gs\in\gsensrg$ and locally uniformly in $\eps\in(0,1)$.
\end{lemma}
With this lemma at hand, we shall prove that most of the initial particles in the process started from $\mueps$ have surviving descendants at time $t'(T)$. Notice that the latter claim should not hold beyond $L(T)^3$, at which point many of the initial particles' offspring should have gone extinct, and the living population is expected to largely come from the few particles that have peaked.
\begin{proof}[Proof of Lemma~\ref{lem:prop:allregimes:UB:npart}]
	Recall Lemma~\ref{lem:ggachangeparam}, which ensures us that we may tighten the barriers on a short time interval. More precisely let $K>3$ a large constant, and recall Lemma~\ref{lem:ggachangeparam} and the notation $\gga^{\anyrg,h',x'}_T$, $\ol\gga^{\anyrg,h',x'}_T$ for $h'>x'>0$. Then, the lemma and a vertical shift of the process yield
	\begin{align*}
		&\bbP_{\gd_{-\gs(0)L(T)}}\Big(\exists\, u\in\cN_{t'}:\forall s\leq t' ,\, X_u(s)\in[\gga^\shifted_T(s),\ol\gga^\anyrg_T(s)] \Big)\\
		&\qquad \geq\; \bbP_{\gd_0}\Big(\exists\, u\in\cN_{t'}:\forall s\leq t' ,\, X_u(s)\in\Big[\gga^{\anyrg,\frac\eps K,\frac\eps{2K}}_T(s),\ol\gga^{\anyrg,\frac\eps K,\frac\eps{2K}}_T(s)\Big] \Big)\,,
	\end{align*}
	Recall the definition of $A^\anyrg_T(\cdot)$ from~\eqref{eq:defA:allregimes}, and let us extend it to a generic pair of barriers $\gga^{\anyrg,h',x'}_T$, $\ol\gga^{\anyrg,h',x'}_T$, for $h'>x'>0$, by writing for $s\in[0,T]$,
	\[
	A^{\anyrg,h',x'}_T(s)\,:=\,\Big\{u\in\cN_t\;\Big|\; \forall s\in[0,t],\, X_u(s)\in \Big[\gga^{\anyrg,h',x'}_{T}(s), \ol\gga^{\anyrg,h',x'}_{T}(s)\Big]\Big\}\,.
	\]
	Then, Paley-Zygmund's inequality gives
	\begin{align*}
		&\bbP_{\gd_0}\Big(\exists\, u\in\cN_{t'}:\forall s\leq t',\, X_u(s)\in\Big[\gga^{\anyrg,\frac\eps K,\frac\eps{2K}}_T(s),\ol\gga^{\anyrg,\frac\eps K,\frac\eps{2K}}_T(s)\Big] \Big)\\
		&\qquad=\; \bbP_{\gd_0}\big(\big|A^{\anyrg,\frac\eps K,\frac\eps{2K}}_T(t'(T))\big|\geq1\big)\;\geq\; \frac{\bbE_{\gd_0}\big[\big|A^{\anyrg,\frac\eps K,\frac\eps{2K}}_T(t'(T))\big|\big]^2}{\bbE_{\gd_0}\big[\big|A^{\anyrg,\frac\eps K,\frac\eps{2K}}_T(t'(T))\big|^2\big]}
		\,.
	\end{align*}
	One easily checks that~\eqref{eq:gthasymp} implies that $L(T)\llbis t'(T)$. Recalling the moment estimates~(\ref{eq:UB:allmoments:1stUB}--\ref{eq:UB:allmoments:2nd}) for the triplet of parameters $(h',x',y)=(\frac\eps K,\frac\eps{2K},0)$, one obtains as $T\to+\infty$,
	\[ \bbE_{\gd_0}\big[\big|A^{\anyrg,\frac\eps K,\frac\eps{2K}}_T(t'(T))\big|\big]\,=\, N^{\frac{\eps}{2K}+o(1)}\,, \qquad\text{and}\qquad \bbE_{\gd_0}\big[\big|A^{\anyrg,\frac\eps K,\frac\eps{2K}}_T(t'(T))\big|^2\big]\,\leq\, N^{\frac{3\eps}{2K}+o(1)}\,. \]
	Therefore, one finally deduces that
	\[
	\bbP_{\gd_{-\gs(0)L(T)}}\Big(\exists\, u\in\cN_{t'}:\forall s\leq t',\, X_u(s)\in[\gga^\shifted_T(s),\ol\gga^\anyrg_T(s)] \Big)\;\geq\; N^{-\frac\eps K + o(1)}\,,
	\]
	as $T\to+\infty$; letting $K\to+\infty$, this finishes the proof of the lemma.
\end{proof}

\begin{proof}[Proof of Proposition~\ref{prop:allregimes:UB:npart}]
	Since one has $\mueps \succ N^{1+\frac\eps2}\gd_{-\gs(0)L(T)}$, we only have to prove~\eqref{eq:prop:allregimes:UB:npart} for the latter initial measure; then the proposition follows from a direct monotonic coupling argument.
	
	We prove this proposition with a union bound, splitting the time interval $[0,t^\anyrg_T]$ into a first part of length $t'(T):= \gth(T)(L(T)^3\wedge T)$, and the remainder into intervals of length 1. Thus,
	\begin{align}\label{eq:allregimes:UB:npart:1}
		&\bbP_{N^{1+\frac\eps2}\gd_{-\gs(0)L(T)}}\big(\exists s\leq t^\anyrg_T:|A_{T,\frac\eps3}^\anyrg(s)| < N\big) \\
		\notag &\quad\leq\; \bbP_{N^{1+\frac\eps2}\gd_{-\gs(0)L(T)}}\big(\exists s\leq t'(T):|A_{T,\frac\eps3}^\anyrg(s)| < N\big) \\
		\notag &\qquad + \sum_{k=t'(T)}^{t^\anyrg_T-1} \bbP_{N^{1+\frac\eps2}\gd_{-\gs(0)L(T)}}\big(\exists s\in[k,k+1]:|A_{T,\frac\eps3}^\anyrg(s)| < N\big)\,,
	\end{align}
	where, again, we respectively wrote $t'(T)$, $t^\anyrg_T-1$ instead of $\lfloor t'(T)\rfloor$, $\lceil t^\anyrg_T-1\rceil$ to lighten notation. Notice that the sum may be empty in the super-critical regime.
	
	We start with the first term in~\eqref{eq:allregimes:UB:npart:1}. Let $M$ denote the number of individuals from the \emph{initial} population which have at least one descendant surviving between $\gga^\shifted_T$ and $\ol\gga^\anyrg_T$ until time $t':=t'(T)$. For a single initial particle in ${-\gs(0)L(T)}$, we write,
	\[p_T \;:=\; \bbP_{\gd_{-\gs(0)L(T)}}\left(\exists u\in\cN_{t'}:\forall s\leq t' ,\, X_u(s)\in[\gga^\shifted_T(s),\ol\gga^\anyrg_T(s)] \right)\;\geq\; N^{o(1)}\,,\]
	where the last inequality is the content of Lemma~\ref{lem:prop:allregimes:UB:npart}. In particular, one has,
	\[p_T\,=\,  \bbP_{\gd_{-\gs(0)L(T)}}(M=1)\,=\,1-\bbP_{\gd_{-\gs(0)L(T)}}(M=0)\,.
	\] 
	Starting from an initial population of $N^{1+\frac\eps2}$ particles, recall that they have independent offspring. Therefore, under $\bbP_{N^{1+\frac\eps2}\gd_{-\gs(0)L(T)}}$, one has that $M$ is a binomial random variable with parameters $(N^{1+\frac\eps2}, p_T)$. Moreover, bounding the first term in~\eqref{eq:allregimes:UB:npart:1} from above by killing particles at $\gga^\shifted_T$, we have,
	\[
	\bbP_{N^{1+\frac\eps2}\gd_{-\gs(0)L(T)}}\big(\exists s\leq t'(T):|A_{T,\frac\eps3}^\anyrg(s)| < N\big)\;\leq\; \bbP_{N^{1+\frac\eps2}\gd_{-\gs(0)L(T)}}(M<N)\;.
	\]
	Then, Paley-Zygmund's inequality yields,
	\begin{align*}
		\bbP_{N^{1+\frac\eps2}\gd_{-\gs(0)L(T)}}(M\geq N)\;&\geq\;  \left(1-\frac{N}{\bbE_{N^{1+\frac\eps2}\gd_{-\gs(0)L(T)}}[M]}\right)^2 \frac{\bbE_{N^{1+\frac\eps2}\gd_{-\gs(0)L(T)}}[M]^2}{\bbE_{N^{1+\frac\eps2}\gd_{-\gs(0)L(T)}}[M^2]} \\&=\; \left(1-N^{-\frac\eps2}p_T^{-1}\right)^2 \frac{N^{1+\frac\eps2}p_T}{1-p_T+N^{1+\frac\eps2}p_T}\,.
	\end{align*}
	This implies $\bbP_{N^{1+\frac\eps2}\gd_{-\gs(0)L(T)}}(M<N)\leq N^{-\frac\eps2 +o(1)}$ as $T\to+\infty$, uniformly in $\gs\in\gsensrg$, which is the announced upper bound for the first term in~\eqref{eq:allregimes:UB:npart:1}. 
	
	We now turn to the second term in~\eqref{eq:allregimes:UB:npart:1}. Recall~\eqref{eq:defA:allregimes}: in particular, let $D(s)$ denotes the set of white particles that end in the interval $[\gga_T^\shifted(s)+\frac\eps{2}\gs(s/T)L(T),\ol\gga_T^\anyrg(s)-\frac\eps{2}\gs(s/T)L(T)]$ at time $s$, that is,
	\[
	D(s)\;:=\;A^\anyrg_{T,[\frac{5\eps}{6},h-\frac{\eps}{2}]}(s)\,.
	\]
	Then, we bound each term of the sum with an union bound, for $k\geq t'(T)$,
	\begin{align}\label{eq:allregimes:UB:npart:1.5}
			&\bbP_{N^{1+\frac\eps2}\gd_{-\gs(0)L(T)}}\big(\exists s\in[k,k+1]:|A_{T,\frac\eps3}^\anyrg(s)| < N\big)\\
			\notag &\quad\leq\; \bbP_{N^{1+\frac\eps2}\gd_{-\gs(0)L(T)}}\big(|D(k)|<N^{1+\frac\eps4}\big) \\
			\notag &\qquad + \bbP_{N^{1+\frac\eps2}\gd_{-\gs(0)L(T)}}\big(|D(k)|\geq N^{1+\frac\eps4}\;;\;\exists s\in[k,k+1]:|A_{T,\frac\eps3}^\anyrg(s)| < N\big)
	\end{align}
	We handle those two terms separately. Applying Paley-Zygmund's inequality, we have
	\begin{align}\label{eq:allregimes:UB:npart:2}
		&\bbP_{N^{1+\frac\eps2}\gd_{-\gs(0)L(T)}}\Big(|D(k)| \geq N^{1+\frac\eps4}\Big) \\
		\notag&\quad \geq\, \left(1-\frac{N^{1+\frac\eps4}}{\bbE_{N^{1+\frac\eps2}\gd_{-\gs(0)L(T)}}\big[|D(k)|\big]}\right)^2\times \frac{\bbE_{N^{1+\frac\eps2}\gd_{-\gs(0)L(T)}}\big[|D(k)|\big]^2}{\bbE_{N^{1+\frac\eps2}\gd_{-\gs(0)L(T)}}\big[|D(k)|^2\big]}.
	\end{align}
	Then, since $k\geq t'(T)\ggbis L(T)$, \eqref{eq:UB:allmoments:1st} implies that
	\[\bbE_{\gd_{-\gs(0)L(T)}}[|D(k)|]\,=\,N^{-\frac{\eps}{6}+o(1)},\]
	as $T\to+\infty$. Moreover,~\eqref{eq:UB:allmoments:2nd} yields
	\[\bbE_{\gd_{-\gs(0)L(T)}}\big[|D(k)|^2\big] \,\leq\, N^{1+o(1)},\]
	as $T\to+\infty$. Noticing that the first moment of $|D(k)|$ is additive in the initial measure, applying~\eqref{eq:PZ:2ndmoment} to its second moment and plugging these estimates into~\eqref{eq:allregimes:UB:npart:2}, we finally obtain
	\begin{align}\label{eq:allregimes:UB:npart:3}
		&\bbP_{N^{1+\frac\eps2}\gd_{-\gs(0)L(T)}}\Big(|D(k)| \geq N^{1+\frac\eps4}\Big) \\
		\notag &\quad\geq\, \left(1-N^{-\frac{\eps}{12}+o(1)}\right)^2\left(1+N^{-\frac{\eps}{6}+o(1)}\right)^{-1} = 1-N^{-\frac\eps{12}+o(1)}\;,
	\end{align}
	where $o(1)$ denotes a term vanishing as $T\to+\infty$ uniformly in $t'(T)\leq k\leq t^\anyrg_T$ and $\gs\in\gsensrg$.
	
	Regarding the second term in~\eqref{eq:allregimes:UB:npart:1.5}, let $\cA$ denote the set of counting measures supported on $[\gga_T^\shifted(k)+\frac\eps2\gs(k/T)L(T),\ol\gga_T^\anyrg(k)-\frac\eps2\gs(k/T)L(T)]$ with total mass at least $N^{1+\frac\eps4}$. Using the Markov property at time $k$, we have
	\begin{align*}
	&\bbP_{N^{1+\frac\eps2}\gd_{-\gs(0)L(T)}}\Big(|D(k)| \geq N^{1+\frac\eps4}\,;\, \exists\,s\in[k,k+1]:|A_{T,\frac\eps3}^\anyrg(s)| < N\Big)\\
	&\quad \leq\, \sup_{\mu\in\cA} \bbP_\mu\big( \exists\,s\leq 1:|\tilde A_{T,\frac\eps3}^\anyrg(s)| < N \big),
	\end{align*}
	with $\tilde A_{T,\frac\eps3}^\anyrg(\cdot)$ being defined similarly to the event $A_{T,\frac\eps3}^\anyrg(\cdot)$ for barriers $\gga^\anyrg_T$, $\ol\gga^\anyrg_T$ shifted in time by $k$. Since adding particles to the initial measure $\mu$ only decreases the probability in the r.h.s. above (this is follows from a direct coupling argument), on can restrict the supremum to measures $\mu$ with total mass exactly $N^{1+\frac\eps4}$. If the total population decreases from $N^{1+\frac\eps4}$ to $N$ at some time $s\leq 1$, this implies that, among the initial particles, at least $N^{1+\frac\eps4}-N$ of them have no living descendant at time $1$. Hence, a union bound yields
	\begin{equation}\label{eq:PZ:allregimes:UB:4}
		\sup_{\mu\in\cA} \bbP_\mu\big( \exists\,s\leq 1: |\tilde A_{T,\frac\eps3}^\anyrg(s)| < N \big) \,\leq\, \binom{N^{1+\frac\eps4}}{N^{1+\frac\eps4}-N} \times \left(\sup_{y} \bbP_{\gd_{y}} \big( |\tilde A_{T,\frac\eps3}^\anyrg(1)| =0 \big)\right)^{N^{1+\frac\eps4}-N},
	\end{equation}
	where the supremum is taken over $y\in [\gga_T^\shifted(k)+\frac\eps2\gs(k/T)L(T),\ol\gga_T^\anyrg(k)-\frac\eps2\gs(k/T)L(T)]$. For any such $y$, we may couple the $N$-BBM starting from $y$ with a Brownian motion $(B_s)_{s\geq0}$ without reproduction, by looking at an arbitrary descendant of $y$. If the $N$-BBM starting from $\gd_y$ goes extinct, the coupled Brownian motion crosses one of the barriers $\gga^\shifted_T$, $\ol\gga^\anyrg_T$ on the time interval $[0,1]$. To do so, it must travel a distance at least $\frac\eps{4}\gs(0)L(T)$, uniformly in $y$: therefore, there exists $c,C>0$, (locally) uniformly in the parameters as in the statement of the proposition, such that, for large $T$,
	\[
	\sup_{y} \bbP_{\gd_{y}} \big( |\tilde A_{T,\frac\eps3}^\anyrg(1)| =0 \big)\;\leq\; \sup_y \bP_{y}\left(\exists\, s\leq 1: |B_s-y|> \tfrac\eps{4}\gs(0)L(T) \right) \;\leq\; Ce^{-cL(T)^2} \le 1/2\,,
	\]
	where the second inequality follows from standard computations on the Brownian motion. Plugging this into~\eqref{eq:PZ:allregimes:UB:4} and using that $\binom ab\leq a^{a-b}$ for any $a,b\in\N$, the second term in~\eqref{eq:allregimes:UB:npart:1.5} is bounded from above by $N^{-c'N^{1+\frac\eps 4})}$, for some $c'>0$, (locally) uniformly in the parameters as in the statement of the proposition. Recollecting~\eqref{eq:allregimes:UB:npart:3} and summing over $O(t^\anyrg_T)$ terms in~\eqref{eq:allregimes:UB:npart:1}, we finally obtain the announced upper bound.
\end{proof}

\subsubsection*{Number of particles killed by the upper barrier}
Let us now turn to the second claim (C-2). We provide some estimates on the number of particles reaching either of the upper barriers before time $t^\anyrg_T$. First we bound from above the number of white particles that reach $\ol\gga^\anyrg_T$, i.e. white particles which are colored red: this is exactly given by $R^\anyrg_T(0,t^\anyrg_T)$ (recall its definition from~\eqref{eq:defR}). Then, for each particle reaching $\ol\gga^\anyrg_T$ for the first time, we estimate the number of (red) particles from its offspring that reach $\ol\gga^\shifted_T$. 
Recall~\eqref{eq:UB:allmoments:R}.

\begin{claim}\label{claim:allregimes:UB:olgga:white}
	One has, as $T\to+\infty$,
	\begin{equation}
		\bbE_{\mueps}\left[\left|R^\anyrg_T(0,t^\anyrg_T)\right|\right]\,\leq\, N^{\frac\eps6+o(1)},
	\end{equation}
	uniformly in $\gs\in\gsensrg$ and locally uniformly in $\eps\in(0,1)$.
\end{claim}

\begin{claim}\label{claim:allregimes:UB:olgga:red}
	Let $t_0\in[0,t^\anyrg_T)$, and consider a BBM starting from one (red) particle in time-space location $(t_0, \ol\gga_T^\anyrg(t_0))$. Then one has, as $T\to+\infty$,
	\begin{equation}
		\bbE_{\gd_{(t_0,\ol\gga_T^\anyrg(t_0))}}\left[\left|\bigcup_{t_0\leq s\leq t^\anyrg_T}\left\{
		u\in\cN_r\,\middle|\,
		\begin{aligned}
			&\forall\, t_0\leq r< s,\,X_u(r)\in (\gga_T^\shifted(r),\ol\gga_T^\shifted(r))\,;\\
			&X_u(s)=\ol\gga_T^\shifted(s)
		\end{aligned}
		\right\}\right|\right]\leq N^{-\frac\eps3+o(1)},
	\end{equation}
	where $o(1)$ vanishes as $T\to+\infty$ uniformly in $t_0\in[0,t^\anyrg_T)$ and $\gs\in\gsensrg$, and locally uniformly in $\eps\in(0,1)$.
\end{claim}

\begin{proof}[Proof of Claims~\ref{claim:allregimes:UB:olgga:white}, \ref{claim:allregimes:UB:olgga:red}]
	The first result is a direct corollary of~\eqref{eq:UB:allmoments:R} and the additivity in the initial measure: indeed, one has for any $0\leq k\leq \eps^{-1}$,
	\begin{align*}
	&\bbE_{N^{k\eps+\frac\eps2}\gd_{-k\eps\gs(0)L(T)}}[R_T^\anyrg(0,t^\anyrg_T)] \\
	&\quad =N^{k\eps+\frac\eps2}\bbE_{\gd_{-k\eps\gs(0)L(T)}}[R_T^\anyrg(0,t^\anyrg_T)]\leq N^{(k\eps+\frac\eps2)-(k\eps+\frac\eps3)+o(1)}=N^{\frac\eps6+o(1)}\,,
	\end{align*}
	so Claim~\ref{claim:allregimes:UB:olgga:white} follows by summing over $k$. Regarding Claim~\ref{claim:allregimes:UB:olgga:red}, it also follows from~\eqref{eq:UB:allmoments:R} applied to a time- and space-shifted BBM on the time interval $[t_0,t^\anyrg_T]$, killed at the barriers $\gga_T^\shifted$, $\ol\gga_T^\shifted$, and starting from a single particle at distance $\frac\eps3\gs(t_0/T)L(T)$ from the upper barrier $\ol\gga^\shifted_T$ (we do not write the details again).
\end{proof}

\begin{remark}\label{rem:coloring}
	Let us point out that the upper bound in Claim~\ref{claim:allregimes:UB:olgga:white} is larger than 1, so, with positive probability, there exist white particles that reach $\ol\gga^\anyrg_T$ before time $t^\anyrg_T$. The reader can check that, in the sub-critical and critical regimes, there \emph{does not exist} a choice of initial configuration $\mu$ and barrier parameters $(h,x)\to(1,1)$ (recall Lemma~\ref{lem:comparison:m:gga}) which yields simultaneously Proposition~\ref{prop:allregimes:UB:npart} and a moment estimate lower than 1 in Claim~\ref{claim:allregimes:UB:olgga:white}. This is not an issue in the super-critical regime since the proof of Proposition~\ref{prop:allregimes:UB:npart} can be simplified (see~\eqref{eq:allregimes:UB:npart:1}), giving more leeway in the choice of $(h,x)$.
\end{remark}

Following these observations, we may finally prove that $\ol\gga_T^\shifted$ does not kill any particle with high probability. Let $R^\tred_T(0,t)$ denote the number of particles killed by $\ol\gga^\shifted_T$. In particular, they remained above $\gga_T^\anyrg$ until some time $\tau_u\leq t^\anyrg_T$ upon which they reached $\ol\gga_T^\anyrg$ and were colored red; then they remained above $\gga_T^\shifted$ throughout $[\tau_u,t^\anyrg_T]$ and reached $\ol\gga_T^\shifted$ at some time $r\in[\tau_u,t]$ (upon which they are killed).
\begin{proposition}\label{prop:allregimes:UB:olgga}
	One has, as $T\to+\infty$,
	\begin{equation}
		\bbE_{\mueps}\left[R^\tred_T(0,t^\anyrg_T)\right]\,\leq\, N^{-\frac\eps6+o(1)},
	\end{equation}
	uniformly in $\gs\in\gsensrg$ and locally uniformly in $\eps\in(0,1)$.
\end{proposition}

With this proposition at hand, claim (C-2) is obtained by applying Markov's inequality to $R^\tred_T(0,t^\anyrg_T)$.

\begin{proof}
	This result is a consequence of Claims~\ref{claim:allregimes:UB:olgga:white},~\ref{claim:allregimes:UB:olgga:red} and a bit of stopping lines theory (recall Section~\ref{sec:coupling:NBBM}). Define
	\[
	\cL:=\big\{(u,s)\,\big|\, \forall r< s,\,X_u(r)\in(\gga_T^\anyrg(r),\ol\gga_T^\anyrg(r))\,;\, X_u(s)=\ol\gga_T^\anyrg(s)\big\},
	\]
	which is the (random) stopping line containing white particles at the moment they hit $\ol\gga_T^\anyrg(\cdot)$ and get colored. Let
	\[
	\cF_\cL:=\gs\left( \left\{\begin{aligned}&\forall r\leq s,\, X_u(r)>\gga_T^\anyrg(r)\,;\\& \forall r< s,\, X_u(r)<\ol\gga_T^\anyrg(r)\,;\\&X_u(s)\in A\end{aligned}\right\};\, s\geq0, u\in\cN_s, A\subset \R \text{ a Borel set}\right).
	\] 
	Informally, $\cF_\cL$ is the sigma-algebra containing all information about white particles. Then the \emph{strong branching property}~\cite[Theorem~4.14]{Jag89} states that, conditionally on $\cF_\cL$, the sub-trees of the process rooted at the pairs $(u,s)\in\cL$ are independent with respective distributions $\bbP_{\gd_{(s,X_u(s))}}=\bbP_{\gd_{(s,\ol\gga_T^\anyrg(s))}}$.
	
	Notice that Claim~\ref{claim:allregimes:UB:olgga:white} implies $|\cL|=R^\anyrg_T(0,t^\anyrg_T)<+\infty$, $\bbP_{\mueps}$-almost surely. Moreover, any particle in $R^\tred_T(0,t^\anyrg_T)$ almost surely has a single ancestor $(u,\tau)\in\cL$ ---in particular this ancestor $u$ was colored red at time $\tau$. Therefore, we obtain by conditioning with respect to $\cF_\cL$ and applying the strong branching property,
	\begin{align*}
		&\bbE_{\mueps}\left[R^\tred_T(0,t^\anyrg_T)\right]\\
		&\!= \bbE_{\mueps}\!\!\Bigg[ \bbE_{\mueps}\!\!\bigg[\sum_{(u,\tau)\in\cL}\Bigg|\!\bigcup_{\tau \leq r\leq t^\anyrg_T}\!\!\Bigg\{\!v\in\cN_r,v\succeq u\;\Bigg|\; \begin{aligned}
			&\forall\, s\in[\tau,r),\,X_v(s)\in(\gga_T^\shifted(s),\ol\gga_T^\shifted(s))\\& \text{and}\quad X_v(r)=\ol\gga_T^\shifted(r)
			\end{aligned}\Bigg\}\Bigg|\,\bigg|\cF_\cL\bigg]\!\Bigg]\\
		&\!= \bbE_{\mueps}\!\!\Bigg[ \sum_{(u,\tau)\in\cL}\bbE_{\gd_{(\tau,\ol\gga_T^\subcrit(\tau))}}\!\!\Bigg[\Bigg|\bigcup_{\tau\leq t\leq t^\anyrg_T}\!\!\Bigg\{\!v\in\cN_r\;\Bigg|\; \begin{aligned}
			&\forall\, s\in[\tau,r),\,X_v(s)\in(\gga_T^\shifted(s),\ol\gga_T^\shifted(s))\\& \text{and}\quad X_v(r)=\ol\gga_T^\shifted(r)
			\end{aligned}\Bigg\}\Bigg|\Bigg] \!\Bigg],
	\end{align*}
	where $v\succeq u$ means that $v\in\cN_r$ is a descendant of $u\in\cN_\tau$, $\tau\leq r$. Plugging Claims~\ref{claim:allregimes:UB:olgga:white},~\ref{claim:allregimes:UB:olgga:red} into this, we finally obtain
	\[\bbE_{\mueps}\left[R^\tred_T(0,t^\anyrg_T)\right]\,\leq\, \bbE_{\mueps}\left[ N^{-\frac\eps3+o(1)}\times\big|\cL\big| \right] = N^{-\frac\eps3+o(1)} \times N^{\frac\eps6+o(1)} = N^{-\frac\eps6+o(1)}, \]
	which concludes the proof.
\end{proof}

\subsubsection*{Particle distribution at the final time}
Finally, we prove the third statement (C-3). We first provide the following two estimates, respectively for white and red particles.

\begin{lemma}\label{lem:allregimes:UB:muK}
	The following statements hold uniformly in $\gs\in\gsensrg$ and locally uniformly in $\eps\in(0,1)$.
	
	$(i)$ For $0\leq j,k\leq \eps^{-1}$, one has, as $T\to+\infty$,
	\begin{align}\label{eq:prop:allregimes:UB:muK:white}
		\inf_{s\in[\gth(T)^{-1}L(T),t^\anyrg_T]}\, \bbP_{N^{k\eps+\frac\eps2}\gd_{-k\eps\gs(0)L(T)}}\left(N^{(j+1)\eps}\leq \Big|A^\anyrg_{T,[h-(j+1)\eps,h-j\eps]}(s)\Big|\leq N^{(j+2)\eps}\right)&\\
		\notag\quad \geq 1-N^{-\frac\eps6+o(1)}\,.&
	\end{align}
	
	$(ii)$ Let $t_0\in[0,t^\anyrg_T)$, and consider a BBM starting from one (red) particle in time-space location $(t_0, \ol\gga_T^\anyrg(t_0))$. Then for $0\leq j\leq \eps^{-1}$, one has, as $T\to+\infty$,
	\begin{align}\label{eq:prop:allregimes:UB:muK:red}
		\sup_{s\in[t_0,t^\anyrg_T]} \, \bbE_{\gd_{(t_0,\ol\gga^\anyrg_T(t_0))}}\left[\left| \left\{
		u\in\cN_{s}\,\middle|\,
		\begin{aligned}
			&\forall\, t_0\leq r< s,\, X_u(r)\in (\gga_T^\shifted(r),\ol\gga_T^\shifted(r))\,;\\
			&\tfrac{X_u(s) -\gga^\anyrg_T(s)}{\gs(s)L(T)} \in \left[h-(j+1)\eps,h-j\eps\right]
		\end{aligned}
		\right\} \right|\right]& \\
		\notag\quad \leq\, N^{(j+1)\eps+o(1)}\,.&
	\end{align}
\end{lemma}

From these results, the claim (C-3) follows naturally: let $(\cX_s^{\twhite-\tred})_{s\in[0,t^\anyrg_T]}$ denote the empirical mass measure on $\R$ of the process defined by white and red particles, that is
\begin{equation}\label{eq:defXwhitered}
	\cX_s^{\twhite-\tred}\,:=\,\sum_{u\in \cN^\twhite_s\cup\cN^\tred_s} \gd_{X_u(s)}\,,\qquad s\in[0,t^\anyrg_T]\,.
\end{equation}
Recall~\eqref{eq:defmuK}. In order to have a condensed statement, let us write $\mu_{\eps,s}^{(y)}$ for the counting measure $\mu_{\eps,s}$ shifted upward by $y$, that is $\mu_{\eps,s}^{(y)}(\cdot):=\mu_{\eps,s}(\cdot-y)$, for $s\in[0,T]$, $y\in\R$.

\begin{proposition}\label{prop:allregimes:UB:muK}
	Let $\eps>0$. Then one has as $T\to+\infty$,
	\begin{equation}
		\inf_{s\in[\gth(T)^{-1}L(T),t^\anyrg_T]}\,
		\bbP_{\mueps} \left(\mu_{\eps,s}^{(\ol\gga^\anyrg_T(s)-\eps\gs(s)L(T))}\,\prec\, \cX_{s}^{\twhite-\tred}\,\prec\, N^{\frac52\eps}\mu_{\eps,s}^{(\ol\gga^\anyrg_T(s))}\right)\,\geq\, 1-N^{-\frac\eps6+o(1)}\,,
	\end{equation}
	uniformly in $\gs\in\gsensrg$ and locally uniformly in $\eps\in(0,1)$.
\end{proposition}

\begin{proof}[Proof of Lemma~\ref{lem:allregimes:UB:muK}]
	$(i)$ Recall~\eqref{eq:UB:parameters} and~(\ref{eq:UB:allmoments:1stUB}--\ref{eq:UB:allmoments:2nd}). On the one hand, Markov's inequality gives
	\[\begin{aligned}
		&\bbP_{N^{k\eps+\frac\eps2}\gd_{-k\eps\gs(0)L(T)}}\left( \Big|A^\anyrg_{T,[h-(j+1)\eps,h-j\eps]}(s)\Big| > N^{(j+2)\eps}\right) \\
		&\qquad\leq\, N^{k\eps+\frac\eps2}N^{-(j+2)\eps} \bbE_{\gd_{-k\eps\gs(0)L(T)}}\left[ \Big|A^\anyrg_{T,[h-(j+1)\eps,h-j\eps]}(s)\Big| \right]\\
		&\qquad\leq\, N^{k\eps+\frac\eps2}N^{-(j+2)\eps}N^{x-k\eps -h+(j+1)\eps+o(1)}=N^{-\frac56\eps+o(1)},
	\end{aligned}\]
	for $0\leq k,j\leq\eps^{-1}$, $s\leq t^\anyrg_T$ and $T$ large. On the other hand, Paley-Zygmund's inequality and~\eqref{eq:UB:allmoments:1stUB}, \eqref{eq:UB:allmoments:2nd} yield  for $s\in[\gth(T)^{-1}L(T),t^\anyrg_T]$, (we leave the details to the reader),
	\begin{align*}
		&\bbP_{N^{k\eps+\frac\eps2}\gd_{-k\eps\gs(0)L(T)}}\left(\Big|A^\anyrg_{T,[h-(j+1)\eps,h-j\eps]}(s)\Big| \geq N^{(j+1)\eps}\right)\\
		&\qquad\geq\, \left(1-N^{-\frac\eps6+o(1)}\right)^2\left(1+N^{-\frac\eps6+o(1)}\right)^{-1}\,\geq\,1-N^{-\frac\eps6+o(1)}\,,
	\end{align*}
	for some $c_1,c_2>0$ and $T$ sufficiently large, which concludes the proof of~\eqref{eq:prop:allregimes:UB:muK:white}.
	
	$(ii)$ Similarly to Claim~\ref{claim:allregimes:UB:olgga:red}, this follows from~\eqref{eq:UB:allmoments:1stUB} applied to a time- and space-shifted BBM on the time interval $[t_0,t^\anyrg_T]$, killed at the barriers $\gga_T^\shifted$, $\ol\gga_T^\shifted$, and starting from a single particle at distance $(h-\frac\eps3)\gs(t_0/T)L(T)$ from the lower barrier $\gga^\shifted_T$ (we leave the details to the reader).
\end{proof}

\begin{proof}[Proof of Proposition~\ref{prop:allregimes:UB:muK}]
	Recall Claim~\ref{claim:allregimes:UB:olgga:white} and~\eqref{eq:prop:allregimes:UB:muK:red}. Using the strong branching property~\cite[Theorem~4.14]{Jag89}, one deduces uniformly in $0\leq j\leq \eps^{-1}$, $s\in[\gth(T)^{-1}L(T),t^\anyrg_T]$,
	\[
	\bbE_{\mueps}\left[\left|\left\{u\in \cN^\tred_{s} \,\middle|\, \tfrac{X_u(s) - \ol\gga^\anyrg_T(s)}{\gs(s)L(T)} \in \big[-(j+1)\eps,-j\eps\big]  \right\}\right|\right]\,\leq\, N^{\frac\eps6+(j+1)\eps+o(1)}\,.
	\]
	Hence, a union bound and Markov's inequality yield,
	\begin{align*}
	\bbP_{\mueps}\left(\exists\, 0\leq j\leq \eps^{-1}:\left|\left\{u\in \cN^\tred_{s} \,;\, \tfrac{X_u(s) - \ol\gga^\anyrg_T(s)}{\gs(s)L(T)} \in [-(j+1)\eps,-j\eps] \right\}\right| > N^{(j+2)\eps}\right)&\\
	\leq\, N^{-\frac56\eps+o(1)}\,.&
	\end{align*}
	Moreover, a union bound and~\eqref{eq:prop:allregimes:UB:muK:white} yield that, uniformly in $s\in[\gth(T)^{-1}L(T),t^\anyrg_T]$,
	\begin{align*}
		&\bbP_{\mueps}\Bigg(\exists\,0\leq j\leq \eps^{-1}: \\
		&\quad \left|\left\{u\in \cN^\twhite_{t^\anyrg_T} \,;\, \tfrac{X_u(t^\anyrg_T) - \ol\gga^\anyrg_T(t^\anyrg_T)}{\gs(t^\anyrg_T)L(T)} \in \big[-(j+1)\eps,-j\eps\big]  \right\}\right|\notin (\eps^{-1}+1)\big[N^{(j+1)\eps},N^{(j+2)\eps}\big]\Bigg)\\
		&\leq \sum_{j=0}^{\eps^{-1}}\sum_{k=0}^{\eps^{-1}} \sup_{s\in[\frac12 t^\anyrg_T,t^\anyrg_T]}\, \bbP_{N^{k\eps+\frac\eps2}\gd_{-k\eps\gs(0)L(T)}}\left(\Big|A^\anyrg_{T,[h-(j+1)\eps,h-j\eps]}(s)\Big|\notin\big[N^{(j+1)\eps}, N^{(j+2)\eps}\big]\right) \\
		&\leq N^{-\frac\eps6+o(1)} \,,
	\end{align*}
	for $T$ large. Therefore, there exists $T_\eps>0$ such that for $T\geq T_\eps$, with large $\bbP_{\mueps}$-probability, for all $0\leq j\leq \eps^{-1}$, there are between $N^{(j+1)\eps}$ and $N^{(j+3)\eps}$ particles (white or red) ending in the interval $\ol\gga^\anyrg_T(s)+\gs(s)L(T)[-(j+1)\eps,-j\eps]$ at time $s$, uniformly in $s\in[\gth(T)^{-1}L(T),t^\anyrg_T]$. Recalling~\eqref{eq:defmuK}, this directly implies the proposition.
\end{proof}

\subsection{Proof of Proposition~\ref{prop:main:UB}}\label{sec:UB:coreprf}
We may finally provide the proof of Proposition~\ref{prop:main:UB}. It shares several similarities to that of Lemma~\ref{lem:main:LB}, notably with the sub-critical case needing additional work (since our moment estimates do not hold until time $T$). Recall that the point measure of the $N$-BBM throughout time is denoted by $(\cX_{t}^{N})_{t\geq0}$. Let $\eps\in(0,1)$,  $\anyrg\in\{\subcrit,\supercrit,\crit\}$, and recall~\eqref{eq:defta}, \eqref{eq:UB:parameters} and the definitions of $\gga^\anyrg_T$, $\ol\gga^\anyrg_T$ from~(\ref{def:gga:supercrit}--\ref{def:gga:crit}).

\begin{lemma}\label{lem:main:UB:coupling}
	Let $\eps\in(0,1)$ and $\anyrg\in\{\subcrit,\supercrit,\crit\}$. There exists $c_1,c_2>0$ such that for $T$ sufficiently large, one has
	\begin{equation}\label{eq:UB:mainineq}
		\inf_{t\in[0,t^\anyrg_T]} \bbP_{\mueps}\big(\max(\cX_{t}^{N}) \leq \ol\gga^\shifted_T(t)  \big) \;\geq\; 1-c_1 N^{-c_2}\,,
	\end{equation}
	and 
	\begin{equation}\label{eq:UB:mainineq:enddist}
		\inf_{t\in[\gth(T)^{-1}L(T),t^\anyrg_T]}
		\bbP_{\mueps} \left(\cX_{t}^{N}\,\prec\, \mu_{\eps,t}^{(\ol\gga^\anyrg_T(t)+3\eps\gs(t/T)L(T))}\right)\,\geq\, 1- c_1 N^{-c_2}\,,
	\end{equation}
	where $c_1,c_2$ are uniformly bounded away from 0 and $\infty$ in $\gs\in\gsensrg$ and locally uniformly in $\eps\in(0,1)$.
\end{lemma}

Let us mention that the proof of Theorem~\ref{thm:main} only requires~\eqref{eq:UB:mainineq}; however~\eqref{eq:UB:mainineq:enddist} is obtained with the same method and is needed to prove Proposition~\ref{prop:endtimedistribution} in Section~\ref{sec:compl}.

\begin{proof}[Proof of Lemma~\ref{lem:main:UB:coupling}]
	Let us start with~\eqref{eq:UB:mainineq}. Denote by $(\cX_{t}^{N+})_{t\geq0}$ the point process of a BBM with the following selection mechanism: a particle is killed whenever it satisfies simultaneously $(i)$ there are $N$ other particles above it, and $(ii)$ it is below the barrier $\gga_T^\anyrg$ \emph{or} it is below $\gga_T^\shifted$ and has an ancestor which is above $\ol\gga_T^\anyrg$. Recalling Definition~\ref{def:N-N+BBM}, this process is an $N^+$-BBM.
	
Recall from~(\ref{eq:defNwhite}--\ref{eq:defNred}) and~\eqref{eq:defXwhitered} the definition of the multi-type process $(\cX_t^{\twhite-\tred})_{t\geq0}$ containing both ``white'' and ``red'' particles. 
We claim that one has for some $c_1,c_2>0$, 
\begin{equation}\label{eq:lem:main:UB:coupling}
\bbP_{\mueps}\big(\max(\cX^{N+}_{t})> \ol\gga^\shifted_T(t)\big)\,\leq\, \bbP_{\mueps}\big(\max(\cX_{t}^{\twhite-\tred})> \ol\gga^\shifted_T(t)\big) + 2c_1N^{-c_2}\,=\,2c_1N^{-c_2}\,,
\end{equation}
for $T$ sufficiently large, uniformly in $t\leq t^\anyrg_T$. Indeed, Proposition~\ref{prop:allregimes:UB:olgga} and Markov's inequality imply that, as $T\to+\infty$,
\begin{equation*}
\bbP_{\mueps}\Big(\exists s\in[0,t^\anyrg_T],\,\exists u\in\cN^\twhite_{s}\cup \cN^\tred_{s}: X_u(s)=\ol\gga^\shifted_T(s) \Big) \;\leq\; c_1N^{-c_2}\,,
\end{equation*}
for some $c_1,c_2>0$, uniformly in $\gs\in\gsensrg$ and locally uniformly in $\eps\in(0,1)$. Furthermore, Proposition~\ref{prop:allregimes:UB:npart} implies that, with probability larger than $1-c_1N^{-c_2}$, there are $N$ particles above $\gga^\shifted_T(\cdot)$ at all time $t\leq t^\anyrg_T$. Therefore, with probability larger than $1-2c_1N^{c_2}$, the two processes $(\cX_{t}^{N+})_{t\geq0}$ and $(\cX_t^{\twhite-\tred})_{t\geq0}$ constructed by applying a selection mechanism to a BBM have the exact same trajectory, yielding~\eqref{eq:lem:main:UB:coupling}. Finally,~\eqref{eq:UB:mainineq} follows directly from the coupling between an $N^+$-BBM and an $N$-BBM given in Lemma~\ref{lem:N-N+couping}.
	
	Regarding~\eqref{eq:UB:mainineq:enddist}, the same coupling argument and Proposition~\ref{prop:allregimes:UB:muK} yield that
	\begin{equation*}
		\inf_{s\in[\frac12 t^\anyrg_T,t^\anyrg_T]}
		\bbP_{\mueps} \left(\cX_{s}^{N}\,\prec\, N^{\frac52\eps}\mu_{\eps,s}^{(\ol\gga^\anyrg_T(s))}\right)\,\geq\, 1- 3c_1 N^{-c_2}\,,
	\end{equation*}
	for large $T$, uniformly in $\gs\in\gsensrg$ and locally uniformly in $\eps\in(0,1)$. Moreover, $\cX_{s}^{N}$ contains at most $N$ particles by definition, whereas $\mu_{\eps,s}$ contains strictly more. Hence, recalling the definition of $\mu_{\eps,s}$~\eqref{eq:defmuK} and shifting it upward by $3\eps\gs(s/T)L(T)$, this finally implies~\eqref{eq:UB:mainineq:enddist}.
\end{proof}

\begin{proof}[Proof of Proposition~\ref{prop:main:UB}, super-critical and critical regimes]
	The result follows directly from Lemma~\ref{lem:main:UB:coupling} (more precisely~\eqref{eq:UB:mainineq}) and Lemma~\ref{lem:comparison:m:gga} in the super-critical and critical regimes. Indeed, recall~(\ref{eq:UB:parameters}--\ref{eq:def:ggashifted}) and that $t^\anyrg_T=T$ in those regimes (see~\eqref{eq:defta}). Recalling the notation $\ol\gga^{\anyrg,h,x}_T$, $h>x>0$, one has
	\[
	\ol\gga^{\shifted}_T(T)\,=\, \ol\gga^{\anyrg,1+\eps,1+\frac{2\eps}3}_T(T) + \frac\eps3\gs(1)L(T)\,.
	\]
	Letting $\eps$ arbitrarily small and applying Lemmata~\ref{lem:main:UB:coupling} and~\ref{lem:comparison:m:gga}, this implies Proposition~\ref{prop:main:UB} in both regimes.
\end{proof}

We now turn to the sub-critical regime. In a similar manner to Section~\ref{sec:LB:coreprf}, we split the interval $[0,T]$ into blocks of length $\frac12 t^\subcrit_T$: more precisely, let $K:=\lfloor 2T/t^\subcrit_T\rfloor$, and for $0\leq k\leq K-1$, let $t_k:= \frac k2t^\subcrit_T$, and $t_{K}=T$ (so $t_{K}-t_{K-1}\in[\frac12t^\subcrit_T,t^\subcrit_T]$). 
However, in contrast to Section~\ref{sec:LB:coreprf}, a first moment method is sufficient to prove Proposition~\ref{prop:main:UB} (no second moment estimate is required), and this proof holds throughout the sub-critical regime (so there is no need to handle the super-polynomial case separately).

\begin{proof}[Proof of Proposition~\ref{prop:main:UB}, sub-critical regime]
	We define an auxiliary process $(\hat\cX^N_t)_{t\in[0,T]}$ as follows: it starts from $N\gd_0$ and evolves as the process $\cX^N$ between times $t_k$ and $t_{k+1}$. Then at each time $t_k$, all particles are displaced to the highest among their positions: in other words, a configuration $\mu$ is replaced with $N\gd_{\max(\mu)}$ (notice that $\hat\cX^N$ always contains exactly $N$ particles). By a coupling argument (recall Corollary~\ref{cor:coupling:N}) and an induction, one may construct a coupling such that $\cX^N_T\prec\hat\cX^N_T$ with probability 1. In particular, it is sufficient to prove the proposition with $\max(\hat\cX^N_T)$ instead of $\max(\cX^N_T)$.
	
	For $1\leq k\leq K$, let us define
	\[
	Y_k\,:=\, \max(\hat\cX^N_{t_k}) - \max(\hat\cX^N_{t_{k-1}})\,,
	\]
	so that $\max(\hat\cX^N_T)=Y_1+\cdots+Y_K$. Let us prove that, for $1\leq k\leq K$, one has
	\begin{equation}\label{Y_k_expectation}
		\bbE[Y_k]\,\leq\, \ol\gga^\subcrit_T(t_k)-\ol\gga^\subcrit_T(t_{k-1})+c_1 L(T)\,,
	\end{equation}
	for some $c_1>0$, and we claim that the Proposition~\ref{prop:main:UB} follows. Indeed, this implies
	\[
	\bbE[\max(\hat\cX^N_T)]\,\leq\, \ol\gga^\subcrit_T(T)+c_1 L(T)\times K\,.
	\]
	Recalling that $K\leq 2T/t^\subcrit_T$ and that $t^\subcrit_T\gg L(T)^3$ (see~\eqref{eq:defta:gthapprox}), Markov's inequality yields that
	\[
	\max(\hat\cX^N_T)\,\leq\, \ol\gga^\subcrit_T(T)+ o_{\bbP}\left(\frac T{L(T)^2}\right)\,.
	\]
	Recall~\eqref{def:gga:subcrit} and that $\max(\hat\cX^N_T)$ dominates stochastically $\max(\cX^N_T)$: hence, letting $\eps\to0$ and applying Lemma~\ref{lem:comparison:m:gga}, this finally yields~\eqref{eq:prop:main:UB}.
	
	Let us prove~\eqref{Y_k_expectation}. Notice that the process $\hat\cX^N$ evolves between times $t_k$ and $t_{k+1}$ as the $N$-BBM with variance profile $\gs(t_k/T+\cdot)$, and that $\hat\cX^N_{t_k}=N\gd_{\max(\hat\cX^N_{t_k})}$ for all $0\leq k<K$. Thus it is enough to show that
	\begin{equation}\label{eq:lem:UB:smallL:1}
		\bbE_{N\gd_0}[\max\cX^N_t]\,\leq\, \bbE_{\mueps}[\max\cX^N_t]+\gh^{-1}L(T)\,\leq\, \ol\gga^\subcrit_T(t)\,+\, c_1 L(T)\,,
	\end{equation}
	for some $c_1>0$ uniform in $\gs\in\gsens$ and $t\in[\frac12 t^\subcrit_T,t^\subcrit_T]$. The first inequality in~\eqref{eq:lem:UB:smallL:1} is obtained by writing $\mueps\succ N\gd_{-\gh^{-1}L(T)}$ and translating the $N$-BBM by $\gh^{-1}L(T)$, so we only have to prove the second one. Recall from Proposition~\ref{prop:coupling:BBM}.$(i)$ that there exists a coupling between $\cX^N$ and a BBM $(\cX_t)_{t\in[0,T]}$ without selection, such that $\cX^N_t\subset\cX_t$ a.s. for all $t\in[0,T]$. Thus for $t\in[\frac12 t^\subcrit_T,t^\subcrit_T]$, one has
	\begin{equation}\label{eq:lem:UB:smallL:2}
		\bbE_{\mueps}\big[\max\cX^N_t\big]\,\leq\, \bbE_{\mueps}\big[\max\cX^N_t\,\ind_{\{\max \cX^N_t\leq \ol\gga^\shifted_T(t)\}}\big] \,+\, \bbE_{\mueps}\big[ \max\cX_t\, \ind_{\{\max \cX^N_t> \ol\gga^\shifted_T(t)\}} \big]\,.
	\end{equation}
	The first term from~\eqref{eq:lem:UB:smallL:2} is clearly bounded by $\ol\gga^\shifted_T(t)\leq \ol\gga^\subcrit_T(t)+\gh^{-1}L(T)$, so it remains to bound the second term. Let $C>0$ a large constant: then one has
	\begin{align*}
	&\bbE_{\mueps}\big[ \max\cX_t\, \ind_{\{\max \cX^N_t> \ol\gga^\shifted_T(t)\}} \big] \\
	&\quad \leq\, C (t^\subcrit_T)^2 \,\bbP_{\mueps}\big(\max \cX^N_t> \ol\gga^\shifted_T(t)\big) \,+\, \bbE_{\mueps}\big[ \max\cX_t, \ind_{\{\max \cX_t> C (t^\subcrit_T)^2\}} \big].
	\end{align*}
	Recalling Lemma~\ref{lem:main:UB:coupling}, more precisely~\eqref{eq:UB:mainineq}, the first term in the r.h.s. above is bounded by $C(t^\subcrit_T)^2 \times c_1N^{-c_2}$ uniformly in $t\in[\frac12 t^\subcrit_T,t^\subcrit_T]$. Moreover,~\eqref{eq:defta} implies that $t^\subcrit_T\leq L(T)^4$, so this term vanishes when $T$ becomes large.
	
	Finally, let us prove that $\bbE_{\mueps}[\max\cX_t\, \ind_{\{\max \cX_t> C (t^\subcrit_T)^2\}}]$ also vanishes as $T\to+\infty$, uniformly in $t\in[\frac12 t^\subcrit_T,t^\subcrit_T]$ and locally uniformly in $\eps$. Recalling Proposition~\ref{prop:coupling:BBM}.$(ii)$, it is sufficient to prove this for a BBM starting from the configuration $N^2\gd_0\succ \mueps$. Moreover, we claim that for any $M\geq2$, $A\geq1$, and $(g_i)_{1\leq i\leq M}$ a centered Gaussian vector such that each $g_i$ has variance $\rho^2>0$, one has
	\begin{equation}
		\bE\big[\max(g_i,1\leq i\leq M)\,\ind_{\{\max(g_i,1\leq i\leq M)\geq A\}} \big]\,\leq\, \frac{M\rho}{\sqrt{2\pi}}  \,\big(1+\rho^2/A^{2}\big) \, e^{-A^2/2\rho^2}\,.
	\end{equation}
	This result is standard, and can be proven similarly to Lemma~\ref{lem:gaussmax} by writing for $A\geq0$ and any real random variable $Y$,
	\[
	\bE[Y\,\ind_{\{Y\geq A\}}]\,=\, A \,\bP(Y\geq A) + \int_A^{+\infty}\bP(Y\geq x)\,\dd x\,.
	\]
	For the sake of conciseness we leave the details to the reader. Therefore, conditioning $\cX$ with respect to its branching epochs and letting $Z_t$ denote its population size at time $t\in[\frac12 t^\subcrit_T,t^\subcrit_T]$, we obtain
	\begin{align*}
		\bbE_{\mueps}\big[\max\cX_t\, \ind_{\{\max \cX_t> C (t^\subcrit_T)^2\}}\big]\,&\leq\, 
		\bbE_{N^2\gd_0}\big[\max\cX_t\, \ind_{\{\max \cX_t> C (t^\subcrit_T)^2\}}\big]\\&\leq\,  c_1\times\bbE_{N^2\gd_0}\big[Z_t\big] \times t^\subcrit_T e^{-c_2 (t^\subcrit_T)^2}\,,
	\end{align*}
	for some $c_1,c_2>0$, where the first inequality follows from Proposition~\ref{prop:coupling:BBM}.(ii). Moreover, one has $\bbE_{N^2\gd_0}[Z_t] = N^2 e^{t/2} \leq N^2 e^{t^\subcrit_T}$, so this concludes the proof of~\eqref{eq:lem:UB:smallL:1}.
\end{proof}

\section{Proofs of complementary results}\label{sec:compl}
In this section we complete the proofs of all remaining statements from Section~\ref{sec:intro}, by showing Propositions~\ref{prop:main:gsdecreasing},~\ref{prop:endtimedistribution} and~\ref{prop:main:detbranching} (the latter and~\eqref{eq:NBBM-NCREMidentity} implying Theorem~\ref{thm:crem}). Moreover, in Theorem~\ref{thm:inhselec} below we present an extension of our results to the $N$-BBM with time-inhomogeneous selection. All of these are either obtained through refinements of arguments from the proof of Theorem~\ref{thm:main} presented before; or they are direct applications of Theorem~\ref{thm:main}, coupling propositions or other results from previous sections. Therefore, let us warn the reader that most of the upcoming proofs are not presented in full details. Indeed, the authors believe that understanding the pivotal arguments from the proof of Theorem~\ref{thm:main}, specifically Corollary~\ref{cor:coupling:N} and the main ideas from Sections~\ref{sec:LB}--\ref{sec:UB}, is vital before moving on to other results. Hence, the focus of this section is put on additional, new arguments which are to be combined with those presented before.

\subsection{Super-critical \texorpdfstring{$L(T)$}{L(T)}, decreasing variance (Proposition~\ref{prop:main:gsdecreasing})}\label{sec:compl:gsdecreasing}
Let $\gs\in\cC^2([0,1])$ be strictly decreasing, let $L(T)=\log N(T) \gg T^{1/3}$ (so the regime is super-critical), and consider the initial configuration $\gd_0$. In~\cite{MZ16}, the authors study the speed of a time-inhomogeneous BBM $(\cX_t)_{t\in[0,T]}$, without selection and with strictly decreasing variance. Recall that $\mathrm{a}_1$ denotes the absolute value of the largest zero of $\Ai$. Starting from a single particle in 0 (so $\shft(\gd_0)=0$), they prove in~\cite[Theorem~1.1]{MZ16} that, under those assumptions,
\begin{equation}\label{eq:MZ16}
	\left(\max(\cX_T)-v(1)T+\frac{\mathrm{a}_1}{2^{1/3}}T^{1/3}\int_0^1\gs(u)^{1/3}|\gs'(u)|^{2/3}\,\dd u +\gs(1)\log(T) \right)_{T\geq0} \quad\text{is tight.} 
\end{equation}

Thus, Proposition~\ref{prop:main:gsdecreasing} can be deduced from the couplings from Corollary~\ref{cor:coupling:N} and Proposition~\ref{prop:coupling:BBM}, by comparing the super-critical regime of the $N$-BBM with the critical regime on the one hand, and with a BBM without selection on the other hand. For $\ga\in\R$, recall the definitions of $m^\crit_T=m^\crit_T(\ga)$ and $m^\supdecr_T$ from~\eqref{eq:defmanyrg}. Then, recalling the asymptotic properties of $\Psi(\cdot)$ (see~\eqref{def:Psi}) and that $\gs'(\cdot)$ is negative, one notices that,
\[
\lim_{\ga\to+\infty}\,\limsup_{T\to+\infty}\, T^{-1/3} \left|m^\crit_T(\ga) - m^\supdecr_T \right|\;=\; 0\,.
\]
Fix $\alpha$ and define $M_\ga(T):=\exp(\ga T^{1/3})$: for $T$ sufficiently large, one has $M_\ga(T)\leq N(T)$. By Corollary~\ref{cor:coupling:N} and Proposition~\ref{prop:coupling:BBM}, for $T$ large, one can construct two couplings with an $M_\ga(T)$-BBM and a BBM without selection (all started from $\gd_0$), such that
\[
\cX_s^{M_\ga(T)}\;\prec\; \cX_s^{N(T)}\;\prec\; \cX_s\;,\qquad\forall\,s\in[0,T]\,.
\]
Recall that, for $\gl>0$, Proposition~\ref{prop:main:LB} yields
\[
\lim_{T\to+\infty}\,\bbP_{\gd_{0}}\left(T^{-1/3}\left(\max\big(\cX_T^{M_\ga(T)}\big) - m^\crit_T(\ga)\right)\leq -\gl \right) \;=\; 0\,.
\]
Moreover, the statement~\eqref{eq:MZ16} directly implies,
\begin{equation}\label{eq:gsdecreasing:upperbound}
	\lim_{T\to+\infty}\,\bbP_{\gd_{0}}\left(T^{-1/3}\big(\max(\cX_T) - m^\supdecr_T\big)\geq \gl \right) \;=\; 0\,.
\end{equation}
Letting $\ga$ be very large (depending on $\gl$) and combining these estimates with the coupling above, we conclude that $T^{-1/3}\big(\max(\cX^{N(T)}_T) - m^\supdecr_T \big)$ converges to 0 in $\bbP_{\gd_0}$-probability as $T\to+\infty$.
\qed

\begin{remark}\label{rem:decreas}
	Let us mention that~\cite{MZ16} makes strong assumptions ($\gs$ strictly decreasing and initial configuration $\gd_0$) in order to obtain a sharp result in~\eqref{eq:MZ16}. If the convergence~\eqref{eq:gsdecreasing:upperbound} were to be proven with $\gs$ non-increasing and a generic initial configuration, then Proposition~\ref{prop:main:gsdecreasing} and its proof would immediately extend to this more general setting. This is discussed again in Remark~\ref{rem:decreas:detbranching} below.
\end{remark}

\subsection{Final-time distribution for the critical and super-critical regimes (Proposition~\ref{prop:endtimedistribution})}\label{sec:compl:enddist}
Let $\anyrg\in\{\crit,\supercrit\}$ (in particular $t^\anyrg_T=T$) and $\gh>0$ throughout this section. We first prove~\eqref{eq:prop:endtimedist:1}, then we use it to deduce~\eqref{eq:prop:endtimedist:2}.

Let $\gl>0$, and recall from Section~\ref{sec:coupling:mainprops} how a coupling argument yields Theorem~\ref{thm:main} subject to Propositions~\ref{prop:main:LB},~\ref{prop:main:UB}. 
The same coupling method can be used to deduce~\eqref{eq:prop:endtimedist:1} from the following lemma. 
\begin{lemma}\label{lem:enddist}
	Let $\anyrg\in\{\crit,\supercrit\}$, and $\gl, \gh>0$. Then, one has
	\begin{align}\label{eq:lem:enddist:UB}
		&\lim_{\eps\to0}\,\limsup_{T\to+\infty} \,\sup_{\gs\in\gsensrg} \,\bbP_{\mu_\eps}\bigg(\exists\,y\in[0,1]: \quad \\
		\notag &\qquad\qquad\qquad \cX_T^{N(T)}\big(\big[\shft(\mu_T)+m^\anyrg_T-y\gs(1)L(T),+\infty\big)\big) \geq N(T)^{y+\gl} \bigg) = 0,
	\end{align}
	and, for any fixed $r\in(0,1)$,
	\begin{align}\label{eq:lem:enddist:LB}
		&\lim_{T\to+\infty} \sup_{\gk\in[0,1]} \sup_{\gs\in\gsensrg} \bbP_{N^{\gk}\gd_{-\gk\gs(0)L(T)}}\bigg(\exists y\in[r,1]: \quad \\
		\notag &\qquad\qquad\qquad \cX_T^{N(T)}\big(\big[\shft(\mu_T)+m^\anyrg_T-y\gs(1)L(T),+\infty\big)\big) \leq N(T)^{y-\gl} \bigg) = 0.
	\end{align}
\end{lemma}

\begin{proof}[Proof of~\eqref{eq:prop:endtimedist:1} subject to Lemma~\ref{lem:enddist}]
	Using arguments very similar to the ones from the proof of Theorem~\ref{thm:main} in Section~\ref{sec:coupling:mainprops}, we obtain from Lemma~\ref{lem:enddist} that for every fixed $r\in (0,1)$, as $T\to\infty$,
	\begin{equation}
		\label{eq:prop:endtimedist:1:proof1}
		\sup_{y\in[r,1]}\left|\,\frac{\log\cX_T^{N(T)}([\shft(\mu_T)+m_T^\anyrg - y\gs(1)L(T),+\infty)) }{L(T)}\,-\,y\,\right|\;\longrightarrow\;0\,,
	\end{equation}
	in $\bbP_{\mu_T}$-probability.
	In particular, we have $\cX_T^{N(T)}([\shft(\mu_T)+ m_T^\anyrg - r\gs(1)L(T),+\infty)) \ge 1$ with high probability, so that we can replace $\log$ by $\log_+$. It remains to consider $y\in [0,r]$. We have
	\begin{align*}
		&\sup_{y\in[0,r]}\left|\,\frac{\log_+\cX_T^{N(T)}([\shft(\mu_T)+m_T^\anyrg - y\gs(1)L(T),+\infty)) }{L(T)}\,-\,y\,\right|\\
		&\qquad\le\, \frac{\log_+\cX_T^{N(T)}([ \shft(\mu_T)+m_T^\anyrg - r\gs(1)L(T),+\infty))}{L(T)} + r,
	\end{align*}
	so that \eqref{eq:prop:endtimedist:1:proof1} implies that the r.h.s. converges to $2r$ in $\bbP_{\mu_T}$-probability as $T\to+\infty$. Letting $r\to 0$, we obtain the result.
\end{proof}

\begin{proof}[Proof of Lemma~\ref{lem:enddist}]
We can assume $\shft(\mu_T)=0$ without loss of generality (recall~\eqref{eq:shft:invariance}).
	We first prove~\eqref{eq:lem:enddist:UB}, i.e. the upper bound on the final-time distribution of the process. Let $\eps\in(0,1)$, $\gl>0$ and recall the definition of $\mu_{\eps,s}$, $s\in[0,T]$ from~\eqref{eq:defmuK}, and of $m^\anyrg_T$ from~\eqref{eq:defmanyrg}. By Lemma~\ref{lem:comparison:m:gga}, one has for $\eps$ sufficiently small and $T$ large that,
	\begin{equation}\label{eq:prfenddist:2}
		\frac1{L(T)} \big| \ol\gga^{\anyrg,1+\eps,1+\frac23\eps}_T(T) - m^\anyrg_T \big|\,\leq\, \frac\gl2\,,
	\end{equation}
	where we wrote the parameters of the barrier $\ol\gga^\anyrg_T$ explicitly. 
	Then, recall that by equation~\eqref{eq:UB:mainineq:enddist} from Lemma~\ref{lem:main:UB:coupling}, we have for $T$ sufficiently large,
	\begin{equation*}
		\inf_{\gs\in\gsensrg}
		\bbP_{\mueps} \left(\cX_{T}^{N}\,\prec\, \mu_{\eps,T}^{(\ol\gga^\anyrg_T(T)+3\eps\gs(1)L(T))}\right)\,\geq\, 1- c_1 N^{-c_2}\,,
	\end{equation*}
	where $c_1$, $c_2$ are constants depending locally uniformly on $\eps$; and $\ol\gga^\anyrg_T(\cdot)$ is defined in~\eqref{def:gga:supercrit} and~\eqref{def:gga:crit} (with parameters $h=1+\eps$, $x=1+\frac23\eps$). In particular, defining
	\[
	I_{T}^{\eps,\gl,\gs}\,\coloneqq\, \left\{ y\in[0,1],\, \mu_{\eps,T}^{(\ol\gga^\anyrg_T(T)+3\eps\gs(1)L(T))}\big(\big[m^\anyrg_T-y\gs(1)L(T),+\infty\big)\big)\geq N(T)^{y+\gl} \right\}\,,
	\]
	we deduce that,
	\begin{align}\label{eq:prfenddist:1}
		&\sup_{\gs\in\gsensrg}\bbP_{\mueps} \left(\exists y\in[0,1]: \cX_{T}^{N}\big(\big[m^\anyrg_T-y\gs(1)L(T),+\infty\big)\big)\geq N(T)^{y+\gl}\right)\\
		\notag &\quad \leq\; \ind_{\{\exists \gs\in\gsensrg\,;\,I_{T}^{\eps,\gl,\gs}\neq\emptyset\}} + c_1 N^{-c_2}\,,
	\end{align}
	Finally, for $y\in[0,1]$ and $\gs\in\gsensrg$, one has
	\begin{align*}
		&\mu_{\eps,s}^{(\ol\gga^\anyrg_T(s)+3\eps\gs(s/T)L(T))}\big( \big[m^\anyrg_T-y\gs(1)L(T),+\infty\big) \big)\\
		&\quad=\, \sum_{k=0}^{\eps^{-1}} N^{k\eps+\frac\eps2} \,\ind_{\{ \ol\gga^\anyrg_T(s)+3\eps\gs(s/T)L(T) -k\eps\gs(1)L(T) \,\geq\, m^\anyrg_T - y\gs(1)L(T) \}}\\
		&\quad\leq\, \sum_{k=0}^{y\eps^{-1}+3\gh^{-2}+\eps^{-1}\frac\gl2} N^{k\eps+\frac\eps2}\;\leq\; (\eps^{-1}(1+\gl/2)+3\gh^{-2}) N^{y+4\gh^{-2}\eps+\frac\gl2}\,,
	\end{align*}
	where we used~\eqref{eq:prfenddist:2}. Assuming $\eps$ is sufficiently small above (depending on $\gl$ and $\gh$), this implies that $I_{T}^{\eps,\gl,\gs}$ is empty for $T$ large enough uniformly in $\gs\in\gsensrg$; 
	this completes the proof of~\eqref{eq:lem:enddist:UB}. 
	\medskip
	
	We now turn to~\eqref{eq:lem:enddist:LB}. Quite similarly to Proposition~\ref{prop:main:LB}, let us first prove that the convergence holds locally uniformly in $\gk\in(0,1)$, then extend it to $\gk\in[0,1]$ via coupling arguments (recall Lemma~\ref{lem:main:LB} from Section~\ref{sec:LB:coreprf}). Let $\eps\in(0,1)$, $h=1-\eps$, $x=(1-\eps)(1-\gk)$, and recall Proposition~\ref{prop:allregimes:LB:endpoint}. Replicating the coupling arguments from Section~\ref{sec:LB:coreprf}, more precisely~(\ref{eq:LBmainineq:coupl1}--\ref{eq:LBmainineq:coupl2}), one obtains for $1\leq j\leq \eps^{-1}$ and $T$ sufficiently large, 
	\begin{equation*}
		\bbP_{N^{\gk}\gd_{0}}\Big( \cX^N_T\big(\big[\ol\gga^\anyrg_T(T)-j\eps\gs(1)L(T),+\infty\big)\big)\,\geq\, N^{j\eps}  \Big)\,\geq\, 1-c_1N^{-c_2}\,,
	\end{equation*}
	where $c_1,c_2>0$ are constants uniform in $\gs\in\gsensrg$, $1\leq j\leq \eps^{-1}$, and locally uniform in $\eps,\gk\in(0,1)$ (we do not reproduce the details). Writing a union bound, this yields
	\begin{equation}\label{eq:prfenddist:3}
		\bbP_{N^{\gk}\gd_{0}}\Big(\forall 1\leq j\leq\eps^{-1},\, \cX^N_T\big(\big[\ol\gga^\anyrg_T(T)-j\eps\gs(1)L(T),+\infty\big)\big)\,\geq\, N^{j\eps}  \Big)\,\geq\, 1-\eps^{-1}c_1N^{-c_2}\,.
	\end{equation}
	Let $\eps'>0$: recalling~\eqref{eq:defmanyrg} and Lemma~\ref{lem:comparison:m:gga}, one has for $\eps$ sufficiently small and $T$ large that,
	\begin{equation}\label{eq:prfenddist:4}
		\frac1{L(T)} \big| \ol\gga^{\anyrg,1-\eps,(1-\eps)(1-\gk)}_T(T) - \gk\gs(0)L(T) - m^\anyrg_T \big|\,\leq\, \eps'\,,
	\end{equation}
	where we used that $\ol\gga^{\anyrg,1-\eps,(1-\eps)(1-\gk)}_T(T) - \gk\gs(0)L(T) = \ol\gga^{\anyrg,1-\eps,1-\eps(1-\gk)}_T(T)$. 
	Moreover, for $y\in[\eps+\eps'\gh,1]$, there exists $1\leq j_y\leq \eps^{-1}$ such that $(y-\eps'\gh^{-1})\eps^{-1}\in[j_y, j_y+1)$; thus we deduce that,
	\begin{align*}
		&\Big\{\exists y\in[\eps+\eps'\gh^{-1},1]: \cX^N_T\big(\big[m^\anyrg_T-y\gs(1)L(T),+\infty\big)\big)< N^{y-\gl}\Big\}\\
		&\quad\subset \Big\{\exists 1\leq j\leq \eps^{-1}:\cX^N_T\big(\big[\ol\gga^{\anyrg,1-\eps,(1-\eps)(1-\gk)}_T(T)-\gk\gs(0)L(T)-j\eps\gs(1)L(T),+\infty\big)\big)\\
		&\qquad\qquad\qquad\qquad\qquad\qquad\qquad\qquad\qquad\qquad\qquad\qquad\qquad\quad< N^{(j+1)\eps+\eps'\gh^{-1}-\gl}\Big\}.
	\end{align*}
	Assume that $\eps',\eps$ are sufficiently small so that $\eps+\eps'\gh^{-1}\leq \min(r,\gl)$. Then, shifting the initial distribution in~\eqref{eq:prfenddist:3} by $-\gk\gs(0)L(T)$, this implies
	\begin{equation}\label{eq:prfenddist:5}
		\bbP_{N^\gk\gd_{-\gk\gs(0)L(T)}}\Big( \exists \,y\in[r,1]: \cX^N_T\big(\big[m^\anyrg_T-y\gs(1)L(T),+\infty\big)\big)< N^{y-\gl} \Big)\,\leq\, \eps^{-1}c_1N^{-c_2}\,,
	\end{equation}
	uniformly in $\gs\in\gsensrg$ and locally uniformly in $\gk\in(0,1)$, which is the expected result.
	
	To finish the proof of Lemma~\ref{lem:enddist}, it remains to extend~\eqref{eq:prfenddist:5} to $\gk$ close to 0 or 1, which is achieved very similarly to the proof of Proposition~\ref{prop:main:LB} subject to Lemma~\ref{lem:main:LB} in Section~\ref{sec:LB:coreprf}. For the case $\gk$ large, assume without loss of generality that $0<r<\eps<\gl$, and notice that for any $\gk\in[1-\eps\gh^{-2},1]$, one has $N^{1-\eps\gh^{-2}}\gd_{-\gs(0)L(T)} \prec N^{\gk}\gd_{-\gk\gs(0)L(T)}$. Hence, Corollary~\ref{cor:coupling:N} yields,
	\begin{align*}
		&\bbP_{N^\gk\gd_{-\gk\gs(0)L(T)}}\Big( \exists y\in[r,1]:\cX^N_T\big(\big[m^\anyrg_T-y\gs(1)L(T),+\infty\big)\big)\,<\, N^{y-\gl}\Big)\\
		&\leq \bbP_{N^{(1-\eps\gh^{-2})}\gd_{-(1-\eps\gh^{-2})\gs(0)L(T)}}\Big(\exists y\in[r,1]: \\
		&\qquad\qquad\qquad\qquad\qquad\quad \cX^N_T\big(\big[m^\anyrg_T-y\gs(1)L(T)+\eps\gh^{-2}\gs(0)L(T),+\infty\big)\big)\,<\, N^{y-\gl}\Big)\\
		&\leq \bbP_{N^{(1-\eps\gh^{-2})}\gd_{-(1-\eps\gh^{-2})\gs(0)L(T)}}\Big(\exists y\in[r-\eps,1]: \\
		&\qquad\qquad\qquad\qquad\qquad\qquad\qquad\qquad\; \cX^N_T\big(\big[m^\anyrg_T-y\gs(1)L(T),+\infty\big)\big)\,<\, N^{y-(\gl-\eps)}\Big),
	\end{align*}
	where we also shifted the process upward by $\eps\gh^{-2}\gs(0)L(T)$ in the first inequality. Applying~\eqref{eq:prfenddist:5} to the latter (with $\gk=1-\eps\gh^{-2}$), this proves that the convergence also holds uniformly in $\gk$ close to 1.
	
	It only remains to treat the case $\gk$ small. Recall Lemma~\ref{lem:LB:standardresults}, which implies that, for an $N$-BBM started from $\gd_0$, one has $\cX^N_{\eps L(T)}\succ N^{\eps/4}\gd_{-2\eps\gh^{-1}L(T)}$ with $\bbP_{\gd_0}$-probability close to 1. Letting $\gk\in[0,\eps]$, writing $\gd_{-\eps\gs(0)L(T)}\prec N^\gk\gd_{-\gk\gs(0)L(T)}$ and applying the Markov property at time $\eps L(T)$, one can again extend the convergence uniformly to $\gk\in[0,\eps]$ through Corollary~\ref{cor:coupling:N}: since this is a carbon copy of arguments presented in Section~\ref{sec:LB:coreprf}, we leave the details to the reader.
\end{proof}

\begin{remark}\label{rem:enddist:subcrit}
	Let us mention that the proof above can be applied to the sub-critical regime, yielding an estimate on the distribution of the process at any time $t\in[\gth(T)^{-1}L(T),t^\subcrit_T]$. Unfortunately, the error term $o(T/L(T)^2)$ in our maximal displacement result renders this obsolete at time $T$. If one manages to compute sharper estimates on the maximal displacement of the $N$-BBM with sub-critical population (with an error at most $o(L(T))$), the arguments above can be adapted straightforwardly. However, it seems very unlikely that such a sharp limit estimate would be non-random: instead we expect it to be obtained through martingale methods or similar arguments, with a non-trivial dependency on the realization of the process near time $t=0$ and on the ``peaking'' events throughout the trajectory (i.e. when a particle rises and reproduces significantly: see~\cite{Mai16} for a study of those in the case of an $N$-BBM with constant $N$, time-homogeneous variance and close to the equilibrium measure).
\end{remark}

\subsubsection*{Proof of~\eqref{eq:prop:endtimedist:2}}
We can assume $\shft(\mu_T)=0$ without loss of generality (recall~\eqref{eq:shft:invariance}). 
We start with the lower bound on the diameter. Let $\gd>0$, and notice that Theorem~\ref{thm:main} implies for $\anyrg\in\{\crit,\supercrit\}$,
\begin{align*}
	&\bbP_{\mu_T}\big(\max(\cX_T^{N(T)}) - \min(\cX_T^{N(T)}) \leq (1-2\gd)\gs(1)L(T)\big)\\
	&\quad\leq\, \bbP_{\mu_T}\big(\min(\cX_T^{N(T)}) \geq m^\anyrg_T - (1-\gd)\gs(1)L(T)\big) + o(1)\\
	&\quad=\, \bbP_{\mu_T}\big(\cX_T^{N(T)} \big([m^\anyrg_T - (1-\gd)\gs(1)L(T),+\infty)\big) \geq N\big) + o(1)\,,
\end{align*}
and by~\eqref{eq:prop:endtimedist:1}, the latter goes to 0 as $T$ goes to $+\infty$, for any $\gd>0$.

Regarding the upper bound on the diameter, it suffices to prove that
\begin{equation}\label{eq:proofdiam:UB:1}
	\bbP_{\mu_T}\Big(\min(\cX_T^{N(T)}) \leq m^\anyrg_T - (1+2\gd)\gs(1)L(T)\Big) = o(1)
\end{equation}
as $T\to+\infty$ for any $\gd>0$, and the result follows similarly from Theorem~\ref{thm:main}. Consider $\cX_{T-\gd^3 L(T)}^{N(T)}$ the particle configuration of the $N$-BBM at time $T-\gd^3 L(T)$, and define
\begin{equation}\label{eq:proofdiam:cA}
	\cA\,:=\,\cX_{T-\gd^3 L(T)}^{N(T)}\cap \big[m^\anyrg_{T}-(1+\gd)\gs(1)L(T),+\infty \big) \,\prec\, \cX_{T-\gd^3 L(T)}^{N(T)}\,.
\end{equation}
Recall~\eqref{def:quantile:M}, and notice that for any $\mu\in\meas_N$ one has $q_N(\mu)=\min(\mu)$ if $\mu(\R)=N$, and $q_N(\mu)=-\infty$ else. Thus, applying the Markov property at time $T-\gd^3L(T)$, one obtains that, 
\begin{align*}
	&\bbP_{\mu_T}\Big(\min(\cX_T^{N(T)}) \leq m^\anyrg_T - (1+2\gd)\gs(1)L(T)\,\Big|\, \cX_{T-\gd^3 L(T)}^{N(T)} \Big)\\
	&\quad \leq\,\bbP_{\mu_T}\Big(q_N(\cX_T^{N(T)}) \leq m^\anyrg_T - (1+2\gd)\gs(1)L(T)\,\Big|\, \cX_{T-\gd^3 L(T)}^{N(T)} \Big)\\
	&\quad =\,\bbP_{(T-\gd^3 L(T),\cX_{T-\gd^3 L(T)}^{N(T)})} \Big(q_N(\cX_{T}^{N(T)}) \leq m^\anyrg_T - (1+2\gd)\gs(1)L(T)\Big)\\
	&\quad \leq \,\bbP_{(T-\gd^3L(T),\cA)} \Big(q_N(\cX_{T}^{N(T)}) \leq m^\anyrg_T - (1+2\gd)\gs(1)L(T)\Big)\,,
\end{align*}
where the last inequality follows from Corollary~\ref{corol:quantile:NBBM}. Moreover, Lemma~\ref{lem:enddist} (more specifically~\eqref{eq:lem:enddist:LB}) implies that, with large probability, $\cA$ contains at least $N^{1-\gd^4}$ particles. Using standard arguments on birth processes (such as Proposition~\ref{prop:LB:standardresults:quote}), one observes that, with large probability, the $N$-BBM started from $(T-\gd^3L(T),\cA)$ contains at time $T$ exactly $N(T)$ particles: in particular, $q_N(\cX_T^{N(T)})=\min(\cX_T^{N(T)})$ with large $\bbP_{(T-\gd^3L(T),\cA)}$-probability, and
\begin{align*}
&\bbP_{(T-\gd^3L(T),\cA)} \Big(q_N(\cX_{T}^{N(T)}) \leq m^\anyrg_T - (1+2\gd)\gs(1)L(T)\Big) \\
&\quad=\, \bbP_{(T-\gd^3L(T),\cA)} \Big(\min(\cX_{T}^{N(T)}) \leq m^\anyrg_T - (1+2\gd)\gs(1)L(T)\Big) +o(1)\,.
\end{align*}
 
Recall from Proposition~\ref{prop:coupling:BBM}.$(i)$ that one can couple $\cX^{N(T)}$ with a BBM started from $(T-\gd^3L(T),\cA)$ such that $\cX^{N(T)}_T\subset \cX_T$. Thus,
\begin{align*}
	&\bbP_{(T-\gd^3L(T),\cA)} \Big(\min(\cX_{T}^{N(T)}) \leq m^\anyrg_T - (1+2\gd)\gs(1)L(T)\Big)\\
	&\quad\leq\, \bbP_{(T-\gd^3L(T),\cA)} \Big(\min(\cX_{T}) \leq m^\anyrg_T - (1+2\gd)\gs(1)L(T)\Big)\\
	&\quad=\, \bbP_{(T-\gd^3L(T),-\cA)} \Big(\max(\cX_{T}) \geq -m^\anyrg_T + (1+2\gd)\gs(1)L(T)\Big)\,,
\end{align*}
where we used the symmetry of the BBM. Since $(-\cA)\prec N\gd_{-m^\anyrg_{T}+(1+\gd)\gs(1)L(T)}$ by definition, one deduces from Proposition~\ref{prop:coupling:BBM} and a shift that
\begin{align*}
&\bbP_{(T-\gd^3L(T),-\cA)} \Big(\max(\cX_{T}) \geq -m^\anyrg_T + (1+2\gd)\gs(1)L(T)\Big)\\
&\quad \leq\, \bbP_{(T-\gd^3L(T),N\gd_0)} \Big(\max(\cX_{T}) \geq \gd\gs(1)L(T)\Big)\,.
\end{align*}
Then, we conclude the proof directly with a union bound: 
indeed, one deduce again from standard arguments on birth processes that, with large $\bbP_{(T-\gd^3L(T),N\gd_0)}$-probability, $\cX$ contains at most $N^{1+\gd^2}$ particles; so letting $Y$ be a centered Gaussian variable with variance $T\int_{1-\gd^3L(T)/T}^1 \gs^2(s)\dd s = (1+o(1))\gs^2(1)\gd^3L(T)$, one deduces from standard Gaussian estimates that,
\begin{align*}
&\bbP_{(T-\gd^3L(T),N\gd_0)} \Big(\max(\cX_{T}) \leq \gd\gs(1)L(T)\Big) \\
&\quad\leq\, o(1)+N^{1+\gd^2}\bP(Y\ge \gd\gs(1)L(T))\\
&\quad\leq\, o(1) + e^{(1+\gd^2) L(T)}\times c_1\sqrt{\gd/L(T)} \,e^{-c_2\gd^{-1}L(T)}\,,
\end{align*}
where $c_1,c_2>0$ are universal constants. Provided that $\gd>0$ is small enough, this concludes the proof of~\eqref{eq:prop:endtimedist:2}. \qed

\subsection{Adaptation of the proof to deterministic branching times (Proposition~\ref{prop:main:detbranching})}\label{sec:compl:detbranching}
In this section we do not present a full proof of Proposition~\ref{prop:main:detbranching}: instead we detail how all our previous results and proofs can be adapted to the time-inhomogeneous BBMdb. Throughout this section, we let $(\tilde\cX_t)_{t\in[0,T]}$ (resp. $(\tilde\cX^N_t)_{t\in[0,T]}$) denote the particle configurations of the BBMdb (resp. $N$-BBMdb). We first focus on adapting the proof of Theorem~\ref{thm:main}, that is Sections~\ref{sec:coupling} through~\ref{sec:UB}.
\medskip

\noindent\emph{Section~\ref{sec:coupling}.} 
We start with the coupling statements on the BBMdb and $N$-BBMdb. 
Regarding Lemma~\ref{lem:N-N+couping}, the construction of the coupling in~\cite[Section~2.3]{Mai16} is done by induction on the sequence of (random) epochs $(t_n)_{n\ge1}$ upon which either a particle branches, or a particle dies. Having the epochs be deterministic does not alter the construction, except for the fact that multiple particles are branching simultaneously. However, the construction can be adapted to that case very straightforwardly: as a matter of fact it has already be done in~\cite[Lemma~1]{BG10} for one step of the (homogeneous) branching random walk with selection. Regarding Proposition~\ref{prop:coupling:BBM}, the definition of the $N$-BBM as a BBM with some selection mechanism $\cL$ is unchanged, so the first statement still holds; and in the proof of the second statement one only needs to replace the BBM's with BBMdb's. 

Finally, the statements of Propositions~\ref{prop:main:LB} and~\ref{prop:main:UB} can be rephrased straightforwardly in temrs of the $N$-BBM: then the same arguments as in Section~\ref{sec:coupling:mainprops} yield that Theorem~\ref{thm:main} also holds for the $N$-BBMdb subject to these propositions.\medskip

\noindent\emph{Sections~\ref{sec:prelim} and~\ref{sec:moments}.} All definitions and notation from Section~\ref{sec:prelim:notation} are unchanged. The main differences are in Section~\ref{sec:prelim:toolbox}, more precisely in the statements of the Many-to-one and Many-to-two lemmas. On the one hand, recall that $|\cN_t|$ denotes the population size of the BBM at time $t\in[0,T]$, and let $|\tilde\cN_t|$ denote that of the BBMdb. Then the Many-to-one lemma still holds for the BBMdb, subject to replacing $\bbE_{\gd_0}[|\cN_t|]=e^{t/2}$ with,
\[
\bbE_{\gd_0}[|\tilde\cN_t|]\,=\,\big(\bbE\xi\big)^{\lfloor t/2\log\bbE\xi \rfloor}\,,
\]
in Lemma~\ref{lem:MtOGirs} (the equality above is proven by induction). Notice that this is of order $e^{t/2+O(1)}$, so that change does not affect the resulting first moment estimates up to constant factors. On the other hand, we provide a discrete-time version of the Many-to-two lemma. It is obtained by standard arguments involving the decomposition of pairs of individuals according to their most recent common ancestor, see e.g. \cite[Appendix II]{Saw76} or \cite[Lemma~3.6]{Mal15}. Let $\tilde A_{T,z}^\anyrg(\cdot)$, $\tilde G(\cdot)$ be defined similarly to $A_{T,z}^\anyrg(\cdot)$, $G(\cdot)$ respectively in~\eqref{eq:defA:allregimes} and~\eqref{eq:defG:allregimes}, where we replace the BBM with a BBMdb. 

\begin{lemma}[Many-to-two lemma, deterministic branching]\label{lem:MtT:detbranching}
	Let $\anyrg\in\{\supercrit,\subcrit,\crit\}$. 
	Let $t\leq T$, $\gga^\anyrg_T, \ol\gga^\anyrg_T \in\cC^1([0,T])$ which satisfy~\eqref{eq:defxh}, and $z\in[0,h)$. Then, one has,
	\begin{align}
		&\bbE_{\gd_0}\big[|\tilde A_{T,z}^\anyrg(t)|^2\big] -  \bbE_{\gd_0}\big[|\tilde A_{T,z}^\anyrg(t)|\big]  \\
		\notag&= \bbE[\xi(\xi-1)] \! \sum_{k=0}^{\lfloor t/(2\log \bbE\xi)\rfloor-1} \!\! \int_0^{h} \! \tilde G^\anyrg(x,y,0,(2\log\bbE\xi)k) \bigg(\!\int_{z}^{h} \tilde G^\anyrg(y,w,(2\log\bbE\xi)k,t) \dd w\!\bigg)^{\!\!2} \!\dd y. 
	\end{align}
\end{lemma}

Replacing Lemma~\ref{lem:MtT} with Lemma~\ref{lem:MtT:detbranching} in the proofs of Propositions~\ref{prop:2ndmom:supercrit},~\ref{prop:2ndmom:subcrit} and~\ref{prop:2ndmom:crit}, and using the adaptation of Lemma~\ref{lem:MtOGirs} discussed above in all the other propositions from Section~\ref{sec:moments}, we obtain the same moment estimates on the BBMdb between barriers as we do for the BBM, in all regimes.\medskip

\noindent\emph{Sections~\ref{sec:LB} and~\ref{sec:UB}.} All statements in these sections immediately hold for the $N$-BBMdb: indeed, the proofs in these sections are entirely based on the moment estimates from Section~\ref{sec:moments}, as well as some technical results on the barriers (such as Lemma~\ref{lem:ggachangeparam}) which are not affected by the choice of branching mechanism. Therefore, this finishes the adaptation of Propositions~\ref{prop:main:LB} and~\ref{prop:main:UB} to the $N(T)$-BBMdb, proving that the results of Theorem~\ref{thm:main} also hold for that process.\medskip

\noindent\emph{Super-critical regime, decreasing variance.} Recall the proof of Proposition~\ref{prop:main:gsdecreasing} from Section~\ref{sec:compl:gsdecreasing}. To the authors knowledge, there is currently no equivalent to~\eqref{eq:MZ16} for the BBMdb (or BRW) in the literature, however~\cite{Mal15} has proven that~\eqref{eq:gsdecreasing:upperbound} holds for a large class of discrete-time, time-inhomogeneous branching random walks (BRW), among which Gaussian BRW, see~\cite[Theorem 1.3]{Mal15}. Therefore, replacing Corollary~\ref{cor:coupling:N} and Proposition~\ref{prop:coupling:BBM} with the coupling arguments discussed above for the $N$-BBMdb, and applying~\cite[Theorem 1.3]{Mal15} to the BBMdb at its final time, the adaptation of Proposition~\ref{prop:main:gsdecreasing} to deterministic branching times is straightforward (we do not replicate the proof).

\begin{remark}\label{rem:decreas:detbranching}
	Let us mention that~\cite[Theorem 1.3]{Mal15} also holds for $\gs$ non-increasing, hence so does~\eqref{eq:mainresult:manyrg} for the superscript $\supdecr$ in the $N$-BBMdb. This provides a slightly more complete statement for the $N$-BBMdb and $N$-CREM than Proposition~\ref{prop:main:gsdecreasing} (where we assumed $\gs$ strictly decreasing), however one still requires an initial configuration $\gd_0$ when quoting~\cite{Mal15}.
\end{remark}
\smallskip

\noindent\emph{Final-time distribution} Regarding~\eqref{eq:prop:endtimedist:1}, it is sufficient to prove that Lemma~\ref{lem:enddist} also holds for the $N$-BBMdb $\tilde\cX^N$. A thorough inspection of the proof shows that all arguments therein can be adjusted to deterministic branching times, with no more work than what has already been presented for the adaptation Theorem~\ref{thm:main} above. In order not to overburden this paper, we do not repeat those arguments here. In the proof of~\eqref{eq:prop:endtimedist:2}, there is one occurrence of the Many-to-one lemma which has to be replaced with the $N$-BBMdb version as above, and the rest of the proof is unchanged. This fully concludes the proof of Proposition~\ref{prop:main:detbranching}.
\qed

\subsection{\texorpdfstring{$N$}{N}-BBM with time-inhomogeneous selection}\label{sec:compl:inhselec}

	It is natural to extend the $N$-BBM by allowing the selection mechanism to be time-inhomogeneous as well. That is, starting from a (time-inhomogeneous) BBM over the time horizon $T>0$, at any time $s\in[0,T]$, keep only the particles at the $N(s,T)$ highest of the whole population, for some function $N(\cdot,T)$ fixed beforehand. Let us call this model the $N(\cdot,T)$-BBM. The results presented in Section~\ref{sec:intro}---namely Theorem~\ref{thm:main} and Proposition~\ref{prop:endtimedistribution}---can be extended to a class of $N(\cdot,T)$-BBM in which the selection does not vary too much: more precisely, the selection remains in the ``same regime'' throughout the time interval $[0,T]$.
	
	Consider some growing function $\hat L(T)\to+\infty$ as $T\to+\infty$, and a positive function $\ell\in\cC^1([0,1])$, note that this implies that $\ell$ is bounded away from $0$ and $+\infty$. Define for $0\leq s\leq T$,
	\begin{equation}\label{eq:def:Linhselec}
		L(s,T)\,=\, \log N(s,T)\,:=\, \ell(s/T) \,\hat L(T)\,,
	\end{equation}
	the log-population size at time $s\in[0,T]$ of the $N(\cdot,T)$-BBM. We say that,
	
	--- $\hat L(T)$ is \emph{sub-critical} if $1\ll \hat L(T) \ll T^{1/3}$,
	
	--- $\hat L(T)$ is \emph{critical} if $\hat L(T)\sim T^{1/3}$,
	
	--- $\hat L(T)$ is \emph{super-critical} if $T^{1/3}\ll \hat L(T)\ll T$.
	
	\noindent Let us adapt the notation from~(\ref{eq:defbanyrg}--\ref{eq:defmanyrg}) by defining for $T\geq0$,
	\begin{equation}\label{eq:defbanyrg:inhselec}
		\hat b^\supercrit_T\,:=\, \hat L(T)\,,\qquad
		\hat b^\crit_T
		\,:=\, T^{1/3}\,,\qquad\text{and}\qquad
		\hat b^\subcrit_T\,:=\, \frac{T}{\hat L(T)^2}\,,
	\end{equation}
	as well as,
	\begin{equation}\label{eq:defmanyrg:inhselec}\begin{aligned}
			\hat m^\supercrit_T\,&:=\, v(1)T + \left[\int_0^1\ell(u)(\gs')^+(u)\,\dd u\right]\hat L(T)\,,\\
			\hat m^\crit_T\,&:=\,v(1)T + \left[\int_0^1\frac{\gs(u)}{\ell(u)^2}\Psi\Big(-\ell(u)^3\frac{\gs'(u)}{\gs(u)}\Big)\dd u\right] T^{1/3}  \,,\\
			\text{and}\qquad \hat m^\subcrit_T\,&:=\, v(1)T  - \frac{\pi^2T}{2\,\hat L(T)^2}\int_0^1\frac{\gs(u)}{\ell(u)^2}\,\dd u \,.
		\end{aligned}
	\end{equation}
	Recall also the definitions of $b^\supdecr_T$, $m^\supdecr_T$ from~(\ref{eq:defbanyrg},~\ref{eq:defmanyrg}). Then we have the following.
	\begin{theorem}\label{thm:inhselec}
		Let $\gs\in\cC^2([0,1])$. Let $\hat L(T)\to+\infty$ as $T\to+\infty$, let $\ell\in\cC^1([0,1])$ and define $N(\cdot,T)$ as in~\eqref{eq:def:Linhselec}. Denote with $\cX_T^{N(\cdot,T)}$ the empirical measure on $\R$ at time $T>0$ of a $N(\cdot,T)$-BBM with infinitesimal variance $\gs^2(\cdot/T)$, started from some initial configuration $\mu_T\in\meas_{N(0,T)}$, $T\geq0$. Let $\max(\cX_T^{N(\cdot,T)})$ denote the maximal displacement of the process at time $T$. Let $\anyrg\in\{\subcrit, \crit, \supercrit\}$ denote the regime satisfied by $\hat L(T)$. Then as $T\to+\infty$, one has
		\begin{equation}\label{eq:thm:inhselec:max}
			\max(\cX_T^{N(\cdot,T)})\;=\; \shft(\mu_T) + \hat m^\anyrg_T + o_{\bbP}\big(b^\anyrg_T\big).
		\end{equation}
		Moreover, if $\gs$ is strictly decreasing, $\mu_T\equiv \gd_0$ and $T^{1/3}\ll \hat L(T)\leq+\infty$, then 
		\begin{equation}\label{eq:thm:inhselec:max:decreas}
			\max(\cX_T^{N(\cdot,T)}) \;=\; m^\supdecr_T + o_{\bbP}\big(b^\supdecr_T\big).
		\end{equation}
		Finally if the regime satisfies $\anyrg \in \{\crit,\sup\}$, one also has
		\begin{equation}\label{eq:thm:inhselec:enddist}
			\sup_{y\in[0,1]}\left|\,\frac{\log_+\cX_T^{N(\cdot,T)}([\shft(\mu_T) + \hat m_T^\anyrg - y\gs(1)\ell(1)\hat L(T)\,,\,+\infty)) }{\ell(1) \hat L(T)}\,-\,y\,\right|\;\longrightarrow\;0\,,
		\end{equation}
		in $\bbP_{\mu_T}$-probability, and
		\begin{equation}\label{eq:thm:inhselec:enddist:2}
			\max(\cX_T^{N(\cdot,T)}) - \min(\cX_T^{N(\cdot,T)}) = \sigma(1)\ell(1)\hat L(T) + o_{\bbP}(L(T)).
		\end{equation}
	\end{theorem}
	\begin{remark}
		The result in the regime sup-d matches Proposition~\ref{prop:main:gsdecreasing} for homogeneous selection: in particular, it does not depend on $\ell(\cdot)$ or the precise asymptotics of $\hat L(T)$.
	\end{remark}
	\begin{remark}
		If $\sigma$ is constant, i.e.~in the case of time-homogeneous diffusion, the critical regime of Theorem~\ref{thm:inhselec} is an analogue for branching Brownian motion of Theorem~4.1 of Mallein~\cite{Mal17}, who considers quite general branching random walks with time-inhomogeneous selection (but time-homogeneous displacement).
	\end{remark}

We now proceed to the proof of Theorem~\ref{thm:inhselec}. The idea is to approximate the $N(\cdot,T)$-BBM by a process with constant selection on a short time interval, and applying Theorem~\ref{thm:main} and Proposition~\ref{prop:endtimedistribution} to the latter.
We first consider the super-critical and critical regimes ($\hat L(T)\gg T^{1/3}$ or $\hat L(T)=T^{1/3}$): we prove~\eqref{eq:thm:inhselec:max} and~\eqref{eq:thm:inhselec:enddist} simultaneously, then~\eqref{eq:thm:inhselec:max:decreas} and~\eqref{eq:thm:inhselec:enddist:2}. Finally, we prove~\eqref{eq:thm:inhselec:max} in the sub-critical regime ($\hat L(T)\ll T^{1/3}$).

Before that, we provide a coupling proposition that extends Corollary~\ref{cor:coupling:N} and Proposition~\ref{prop:coupling:BBM}.$(i)$ to a setting with time-inhomogeneous selection. Notice that, with the definitions above, for any $T>0$ there are finitely many $t\in[0,T]$ upon which $t\mapsto N(t,T)$ changes its integer part.

\begin{proposition}\label{prop:coupling:inhselec}
	Let $T\geq0$ fixed. Let $N_1(\cdot,T)$, $N_2(\cdot,T)$ two positive functions on $[0,T]$ such that $N_1(t,T)\leq N_2(t,T)$ for all $0\leq t\leq T$. Assume that their integer parts change finitely many times in $t\in[0,T]$. Let $\mu_1\in\meas_{N_1(0,T)}$ and $\mu_2\in\meas_{N_2(0,T)}$ which satisfy $\mu_1\prec\mu_2$: then there exists $(\cX^{N_1(\cdot,T)}_{t})_{t\in[0,T]}$ and $(\cX^{N_2(\cdot,T)}_{t})_{t\in[0,T]}$, respectively an $N_1(\cdot,T)$- and an $N_2(\cdot,T)$-BBM, such that $\cX^{N_1(\cdot,T)}_{0}=\mu_1$, $\cX^{N_2(\cdot,T)}_{0}=\mu_2$ and $\cX^{N_1(\cdot,T)}_{t}\prec\cX^{N_2(\cdot,T)}_{t}$ for all $t\in[0,T]$ with probability 1.
	
	Moreover, there also exists a coupling with a time-inhomogeneous BBM without selection started from $\mu_1$, such that $\cX^{N_1(\cdot,T)}_{t}\subset\cX_{t}$ for all $t\in[0,T]$ with probability 1.
\end{proposition}

\begin{proof}
This is very similar to the adaptation made in Section~\ref{sec:compl:detbranching} above: the only difference between this proposition and the setting of Corollary~\ref{cor:coupling:N} and Proposition~\ref{prop:coupling:BBM}.$(i)$ is that particles may be killed at some deterministic epochs, i.e.\ when the integer part of $N_1(\cdot,T)$ or $N_2(\cdot,T)$ diminishes. Since we assumed that there are finitely many such epochs, the constructions of the coupling and the stopping line can be adapted straightforwardly.
\end{proof}

We resume the proof of Theorem~\ref{thm:inhselec}, starting with~\eqref{eq:thm:inhselec:max} and~\eqref{eq:thm:inhselec:enddist} in the critical and super-critical regimes.

\begin{proof}[Proof of~\eqref{eq:thm:inhselec:max} and~\eqref{eq:thm:inhselec:enddist}, super-critical and critical regimes]
	First of all recall the definition of $\shft(\cdot)$ from \eqref{eq:defshift}, and let us extend it into
	\begin{equation}\label{eq:defshift:endtime} 
		\shftend{s}(\mu)\;:=\; \sup\big\{q\in\R\,\big|\, \exists \gk\in[0,1]: \mu\big([q-\gk\gs(s/T)L(T),+\infty)\big)\,\geq\, N(T)^{\gk}\big\}\,,\quad s\in[0,T].
	\end{equation}
	Recall Proposition~\ref{prop:endtimedistribution}: then we have the following. 
	\begin{lemma}\label{lem:enddist:inhselec}
		Let $\anyrg\in\{\crit, \supercrit\}$ and $\gh>0$. Let $\mu_T\in\measN$, and denote with $\cX^{N(T)}$ an $N(T)$-BBM (with time-homogeneous selection) started from $\mu_T$, $T\geq0$, with infinitesimal variance $\gs^2(\cdot/T)$, $\gs\in\gsensrg$. Then as $T\to+\infty$, one has
		\begin{equation}
			\sup_{\gs\in\gsensrg}\,\frac1{L(T)} \left|\shftend{T}(\cX^{N(T)}_T)-\shft(\mu_T)-m^\anyrg_T\right|\,\longrightarrow\, 0\,,\qquad\text{in }\bbP_{\mu_T}\text{-probability.}
		\end{equation}
	\end{lemma}
	\begin{proof}[Proof of Lemma~\ref{lem:enddist:inhselec}]
		This is a direct corollary of Proposition~\ref{prop:endtimedistribution} ---more precisely Lemma~\ref{lem:enddist}--- and the definition~\eqref{eq:defshift:endtime}. Indeed, the upper bound on $\shftend{T}(\cX^{N(T)}_T)$ follows from the fact that~\eqref{eq:lem:enddist:UB} provides an upper bound on $\cX^{N(T)}_T\big([m^\anyrg_T-y\gs(1)L(T),+\infty)\big)$ holding for \emph{all} $y\in[0,1]$ with large probability; the lower bound is obtained through a direct application of~\eqref{eq:lem:enddist:LB}. Finally, as stated in Lemma~\ref{lem:enddist}, the result is uniform in $\gs\in\gsensrg$. 
	\end{proof}
	
	With Proposition~\ref{prop:coupling:inhselec} and Lemma~\ref{lem:enddist:inhselec} at hand, Theorem~\ref{thm:inhselec} is obtained quite naturally in the critical and super-critical regimes, by writing a block decomposition of the process. Indeed, let us first consider the critical regime, i.e. $\hat L(T)=T^{1/3}$. Recall the definition of $\hat m^\crit_T$ from~\eqref{eq:defmanyrg:inhselec}. Let $\eps>0$ small (assume $\eps^{-1}\in\N$ for the sake of simplicity), and part $[0,T]$ into $\eps^{-1}$ intervals of length $\eps T$. Define for $1\leq i\leq \eps^{-1}$,
	\[
	\ul\ell_i\,:=\, \inf\{\ell(u),\,u\in[(i-1)\eps,i\eps]\}\,,\qquad\text{and}\quad \ol\ell_i\,:=\, \sup\{\ell(u),\,u\in[(i-1)\eps,i\eps]\}\,,
	\]
	and define similarly $\ul\gs_i$, $\ol\gs_i$, $\ul\gs_i'$ and $\ol\gs_i'$. Define also,
	\[
	\ul N_i(T)\,:=\, \exp(\ul\ell_i T^{1/3})\,,\qquad\text{and}\quad \ol N_i(T)\,:=\, \exp(\ol\ell_i T^{1/3})\,.
	\]
	Finally, we let for $1\leq i\leq \eps^{-1}$,
	\begin{align*}
		&\ul m_{i,\eps}\,:=\, T\int_{(i-1)\eps T}^{i\eps T} \gs(u)\dd u\,+\, \eps \,T^{1/3}\inf\left\{\frac{\gs}{\ell}\,\Psi\left(-\ell^3\frac{\gs'}{\gs}\right) \,\middle|\; \begin{aligned}
		&\gs\in[\ul\gs_i,\ol\gs_i], \gs'\in[\ul\gs_i',\ol\gs_i']\\
		&\text{and}\quad \ell\in[\ul\ell_i,\ol\ell_i]\end{aligned}\right\}\,,\\
		&\ol m_{i,\eps}\,:=\, T\int_{(i-1)\eps T}^{i\eps T} \gs(u)\dd u\,+\, \eps \,T^{1/3}\sup\left\{\frac{\gs}{\ell}\,\Psi\left(-\ell^3\frac{\gs'}{\gs}\right) \,\middle|\; \begin{aligned}
		&\gs\in[\ul\gs_i,\ol\gs_i], \gs'\in[\ul\gs_i',\ol\gs_i']\\
		&\text{and}\quad \ell\in[\ul\ell_i,\ol\ell_i]\end{aligned}\right\}.
	\end{align*}
	Using a Riemann sum approximation and that $\gs\in\gsensrg$, $\ell\in\cC^1([0,1])$, one observes that
	\begin{equation}\label{eq:prf:inhselec:riemann}
		\lim_{\eps\to0}\,T^{-1/3}\,\Bigg|\hat m_T^\crit - \sum_{i=1}^{\eps^{-1}} \ul m_{i,\eps}\Bigg|\;=\;
		\lim_{\eps\to0}\,T^{-1/3}\,\Bigg|\hat m_T^\crit - \sum_{i=1}^{\eps^{-1}} \ol m_{i,\eps}\Bigg|\;=\;0\,.
	\end{equation}
	Let $(\cF_s)_{s\in[0,T]}$ denote the natural filtration of the $N(\cdot,T)$-BBM. Following from Proposition~\ref{prop:coupling:inhselec}, there exists an $\ul N_i(T)$- and a $\ol N_i(T)$-BBM, both starting from the initial measure $\cX^{N(\cdot,T)}_{(i-1)\eps T}$ at time $(i-1)\eps T$, such that,
	\begin{equation}\label{eq:prf:inhselec:coupl}
		\bbP_{\mu_T}\left(\cX_{i\eps T}^{\ul N_i(T)}\,\prec\, \cX_{i\eps T}^{N(\cdot,T)} \,\prec\, \cX_{i\eps T}^{\ol N_i(T)}\,\middle|\, \cF_{(i-1)\eps T}\,;\,\cX^{\ul N_i(T)}_{(i-1)\eps T}=\cX^{N(\cdot,T)}_{(i-1)\eps T}=\cX^{\ol N_i(T)}_{(i-1)\eps T} \right)\;=\;1\,.
	\end{equation}
	Moreover, Lemma~\ref{lem:enddist:inhselec} states that, for $\gl>0$, there exists $T_0>0$ (depending only on $\gh,\eps,\gl$) such that for $T\geq T_0$,
	\begin{align}\label{eq:prf:inhselec:1}
			&\bbP_{\mu_T}\bigg( \frac1{T^{1/3}}\bigg(\shftend{i\eps T}\big(\cX^{\ul N_i(T)}_{i\eps T}\big) - \shftend{(i-1)\eps T}\big(\cX^{\ul N_{(i-1)}(T)}_{(i-1)\eps T}\big) - \ul m_{i,\eps} \bigg) <-\eps\gl \,\bigg|\, \cF_{(i-1)\eps T} \bigg)\leq \eps^2,\\
			\notag \text{and}\;\; &\bbP_{\mu_T}\bigg( \frac1{T^{1/3}}\bigg(\shftend{i\eps T}\big(\cX^{\ol N_i(T)}_{i\eps T}\big) - \shftend{(i-1)\eps T}\big(\cX^{\ol N_{(i-1)}(T)}_{(i-1)\eps T}\big) - \ol m_{i,\eps} \bigg) >\eps\gl \,\bigg|\, \cF_{(i-1)\eps T} \bigg)\leq \eps^2.
	\end{align}
	Recalling Theorem~\ref{thm:main}, one obtains similarly for $T\geq T_0$,
	\begin{align}\label{eq:prf:inhselec:2}
			&\bbP_{\mu_T}\bigg( \frac1{T^{1/3}}\bigg(\!\max\!\big(\cX^{\ul N_i(T)}_{i\eps T}\big) - \shftend{(i-1)\eps T}\big(\cX^{\ul N_{(i-1)}(T)}_{(i-1)\eps T}\big) - \ul m_{i,\eps} \bigg) <-\eps\gl \,\bigg|\, \cF_{(i-1)\eps T} \bigg)\leq \eps^2,\\
			\notag \text{and}\;\; &\bbP_{\mu_T}\bigg( \frac1{T^{1/3}}\bigg(\!\max\!\big(\cX^{\ol N_i(T)}_{i\eps T}\big) - \shftend{(i-1)\eps T}\big(\cX^{\ol N_{(i-1)}(T)}_{(i-1)\eps T}\big) - \ol m_{i,\eps} \bigg) >\eps\gl \,\bigg|\, \cF_{(i-1)\eps T} \bigg)\leq \eps^2.
	\end{align}
	Apply the Markov property at times $(i-1)\eps T$, $1\leq i\leq \eps^{-1}-1$, one deduces with a union bound and~\eqref{eq:prf:inhselec:coupl},
	\begin{align*}
		&\bbP_{\mu_T}\left( \frac1{T^{1/3}}\bigg(\max\big(\cX^{N(\cdot,T)}_{T}\big) - \shft(\mu_T) - \sum_{i=1}^{\eps^{-1}} \ul m_{i,\eps} \bigg) <-\gl \right)\\
		&\quad\leq \sum_{i=1}^{\eps^{-1}-1} \bbP_{\mu_T}\bigg( \frac1{T^{1/3}}\bigg(\shftend{i\eps T}\big(\cX^{\ul N_i(T)}_{i\eps T}\big) - \shftend{(i-1)\eps T}\big(\cX^{\ul N_{(i-1)}(T)}_{(i-1)\eps T}\big) - \ul m_{i,\eps} \bigg) <-\eps\gl \\
		& \qquad\qquad\qquad\qquad\qquad\qquad\qquad\qquad\qquad\qquad\qquad\qquad\qquad \bigg|\, \cX^{\ul N_i(T)}_{(i-1)\eps T}=\cX^{N(\cdot,T)}_{(i-1)\eps T} \bigg)\\
		&\qquad+ \bbP_{\mu_T}\bigg( \frac1{T^{1/3}}\bigg(\max\big(\cX^{\ul N_{\,\eps^{-1}}(T)}_{T}\big) - \shftend{(1-\eps)T}\big(\cX^{\ul N_{(\eps^{-1}-1)}(T)}_{(1-\eps)T}\big) - \ul m_{\eps^{-1},\eps} \bigg) <-\eps\gl \\
		& \qquad\qquad\qquad\qquad\qquad\qquad\qquad\qquad\qquad\qquad\qquad\qquad\qquad \bigg|\, \cX^{\ul N_{\eps^{-1}}(T)}_{(1-\eps)T}=\cX^{N(\cdot,T)}_{(1-\eps)T} \bigg),
	\end{align*}
	for $T\geq T_0(\gh,\gl,\eps)$; and by~(\ref{eq:prf:inhselec:1}--\ref{eq:prf:inhselec:2}), the latter sum is bounded by $\eps$. Recalling~\eqref{eq:prf:inhselec:riemann} and letting $\eps\to0$, this finally proves the lower bound in~\eqref{eq:thm:inhselec:max}; and the upper bound is obtained similarly. Moreover,~\eqref{eq:thm:inhselec:enddist} is obtained with an analogous computation, by applying Proposition~\ref{prop:endtimedistribution} instead of~\eqref{eq:prf:inhselec:2} to the last block ($i=\eps^{-1}$) of the decomposition above. This concludes the proof of Theorem~\ref{thm:inhselec} in the critical regime.
	\smallskip
	
	Regarding the super-critical case, it is proven by replicating the proof above \emph{mutatis mutandis}: that is, replacing $T^{1/3}$ with $\hat L(T)\gg T^{1/3}$, $\hat m_T^\crit$ with $\hat m_T^\supercrit$, and adapting the definitions of $\ul m_{i,\eps}$, $\ol m_{i,\eps}$ above accordingly. This is very straightforward, so we leave the details to the reader.
\end{proof}

\begin{proof}[Proof of~\eqref{eq:thm:inhselec:max:decreas}]
	This is immediate: recall that, in the super-critical regime with decreasing variance, the convergence~\eqref{eq:thm:inhselec:max:decreas} \emph{does not} depend on the specifics of $\ell(\cdot)$, $\hat L(T)$. Hence the super-critical $N(\cdot,T)$-BBM can still be coupled with, on the one hand a critical $M_\ga(T)$-BBM with population size $M_\ga(T):=\exp(\ga T^{1/3})\ll\inf\{N(s,T),0\leq s\leq T\}$, $\ga>0$; and on the other hand a BBM without selection (see Proposition~\ref{prop:coupling:inhselec}). Therefore, the proof of Proposition~\ref{prop:main:gsdecreasing} above fully accommodates to super-critical, time-inhomogeneous selection. We leave the details to the reader.
\end{proof}

\begin{proof}[Proof of~\eqref{eq:thm:inhselec:enddist:2}]
	This is also straightforward: recalling the proof of (\ref{eq:thm:inhselec:max},~\ref{eq:thm:inhselec:enddist}), it suffices to estimate $\min(\cX^{N(\cdot,T)}_T)$ in the last block of the decomposition. One can estimate both $\min(\cX^{\ol N_{\eps^{-1}}}_T)$ and $\min(\cX^{\ul N_{\eps^{-1}}}_T)$ by using~\eqref{eq:prop:endtimedist:2} and Theorem~\ref{thm:main}; and reproducing the coupling arguments from the proof of~\eqref{eq:prop:endtimedist:2} (we do not write the details again, but let us recall that they do not immediately follow from the definition of stochastic domination), this gives the expected estimate on $\min(\cX^{N(\cdot,T)}_T)$.
\end{proof}

\begin{proof}[Proof of~\eqref{eq:thm:inhselec:max}, sub-critical regime]
	In that regime there is no equivalent to Proposition~\ref{prop:endtimedistribution} or Lemma~\ref{lem:enddist:inhselec}, so one cannot split the interval $[0,T]$ into blocks of length $\eps T$; however, the proofs of Lemma~\ref{lem:main:LB} (for the lower bound) and Proposition~\ref{prop:main:UB} (upper bound), in the sub-critical regime, already rely on a block decomposition. Hence, Theorem~\ref{thm:inhselec} can be obtained by approximating the $N(\cdot,T)$-BBM by a process with piecewise-constant selection, and reproducing the arguments from the proof of Theorem~\ref{thm:main}. 
	
	More precisely, let us first discuss the lower bound. Let $K=\lfloor 2T/t^\subcrit\rfloor$ and for $0\leq k\leq K-1$, $t_k=\frac k2t^\subcrit_T$, and $t_{K}=T$. Let $N_k=\inf\{N(s,T), t_{k-1}\leq s\leq t_k\}$, and recall the construction of the auxiliary process $\ol \cX^{N}$ from Section~\ref{sec:LB:coreprf}: let us tweak it such that $(1)$ at times $t_k$, $1\leq k\leq K$, all particles but the $\lfloor N_k\rfloor$ top-most ones are removed, and the remaining particles are set to the lowest among their positions; and $(2)$ on the interval $[t_{k-1},t_k]$, the process $\ol \cX^{N}$ evolves like an $N_k$-BBM, where $N_k$ is constant. Therefore, this process is still stochastically dominated by the $N$-BBM, and the remainder of the proof from Section~\ref{sec:LB:coreprf} still holds for this process (since the selection is constant on each interval $[t_{k-1},t_k]$), both for the super-polynomial and small $N(T)$ cases. We deduce from this a lower bound on the maximal displacement of the $N(\cdot,T)$-BBM, and a Riemann sum approximation (analogous to~\eqref{eq:prf:inhselec:riemann}) finally shows that it is of order $\hat m^\subcrit_T - o(T/L(T)^2)$ for $T$ large. Since all these adaptations are very straightforward, and rely on arguments that were already applied in other parts of the paper, we leave the remaining details to the reader. Finally, the upper bound is obtained with analogous arguments.
\end{proof}

\begin{remark}\label{rem:inhselec}
	Let us point out that Theorem~\ref{thm:inhselec} assumes that the ``regime'' for the selection (i.e. $\hat L(T)$) remains the same throughout $[0,T]$. If the regime changes finitely many times, e.g. $\hat L(T)$ is replaced with some $\hat L_k(T)$ on an interval $[t_{k-1},t_k]$, $1\leq k\leq K$, one can derive similar results by induction: indeed,~\eqref{eq:thm:inhselec:enddist} and Lemma~\ref{lem:enddist:inhselec} provide an estimate on $\shftend{t_k}(\cX^{N(\cdot,T)}_{t_k})$ if the regime is critical or super-critical on $[t_{k-1},t_k]$; and in the sub-critical case $\hat L_k(T)\ll T^{1/3}$, one has,
	\[\shftend{t_k}(\cX^{N(\cdot,T)}_{t_k})=\shftend{t_k}(\gd_{\max \cX^{N(\cdot,T)}_{t_k}})+O(\hat L_k(T)).\]
	Then, one can apply Theorem~\ref{thm:inhselec} on each interval $[t_{k-1},t_k]$ inductively. We do not write any statement or proof for this fact, since this follows from a carbon copy of arguments presented above.
\end{remark}


\section*{Acknowledgments}
The authors would like to thank Marc Lelarge for mentioning the relation of our work with the beam-search algorithm, and G\'erard Ben Arous for fruitful discussions on the topic of spin glasses. We also thank four anonymous referees for their constructive comments that improved the
quality of this paper.

\section*{Funding}
A. Legrand acknowledges support from the ANR projects ``REMECO'', ANR-20-CE92-0010 and ``LOCAL'', ANR-22-CE40-0012. P. Maillard acknowledges support from  the ANR-DFG project ``REMECO'', ANR-20-CE92-0010 and Institut Universitaire de France (IUF). 

\bibliographystyle{abbrv}
\bibliography{biblio}

\end{document}